\documentclass[11pt]{amsart}
\usepackage{amssymb, amsfonts, latexsym, amsthm, amsmath, verbatim, enumerate}
\pdfoutput=1
\usepackage{soul} 
\usepackage{paralist}
\usepackage[all]{xy}
\usepackage{amscd}
\usepackage{mathrsfs} 
\usepackage[english]{babel}
\usepackage{tikz}
\usetikzlibrary{cd}
\usepackage{graphicx}
\usepackage{epstopdf}
\definecolor{darkblue}{rgb}{0,0,0.4} 
\usepackage[colorlinks=true, citecolor=darkblue, filecolor=darkblue, linkcolor=darkblue,urlcolor=darkblue]{hyperref}
\usetikzlibrary{arrows}
\usetikzlibrary{calc}
\usetikzlibrary{decorations.pathmorphing}
\tikzstyle{crossing}=[circle,fill=white,minimum height=6pt,inner sep=0pt, outer sep=0pt, style={transform shape=false}]
\usepackage{amscd}
\usepackage{enumitem}
\usepackage{todonotes}

\newcommand{\inv}{^{-1}} 
\newcommand{\map}[1]{\xrightarrow{#1}}
\newcommand{\resp}{resp.\ }
\newcommand{\tate}{\mathrm{Tate}} 

\makeatletter
\providecommand\@dotsep{5}
\def\listtodoname{List of Todos}
\def\listoftodos{\@starttoc{tdo}\listtodoname}
\makeatother

\tikzstyle{crossing}=[circle,fill=white,minimum height=6pt,inner sep=0pt, outer sep=0pt, style={transform shape=false}]

\numberwithin{equation}{section}
\numberwithin{figure}{section}

\theoremstyle{plain}
\newtheorem{thm}[equation]{Theorem}
\newtheorem{construction}[equation]{Construction}
\newtheorem*{thm*}{Theorem}
\newtheorem{prop}[equation]{Proposition}
\newtheorem{cor}[equation]{Corollary}
\newtheorem{lem}[equation]{Lemma}

\theoremstyle{definition}
\newtheorem{rmk}[equation]{Remark}
\newtheorem{examp}[equation]{Example}

\newtheorem{defn}[equation]{Definition}
\newtheorem{notation}[equation]{Notation}

\newcommand{\f}{\mathbb{F}}
\newcommand{\forgot}{\mathcal{F}}
\newcommand{\hoco}{\mathrm{hocolim}} 

\newcommand{\X}{\mathcal{X}}
\newcommand{\col}{\colon}
\newcommand{\Id}{\mathrm{Id}}
\newcommand{\topp}{\mathrm{Top}_*}
\newcommand{\Hom}{\mathrm{Hom}} 
\newcommand{\from}{\colon}
\newcommand{\into}{\hookrightarrow}

\newcommand{\Kh}{\mathit{Kh}} 
\newcommand{\kho}{\Kh_o} 
\newcommand{\KhGen}{\mathit{Kg}}
\newcommand{\KhCx}{\mathit{Kc}}
\newcommand{\oddKhCx}{\KhCx_o}

\newcommand{\khoh}{\X_e} 

\newcommand{\burn}{\mathscr{B}} 
\newcommand{\oddb}{\burn_{K}} 
\newcommand{\two}{\underline{2}} 
\newcommand{\Tot}{\mathrm{Tot}} 

\newcommand{\s}{\mathfrak{s}}
\newcommand{\refl}{\mathfrak{r}}
\newcommand{\dig}{\mathfrak{d}} 

\newcommand{\ZZ}{\mathbb{Z}}
\newcommand{\Ob}{\mathrm{Ob}}
\renewcommand{\th}{^{\text{th}}}
\newcommand{\Map}{\mathrm{Map}}

\newcommand{\Cat}{\mathscr{C}}
\newcommand{\Dat}{\mathscr{D}}
\newcommand{\cellC}{C_{\mathrm{cell}}}
\newcommand{\eqcellC}{C_{\mathrm{cell}}^{\ZZ_p,*}} 
\newcommand{\redcellC}{\widetilde{C}_{\mathrm{cell}}}

\newcommand{\AbFunc}{\mathfrak{F}}
\newcommand{\Abelianize}{\mathcal{A}}

\newcommand{\RR}{\mathbb{R}}
\newcommand{\cell}{\mathcal{C}} 

\newcommand{\CRealize}[1]{\|#1\|}
\newcommand{\Realize}[1]{|#1|}

\newcommand{\op}{\mathrm{op}}
\newcommand{\Mod}{\text{-Mod}}
\renewcommand{\emptyset}{\varnothing}
\newcommand{\basis}[1]{\langle #1\rangle}

\newcommand{\ff}{F} 

\newcommand{\hodit}{\tilde{F}^+} 
\newcommand{\Akh}{AKh} 
\newcommand{\Akhspace}{\mathcal{AKH}} 
\newcommand{\Akc}{\mathit{AKc}} 
\newcommand{\oddAkc}{\mathit{AKc}_o} 
\newcommand{\oddAkh}{AKh_o} 
\newcommand{\Arr}{\mathrm{Hom}} 
\newcommand{\degh}{\mathfrak{d}} 

\newcommand{\khfunc}{\AbFunc_e} 
\newcommand{\khofunc}{\AbFunc_o} 
\newcommand{\khburn}{\mathcal{KH}} 
\newcommand{\khoburn}{\mathcal{KHO}} 
\newcommand{\akhoburn}{\mathcal{AKHO}} 
\newcommand{\akhburn}{\mathcal{AKH}} 

\newcommand{\detg}{\mathrm{Det}}
\newcommand{\pt}{\mathrm{pt}}

\newcommand{\khon}{\X_n} 

\newcommand{\gr}{\mathrm{gr}} 
\newcommand{\gradedmod}{\ZZ\text{-gMod}} 
\newcommand{\Sym}{\mathrm{Sym}} 
\newcommand{\afunc}{\mathfrak{F}_{\mathrm{Ann}}} 
\newcommand{\oddafunc}{\mathfrak{F}_{\mathrm{Ann}_o}} 
\newcommand{\indset}{\mathbb{I}} 
\newcommand{\verti}{\mathrm{Vert}} 

\title{Localization in Khovanov homology}

\author{Matthew Stoffregen}
\address {Department of Mathematics, Michigan State University, East Lansing, MI 48824}
\email{stoffre1@msu.edu}

\author[M. Zhang]{Melissa Zhang}
\address{Department of Mathematics, UC Davis, One Shields Ave., Davis, CA 95616-8633, U.S.A.}
\email{mlzhang@ucdavis.edu}

\begin{document}

\begin{abstract}
We construct equivariant Khovanov spectra for periodic links using the Burnside functor construction introduced by Lawson, Lipshitz, and Sarkar. By identifying the fixed-point sets, we obtain rank inequalities for odd and even Khovanov homologies, and their annular filtrations, for prime-periodic links in $S^3$.  
\end{abstract} 
\maketitle

\tableofcontents

\section{Introduction}

\subsection{Motivation}

In \cite{kho1} Khovanov categorified the Jones polynomial: to a link
diagram $L$, he associated a bigraded chain complex, whose graded
Euler characteristic is (a certain normalization of) the Jones
polynomial of $L$, and whose (graded) chain homotopy type is an
invariant of the underlying link.  Several generalizations were soon
constructed; for example, \cite{khovtanglefunctor}, \cite{natantangle} developed theories for tangles.  Ozsv{\'a}th-Rasmussen-Szab{\'o} \cite{ors} constructed a version, \emph{odd Khovanov homology}, also categorifying the Jones polynomial, and agreeing with Khovanov homology over the field of two elements.  A further generalization, \emph{annular Khovanov homology}, an invariant of links in the thickened annulus, was introduced by Asaeda-Przytycki-Sikora \cite{aps}; this was further generalized to \emph{odd annular Khovanov homology} by Grigsby-Wehrli in \cite{grigsby-wehrli-gl11}.  Other generalizations, for other polynomials, were given by \cite{kr1,kr2} and others, and have since been extensively developed.

 The purpose of the present paper is to investigate the structure of Khovanov homology in the presence of symmetry; that is, we study the Khovanov homology of \emph{periodic links}.  We say that a link $\tilde{L}\subset S^3$ is $p$-\emph{periodic} if there is a $\ZZ_p=\ZZ/p\ZZ$-action on $(S^3,\tilde{L})$ which preserves $\tilde{L}$ and whose fixed-point set is an unknot $\tilde{U}$ disjoint from $\tilde{L}$.  A particular application of the techniques of this paper is the following:

\begin{thm}\label{thm:main}
Let $\tilde{L}$ be a $p^n$-periodic link, for a prime $p$, with quotient link $L$.  Let $\Kh(\tilde{L};\mathbb{F}_p)$ (\resp  $\kho(\tilde{L};\mathbb{F}_p)$) denote the Khovanov homology (\resp  odd Khovanov homology) of $\tilde{L}$, with coefficients in $\mathbb{F}_p$, the field of $p$ elements.  Let $\Akh(L;\mathbb{F}_p)$ (\resp  $\oddAkh(L;\f_p)$) denote the (\resp  odd) annular Khovanov homology of $L$, viewed in the complement of $U=\tilde{U}/\ZZ_p$.  Then, 
\[
\dim \Kh(\tilde{L}; \f_p) \geq \dim \Akh(L;\f_p) 
 \quad \mbox{  and } \quad
 \dim \kho(\tilde{L}; \f_p) \geq \dim \oddAkh(L;\f_p).
\]
\end{thm}

The motivation for this study comes from the both the application of classical Smith theory to Floer theories, and the general perspective of studying Floer and Khovanov invariants via the (often only conjectural) spectra underlying these theories. 

Let $G$ be a group of order $p^n$ with $p$ prime, acting on a finite-dimensional topological space $M$, with fixed-point set $M^G$. A version of the classical Smith inequality states \cite{smith-classical, bredon}: 
\begin{equation}\label{eq:1}
\dim  H^*(M;\f_p) \geq \dim H^*(M^G;\f_p).
\end{equation}
In low-dimensional topology and symplectic geometry, many results have been developed in analogy with the Smith inequality, relating the Floer homology of some object with symmetries with the Floer homology of its `quotient,' when the latter notion makes sense.  In particular, Seidel-Smith \cite{seidel-smith-loc} proved an analogue of the Smith inequality for $p=2$ in Lagrangian Floer theory. In fact, one of the motivations for \cite{seidel-smith-loc} was its application to symplectic Khovanov homology: Seidel-Smith prove a localization result for the symplectic Khovanov homology \cite{seidel-smith-symp} of $2$-periodic links.  Seidel-Smith further remark in \cite{seidel-smith-loc} that the combinatorial analogue to their symplectic Khovanov rank inequality was not known to hold at the time;  Corollary \ref{cor:cascade} (a consequence of Theorem \ref{thm:main}) of the present paper asserts that this analogue does indeed hold.  (Note that Khovanov homology and symplectic Khovanov homology are known to agree in characteristic 0 by work of Abouzaid-Smith \cite{abouzaid-smith}, but Smith-type inequalities from $\ZZ_p$-localization only hold in finite characteristic.) 

The Seidel-Smith inequality led to many further developments in low-dimensional topology. For instance, Hendricks \cite{hendricks-bdc} showed that the knot Floer homology of a knot $K\subset S^3$ has rank at most as large as that of the knot Floer homology of the preimage $\hat{K}$ in the branched double cover $\Sigma(K)$, and also obtained relationships between knot Floer homology of $2$-periodic knots and that of their quotients \cite{hendricks-periodic} (see also \cite{hls}, \cite{keegan-hfk}, and \cite{large}).

From the perspective of the present paper, the Seidel-Smith inequality reflects the extent to which Floer theories contain more information than just the resulting chain complex (indeed, the Smith inequality is a fact about \emph{spaces}, not about chain complexes).  A particularly striking formulation of this principle is found in Lidman-Manolescu \cite{Lidman-Manolescu-covering}, where they showed that, roughly, for a $p^n$-sheeted regular cover $\pi \from \tilde{Y}\to Y$ there is an action of a group $G$ of order $p^n$ on the Seiberg-Witten Floer space $\mathit{SWF}(\tilde{Y};\pi^*\mathfrak{s})$, so that the fixed-point set is $\mathit{SWF}(Y,\mathfrak{s})$, the Seiberg-Witten Floer space of the quotient.  They thus obtain a rank inequality, by applying the classical Smith inequality: 
\[
\sum_i \dim \tilde{H}_i(\mathit{SWF}(\tilde{Y},\pi^*\mathfrak{s});\f_p)\geq \sum_i \dim \tilde{H}_i(\mathit{SWF}(Y,\mathfrak{s});\f_p).
\]
Recall that Lidman-Manolescu \cite{Lidman-Manolescu-ident} identified the reduced homology of $\mathit{SWF}(Y,\mathfrak{s})$ with the tilde flavor of monopole Floer homology $\widetilde{HM}_*(Y,\mathfrak{s})$.  Further, Colin-Ghiggini-Honda and Kutluhan-Lee-Taubes \cite{cgh1},\cite{klt1} proved $\widetilde{HM}_*(Y,\mathfrak{s})=\widehat{HF}_*(Y,\mathfrak{s})$.  
Then the result of \cite{Lidman-Manolescu-covering} gives an inequality of ranks of Heegaard Floer homology, and in particular, strong constraints on $L$-spaces arising as regular covers.  

\subsection{Results}

In the present paper, we relate Khovanov space-level invariants of a periodic link $\tilde{L}$ with those of the quotient link $L$. This space-level relationship leads to a relationship on the level of homology that does not seem to follow in a simple way from the chain complex description of Khovanov homology.  A priori, it is difficult to relate any given Khovanov chain complex of a periodic link with any given Khovanov chain complex of the quotient, since without further information these are just chain complexes without further structure. 
However, the second author showed in \cite{z-annular-rank}, without using space-level invariants, that there is a spectral sequence relating the annular Khovanov homology of a $2$-periodic link with that of its quotient. 
This took advantage of a bonus grading in annular Khovanov homology, which is a richer invariant than Khovanov homology itself \cite{GLW-schur-weyl}; the extra structure was essential to that result.  

To set up notation, recall that for a link $L\subset S^3$, Lipshitz-Sarkar \cite{lshomotopytype} constructed a CW spectrum $\khoh(L)$ whose stable homotopy type is an invariant of the underlying link $L$, and whose reduced cellular chain complex is precisely the Khovanov chain complex $\KhCx(L)$. 
Their construction readily generalizes to produce an annular Khovanov spectrum of a link $L$ in the thickened annulus.
Further, in \cite{oddkh}, a family $\khon(L)$ of CW spectra was constructed for $n\in\ZZ_{\geq 0}$, so that $\X_0(L)=\khoh(L)$ and so that the reduced cellular chain complex $\tilde{C}_{\mathrm{cell}}(\khon(L))$ is the even Khovanov chain complex $\KhCx(L)$ for $n$ even, and is the odd Khovanov chain complex $\oddKhCx(L)$ for $n$ odd.  It is again straightforward to construct an annular Khovanov spectrum $\Akhspace_n(L)$ for any $n\in \ZZ_{\geq 0}$, whose reduced cellular chain complex $\tilde{C}_{\mathrm{cell}}(\Akhspace_n(L))$ is the even annular Khovanov chain complex $\Akc(L)$ if $n$ is even, and the odd annular Khovanov chain complex $\oddAkc(L)$ if $n$ is odd.  The Khovanov spaces and spectra split as a wedge sum according to quantum gradings, and in the case of annular, $(k)$-gradings as well, as $\khon(L)=\vee_j \khon^j(L)$ and $\Akhspace_n(L)=\vee_{j,k}\Akhspace_n^{j,k}(L)$, respectively. 
Furthermore, for $n\geq 1$,  $\khon(L)$ (\resp $\Akhspace_n(L)$) is a $\ZZ_2$-equivariant spectrum with geometric fixed points $\Sigma^{-1}\X_{n-1}(L)$ (\resp $\Sigma^{-1}\Akhspace_{n-1}(L)$).
See also \cite{akw-quantum}, where Akhmechet-Krushkal-Willis construct a stable homotopy refinement of Beliakova-Putyra-Wehrli's quantum annular Khovanov homology \cite{bpw-16}, and \cite{akw-towards}.

The main result of the present paper is the following:
\begin{thm}\label{thm:submain}
Fix $p>1$ and let $\tilde{L}$ be a $p$-periodic link with quotient link $L$.  For each quantum grading $j$, there is a well-defined structure of a $\ZZ_p$-equivariant spectrum on $\X_0^j(\tilde{L})$ ($\Akhspace^{j,k}_0(\tilde{L})$), whose $\ZZ_p$-equivariant stable homotopy type is an invariant of the $p$-periodic link $\tilde{L}$ (that is, the equivariant stable homotopy type is preserved by equivariant isotopies and equivariant Reidemeister moves of a diagram $\tilde{D}$ of $\tilde{L}$).  Further, the geometric fixed points are given by:
\begin{align}\label{eq:fixed-point-main}
\X^j_0(\tilde{L})^{\ZZ_p}&=\bigvee_{\{a,b\mid pa-(p-1)b=j\}}\Akhspace^{a,b}_0(L),\\
\Akhspace^{pj-(p-1)k,k}_0(\tilde{L})^{\ZZ_p}&=\Akhspace^{j,k}_0(L). \nonumber
\end{align}
  Moreover, if $n\geq 1$ and $p$ is odd, $\khon^j(\tilde{L})$ ($\Akhspace^j_n(\tilde{L})$) is naturally a $\ZZ_2\times \ZZ_p$-equivariant spectrum, whose $\ZZ_2\times \ZZ_p$-equivariant stable homotopy type is an invariant of the $p$-periodic link $\tilde{L}$.  Then, as $\ZZ_2$-equivariant spectra,
\begin{align}\label{eq:fixed-point-main-2}
\khon^j(\tilde{L})^{\ZZ_p}&=\bigvee_{\{a,b\mid pa-(p-1)b=j\}} \Akhspace^{a,b}_n(L),\\
\Akhspace^{pj-(p-1)k,k}_n(\tilde{L})^{\ZZ_p}&=\Akhspace^{j,k}_n(L). \nonumber
\end{align}
In fact, if $p$ is odd and $V$ is any finite-dimensional orthogonal $\ZZ_2\times \ZZ_p$-representation, there is a $\ZZ_2\times \ZZ_p$-equivariant spectrum $\mathcal{X}^j_{V}(\widetilde{L})$ ($\mathcal{AKH}_V^j(\widetilde{L})$), whose $\ZZ_2\times \ZZ_p$-equivariant stable homotopy type is an invariant of the $p$-periodic link $\widetilde{L}$.  Let $V^{\ZZ_p}$ be the $\ZZ_p$-fixed subspace of $V$.  Moreover, $\mathcal{X}^j_{V}(\widetilde{L})$ ($\mathcal{AKH}_V^j(\widetilde{L})$) is $\ZZ_2$-stable homotopy equivalent to $\Sigma^{V-V^{\ZZ_p}}\mathcal{X}^j_{\dim V^{\ZZ_p}}(\widetilde{L})$  ($\Sigma^{V-V^{\ZZ_p}}\mathcal{AKH}_{\dim V^{\ZZ_p}}^j(\widetilde{L})$).   Then, as $\ZZ_2$-equivariant spectra,
\begin{align}\label{eq:fixed-point-main-3}
\mathcal{X}_V^j(\tilde{L})^{\ZZ_p}&=\bigvee_{\{a,b\mid pa-(p-1)b=j\}} \mathcal{AKH}^{a,b}_{\dim V^{\ZZ_p}}(L),\\
\Akhspace^{pj-(p-1)k,k}_V(\tilde{L})^{\ZZ_p}&=\Akhspace^{j,k}_{\dim V^{\ZZ_p}}(L). \nonumber
\end{align}
\end{thm}

\emph{Proof of Theorem \ref{thm:main}:}  We begin by noting that, in Theorem \ref{thm:submain}, all the involved objects are suspension spectra of compact spaces, and the statements in Theorem \ref{thm:submain} continue to hold at the level of the underlying topological spaces
(see Lemma \ref{lem:fix-point-coincidence}  and Theorem \ref{thm:burnside-fixed-pts}).  
Then, $\khoh(\tilde{L})$ (here, a compact topological space) admits a $\ZZ_{p^n}$-action 
with fixed-point set $\Akhspace_0(L)$.  The homology satisfies $\tilde{H}(\khoh(\tilde{L}))=\Kh(\tilde{L})$, while $\tilde{H}(\Akhspace_0(L))=\Akh(L)$.  Applying (\ref{eq:1}) to $M=\khoh(\tilde{L})$, Theorem \ref{thm:main} follows for the even case.  The odd case is similar. \qed 

Further, we expect that the Tate spectral sequence arising from the proof of Theorem \ref{thm:submain} should be compatible with spectral sequences from Khovanov to Floer theories, perhaps being related to Hendricks's \cite{hendricks-periodic}, Roberts's \cite{roberts-dbc}, or Xie's \cite{xie-ss} spectral sequences.

We mention a few further possible connections of Theorem \ref{thm:submain} to other work.  First, recall from \cite{bpw-16} that annular Khovanov homology of a link $L$ can be realized as the Hochschild homology of an appropriate bimodule over the platform algebra (see \cite{chenkhovanov},\cite{stroppel-5}).  Recall moreover that Lawson-Lipshitz-Sarkar \cite{lls-tangle} have given a spectrum-level version of Khovanov's invariant for tangles \cite{khovtanglefunctor}.  From these developments, it seems natural to conjecture that the annular Khovanov spectrum of a link is realized as the topological Hochschild homology of an appropriate spectral bimodule; this conjecture has been proved in \cite{LLS-annular} after the appearance of the present paper. Given the result of \cite{LLS-annular}, it is interesting to ask whether and how the actions constructed in this paper pass over to give actions on the topological Hochschild homology.  See also \cite{lipshitz-treumann}.

Independently,
Borodzik-Politarczyk-Silvero \cite{bps-2018} use equivariant flow categories to also show that $\X_0(\tilde{L})=\X_e(\tilde L)$ admits a $\ZZ_p$-action; the main Theorem 1.2 of \cite{bps-2018} is the first sentence of Theorem \ref{thm:submain} in the present paper, although it is not clear that the action constructed in \cite{bps-2018} and that constructed here (in the case $n=0$) agree.  In \cite{bps-2018}, they further relate the Borel equivariant cohomology of $\X_e(\tilde{L})$ to Politarczyk's equivariant Khovanov homology \cite{politarczyk-kh}.
Jeff Musyt has also constructed a $\ZZ_p$-equivariant Khovanov stable homotopy type using methods similar to ours \cite{musyt-thesis}.

One final potentially surprising point is that, in the odd case, there are several $\ZZ_2\times \ZZ_p$-equivariant Khovanov spectra underlying any of the $\ZZ_2$-equivariant spectra $\mathcal{X}^j_{n}(\widetilde{L})$, with potentially different $\ZZ_2\times \ZZ_p$-equivariant stable homotopy types. Indeed, as in Theorem \ref{thm:submain}, any of the $\mathcal{X}^j_V(\widetilde{L})$, for a $\ZZ_2\times \ZZ_p$-representation $V$ with $\dim V^{\ZZ_p}=n$, is $\ZZ_2$-equivariantly stable homotopy equivalent to $\Sigma^{V-V^{\ZZ_p}}\mathcal{X}^j_n(\widetilde{L})$.  It is not known to us if the $\ZZ_2\times \ZZ_p$-equivariant stable homotopy type of $\mathcal{X}_V^j(\tilde{L})$ is independent of $V$.

\subsection{Techniques}

We summarize the machinery and organization of this paper.  This paper uses the machinery of Burnside functors (roughly speaking, these are functors to the Burnside category, defined below),  introduced in \cite{hkk} and \cite{lls1}, to study the Khovanov spectrum.  The machinery of Burnside functors first appeared in \cite{lls1} to handle the product formula for Khovanov spectra, by giving a construction of the Khovanov spectrum as a certain homotopy colimit, which is more convenient for many applications.  We will use a slight generalization of Burnside functors from \cite{oddkh}, `decorated' Burnside functors, introduced to generalize the construction of \cite{lls1} to produce an odd Khovanov space.  We first review the construction of \cite{lls1}, in order to explain what is done in the present paper.

In \cite{lls1}, the dual of the Khovanov chain complex of a link diagram with $n$ ordered crossings is viewed as a diagram of abelian groups:
\[
\AbFunc_e\from (\two^n)^\op \to \ZZ\Mod,
\]  
and similarly in \cite{oddkh}, the odd Khovanov chain complex is viewed as a diagram:
\[
\AbFunc_o\from (\two^n)^\op \to \ZZ\Mod .
\]

Let us recall, for $K$ a finite group, the $K$-decorated \emph{Burnside category} $\burn_K$ (written $\burn$, if $K=\{1\}$), whose objects are finite sets, whose $1$-morphisms are finite correspondences decorated by elements of $K$, and whose $2$-morphisms are bijections respecting decorations.  The $2$-category $\burn$ naturally comes with a forgetful functor to abelian groups $\burn \to \ZZ\Mod$ by sending a set $S$ to the free abelian group $\ZZ\basis{S}$ generated by $S$.
The Khovanov stable homotopy type arises from a lift, according to \cite{lls1}:
\[
\begin{tikzpicture}[baseline={([yshift=-.8ex]current  bounding  box.center)},xscale=2.5,yscale=1.5]
\node (a0) at (0,0) {$\two^n$};
\node (a1) at (1,0) {$\ZZ\Mod$};
\node (b1) at (1,1) {$\burn$};

\draw[->] (a0) -- (a1) node[pos=0.5,anchor=north] {\scriptsize
  $\AbFunc^\op_e$}; \draw[->] (b1) -- (a1) node[pos=0.2,anchor=east] {};
\draw[->,dashed] (a0) -- (b1) node[pos=0.5,anchor=south east]{\scriptsize $\khburn$};

\end{tikzpicture}
\]

On the other hand, given a homomorphism $\epsilon \from K \to \ZZ_2$, there is a forgetful functor $\burn_K \to \ZZ\Mod$, again by sending a set $S$ to the free abelian group $\ZZ\basis{S}$ generated by $S$, and with $\ZZ\Mod$-morphisms twisted by $\epsilon$.  The odd Khovanov stable homotopy type arises from a lift:

\[
\begin{tikzpicture}[baseline={([yshift=-.8ex]current  bounding  box.center)},xscale=2.5,yscale=1.5]
\node (a0) at (0,0.5) {$\two^n$};
\node (a1) at (1,0) {$\ZZ\Mod$};
\node (b1) at (1,1) {$\burn_{\ZZ_2}$};
\node (a2) at (2,0) {$\ZZ\Mod$};

\draw[->] (a0) -- (a1) node[pos=0.5,anchor=north] {\scriptsize
  $\AbFunc^\op_e$}; 
\draw[->] (b1) -- (a1) node[pos=0.2,anchor=east] {\scriptsize$\epsilon=0$};
\draw[->] (b1) -- (a2) node[pos=0.2,anchor=west] {\scriptsize$\epsilon=\Id$};
\draw[->,dashed] (a0) -- (b1) node[pos=0.5,anchor=south east]{\scriptsize $\khoburn$};

\draw[->] (a0) -- (a2) node[pos=0.7,anchor=south west] {\scriptsize
  $\AbFunc^\op_o$}; 
\end{tikzpicture}
\]
The even Burnside functor $\khburn$ is obtained by forgetting the $\ZZ_2$-decorations on $\khoburn$.  

Given a Burnside functor $F$, \cite{lls1} gives a recipe, called \emph{realization} (see Section \ref{sec:disk}), for how to construct a  space, $||F_e(L)||$, as a homotopy colimit of a certain homotopy coherent diagram constructed from $F$.  This is generalized in \cite{oddkh}, for the case of nontrivial $K$, and allows for the case the construction of an odd Khovanov space $||F_o(L)||$ in a similar way.

The goal of the present paper is to investigate extra structure on the realizations $||F||$, for $F=\khburn(\tilde{L})$ or $\khoburn(\tilde{L})$ for $\tilde{L}$ $p$-periodic.  A natural expectation is that $||F||$ should admit a $\ZZ_p$-action.  The first technical work of the present paper consists of developing the correct notion of `actions' on Burnside functors $F \from \Cat\to \burn_K$, for $\Cat$ a small category, and on homotopy coherent diagrams $\Cat \to \topp$, where $\topp$ is the category of pointed topological spaces.  

First, we briefly explain the notion of `action' on Burnside functors.  A first guess is that a Burnside functor $F$ with action should be a diagram $BG \times \Cat \to \burn_K$, where $BG$ is the category with one object, and morphisms $G$; in analogy with viewing a pointed $G$-space as a diagram $BG \to \topp$.  The main technical difficulty is that, for the Khovanov-Burnside functor, $G=\ZZ_p$ acts on the category $\Cat$ itself.  In Section \ref{sec:burn}, we define a notion of \emph{external action} of a group $G$ on a Burnside functor $F$ as a kind of twist of the above definition. 

We must next see how the realization process of \cite{lls1} behaves on a Burnside functor $F$ with action.  As before, the problem is that we obtain a homotopy coherent diagram where the index category itself admits a $G$-action (we call such a diagram a \emph{diagram with external action} by $G$).  Note that a homotopy coherent diagram with a $G$-action (so that $G$ acts trivially on the index category) is simply a homotopy coherent diagram in the category of $G$-spaces, which would be readily handled along the lines of \cite{oddkh}.  

In Section \ref{sec:external}, we develop some machinery for homotopy colimits for homotopy coherent diagrams with an external action.   We do not pursue the greatest level of generality here; indeed, a more satisfactory treatment would be to essentially generalize the bulk of \cite{Vogt} to this situation; see also work of Dotto-Moi \cite{Dotto-Moi}.   The main results are Proposition \ref{prop:preconditions} and Lemma 5.6, while the main application to realizations of Burnside functors is Proposition \ref{prop:stable-equivalences-realization}.  In fact, including Proposition \ref{prop:stable-equivalences-realization} increases substantially the preliminaries we need, but is not needed in order to show that the Khovanov spaces of $p$-periodic links admit a $\ZZ_p$-action.  Instead, Proposition \ref{prop:stable-equivalences-realization} is only needed to show that the resulting $\ZZ_p$-action is well-defined.  In Section \ref{sec:app}, we show that $\khburn$ and $\khoburn$ have external actions under suitable circumstances, and find the fixed-point functors.  This involves a reasonably detailed study of the relationship of resolution configurations in a periodic link with those in its quotient.  It is somewhat interesting that the case of odd Khovanov homology here is substantially more involved than the even case.

We conclude the introduction with a few remarks.  First, in sections dealing with homotopy coherent diagrams, we work with diagrams in $K$-spaces for a group $K$, although for all of our applications $K$ will always be $\ZZ_2$ or trivial.  We include the more general case because it is no more complicated, and also on account of a conjecture of \cite{oddkh}.  

To explain this conjecture, recall that there are an infinite family of Khovanov spaces $\X_n(L)$ of a link $L$ for $n\in \ZZ_{\geq 0}$, where the $n$-th space has cellular chain complex equal to the even (\resp odd) Khovanov chain complex if $n$ is even (\resp odd).  The conjecture of \cite{oddkh} is that there should be stable homotopy equivalences 
\begin{equation}\label{eq:conjecture}
\X_n(L) \simeq \X_{n+2}(L).
\end{equation}  An attractive method of proving this conjecture would be the construction of a further Burnside functor $\khburn_\ZZ\from (\two^n)^\op\to \burn_\ZZ$ recovering $\khoburn(L)$ by taking $\ZZ \to \ZZ_2$.  If such a functor could be constructed, the techniques of the present paper would apply immediately to its realizations.  Note that even if (\ref{eq:conjecture}) holds, Theorem \ref{thm:submain} is not entirely boring for $n\geq 2$.  Indeed, the statement (\ref{eq:conjecture}) requires a choice of homotopy equivalence, and we expect that the natural family of homotopies realizing this equivalence (constructed from the putative $\khburn_\ZZ$) is not contractible.  That is, there may be no preferred homotopy equivalence $\X_n\to \X_{n+2}$.
 
\subsection*{Acknowledgements} We are grateful to John Baldwin, who suggested the problem to the second author, Eli Grigsby, Robert Lipshitz, Sucharit Sarkar, and David Treumann for much helpful input. The second author would like to thank Patrick Orson for organizing a learning seminar at Boston College which got her interested in space-level tools.  Part of Section \ref{sec:burn} was developed in the course of work on \cite{oddkh} in the hope of proving (\ref{eq:conjecture}); we thank Sucharit Sarkar and Chris Scaduto for allowing us to include it here.
We also thank the referee for many helpful insights and suggestions for making this paper more readable.

MS was supported by NSF Grant DMS-1702532.
MZ was partially supported by NSF Grant DMS-1510444.

\section{Khovanov homologies and periodic links}\label{sec:top-prelims}

In this section, we briefly review the definition and basic properties of several Khovanov homology theories.  For an oriented link $L\subset S^3$, we review the \emph{even} Khovanov homology $\Kh(L)=\Kh_e(L)$, defined by Khovanov \cite{kho1}, and the {\emph{odd}} Khovanov homology $\kho(L)$ defined by Ozsv\'{a}th-Rasmussen-Szab\'{o} \cite{ors}.  For an oriented link $L$ in the thickened annulus $(S^1\times [0,1])\times [0,1]$, we review the annular Khovanov homology $\Akh(L)$ defined by Asaeda-Przytycki-Sikora in \cite{aps}, as well as the odd annular Khovanov homology $\oddAkh(L)$, defined in \cite{grigsby-wehrli-gl11} by Grigsby-Wehrli.  For a more detailed introduction to Khovanov homology, see \cite{kho1}.  Our exposition follows \cite{lls1} closely.

\subsection{The cube category}\label{subsec:cube-cat}
We first recall the cube category.  Call $\two=\{0,1\}$ the
one-dimensional cube, viewed as a partially ordered set by setting
$1>0$, or as a category with a single non-identity morphism from $1$
to $0$.

Call $\two^n=\{0,1\}^n$ the $n$-dimensional cube, with the partial
order given by
\[ u=(u_1,\dots,u_n) \geq v=(v_1,\dots,v_n) \text{ if and only if }
\forall \; i \; (u_i \geq v_i).
\] 
It has the categorical structure induced by the partial order, where
$\Hom_{\two^n}(u,v)$ has a single element if $u \geq v$ and is empty
otherwise.  Write $\phi_{u,v}$ for the unique morphism $u \to v$ if it
exists.  The cube carries a grading given by $|v|=\sum_i v_i$.  Write
$u\geqslant_k v$ if $u\geq v$ and $|u|-|v|=k$. When $u\geqslant_1 v$,
we call the corresponding morphism $\phi_{u,v}$ an {\emph{edge}}, and call $v$ an \emph{immediate successor} of $u$. 

We will study chain complexes refining the cube category whose homological gradings correspond to the gradings of the vertices. When we work with homotopy colimits, it is most useful for us to work with commutative cubes, i.e.\ cubes where the 2-dimensional faces commute. However, in order for $\partial^2=0$ to hold in the chain complex, we must assign signs to the edges of the cube to force each face to instead anticommute, leading to the following definition.

\begin{defn}\label{def:signassign} The {\emph{standard sign
      assignment}} $s$ is the following function from edges of $\two^n$
  to $\f_2$. For $u\geqslant_1 v$, let $k$ be the unique element in
  $\{1,\dots,n\}$ with $u_k > v_k$. Then
\[ 
	s_{u,v} \; := \; \sum^{k-1}_{i=1} u_i \bmod{2}.
\]
\end{defn}
Note that $s$ may be viewed as a 1-cochain in
$\cellC^*([0,1]^n;\f_2)$. In general, $s+c$ is called a
\emph{sign assignment} for any $1$-cocycle $c$ in
$\cellC^*([0,1]^n;\f_2)$.

\subsection{Even Khovanov homology $\Kh$}
\label{subsec:kh-cpx}

Khovanov homology, introduced in \cite{kho1}, is a combinatorial link invariant computed from a planar diagram of an oriented link by considering the \emph{cube of resolutions}.  The result is a bigraded homology theory associated to an oriented link. We sometimes refer to this theory as \emph{even} Khovanov homology to distinguish it from odd Khovanov homology.  For a more complete introduction to this theory, see \cite{natantangle}.

Let $D$ be a link diagram with $n$ ordered crossings.  Each crossing 
$\begin{tikzpicture}[scale=0.04,baseline={([yshift=-.8ex]current
    bounding box.center)}] \draw (0,10) -- (10,0); \node[crossing] at
  (5,5) {}; \draw (0,0) -- (10,10);
\end{tikzpicture}$
can be \emph{resolved} as the \emph{$0$-resolution}
$\begin{tikzpicture}[scale=0.04,baseline={([yshift=-.8ex]current
    bounding box.center)}]
\draw (0,0) .. controls (4,4) and (4,6) .. (0,10);
\draw (10,0) .. controls (6,4) and (6,6) .. (10,10);
\end{tikzpicture}$ or the \emph{$1$-resolution}
$\begin{tikzpicture}[scale=0.04,baseline={([yshift=-.8ex]current
    bounding box.center)}] \draw (0,0)
      .. controls (4,4) and (6,4) .. (10,0); \draw (0,10) .. controls
      (4,6) and (6,6) .. (10,10);
    \end{tikzpicture}$.  

We will view Khovanov homology as coming from a functor
\[
\AbFunc_e : (\two^n)^\op
    \longrightarrow \ZZ\Mod
\]
which we define below. 
The theory is also defined similarly over more general rings. In the context of Smith inequalities (Subsection \ref{subsec:smith}), we will use field coefficients. 

\subsubsection*{Generators}
For each $v\in \two^n$, let $D_v$ be the complete resolution of $D$ formed by taking the $0$-resolution at the $i\th$ crossing if $v_i=0$, or the $1$-resolution otherwise.  The diagram $D_v$ is a planar diagram of embedded circles.  We write $Z(D_v)$ for the set of embedded circles (which we just call circles) in $D_v$. A \emph{Kauffman state} at $v$ will be an element of the powerset of $Z(D_v)$.  Let $\AbFunc_e(v)$ be the free $\ZZ$-module generated by Kauffman states at $v$.  We can think of Kauffman states as the monomials in the symmetric algebra generated by the circles $Z(D_v)$, modulo $x_i^2=0$ for each circle $x_i\in Z(D_v)$, that is, as an element of $\Sym(Z(D_v))/(x^2)_{x\in Z(D_v)}$.

\subsubsection*{Arrows}
Let $v,u \in \Ob(\two^n)$ where $u \geqslant_1 v$.
Since $D_u$ and $D_v$ differ only at the resolution of one crossing, either two circles in $D_v$ \emph{merge} to become one circle in $D_u$, or, dually, one circle in $D_v$ \emph{splits} to become two circles in $D_u$.  Let $\phi_{v,u}^\op\from v \to u$ be the arrow opposite $\phi_{u,v}$.  

First, say that two circles $a_1,a_2\in Z(D_v)$ merge to a circle $a\in Z(D_u)$.  Note that the complements $Z(D_v)\backslash\{a_1,a_2\}$ and $Z(D_u)\backslash\{a\}$ are naturally identified.  Define $\AbFunc_e(\phi^\op_{v,u})$ as the $\ZZ$-algebra map 
\[
\Sym(Z(D_v))/(x^2)_{x\in Z(D_v)}\to \Sym(Z(D_u))/(x^2)_{x\in Z(D_u)}
\]
determined by sending $a_1,a_2$ to $a$, and sending other circles by the identity.

Next, say that one circle $a\in Z(D_v)$ splits to circles $a_1,a_2\in Z(D_u)$.  Define 
\[
\AbFunc_e(\phi^\op_{v,u})(x)=(a_1+a_2)x
\]
where we have used the natural identification of $Z(D_v)\backslash\{a\}$ with $Z(D_u)\backslash\{a_1,a_2\}$.  One readily checks that, with these definitions, $\AbFunc_e$ defines a functor $(\two^n)^\op\to \ZZ\Mod$.  We call $\AbFunc_e$ the \emph{Khovanov functor} of the diagram $D$.

\subsubsection*{Gradings}

There are two gradings associated to the Khovanov complex: first, there is the \emph{homological grading} (or `$h$-grading') $\gr_h$, and an additional \emph{quantum grading} (or `$q$-grading') $\gr_q$ that allows for decategorification to the Jones polynomial.

Let $D $ be a diagram for an oriented link $L$, $n$ the number of crossings in $D$, and $n_+$ and $n_-$ the number of positive and negative crossings (where a negative crossing is locally $\begin{tikzpicture}[scale=0.04,baseline={([yshift=-.8ex]current
    bounding box.center)}] \draw[->] (0,10) -- (10,0); \node[crossing] at
  (5,5) {}; \draw[->] (0,0) -- (10,10);
\end{tikzpicture}$) in $D$, respectively.  
Let $x=a_1\dots a_\ell\in \AbFunc_e(D_u)$ (where $a_i\in Z(D_u)$); then the gradings of $x$ are given by 
\[
	\gr_h(x) \; = \; |v|-n_-,\qquad \gr_q(x) \; = \; |Z(D_v)|-2\ell+|v|+n_+-2n_-.
\]

Note that the morphisms $\AbFunc_e(\phi_{v,u}^\op)$ increase homological grading by $1$ and preserve quantum grading.  In particular, we can regard 
\[
\AbFunc_e\colon (\two^n)^{\mathrm{op}} \to \gradedmod
\]
where $\gradedmod$ is the category of graded $\ZZ$-modules.  We write $\AbFunc^j_e$ for the functor taking $(\two^n)^\op$ to the $j$-graded component of $\AbFunc_e$.

\subsection{Homology from functors}

Khovanov homology is defined from $\AbFunc_e$ as follows.  Let 
\[
\KhCx(L) \; = \; \bigoplus_{v\in\two^n}
\AbFunc_e(v) , \qquad\;\; \partial_{Kh} \; = \;
\sum_{v\geqslant_1 w}
(-1)^{s_{v,w}}\,\AbFunc_e(\phi_{w,v}^\op).
\]

Here $s$ is the standard sign assignment from Definition \ref{def:signassign}.  The chain homotopy type of the resulting complex is an invariant of the oriented link $L$, \cite[Theorem 1]{kho1}.  Note that $\KhCx(L)$ decomposes, over quantum grading, as a chain complex $\KhCx(L)=\KhCx^j(L)$.  The resulting homology $\Kh^{i,j}(L)=H^i(\KhCx^j(L))$ is the \emph{Khovanov homology} of $L$.

\subsection{Odd Khovanov homology $\kho$}
\label{subsec:okh-cpx}

Odd Khovanov homology, introduced in \cite{ors}, is structurally very similar to even Khovanov homology, but instead uses exterior algebra operations to define the differential, introducing signs to the differential within edges. We will view odd Khovanov homology as coming from a functor
\[
\AbFunc_o : (\two^n)^\op
    \longrightarrow \ZZ\Mod
\]

In order to define odd Khovanov homology from a link diagram $D$ with $n$ ordered crossings, we further equip $D$ with an \emph{orientation of crossings}, which is a choice of an arrow at each crossing.  Note that an orientation of the link can be used to acquire an orientation of crossings.  The resolution of a diagram $D$ with an orientation of crossings assigns to $v\in \two^n$ a collections of embedded circles, along with embedded oriented arcs joining the circles.  That is, locally the $0$-resolution of $\begin{tikzpicture}[scale=0.04,baseline={([yshift=-.8ex]current
    bounding box.center)}] \draw (0,10) -- (10,0); \node[crossing] at
  (5,5) {}; \draw (0,0) -- (10,10);
  \draw[thick,->,red] (0,5) --(10,5);
\end{tikzpicture}$ (\resp $\begin{tikzpicture}[scale=0.04,baseline={([yshift=-.8ex]current
    bounding box.center)}] \draw (0,10) -- (10,0); \node[crossing] at
  (5,5) {}; \draw (0,0) -- (10,10);
  \draw[thick,<-,red] (0,5) --(10,5);
\end{tikzpicture}$) is $\begin{tikzpicture}[scale=0.04,baseline={([yshift=-.8ex]current
    bounding box.center)}]
\draw (0,0) .. controls (4,4) and (4,6) .. (0,10);
\draw (10,0) .. controls (6,4) and (6,6) .. (10,10);
  \draw[thick,->,red] (3,5) --(7,5);
\end{tikzpicture}$  (\resp $\begin{tikzpicture}[scale=0.04,baseline={([yshift=-.8ex]current
    bounding box.center)}]
\draw (0,0) .. controls (4,4) and (4,6) .. (0,10);
\draw (10,0) .. controls (6,4) and (6,6) .. (10,10);
  \draw[thick,<-,red] (3,5) --(7,5);
\end{tikzpicture}$) and the $1$-resolution is $\begin{tikzpicture}[scale=0.04,baseline={([yshift=-.8ex]current
    bounding box.center)}] \draw (0,0)
      .. controls (4,4) and (6,4) .. (10,0); \draw (0,10) .. controls
      (4,6) and (6,6) .. (10,10);
      \draw[thick,red,->] (5,7)--(5,3);
    \end{tikzpicture}$ (\resp $\begin{tikzpicture}[scale=0.04,baseline={([yshift=-.8ex]current
    bounding box.center)}] \draw (0,0)
      .. controls (4,4) and (6,4) .. (10,0); \draw (0,10) .. controls
      (4,6) and (6,6) .. (10,10);
      \draw[thick,red,<-] (5,7)--(5,3);
    \end{tikzpicture}$).

For objects $v\in \two^n$, set $\AbFunc_o(v)=\Lambda(Z(D_v))$, the exterior algebra, over $\ZZ$, on the set of circles $Z(D_v)$.  This can be identified with $\AbFunc_e(v)$, but the identification is not canonical.  To define $\AbFunc_o$ on morphisms, we start with an auxiliary assignment $\AbFunc_o'$ (with the same objects) defined on edges $u\geqslant_1 v$; the functor $\AbFunc_o$ is obtained by changing suitable signs of $\AbFunc_o'$.  We will call $\AbFunc_o'$ the \emph{projective odd Khovanov functor}.

For $u\geqslant_1 v$ such that circles $a_1,a_2 \in Z(D_v)$ merge to a circle $a\in Z(D_u)$, set $\AbFunc_o'(\phi^\op_{v,u})$ to be the $\ZZ$-algebra map $\Lambda(Z(D_v))\to \Lambda(Z(D_u))$ determined by sending $a_1,a_2 \mapsto a$ and by identifying the other generators.

For $u\geqslant_1 v$ such that a circle $a\in Z(D_v)$ splits into circles $a_1,a_2\in Z(D_u)$, and such that the arc in $D_u$ points from $a_1$ to $a_2$, set
\[
\AbFunc_o'(\phi_{v,u}^\op)(x)=(a_1-a_2)x
\]
where we view $\Lambda(Z(D_v))$ as a subalgebra of $\Lambda(Z(D_u))$ by sending $a$ to either $a_1$ or $a_2$ and identifying the other generators; one can quickly check that $\AbFunc_o'$ does not depend on whether $a$ is sent to $a_1$ or $a_2$.  It will be convenient later to have the following terminology from \cite{lshomotopytype}:

\begin{defn}[Definition 2.1 \cite{lshomotopytype}]\label{def:reso-config}
A \emph{resolution configuration} $C$ is a pair $(Z(C),A(C))$ where $Z(C)$ is a collection of pairwise-disjoint embedded circles in $S^2$, and $A(C)$ is a totally ordered collection of arcs embedded in $S^2$ with $A(C) \cap Z(C)=\partial A(C)$.  The number of arcs will be called the \emph{index} of a resolution configuration.

An \emph{odd} resolution configuration will be a resolution configuration as above, but where the arcs are oriented.
\end{defn}

For a link diagram $D$ and $u\geq_i w\in \two^n$, we write $D_{u,w}$ for the resolution configuration obtained by performing the $w$-resolution, and then drawing the $i$ arcs corresponding to the difference between $u$ and $w$. 

The assignment $\AbFunc_o'$ on the edges of $(\two^n)^\op$ commutes \emph{up to a sign} along $2$-dimensional faces. We can adjust $\AbFunc_o'$ to give a genuine functor from the cube category, as follows.  
    The two-dimensional odd resolution configurations
    can be divided into four categories as follows (with unoriented
    arcs being orientable in either direction).
   \begin{equation}\label{diagram:type}
     \begin{split}
       \mathrm{A}:\quad&
       \begin{tikzpicture}[scale=0.06,baseline={([yshift=-.8ex]current  bounding  box.center)}]
         \draw (0,0) circle (5cm);
         \draw (11,0) circle (5cm);
         \draw[thick,red] (-5,0) -- (5,0);
         \draw[thick,red] (11-5,0) -- (11+5,0);
       \end{tikzpicture},
       \begin{tikzpicture}[scale=0.06,baseline={([yshift=-.8ex]current  bounding  box.center)}]
         \draw (0,0) circle (5cm);
         \clip (0,0) circle (5cm);
         \draw[thick,red] (-5,0) circle (4cm);
         \draw[thick,red] (5,0) circle (4cm);
       \end{tikzpicture},
       \begin{tikzpicture}[scale=0.06,baseline={([yshift=-.8ex]current  bounding  box.center)}]
         \node[inner sep=0pt, outer sep=0pt,draw,shape=circle,minimum width=0.6cm,style={transform shape=False}] (a) at (0,0) {};
         \node[inner sep=0pt, outer sep=0pt,draw,shape=circle,minimum width=0.6cm,style={transform shape=False}] (b) at (15,0) {};
         \draw[thick,red,->] (a) to[out=30,in=150] (b);
         \draw[thick,red,->] (a) to[out=-30,in=-150] (b);
       \end{tikzpicture},
       \begin{tikzpicture}[scale=0.06,baseline={([yshift=-.8ex]current  bounding  box.center)}]
         \draw[thick,red] (-5,0) circle (4cm);
         \draw[fill=white] (0,0) circle (5cm);
         \draw[thick,red] (0,5) -- (0,-5);
       \end{tikzpicture}.
       \\
       \mathrm{C}:\quad&
       \begin{tikzpicture}[scale=0.06,baseline={([yshift=-.8ex]current  bounding  box.center)}]
         \draw (0,0) circle (5cm);
         \draw (15,0) circle (5cm);
         \draw (30,0) circle (5cm);
         \draw[thick,red] (5,0) -- (10,0);
         \draw[thick,red] (20,0) -- (25,0);
       \end{tikzpicture},
       \begin{tikzpicture}[scale=0.06,baseline={([yshift=-.8ex]current  bounding  box.center)}]
         \draw (0,0) circle (5cm);
         \draw (15,0) circle (5cm);
         \draw[thick,red] (5,0) -- (10,0);
         \draw[thick,red] (15,5) -- (15,-5);
       \end{tikzpicture},
       \begin{tikzpicture}[scale=0.06,baseline={([yshift=-.8ex]current  bounding  box.center)}]
         \node[inner sep=0pt, outer sep=0pt,draw,shape=circle,minimum width=0.6cm,style={transform shape=False}] (a) at (0,0) {};
         \node[inner sep=0pt, outer sep=0pt,draw,shape=circle,minimum width=0.6cm,style={transform shape=False}] (b) at (15,0) {};
         \draw[thick,red,<-] (a) to[out=30,in=150] (b);
         \draw[thick,red,->] (a) to[out=-30,in=-150] (b);
       \end{tikzpicture},
       \begin{tikzpicture}[scale=0.06,baseline={([yshift=-.8ex]current  bounding  box.center)}]
         \draw (0,0) circle (5cm);
         \draw (15,0) circle (5cm);
         \draw (26,0) circle (5cm);
         \draw (41,0) circle (5cm);
         \draw[thick,red] (5,0) -- (10,0);
         \draw[thick,red] (31,0) -- (36,0);
       \end{tikzpicture},
       \begin{tikzpicture}[scale=0.06,baseline={([yshift=-.8ex]current  bounding  box.center)}]
         \draw (0,0) circle (5cm);
         \draw (11,0) circle (5cm);
         \draw (26,0) circle (5cm);
         \draw[thick,red] (16,0) -- (21,0);
         \draw[thick,red] (-5,0) -- (5,0);
       \end{tikzpicture}.
       \\
       \mathrm{X}:\quad&
       \begin{tikzpicture}[scale=0.06,baseline={([yshift=-.8ex]current  bounding  box.center)}]
         \node[inner sep=0pt, outer sep=0pt,draw,shape=circle,minimum width=0.6cm,style={transform shape=False}] (a) at (0,0) {};
         \draw[thick,red,<-] (-5,0) --(5,0);
         \draw[thick,red,->] (a) to[out=150,in=90] (-10,0) to[out=-90,in=210] (a);
       \end{tikzpicture}.
       \\
       \mathrm{Y}:\quad&
       \begin{tikzpicture}[scale=0.06,baseline={([yshift=-.8ex]current  bounding  box.center)}]
         \node[inner sep=0pt, outer sep=0pt,draw,shape=circle,minimum width=0.6cm,style={transform shape=False}] (a) at (0,0) {};
         \draw[thick,red,<-] (-5,0) --(5,0);
         \draw[thick,red,<-] (a) to[out=150,in=90] (-10,0) to[out=-90,in=210] (a);
       \end{tikzpicture}.
     \end{split}
   \end{equation}
Note that $\AbFunc_o'$ commutes on faces of type C, and anticommutes on faces of type A.  Meanwhile, $\AbFunc_o'$ both commutes and anticommutes on faces of type X and type Y (that is, $\AbFunc_o'(\phi_{v,u}^\op)\AbFunc_o'(\phi^\op_{w,v})=0$ on faces of type X and type Y).  For later reference, we call type X and type Y odd resolution configurations (as well as their underlying resolution configurations) \emph{ladybug configurations}. 

We can define \emph{obstruction cocycles} $\Omega(D)\in \cellC^2([0,1]^n;\ZZ_2)$ as follows ($\ZZ_2=\{1,-1\}$ will be written multiplicatively).  Define the type X (\resp type Y) obstruction cocycle $\Omega(D)^X\in \cellC^2([0,1]^n;\ZZ_2)$ (\resp $\Omega(D)^Y$) by setting $\Omega(D)^X_{u,w}=-1$ on faces of type A and type X (\resp type A and type Y), and $\Omega(D)^X_{u,w}=1$ on faces of type C and type Y (\resp type C and type X).  In the sequel we will usually omit the superscript from $\Omega(D)^X$, and we will choose to work with the type X obstruction cocycle. 

Note that the obstruction cocycle cannot \textit{a priori} be determined from the projective functor $\AbFunc_o'\from (\two^n)^\op\to \ZZ\Mod$ itself; the value $\Omega(D)_{u,w}$ on faces $u\geq_2 w\in \two^n$ so that $\AbFunc_o'(\phi_{v,u}^\op)\AbFunc_o'(\phi^\op_{w,v})\neq 0$ is determined by $\AbFunc_o'$, but for faces with $\AbFunc_o'(\phi_{v,u}^\op)\AbFunc_o'(\phi^\op_{w,v})=0$, we need the type of $D_{u,w}$ to specify $\Omega(D)_{u,w}$.

It is shown in \cite{ors} that $\Omega(D)$ (for either type X or Y) is a cocycle, and thus also a coboundary, since $H^2(\cellC([0,1]^{n};\ZZ_2))=0$.  That is, there exists some element $\epsilon \in\cellC^1([0,1]^n;\ZZ_2)$ such that $\delta \epsilon = \Omega(D)$, where $\delta$ denotes the coboundary of $\cellC([0,1]^n;\ZZ_2)$.  An element $\epsilon\in \cellC^1([0,1]^n;\ZZ_2)$ satisfying $\delta \epsilon=\Omega(D)$ will be called an \emph{edge assignment}.  Moreover, for edge assignments $\epsilon_1$ and $\epsilon_2$, the product $\epsilon_1\epsilon_2$ is a cocycle in $\cellC^1([0,1]^n;\ZZ_2)$.  Since $H^1(\cellC([0,1]^{n};\ZZ_2))=0$, any two edge assignments differ by multiplication by a coboundary in $\cellC^1([0,1]^n;\ZZ_2)$.

We define 
\[
\AbFunc_o(\phi^\op_{v,u})=\epsilon_{u,v}\AbFunc_o'(\phi^\op_{v,u}),
\]
which gives a functor $\AbFunc_o$  from the opposite cube category $(\two^n)^\op\to \gradedmod$. Although the identification of $\AbFunc_o(D_u)$ and $\AbFunc_e(D_u)$ is noncanonical, all choices result in the same grading on $\AbFunc_o(D_u)$.  Moreover, it is clear that the arrows $\AbFunc_o(\phi)$ preserve $q$-grading and increase $h$-grading by $1$.  

Odd Khovanov homology is constructed from this functor via

 \[
\oddKhCx(L) \; = \; \bigoplus_{v\in\two^n}
\AbFunc_o(v) , \qquad\;\; \partial_{Kh_o} \; = \;
\sum_{v\geqslant_1 w}
(-1)^{s_{v,w}}\,\AbFunc_o(\phi_{w,v}^\op).
\]
 The homology $H^i(\oddKhCx^j,\partial_{Kh_o})=\kho^{i,j}(L)$ is called the \emph{odd Khovanov homology} of $L$, and its isomorphism class is an invariant of the isotopy class of the oriented link $L$ \cite{ors}.  We will write $\kho^j(L)$ for the sum $\oplus_i \kho^{i,j}(L)$, and similarly write $\Kh^j(L)$ for the sum $\oplus_i \Kh^{i,j}(L)$ for even Khovanov homology.

We will also need to fix bases for the various $\ZZ$-modules considered above.  For the even case, a natural set of generators is given by elements $a_1\otimes \dots\otimes a_k\in \Sym(Z(D_v))/(x^2)_{x\in Z(D_v)}$ where the $a_i\in Z(D_v)$ are distinct.  We refer to the elements $a_1\otimes \dots \otimes a_k$ as \emph{even Khovanov generators}.  For the odd case, in order to choose a basis, we fix at every vertex $v\in \two^n$ a total ordering $>$ on the set $Z(D_v)$.  The set
\[
\KhGen(v) \; =  \; \{ a_{1}\otimes \cdots
\otimes a_{k} : \; a_i \in Z(D_v), \; a_1 < \cdots < a_k \}
\]
is called the set of \emph{odd Khovanov generators at} $v$.  We will usually suppress `even' and `odd' from the notation for Khovanov generators when the appropriate adjective is clear from context.

\begin{rmk} \label{rmk:summary-conventions}
We summarize our conventions with the following minimal example. 
Consider a knot diagram $D$ with one crossing. 
The cohomological functors $\Kh, \Kh_o: \text{link diagrams} \to \ZZ\Mod$ arise from functors $\AbFunc_e, \AbFunc_o$ whose source category is $(\two)^\op$, which is $0 \map{\phi_{0,1}^\op} 1$, where $1 \geqslant_1 0$. 

We have chosen these conventions to match existing literature on our most pertinent tools. `Khovanov homology' was defined with the differentials increasing homological grading \cite{kho1}. Lipshitz-Sarkar constructed their stable homotopy type $\mathcal{X}$ using framed flow categories \cite{lshomotopytype}; in this context, the category $\two = 1 \to 0$ is more natural, for the same reason Morse homology is defined homologically. The (singular) cohomology of $\mathcal{X}$ is then `Khovanov homology.' 
\end{rmk}

\subsection{Annular filtrations}
\label{subsec:ann-filt}

We call a link $L \subset (\mathbb{R}^2 -\{0\}) \times [0,1]$ an \emph{annular} link; in this section we recall the definition of the annular and odd annular Khovanov homologies of annular links.  The former is first defined by \cite{aps}, and the latter is a generalization of their construction, first appearing in \cite{grigsby-wehrli-gl11}.  
  
It is convenient to think of annular links as drawn on $S^2 = \mathbb{R}^2 \cup \{\infty\}$ with two basepoints, with $\mathbb{X}$ at the origin and $\mathbb{O}$ at $\infty$. 
The presence of these basepoints filters both the even and odd Khovanov complexes by a filtration grading $\gr_k$ discussed below, and the associated graded objects are the \emph{annular Khovanov} and the \emph{odd annular Khovanov} complexes. We will denote their homologies by $\Akh$ and $\oddAkh$, respectively.

Fix an annular link diagram $D$.  To obtain the \textit{annular grading} (also called the `($k$)-grading') $\gr_k$, we choose an oriented arc $\gamma$ from $\mathbb{X}$ to $\mathbb{O}$ that misses all crossings of $D$; the resulting grading will be independent of the choice of $\gamma$.  For each Kauffman state of a resolution $D_u$, viewed as a monomial $x_{a_1}\dots x_{a_t}$ in the circles $Z(D_u)$, we obtain an orientation of the circles $Z(D_u)$, where the circles $x_{a_i}$ for $i=1,\dots,t$ are oriented clockwise and the other circles are oriented counterclockwise.  View the collection of oriented circles (associated to a Kauffman state) $Z(D_u)$ as an embedded $1$-manifold $z$.  The ($k$)-grading of $x=x_{a_1}\dots x_{a_t}$, written $\gr_k(x)$, is defined by $\gr_k(x)=I(\gamma,z)$, the algebraic intersection number of $\gamma$ and $z$.

One can check that the maps $\AbFunc_e(\phi^\op_{v,u})$ and $\AbFunc_o(\phi^\op_{v,u})$ (and thus also the differentials $\partial_{\Kh}$ and $\partial_{\kho}$) can only preserve or decrease the ($k$)-grading.   We set $\afunc^{j,k}(v)$ to be the summand of $\AbFunc^j_e(v)$ concentrated in annular grading $k$; equivalently, this is the span of generators of $\AbFunc^j_e(v)$ with annular grading equal to $k$. 
Let $\iota_k\from \afunc^{j,k}(v)\to \AbFunc^j_e(v)$ be the inclusion, and let $\pi_k \from \AbFunc^j_e(v)\to \afunc^{j,k}(v)$ be the projection.  
We define the morphisms $\afunc^{j,k}(\phi_{v,u}^\op)$ to be the ($k$)-grading preserving part of $\AbFunc^j_e(\phi_{v,u}^\op)$; that is, $\afunc^{j,k}(\phi_{v,u}^\op)=\pi_k\AbFunc^j_e(\phi_{v,u}^\op)\iota_k$.  
 Let  $\afunc=\oplus_{j,k} \afunc^{j,k}$ (\resp $\oddafunc$). Then $\afunc$ is a functor
\[
\afunc : (\two^n)^\op
    \longrightarrow \ZZ\Mod,
\] 
which we call the \emph{even annular Khovanov functor}. 

The definitions for $\oddafunc^{j,k}(v)$ and the \emph{odd annular Khovanov functor} $\oddafunc$ are entirely analogous.
It will also be convenient to define $\oddafunc'$, the \emph{(odd) annular Khovanov projective functor}, as the associated graded object of $\AbFunc_o'$ (with respect to the $(k)$-grading).
 
The even annular Khovanov homology of $L$ at $(\gr_q, \gr_k)$-bigrading $(j,k)$, denoted $\Akh^{i,j,k}(L)=H^i(\Akc^{j,k}(L))$, is defined as the homology of the complex
 \[
\Akc^{j,k}(L) \; = \; \bigoplus_{v\in\two^n}
\afunc^{j,k}(v) , \qquad\;\; \partial \; = \;
\sum_{v\geqslant_1 w}
(-1)^{s_{v,w}}\,\afunc(\phi_{w,v}^\op).
\]
The even annular Khovanov homology $\Akh(L)$ of $L$ is the homology of $\Akc(L)$, the direct sum of the above complexes over all bigradings $(j,k)$.
Similarly, the odd annular Khovanov homology $\oddAkh(L)$ is the homology of the analogous complex $\oddAkc(L)$, where $\oddafunc(\phi_{w,v}^\op)$ replaces $\afunc(\phi_{w,v}^\op)$ in the differential $\partial$.  The isomorphism classes of $\Akh(L)$ and $\oddAkh(L)$ are invariants of the annular isotopy class of $L$.

We can also describe the maps $\afunc(\phi^\op_{v,u})$ in local pictures.  It will be useful later to define \emph{annular resolution configurations} (\resp \emph{odd annular resolution configurations}) as in the definition of resolution configurations, except that we require that the embedded circles lie in $S^2-\{\mathbb{X},\mathbb{O}\}$ rather than $S^2$.
Note that an (odd) annular resolution configuration has a well-defined underlying (odd) resolution configuration.  We sometimes abuse notation and refer to any of these types of resolution configurations as  \emph{configurations}.

There are two types of circles in an annular resolution: we call a circle \emph{nontrivial} if it separates $\mathbb{O}$ and $\mathbb{X}$; otherwise, the circle is \emph{trivial}. 
When the annular grading is relevant, we associate nontrivial and trivial circles with the labels $\mathbb{V}$ and $\mathbb{W}$, respectively. 
(Similar to the notation in \cite{GLW-schur-weyl, AGW-skh-hh}, $\mathbb{V}$ 
(\resp $\mathbb{W}$) represents a 2-dimensional vector space with generators in $(\gr_q, \gr_k)$ bigradings $(1, 1)$ and $(-1,-1)$ (\resp $(0,0)$ for both).) 
A saddle (merge or split) cobordism in the annulus corresponds to one of six situations, which are captured by the isotopy classes of index-$1$ annular resolution configurations; see Figure \ref{fig:ann-cobords} for explicit descriptions of the corresponding differentials.
 For an elementary cobordism $S\from D_v \to D_u$, we call a circle $x$ in $Z(D_v)$ or $Z(D_u)$ \textit{active} if the component of $S$ containing $x$ is not homeomorphic to a cylinder; otherwise, we call $x$ a \textit{passive} circle.  The maps $\afunc(\phi^\op_{v,u})$ (and $\oddafunc(\phi^\op_{v,u})$) are obtained from the maps in Figure \ref{fig:ann-cobords} by tensoring with the identity map on generators corresponding to passive circles.
 
\begin{figure} 
\centering 
\renewcommand{\arraystretch}{1.5}
\begin{tabular}{|c | c| c | }
	\hline
	merge map & annular interaction & split map \\
	\hline
	$\begin{array}{c}
	(1 \mapsto 1) \\
	v_1, v_2 \mapsto w
	\end{array}$
	&
	\begin{tikzpicture}[scale=1,baseline={(current bounding box.center)}]
	\node (ghost) at (0,1.5) {}; 
	\node (x) at (0,0) {$\mathbb{X}$};
	\draw circle  (1cm)
	circle (.5cm);
	\draw[->, densely dotted] (0,0) -- node[auto]{$\gamma$} (0,1.5);
	\node (vv) at (0,-1.5) {$\mathbb{V} \otimes \mathbb{V}$}; 
	\end{tikzpicture}
	\begin{tikzpicture}[scale=1,baseline={(current bounding box.center)}]
	\draw[<->] (-.4,0) -- (.4,0);
	\end{tikzpicture}
	\begin{tikzpicture}[scale=1,baseline={(current bounding box.center)}]
	\node (ghost) at (0,1.5) {}; 
	\node (x) at (0,0) {$\mathbb{X}$};
	\draw[->, densely dotted] (0,0) -- node[auto]{$\gamma$} (0,1.5);
	\draw(-75:0.5cm) -- (-75:1cm)
	arc (-75:255:1cm) -- (255:0.5cm) arc (255:-75:0.5cm) ;
	\node (w) at (0,-1.5) {$\mathbb{W}$}; 
	\end{tikzpicture}
	&
	$\begin{array}{c}
	1 \mapsto v_1 + v_2 \\
	(w \mapsto v_1v_2)
	\end{array}$\\
	\hline
	$\begin{array}{c}
	1 \mapsto 1 \\
	v_1 \mapsto v \\
	(w_1 \mapsto v)
	\end{array}$
	&
	\begin{tikzpicture}[scale=1,baseline={(current bounding box.center)}]
	\node (ghost) at (0,1.5) {}; 
	\node (x) at (0,0) {$\mathbb{X}$};
	\node (x) at (0,0) {$\mathbb{X}$};
	\draw[->, densely dotted] (0,0) -- node[auto]{$\gamma$} (0,1.5);
	\draw(-75:0.75cm) -- (-75:1cm)
	arc (-75:255:1cm) -- (255:0.75cm) arc (255:-75:0.75cm) ;
	\draw circle (.5cm);
	\node (vw) at (0,-1.5) {$\mathbb{V}\otimes\mathbb{W}$}; 
	\end{tikzpicture}
	\begin{tikzpicture}[scale=1,baseline={(current bounding box.center)}]
	\draw[<->] (-.4,0) -- (.4,0);
	\end{tikzpicture}
	\begin{tikzpicture}[scale=1,baseline={(current bounding box.center)}]
	\node (ghost) at (0,1.5) {}; 
	\node (x) at (0,0) {$\mathbb{X}$};
	\draw[->, densely dotted] (0,0) -- node[auto]{$\gamma$} (0,1.5);
	\draw
	(255:1cm) 
	-- (255:.75cm) 
	arc (255:15:.75cm)
	-- (15:.5cm)
	arc (15:345:.5cm)
	-- (-15:.75cm) 
	arc (-15:-75:.75cm)
	-- (-75:1cm)
	arc (-75:255:1cm);
	\node (v) at (0,-1.5) {$\mathbb{V}$}; 
	\end{tikzpicture}
	&
	$\begin{array}{c}
	1 \mapsto w_1 (+v_1)\\
	v \mapsto v_1w_1
	\end{array}$\\
	\hline
	%
	%
	$\begin{array}{c}
	1 \mapsto 1 \\ w_1, w_2 \mapsto w
	\end{array}$
	&
	\begin{tikzpicture}[scale=1,baseline={(current bounding box.center)}]
	\node (ghost) at (0,1.5) {}; 
	\node (x) at (0,0) {$\mathbb{X}$};

	\draw(-25:0.5cm) -- (-25:1cm)
	arc (-25:205:1cm) -- (205:0.5cm) arc (205:-25:0.5cm) ;
	\draw(235:0.5cm) -- (235:1cm)
	arc (235:305:1cm) -- (305:0.5cm) arc (305:235:0.5cm) ;
	\draw[->, densely dotted] (0,0) -- node[auto]{$\gamma$} (0,1.5);
	\node (ww) at (0,-1.5) {$\mathbb{W}\otimes \mathbb{W}$}; 
	\end{tikzpicture}
	\begin{tikzpicture}[scale=1,baseline={(current bounding box.center)}]
	\draw[<->] (-.4,0) -- (.4,0);
	\end{tikzpicture}
	\begin{tikzpicture}[scale=1,baseline={(current bounding box.center)}]
	\node (ghost) at (0,1.5) {}; 
	\node (x) at (0,0) {$\mathbb{X}$};
	\draw[->, densely dotted] (0,0) -- node[auto]{$\gamma$} (0,1.5);
	\draw
	(-125:.5cm)
	-- (-125:1cm)
	arc (-125:205:1cm)
	-- (205:.5cm)
	arc (205:-125:.5cm);
	\node (w) at (0,-1.5) {$\mathbb{W}$}; 
	\end{tikzpicture}
	&
	$\begin{array}{c}
	1 \mapsto w_1 + w_2 \\
	w \mapsto w_1w_2
	\end{array}$\\
	\hline
\end{tabular}
\caption{The six types of saddle interactions between circles in the annulus. Components of the (even) Khovanov differential are listed in the side columns, with components that fail to preserve the annular $(k)$-grading in parentheses; these decrease $\gr_k$ by exactly $-2$. For the odd case, the signs may differ, depending on context. }
\label{fig:ann-cobords} 
\end{figure}

There is another grading $\gr_{j_1}$ special to the annular case that we are tempted to call the \emph{annular quantum grading}, as it appears to be more relevant in annular Khovanov homology than the quantum grading; it was first introduced in \cite{GLW-schur-weyl} as the `filtration-adjusted quantum grading' and is defined by $\gr_{j_1} = \gr_q - \gr_k$. This grading will play an important role when we study the Khovanov complexes for periodic links.

Given an annular link diagram, the Khovanov generators $\KhGen(v)$ inherit a well-defined ($k$)-grading, and we write $\KhGen^{j,k}(v)$ for the Khovanov generators at $v\in \two^n$ with $\gr_q=j$ and $\gr_k=k$.

\subsection{Periodic links}\label{subsec:periodic-links}

Let $p$ be an integer greater than $1$.  A $p$-periodic link $(\tilde{L},\psi)$ is a link $\tilde L \subset S^3$  together with an orientation-preserving $\ZZ_p$-action $\psi$ on the pair $(S^3, \tilde L)$ such that the fixed-point set of $\psi$ on $S^3$ is an unknot $\tilde{U}$ disjoint from $\tilde L$. (Often, we will confound notation, and write $\psi$ for a generator of this action.)  We will often write $\tilde{L}$ for such a periodic link, with the action $\psi$ suppressed from the notation.

For a $p$-periodic link $\tilde{L}$, the image of $\tilde L$ under the quotient map $S^3 \to S^3 / \psi$ is called the \emph{quotient link}, and is denoted $L$. 
Observe that if we remove the fixed-point set, an equivariant isotopy from $p$-periodic link $(\tilde{L}_0,\psi_0)$ to another $p$-periodic link $(\tilde{L}_1,\psi_1)$ can be viewed as an equivariant ambient isotopy in the solid torus. Quotienting by the action of $\psi$, we see that an equivariant isotopy between $\tilde{L}_0$ and $\tilde{L}_1$ is a lift of an annular isotopy from $L_0$ to $L_1$. 

We will need a particularly convenient form of link diagrams for periodic links.  A $p$-\emph{periodic link diagram} will be an annular link diagram $\tilde{D}$ in $\mathbb{R}^2$ such that the action of $\ZZ_p$ by counterclockwise rotation on $\mathbb{R}^2$ preserves $\tilde{D}$ (setwise).  Such a diagram describes a $p$-periodic link $\tilde{L}$ in $S^3$, and every $p$-periodic link admits such a diagram. 
Then $D = \tilde D/ \psi$ is a diagram for the quotient link $L$. 
We will assume that all of our diagrams for $p$-periodic links are $p$-periodic diagrams.
 
Note also that given an annular diagram $D$, we can form a $p$-periodic link diagram $\tilde{D}$, called the $p$-\emph{cover} of $D$, by taking $p$ copies $\{D_i\}_{i=1,\dots,p}$ of $D$ cut along an arc $\gamma$ as in Figure \ref{fig:ann-cobords}, and gluing (reversing orientation on the boundary) $D_i$ to $D_{i+1}$ along one boundary component of the cut diagram (with subscripts interpreted cyclically).

Two $p$-periodic diagrams $\tilde{D}_1$ and $\tilde{D}_2$ represent the same periodic link if and only if they are related by equivariant isotopies and \emph{equivariant Reidemeister moves}, which are the lift of Reidemeister moves on the quotient diagrams $D_1,D_2$ (see \cite{politarczyk-kh}).  See Figure \ref{fig:reid-1} for an example. In particular, equivariant Reidemeister moves do not interact with the basepoint $\mathbb{X}$ in the diagram.

\begin{notation} \label{notation:lift-of}
For bookkeeping purposes, we introduce the notation that $\tilde \cdot$ generally means `lift of,' as well as the following rules.
Given an ordering of crossings of a diagram $D$, we obtain an ordering of crossings on $\tilde{D}$ as follows.  Recall that in the definition of annular Khovanov homology we used an arc $\gamma$.  As the quotient of a periodic diagram $\tilde{D}$, the diagram $D$ is naturally an annular diagram, and we fix some arc $\gamma$ connecting from $\mathbb{X}$ to $\mathbb{O}$, as in the definition of annular Khovanov homology in the previous section.  Let $\tilde{\gamma}$ be a lift of $\gamma$ to $\tilde{D}$.  We divide the plane containing $\tilde{D}$ into \emph{sectors}, that is, the connected components of $\mathbb{R}^2-\ZZ_p\tilde{\gamma}$, where $\ZZ_p\tilde{\gamma}$ denotes the orbit of $\tilde{\gamma}$ under the rotation action of $\ZZ_p$.  The sectors are labeled $S_1,\dots,S_p$, where $S_i$ is the sector between $\psi^{i-1}\tilde{\gamma}$ and $\psi^i\tilde{\gamma}$.  The crossings of $\tilde{D}$ are ordered by requiring that the first $n$ crossings are those contained in $S_1$, ordered according to their ordering in the quotient, the next $n$ are the crossings of $S_2$, and so on.  From now on, unless otherwise stated, given an annular diagram $D$ with ordered crossings, we will assume its $p$-cover $\tilde{D}$ has this ordering of crossings.
\end{notation}

\subsection{Periodic links and Khovanov homologies}\label{subsec:periodic-links-kh}
Fix an integer $p>1$ and a  $p$-periodic diagram $\tilde{D}$.  The rotation action on resolution diagrams induces an action $\psi$ on the Khovanov generators, which we describe below. We first observe that for $g\in \ZZ_p$, and a circle $x$ of a resolution $v$, there is a circle $gx$ in the resolution $gv$, obtained by rotating $x$ through $(g/p)2\pi$ (counterclockwise).  The group $\ZZ_p$ acts on $\oplus_v \AbFunc_e(v)$ by sending a Kauffman state $x_1\dots x_t$ to $y_1\dots y_t$, where $y_i=gx_i$.  For the above ordering of the crossings of $\tilde{D}$ and $D$, this action lies over the action of $\ZZ_p$ on $(\two^{n})^p$ by cyclic permutation.  To be specific, the action of $\ZZ_p$ on $(\two^{n})^p$ is defined by the property that the generator $1\in \ZZ_p$ sends $(x_1,x_2,\ldots,x_{p-1},x_p)\in (\two^n)^{p}$ to $(x_p,x_1\ldots,x_{p-2},x_{p-1})$. We call a Khovanov generator an \emph{invariant generator} if it is invariant under the action of $\ZZ_p$. Meanwhile, $\ZZ_p$ acts by bijections on the set $\KhGen(\tilde{D})$, but one can say somewhat more.  That is, $\ZZ_p$ may send odd Khovanov generators to $\pm$-multiples of odd Khovanov generators.  Let a \emph{signed bijection} $X\from S_1 \to S_2$ between two finite sets $S_1,S_2$ be a bijection along with a `sign' map $\sigma \from S_1 \to \ZZ_2$.  (Really, we should view $X$ as a correspondence between $S_1$ and $S_2$ along with a `sign' map $\sigma \from X \to \ZZ_2$; see Section \ref{sec:burn} for more details.)  Then the generator $\psi$ of $\ZZ_p$ acts by signed bijections, $\KhGen(u)\to \KhGen(\psi u)$, where the sign of $x\in \KhGen(u)$ records the sign of the generator $\psi(x)$ as a Khovanov generator of $\AbFunc_o(\psi u)$.  We write $\KhGen(\tilde{D})^{\ZZ_p}$ for the set of invariant Khovanov generators (where invariant just means invariant under the $\ZZ_p$-action, and does not involve the sign map of the $\ZZ_p$-action).

We conclude this section by discussing the relationship between generators in $\KhCx(D)$ and their lifts in $\KhCx(\tilde{D})$. In particular, the relationship between gradings of generators in $\KhCx(D)$ and $\KhCx(\tilde{D})$ explains the role annular filtrations play in the present localization of Khovanov homology.

\begin{prop}[{cf. \cite[Proposition 29]{z-annular-rank}}]
\label{prop:identify-equiv-gens}
There is a bijection between the (even) generators of $\KhCx(D)$ and the (even) invariant generators of $\KhCx(\tilde{D})$, given by $x \mapsto \tilde x$, such that
\[\gr_k(\tilde x) = \gr_k(x), \qquad \gr_h(\tilde x) = p \gr_h(x), \quad \mbox{  and } \quad \gr_q(\tilde x) = p \gr_q(x) - (p-1) \gr_k(x).
\]
In particular, this implies $\gr_{j_1}(\tilde x) = p \gr_{j_1}(x)$.  
\end{prop}
\begin{proof}
(See Notation \ref{notation:lift-of}.)
Note that $\tilde n_+ = p n_+$,  $\tilde n_- = p n_-$, and $|\tilde u| = p |u|$.
Let $x \in \KhCx(D)$ be a generator lying at vertex $u \in \two^n$. 
Suppose $x$ has 
\begin{itemize}
	\item $\alpha$ nontrivial counterclockwise circles,
	\item $\beta$ nontrivial clockwise circles,
	\item $\gamma$ trivial counterclockwise circles, and
	\item $\delta$ trivial clockwise circles.
\end{itemize}

Let $S$ be a circle in $D_u$. If $S$ is nontrivial, then its lift in $D_{\tilde u}$ consists of a single equivariant nontrivial circle. On the other hand, if $S$ is trivial, then its lift consists of $p$ copies of a nontrivial circle.  
We may then compute the gradings for $\tilde x$:
\begin{align*}
\gr_k(\tilde x) &= \alpha - \beta = \gr_k(x) \\
\gr_h(\tilde x) &= |\tilde u| - \tilde n_- = p|u| - pn_- = p \gr_h (x) \\
\gr_q(\tilde x) &= |\tilde u| + \alpha - \beta + p \gamma - p \delta + \tilde n_- - 2 \tilde n_+ \\
       &= p |u| + p(\alpha - \beta + \gamma - \delta) - (p-1)(\alpha - \beta) + p(n_- - 2 n_+) \\
       &= p \gr_q(x) - (p-1)\gr_k(x).
\end{align*}
The $\gr_{j_1}$ relationship follows directly.
\end{proof}

Moreover, Proposition \ref{prop:identify-equiv-gens} extends to a  bijection $\KhGen(D)\to \KhGen(\tilde{D})^{\ZZ_p}$, when the order of circles upstairs is chosen to satisfy the following.  First, if circles $a_1, a_2\in Z(D_u)$ satisfy $a_1<a_2$, then any circles over them, say $\tilde{a}_1$ and $\tilde{a}_2$, satisfy $\tilde{a}_1<\tilde{a}_2$.  For $a \in Z(D_u)$, let $\tilde{a}_1$ be the circle upstairs that is closest to $\tilde{\gamma}$, proceeding counterclockwise from $\tilde{\gamma}$, for those $a$ which do not intersect $\gamma$. (For nontrivial circles, which necessarily intersect $\tilde{\gamma}$, there is no ambiguity.)  For trivial $a$ that intersects $\gamma$, let $\tilde{a}$ denote the circle above $a$ that intersects $\tilde{\gamma}$ furthest from $\mathbb{X}$.   
Define $\tilde{a}_i=\psi^{i-1} \tilde{a}$ for $1\leq i\leq p$.  We require $\tilde{a}_1<\dots<\tilde{a}_{p}$.  The bijection $\KhGen(D) \to \KhGen(\tilde{D})^{\ZZ_p}$ is determined by taking nontrivial circles to nontrivial circles, and takes a trivial circle $a$ to $\tilde{a}_1\dots \tilde{a}_{p}$. 

If $p$ is odd, the bijection $\KhGen(D)\to \KhGen(\tilde{D})^{\ZZ_p}$ can be described more simply.  
Each invariant generator in $\KhGen(\tilde{D}_u)$ is a product of terms coming from nontrivial circles in $\tilde{D}_u$ and products $x_{i_1}\ldots x_{i_{p}}$ of trivial circles related by rotation.  Say $y_1\ldots y_k$ is an element of $\KhGen(D)$, with $y_1<\ldots< y_k$.  If $y_i$ is a nontrivial circle, let $\tilde{y}_i$ be the unique circle over $y_i$, otherwise let $\tilde{y}_i$ be the product $\tilde{y}_{i,1}\ldots \tilde{y}_{i,p}$ of trivial circles over $y_i$, where $\tilde{y}_{i,1}$ is any trivial circle over $y_i$, and $\tilde{y}_{i,j}=\psi^{j-1}\tilde{y}_{i,1}$.  Because $p$ is odd, the product $\tilde{y}_{i,1}\ldots \tilde{y}_{i,p}$ is independent of the choice of $\tilde{y}_{i,1}$.  Then the bijection $\KhGen(D)\to\KhGen(\tilde{D})^{\ZZ_p}$ is given by sending
\[
y_1\ldots y_k \mapsto \tilde{y}_1\ldots \tilde{y}_k.
\]


\section{Burnside categories and functors}\label{sec:burn}
In this section we recall the machinery of Burnside functors from \cite{lls1},\cite{lls2}.  We will also record a slight generalization of the signed Burnside functors of \cite{oddkh}.  Sections \ref{subsec:burnside}--\ref{subsec:func-to} are essentially a review of material from \cite{lls1}-\cite{oddkh}.  In Section \ref{subsec:external-3}, we introduce external actions on Burnside functors and prove basic properties.  The rest of the section consists of generalizing notions of \cite{lls1} to Burnside functors with external action.

\subsection{The Burnside category}\label{subsec:burnside}
Given finite sets $X$ and $Y$, a correspondence from $X$ to $Y$ is a
triple $(A,s,t)$ consisting of a finite set $A$ and set maps $s
\from A\to X$ and $t \from A \to Y$.  The maps $s$ and $t$ are called the
\emph{source} and \emph{target} maps, respectively. The correspondence
$(A,s,t)$ is depicted:
\[
\begin{tikzpicture}[scale=1]
\node (X) at (-2,0) {$X$};
\node (Y) at (2,0) {$Y$};
\node (A) at (0,1) {$A$};
\draw[->] (A) -- (X) node[midway, above] {$s\;\;\;$};
\draw[->] (A) -- (Y) node[midway,above] {$\;\;t$};
\end{tikzpicture}
\]

For correspondences $(A,s_A,t_A)$ from $X$ to $Y$ and $(B,s_B,t_B)$ from $Y$ to $Z$, define the composition $(B,s_B,t_B)\circ
(A,s_A,t_A)$ to be the correspondence $(C,s,t)$ from $X$ to $Z$ given
by the fiber product $C=B \times_Y A =\{ (b,a) \in B \times A \mid
t(a) = s(b)\}$ with source and target maps $s(b,a)=s_A(a)$ and
$t(b,a)=t_B(b)$. There is also the identity correspondence from a set
$X$ to itself, $(X,\Id_X,\Id_X)$ from $X$ to $X$, where $\Id_X\colon X \to X$ is the identity map.  Given
correspondences $(A,s_A,t_A)$, $(B,s_B,t_B)$ from $X$ to $Y$, a
morphism of correspondences from $(A,s_A,t_A)$ to $(B,s_B,t_B)$ is a
bijection $f \from A \to B$ commuting with the source and target maps.

Composition (of set maps) gives the set of correspondences from $X$ to
$Y$ the structure of a category.  Informally, the \emph{Burnside category}
$\burn$ is the weak $2$-category whose objects are finite sets,
morphisms are finite correspondences, and $2$-morphisms are maps of
correspondences.

Recall that in a weak $2$-category, composition need only be
associative up to an equivalence, and similarly the identity axiom need
only hold after composing with a $2$-morphism.  To be explicit, for
finite sets $X,Y$ and $(A,s,t)$ a correspondence from $X$ to $Y$,
neither $(Y,\Id_Y,\Id_Y) \circ (A,s,t)$ nor $(A,s,t)\circ
(X,\Id_X,\Id_X)$ needs to equal $(A,s,t)$. Rather, there are natural
$2$-morphisms called left and right unitors,
\[
\lambda\from Y \times_Y A \to A \qquad \text{and} \qquad \rho\from A \times_X X \to A,
\]
such that $\lambda(y,a)=a$ and $\rho(a,x)=a$. Further, the composition
$C\circ (B\circ A)$, for $A$ from $W$ to $X$, $B$ from $X$ to $Y$, and
$C$ from $Y$ to $Z$, is not necessarily identical to $(C \circ B) \circ A$. Rather,
there is an associator
\[
\alpha \from (C \times_Y B)\times_X A \to C \times_Y (B \times_X A)
\]
given by $\alpha((c,b),a)=(c,(b,a))$.

We will work with a variant, as in \cite[Section 2.11]{lls-tangle}, in which composition is strictly associative.  Here, the objects of $\burn$ are finite sets, and for objects $X,Y$, the morphism set $\Hom(X,Y)$ is the set of pairs consisting of an integer $n$ and a $(Y\times X)$-matrix $(A_{y,x})_{x\in X,y\in Y}$ of finite subsets $A_{y,x}$ of $\mathbb{R}^n$ satisfying
\[
A_{y,x}\cap A_{y',x}=\emptyset  \mbox{ if }  y\neq y' 
\qquad 
\mbox{ and } 
\qquad 
A_{y,x}\cap A_{y,x'}=\emptyset  \mbox{ if }  x\neq x'.
\]
Note that for $A\subset \mathbb{R}^n$ and $B\subset \mathbb{R}^m$, $A\times B$ is a subset of $\mathbb{R}^{n+m}$.  Composition is then given by
\[
(A_{z,y})_{y\in Y,z\in Z} \circ(A_{y,x})_{x\in X,y\in Y}=\big( \bigcup_{y\in Y}A_{z,y}\times A_{y,x}\big)_{x\in X,z\in Z}.
\]
Meanwhile, $2$-morphisms are bijections of correspondences (for which the embedding information is not needed).  For more details, see \cite[Section 2.11]{lls-tangle}.

Throughout, when we refer to the `Burnside category,' we will mean this \emph{strict} version of the category.  However, for everything that appears in this paper, the embedding data can be chosen arbitrarily, and so we will not specify it.

In a 2-category, there are two kinds of composition for 2-morphisms. For objects $x,y,z$, and 1-morphisms $f,g\from x \to y$ and $f',g'\from y \to z$, as well as 2-morphisms $\beta\from f\to g$ and $\gamma \from f'\to g'$, there is the \emph{horizontal composite} $\gamma\circ_1\beta\from f'\circ f\to g'\circ g$.  Meanwhile, for objects $x,y$, the set of morphisms $\Hom(x,y)$ is a category, in which we have composition.  That is, for fixed morphisms $f,g,h\from x\to y$, if $\beta\from f\to g$ and $\gamma \from g\to h$ are 2-morphisms, then there is a well-defined \emph{vertical composite} $\gamma\circ_2\beta$.  When it is clear which of the two compositions ($\circ_1$ or $\circ_2$) we are referring to, we will omit the subscript.  For more details on 2-categories, we refer the reader to \cite{Benabou-other-bicategories}.

\subsection{Decorated Burnside categories} \label{subsec:oddburn} 
Fix a group $K$; for our purposes, $K$ is usually the cyclic group $\ZZ_2=\{1,-1\}$, written multiplicatively.  Given finite sets $X$ and $Y$, a \emph{decorated correspondence} is a
correspondence $(A,s_A,t_A)$ equipped with a map
$\sigma_A \from A \to K,$ regarded as a tuple
$(A,s_A,t_A,\sigma_A)$; we call $\sigma_A$ the `decoration'
of the correspondence (or the `sign' if $K=\ZZ_2$):
\[
\begin{tikzpicture}[scale=1]
\node (X) at (-2,0) {$X$};
\node (Y) at (2,0) {$Y$};
\node (A) at (0,1) {$A$};
\node (S) at (0,2.5) {$K$};
\draw[->] (A) -- (X) node[midway, above] {$s_A\;\;\;$};
\draw[->] (A) -- (Y) node[midway,above] {$\;\;t_A$};
\draw[->] (A) -- (S) node[midway,left] {$\sigma_A$};
\end{tikzpicture}
\]
In the sequel we often write `correspondence' for `decorated
correspondence,' where it will not cause any confusion.  

Let $(A,s_A,t_A,\sigma_A)$ be a correspondence from $X$ to $Y$ and
$(B,s_B,t_B,\sigma_B)$ a correspondence from $Y$ to $Z$;  
we define a
composition $(B,s_B,t_B,\sigma_B)\circ (A,s_A,t_A,\sigma_A)$ of decorated
correspondences by $(C,s,t,\sigma)$, where
$(C,s,t)$ is the composition $(B,s_B,t_B) \circ (A,s_A,t_A)$ and
$\sigma(b,a)=\sigma_B(b)\sigma_A(a)$.  Also, we define the identity correspondence by $(X, \Id_X, \Id_X, 1)$ (i.e., the identity
correspondence takes the decoration $1$ on all elements of $X$).

We define maps of decorated correspondences $f \from (A,s_A,t_A,\sigma_A)
\to (B,s_B,t_B,\sigma_B)$ to be morphisms of correspondences $f \from
(A,s_A,t_A) \to (B,s_B,t_B)$ such that $\sigma_B\circ f=\sigma_A$.  We
may then define the $K$-\emph{Burnside category} $\burn_K$ to be the
 $2$-category with objects finite sets, morphisms given by decorated
correspondences along with an embedding as in the definition of the ordinary Burnside category from Section \ref{subsec:burnside}, and $2$-morphisms given by maps of decorated
correspondences.  Note that the structure maps $\lambda, \rho,\alpha$ of
Subsection \ref{subsec:burnside} are easily seen to respect the decoration, 
confirming that $\oddb$, with the embedding information forgotten, is indeed a weak $2$-category. There is a
forgetful 2-functor $\forgot\from \oddb \to \burn$ which forgets
decorations.  As with the ordinary Burnside category, we will work with the \emph{strict} version of the $K$-Burnside category, but we will not specify the embedding data when it may be chosen arbitrarily.  

For a homomorphism $\degh\from K \to \ZZ_2$, we define a functor
$\Abelianize_\degh\from \burn_K \to\ZZ\Mod$ by sending an
object $X$ of $\burn_K$ to the free abelian group
generated by $X$, denoted $\Abelianize_\degh(X)$.  For a decorated
correspondence $\phi = (A,s,t,\sigma)$ from $X$ to $Y$, we define
$\Abelianize_\degh(\phi)\from\Abelianize_\degh(X)\to \Abelianize_\degh(Y)$ by
\begin{equation}
  \Abelianize_\degh(\phi)(x)=\sum_{a\in A\,\mid\,s(a)=x}\degh(\sigma(a))t(a)\label{eq:amorph}
\end{equation}
for elements $x \in X$, extended linearly over $\ZZ$.  When $\degh$ is the trivial morphism, we write $\Abelianize$ for $\Abelianize_\degh$.

We also define a functor $\Abelianize_K\from \burn_K\to \ZZ[K]\Mod$ by sending an object $X$ of $\burn_K$ to the free $\ZZ[K]$-module generated by $X$, denoted $\Abelianize_K(X)$.  For a decorated correspondence $\phi=(A,s,t,\sigma)$ from $X$ to $Y$, we define $\Abelianize_\degh(\phi)\from\Abelianize_\degh(X)\to \Abelianize_K(Y)$ by
\begin{equation}\label{eq:kmorph}
\Abelianize_K(\phi)(x)=\sum_{\{a\in A\,\mid\,s(a)=x\}}\sigma(a)t(a)
\end{equation}
for elements $x \in X$, extended linearly over $\ZZ[K]$.  Note that a homomorphism $\degh\colon K \to \ZZ_2$ defines a homomorphism of group rings $\ZZ[K]\to \ZZ$ by sending $k\in K$ to $\degh(k)\in \{\pm 1\}$ and extending linearly.  The functor $\Abelianize_\degh$ is obtained by applying extension of scalars, along $\ZZ[K]\to \ZZ$, to $\Abelianize_K$.

\subsection{Functors to Burnside categories}\label{subsec:func-to} 

We now consider functors from the
cube category $\two^n$ to the Burnside categories introduced thus far. The functors
$F\from\two^n\to \burn_K$ we consider will be strictly unitary
lax 2-functors, defined below.

\begin{defn}\label{def:lax}
	Let $\Cat$ be a 1-category and $\Dat$ a weak 2-category.  A \emph{strictly unitary lax 2-functor} $F\from \Cat\to \Dat$ consists of the following data:
\begin{enumerate}[leftmargin=*]
\item\label{itm:32-1} For each object $x$ of $\Cat$, an object
  $F(x)$ of $\Dat$.
\item\label{itm:32-2} For any morphism $\phi\from x\to y$ in $\Cat$, a 1-morphism $F(\phi)$ in
  $\Dat$ from $F(x)$ to $F(y)$.  For $x$ an object of a 1-category or 2-category, let $\Id_x$ denote the identity morphism at $x$.  We require that, for all objects $x$ of $\Cat$, $F(\Id_{x})$
  is the identity morphism $\Id_{F(x)}$.
\item\label{itm:32-3} Finally, for any objects $x,y,z$ of $\Cat$ and morphisms $\beta\from x \to y$ and $\gamma \from y \to z$, a 2-morphism $F_{\beta,\gamma}$ in
  $\burn_K$ from $F(\gamma)\circ F(\beta)$ to
  $F(\gamma\circ \beta)$ that agrees with $\lambda$ (\resp $\rho$)
  when $\gamma=\Id_y$ (\resp $\beta=\Id_x$), such the diagram below commutes. Here, $w,x,y,z$ are objects of $\Cat$ with morphisms $\beta\from w\to x$, $\gamma\from x\to y$, and $\delta\from y\to z$.

\[
\begin{tikzpicture}[scale=1,baseline={(current  bounding  box.center)}]
\node (h0) at (0,0)   {\scriptsize $(F(\delta)\circ F(\gamma))\circ F(\beta)$};
\node (h1) at (2*2.5,0)    {\scriptsize $F(\delta)\circ(F(\gamma)\circ F(\beta))$};
\node (g1) at (0,-2.5)  {\scriptsize $F(\delta\circ\gamma)\circ F(\beta)$};
\node (g2) at (2*2.5,-2.5) {\scriptsize $F(\delta)\circ F(\gamma\circ\beta)$};
\node (c1) at (2.5,-5)  {\scriptsize $F(\delta\circ\gamma\circ\beta)$};

\draw[->]
(h0) edge node[auto] {\scriptsize $\alpha$} (h1);

\draw[->](h1) edge node[auto] {\scriptsize $\Id \circ_1 F_{\beta,\gamma}$} (g2)
(g2) edge node[auto] {\scriptsize $F_{\gamma\circ\beta,\delta}$} (c1);
\draw[->]
(h0) edge node[left] {\scriptsize$F_{\gamma,\delta}\circ_1 \Id$} (g1)
(g1) edge node[auto] {\scriptsize $F_{\beta,\delta\circ\gamma}$} (c1);
\end{tikzpicture}
\]
\end{enumerate}
\end{defn}

The more general definition of \emph{strictly unitary lax 2-functors} between weak 2-categories $\Cat$ and $\Dat$ can be found in Definitions 4.2 and 4.3 of \cite{lls1}. 
We call strictly unitary lax 2-functors simply `2-functors' or `functors' when it will not cause confusion.  If the target of such a functor is the Burnside category or a variant thereof, we may also refer to such functors as `Burnside functors.'

When $\Cat=\two^n$, and $u\geq v\geq w$ are objects of $\Cat$, we write $F_{u,v,w}$ for the 2-morphism $F(\phi_{v,w})\circ F(\phi_{u,v})\to F(\phi_{u,w})$.

\begin{lem}[{\cite[Lemma 3.2]{oddkh}}] \label{lem:212}
  Consider the data consisting of objects $F(v)$ for $v \in \two^n$,
  a collection of 1-morphisms $F(\phi_{v,w})$ in $\burn_K$ for edges $v\geqslant_1 w$, and 2-morphisms
  $F_{u,v,v',w} \from F(\phi_{v,w}) \circ F(\phi_{u,v}) \to
  F(\phi_{v',w}) \circ F (\phi_{u,v'})$ for each 2-dimensional face described by
  $u\geqslant_1 v,v'\geqslant_1 w$, such that the following
  compatibility conditions are satisfied:
  \begin{enumerate}[leftmargin=*]
  \item For any 2-dimensional face $u,v,v',w$ as above, $F_{u,v,v',w}=F^{-1}_{u,v',v,w}$;
  \item For any 3-dimensional face in $\two^n$ on the left, the hexagon on
    the right commutes: \label{cond:c2}
  \end{enumerate}
\[
\begin{tikzpicture}[node distance=2.5cm,
  back line/.style={densely dotted},
  cross line/.style={preaction={draw=white, -, line width=6pt}},baseline={(current  bounding  box.center)}]
  \node (u) {$u$};
  \node (v') [right of=u] {$v'$};
  \node [below of=u] (v'') {$v''$};
  \node [right of=v''] (w) {$w$};
  \node (v) [right of=u, above of=u, node distance=1cm] {$v$};
  \node (w'') [right of=v] {$w''$};
  \node (w') [below of=v] {$w'$};
  \node (z) [right of=w'] {$z$};
  \draw[back line, ->] (v) to (w');
  \draw[back line, ->] (v'') to (w');
  \draw[back line, ->] (w') to (z);
  \draw[->] (w'') to (z);
  \draw[cross line, ->] (u) to (v);
  \draw[cross line, ->] (u) to (v');
  \draw[cross line, ->] (u) to (v'');
  \draw[cross line, ->] (v') to (w'');
  \draw[cross line, ->] (v'') to (w);
  \draw[cross line, ->] (v) to (w'');
  \draw[cross line, ->] (w) to (z);
  \draw[cross line, ->] (v') to (w);
\end{tikzpicture}\qquad\qquad
\begin{tikzpicture}[scale=0.6,baseline={(current  bounding  box.center)}]
  \def\radius{3.7cm} 
  \node (h0A) at (60:\radius)   {$\circ$};
  \node (h0C) at (0:\radius)    {$\circ$};
  \node (h1B) at (-60:\radius)  {$\circ$};
  \node (h1A) at (-120:\radius) {$\circ$};
  \node (h1C) at (180:\radius)  {$\circ$};
  \node (h0B) at (120:\radius)  {$\circ$};
  
  \draw[->]
  (h0A) edge node[auto] {$\Id\times F_{u,v,v'',w'}$} (h0C)
  (h1B) edge node[right] {$F_{v'',w',w,z}\times\Id$} (h0C)
  (h1A) edge node[auto] {$\Id\times F_{u,v'',v',w}$} (h1B)
  (h1C) edge node[left] {$F_{v',w,w'',z}\times\Id$} (h1A)
  (h1C) edge node[auto] {$\Id\times F_{u,v',v,w''}$} (h0B)
  (h0B) edge node[auto] {$F_{v,w'',w',z}\times\Id$} (h0A);
\end{tikzpicture}
\]
This collection of data can be extended to a strictly unitary functor
$F\from\two^n\to\burn_K$, uniquely up to natural
isomorphism, so that $F_{u,v,v',w}=F_{u,v',w}^{-1}\circ_2 F_{u,v,w}$.
\end{lem}

\begin{defn}\label{def:total}
 Given a functor
$F\from\two^n \to \oddb$ and $\degh \from K \to \ZZ_2$, we construct a chain complex
denoted $\Tot_\degh(F)$, called the {\emph{totalization}} of the
functor $F$.  We usually suppress $\degh$ from notation when it is clear.  The underlying chain group of $\Tot_\degh(F)$ is
\[
\Tot_\degh(F) \;= \; \bigoplus_{v \in \two^n} \Abelianize_\degh(F(v)).
\]
We set the homological grading of the summand $\Abelianize_\degh(F(v))$ to
be $|v|$. The differential is given by defining the components
$\partial_{u,v}$ from $\Abelianize_\degh(F(u))$ to $\Abelianize_\degh(F(v))$ by
\[
\partial_{u,v}=\begin{cases}
(-1)^{s_{u,v}}\Abelianize_\degh(F(\phi_{u,v})) & \mbox{if } u\geqslant_1 v \\
0 &\mbox{otherwise.}
\end{cases}
\] 

We just write $\Tot$ for $\Tot_\degh$ when $\degh$ is the map that sends all elements of $K$ to the identity of $\ZZ_2$.
Similarly, we construct a chain complex $\Tot_K(F)$, called the $K$-totalization of $F$.  The underlying chain group of $\Tot_K(F)$ is
\[
\Tot_K(F) \;= \; \bigoplus_{v \in \two^n} \Abelianize_K(F(v)),
\] 
and the homological grading of the summand $\Abelianize_K(F(v))$ is $|v|$. The differential is given by defining the components
$\partial_{u,v}$ from $\Abelianize_K(F(u))$ to $\Abelianize_K(F(v))$ by
\[
\partial_{u,v}=\begin{cases}
(-1)^{s_{u,v}}\Abelianize_K(F(\phi_{u,v})) & \mbox{if } u\geqslant_1 v \\
0 &\mbox{otherwise.}
\end{cases}
\] 
\end{defn}

\subsection{External actions on Burnside functors}\label{subsec:external-3}
We will be especially interested in functors to the Burnside category that admit `extra symmetries,' as follows.  For a group $G$, let $BG$ denote the category with one object, and morphism set $G$.   The composition in $BG$ is given by $g\circ h= hg$.  Let $\mathrm{Cat}$ denote the category of small categories.

\begin{defn}\label{def:group-acts-on-category}
Let $G$ be a group and $\Cat$ a small category.  A group action of $G$ on $\Cat$ is a functor $\psi \from BG\to \mathrm{Cat}$, so that the object of $BG$ is sent to $\Cat$.  Alternatively, a group action of $G$ on $\Cat$ consists of a group action $\psi$ of $G$ on $\Ob(\Cat)$, along with an isomorphism of sets $\psi_g \col \Arr(x,y) \to \Arr(\psi_g x,\psi_g y)$ for each $g\in G$, compatible with composition of morphisms in $\Cat$ and so that $\psi_h\psi_g=\psi_{hg}$.  We further require that the group action preserves identity morphisms.  
\end{defn}

\begin{defn}\label{def:external-action}
Fix a Burnside functor $F \from \Cat \to \burn_K$, for $\Cat$ a small category.  Say there exists an action of $G$ by $\psi$ on $\Cat$.  An \emph{external action} on $F$ \emph{compatible with} $\psi$ consists of the following data.  In the following, to ease the notation, for an object $v$ of $\Cat$, and an element $g\in G$, we will write $gv$ for the object of $\Cat$ obtained by acting by $g$ on $v$, and similarly for morphisms of $\Cat$.
\begin{enumerate}[leftmargin=*]
\item\label{itm:34-1} A collection of $1$-isomorphisms
\[
\psi_{g,v} \from F(v) \to F(gv),
\]
in $\burn_K$, for all $g\in G$ and $v$ objects of $\Cat$.  We also require, for each $g,h\in G$ and $v$ an object of $\Cat$, that there exists a 2-morphism $\psi_{g,h,v}\from \psi_{gh,v}\to\psi_{g,hv}\circ \psi_{h,v}$ (note that if such a $2$-morphism exists, it is unique).

\item\label{itm:34-2} For each morphism $A\from x\to y$ in $\Cat$ and each $g\in G$, there is a 2-morphism, which is part of the data of an external action,
\[
\psi_{g,A} \from \psi_{g,y} \circ F(A) \to F(gA)\circ \psi_{g,x}.
\]
\end{enumerate}
The data are subject to the following conditions:
\begin{enumerate}[leftmargin=*,label=(E-\arabic*)]
\item \label{itm:e1} Let $A\from u\to v$ be a 1-morphism in $\Cat$, for objects $u,v$ of $\Cat$.  The $2$-morphism $\psi_{gh,A}$ is given by the composite:
\begin{align*}
\psi_{gh,v} \circ F(A) &\stackrel{\psi_{g,h,v}\circ_1\Id}{\longrightarrow} \psi_{g,hv}\psi_{h,v} \circ F(A) \stackrel{\Id\circ_1\psi_{h,A}}{\longrightarrow} \psi_{g,hv} \circ F(hA)\circ \psi_{h,u}\\
&\stackrel{\psi_{g,hA}\circ_1\Id}{\longrightarrow} F(ghA) \circ \psi_{g,hu}\psi_{h,u} \stackrel{\Id\circ_1\psi_{g,h,u}}{\longrightarrow} F(ghA)\circ \psi_{gh,u}.
\end{align*}
It is convenient to record relations such as this using certain schematics.  We represent 2-morphisms by double arrows and 1-morphisms by single arrows.  Then the above equation can be represented schematically by:
\[
\begin{tikzpicture}[scale=1.4,baseline={(current bounding box.center)}]
\node (a1) at (-.3,0) {$F(u)$};
\node (a2) at (1,0) {$F(hu)$};
\node (a3) at (2.3,0) {$F(ghu)$};
\node (b1) at (-.3,-1) {$F(v)$};
\node (b2) at (1,-1) {$F(hv)$};
\node (b3) at (2.3,-1) {$F(ghv)$};
\node (x1) at (1,1) {};
\node (y1) at (1,-2) {};
\draw[->] (a1) -- (a2);
\draw[->] (a2) -- (a3);
\draw[->] (b1) -- (b2);
\draw[->] (b2) -- (b3);
\draw[->] (a1) -- (b1);
\draw[->] (a2) -- (b2);
\draw[->] (a3) -- (b3);
\draw[double,->] (b1) -- (a2);
\draw[double,->] (b2) -- (a3);
\path[->] 

(a1) edge [bend left=80] node [left] {} (a3) ;

\path[->] 

(b1) edge [bend right=80] node [left] {} (b3) ;

\draw[double,->] (a2) -- (x1);

\draw[double,->] (y1) -- (b2);

\end{tikzpicture} 
\qquad \qquad = \qquad \qquad 
\begin{tikzpicture}[scale=1.4,baseline={(current bounding box.center)}]
\node (a1) at (0,0) {$F(u)$};
\node (a2) at (1.3,0) {$F(ghu)$};
\node (b1) at (0,-1) {$F(v)$};
\node (b2) at (1.3,-1) {$F(ghv)$};
\draw[->] (a1) -- (a2);
\draw[->] (b1) -- (b2);
\draw[->] (a1) -- (b1);
\draw[->] (a2) -- (b2);
\draw[double,->] (b1) -- (a2);

\end{tikzpicture} 
\]

The figure on the left-hand side represents a 2-morphism from $\psi_{gh,v}\circ F(A)$ to $F(ghA)\circ\psi_{gh,u}$, as follows.  The horizontal and curved 1-morphisms are of the form $\psi_{k,x}$ for $k\in G$ and $x$ an object of $\Cat$, while the vertical 1-morphisms are $A\from u \to v$, $hA\from hu\to hv$ and $ghA\from ghu\to ghv$.  Associated to each path, by traversing single arrows from $F(u)$ to $F(ghv)$, there is a 1-morphism $F(u)\to F(ghv)$.  For instance, $\psi_{gh,v}\circ F(A)$ is obtained from the arrow $F(u)\to F(v)$, composed with the curved arrow $v\to ghv$.  Each square in the rectangle records a 2-morphism as in item (\ref{itm:34-2}).  The semicircular regions have 2-morphisms as in item (\ref{itm:34-1}) above.  The 2-morphism represented by the figure is the composite of the $2$-morphisms in the squares and top and bottom regions.  The square on the other side of the schematic equality represents the 2-morphism $\psi_{gh,A}$.  For more on this notation, we refer to \cite[Section 2]{lauda-introduction}.
 
\item \label{itm:e2} Let $u,v,w$ be objects of $\Cat$ and let $A\from u\to v$ and $B\from v\to w$ be 1-morphisms in $\Cat$. We require that the following pentagon commutes (where three additional associators have been suppressed):

\[
\begin{tikzpicture}[scale=0.8,baseline={(current  bounding  box.center)}]
  \node (h0) at (0,0)   {\scriptsize $\psi_{g,w}\circ F(B)\circ F(A)$};
  \node (h1) at (2.5,2.5)    {\scriptsize $F(gB) \circ \psi_{g,v} \circ F(A)$};
  \node (h2) at (2*2.5,0)  {\scriptsize $F(gB)\circ F(gA) \circ \psi_{g,u}$};
  \node (g2) at (2*2.5,-2.5) {\scriptsize $F(gB\circ gA) \circ \psi_{g,u}$};
  \node (g1) at (0,-2.5)  {\scriptsize $\psi_{g,w} \circ F(B\circ A)$};
  
  \draw[->]
  (h0) edge node[auto] {\scriptsize $\psi_{g,B}\circ_1 \Id$} (h1)
  (h1) edge node[right] {\scriptsize $\Id \circ_1 \psi_{g,A}$} (h2)
  (h2) edge node[auto] {\scriptsize $F_{gA,gB} \circ_1 \Id$} (g2)
  (h0) edge node[left] {\scriptsize$\Id \circ_1 F_{A,B}$} (g1)
  (g1) edge node[auto] {\scriptsize $\psi_{g,B\circ A}$} (g2);
\end{tikzpicture}
\]

Schematically,
\[
\begin{tikzpicture}[scale=1.4,baseline={(current bounding box.center)}]
\node (a1) at (-0.3,0) {$F(u)$};
\node (a2) at (1,0) {$F(v)$};
\node (a3) at (2.3,0) {$F(w)$};
\node (b1) at (-0.3,-1) {$F(gu)$};
\node (b2) at (1,-1) {$F(gv)$};
\node (b3) at (2.3,-1) {$F(gw)$};
\draw[->] (a1) -- (a2);
\draw[->] (a2) -- (a3);
\draw[->] (b1) -- (b2);
\draw[->] (b2) -- (b3);
\draw[->] (a1) -- (b1);
\draw[->] (a2) -- (b2);
\draw[->] (a3) -- (b3);
\draw[double,->] (a2) -- (b1);
\draw[double,->] (a3) -- (b2);

\node (x1) at (1,1) {};
\node (y1) at (1,-2) {};
\path[->] 

(a1) edge [bend left=80] node [left] {} (a3) ;

\path[->] 

(b1) edge [bend right=80] node [left] {} (b3) ;

\draw[double,<-] (a2) -- (x1);

\draw[double,<-] (y1) -- (b2);
\end{tikzpicture}\qquad \qquad = \qquad \qquad
\begin{tikzpicture}[scale=1.4,baseline={(current bounding box.center)}]

\node (a1) at (0,0) {$F(u)$};
\node (a2) at (1.3,0) {$F(w)$};
\node (b1) at (0,-1) {$F(gu)$};
\node (b2) at (1.3,-1) {$F(gw)$};
\draw[->] (a1) -- (a2);
\draw[->] (b1) -- (b2);
\draw[->] (a1) -- (b1);
\draw[->] (a2) -- (b2);
\draw[double,->] (a2) -- (b1);
\end{tikzpicture}
\]
 
\end{enumerate}

\end{defn}

\begin{rmk}
One can view the condition \ref{itm:e1} of \ref{def:external-action} as stating that the `action' by $G$ on $F$ is compatible with multiplication in $G$, while \ref{itm:e2} says that the `action' of $G$ on $F$ is compatible with composition in the category $\Cat$.  
\end{rmk}

Note that $\two^{np}=(\two^n)^p$, and let $\ZZ_p$ act by cyclic permutation of the factors $\two^n$ of $\two^{np}$, as in Section \ref{subsec:periodic-links-kh}.  Let $F\from \two^{np}\to \burn_K$ be a Burnside functor with compatible external action.  The complex $\Tot_K(F)$, for $F\from \two^{np}\to \burn_K$ admitting an external $\ZZ_p$-action compatible with permutation of the coordinates, admits its own $\ZZ_p$-action as follows.  For each $u\in \Ob(\two^{np})$, let $\tau(u)=(-1)^{(\# \{i\leq n(p-1) \mid u_i=1\})(\# \{i >n(p-1)\mid u_i=1\})}\in \{\pm 1\}$.  Define, for $g\in \ZZ_p$, $g_*\from \Tot_K(F)\to \Tot_K(F)$ as follows.  For $g=1_p$ the generator of $\ZZ_p$ (written additively), and  $x\in F(v)$, set $g_*(x)=\tau(v)\sigma(g,x)g(x)$ and extend linearly, where $g(x)$ and $\sigma(g,x)$ are defined as follows.  Let $\psi_{g,v}$ be the correspondence from $F(v)$ to $F(gv)$ as in Definition \ref{def:external-action} with source map $s$, target map $t$, and decoration $\sigma\from \psi_{g,v}\to K$, and set $\sigma(g,x)=\sigma(s^{-1}(x))$, and $g(x)=t(s^{-1}(x))$.  The action of general $g=\ell(1_p)\in \ZZ_p$ is defined by $(1_p)_*^\ell$.

It is a direct but tedious check to see that $g_*$ is a chain map for all $g\in \ZZ_p$, and a similar check gives $(1_p)_*^p=\Id$, so we obtain: 

\begin{lem}\label{lem:total-equiv}
Let $\ZZ_p$ act on $\two^{np}$ as above.  For $F\from \two^{np}\to \burn_K$ a functor with compatible external action, the complex $\Tot_{\ZZ_p}(F)$ is naturally a chain complex over $\ZZ[\ZZ_p\times K]$.  It follows that, for each homomorphism $\degh\from K \to \ZZ_2$, the complex $\Tot_\degh(F)$ is a chain complex over $\ZZ[\ZZ_p]$.  
	\end{lem}

For a small category $\Cat$ with an action by a group $G$, along with functors $F_1,F_2\from \Cat \to \burn_K$ with compatible external actions $\psi_1,\psi_2$ respectively, we say that $F_1$ and $F_2$ are \emph{$G$-equivariantly naturally isomorphic} if there is a functor $J\from \Cat\times \two \to \burn_K$ with external action (where the action on $\two$ is trivial) so that $J|_{\Cat\times 0}$ is $F_1$ and $J|_{\Cat\times 1}$ is $F_2$, and so that $J(\Id_v \times \phi_{1,0})$ is an isomorphism for all objects $v$ of $\Cat$.  
\begin{lem}\label{lem:212-equivariant}
Let $\ZZ_p$ act on $(\two^n)^p$ by cyclic permutation of the factors of $\two^n$.  Consider the data $F$ as in Lemma \ref{lem:212} along with the following data:
\begin{enumerate}
\item\label{itm:hyp1} For each object $v$ of $(\two^n)^p$ and element $g$ of $\ZZ_p$, a $1$-isomorphism (in $\burn_K$) $\psi_{g,v} \from F(v) \to F(gv)$.  We also require, for all $g,h\in G$ and $v$ an object of $(\two^n)^p$, a 2-morphism $\alpha_{g,h,v}\from \psi_{gh,v}\to\psi_{g,hv}\psi_{h,v}$ (note that if such a 2-morphism exists, it is unique).  
\item\label{itm:hyp2} For each $g\in \ZZ_p$ and pair $u\geqslant_1v \in \two^{np}$, a $2$-morphism $\psi_{g,u,v} \from \psi_{g,v} \circ F(\phi_{u,v}) \to F(\phi_{gu,gv})\circ \psi_{g,u}$.
\end{enumerate} 
Assume that the data satisfies the following conditions:
\begin{enumerate}[leftmargin=*,label=(E-\arabic*$'$)]
\item \label{itm:e1'} For any $u\geqslant_1v\in \two^{np}$ and for all $g,h\in G$, we have
\[
\psi_{gh,u,v}=\alpha_{g,h,u}^{-1}\circ_2(\psi_{g,hu,hv}\circ \Id) \circ_2 (\Id \circ \psi_{h,u,v})\circ_2 \alpha_{g,h,v}.
\]
  That is, the data $(\psi_{g,v},\psi_{g,u,v},\alpha_{g,h,v})$ satisfy \ref{itm:e1} for length-$1$ morphisms.
\item \label{itm:e2'} For any objects $u,v$ of $(\two^n)^p$, write $F(\phi_{u,v})=A_{u,v}$ to ease the notation.  For objects $u\geqslant_1 v,v'\geqslant_1 w$ of $(\two^n)^p$ and $g\in G$, the following hexagon commutes:
\[
\begin{tikzpicture}[scale=1.1,baseline={(current  bounding  box.center)}]
  \def\radius{3.7cm} 
  \node (a3) at (60:\radius)   {\scriptsize $A_{gv,gw}\circ A_{gu,gv}\circ \psi_{g,u}$};
  \node (a4) at (0:\radius)    {\scriptsize $A_{gv',gw}\circ A_{gu,gv'}\circ \psi_{g,u}$};
  \node (b2) at (-60:\radius)  {\scriptsize $A_{gv',gw}\circ \psi_{g,v'}\circ A_{u,v'}$};
  \node (b1) at (-120:\radius) {\scriptsize $\psi_{g,w} \circ A_{v',w} \circ A_{u,v'} $};
  \node (a1) at (180:\radius)  {\scriptsize $\psi_{g,w}\circ A_{v,w}\circ A_{u,v}$};
  \node (a2) at (120:\radius)  {\scriptsize $A_{gv,gw} \circ \psi_{g,v}\circ A_{u,v}$};
  
  \draw[->]
  (a1) edge node[auto] {\scriptsize $\psi_{g,v,w}\circ\Id $} (a2)
  (a2) edge node[auto] {\scriptsize $\Id \circ \psi_{g,u,v}$} (a3)
  (a3) edge node[auto] {\scriptsize $F_{gu,gv,gv',gw}\circ_2 \Id$} (a4)
  (a1) edge node[auto] {\scriptsize  $\Id \circ F_{u,v,v',w}$} (b1)
  (b1) edge node[auto] {\scriptsize $\psi_{g,v',w}\circ\Id$} (b2)
  (b2) edge node[auto] {\scriptsize $\Id\circ \psi_{g,u,v'}$ } (a4);
\end{tikzpicture}
\]
\end{enumerate}

Then this collection of data extends to a strictly unitary functor $F\from (\two^{n})^p\to \burn_K$ admitting an external $\ZZ_p$-action, which is unique up to $\ZZ_p$-equivariant natural isomorphism.
\end{lem}
\begin{proof}
We will need to briefly describe the argument for Lemma \ref{lem:212}, which is identical to that of Proposition 4.3 \cite{lls2}.  The functor $F$ constructed in Lemma \ref{lem:212} is defined by, for each $\phi_{u,v}$, choosing a sequence $u\geqslant_1 u_1\dots\geqslant_1 u_{i-1}\geqslant_1 u_i=v$ and then setting $F(\phi_{u,v})=F(\phi_{u_{i-1},v})\circ \cdots \circ F(\phi_{u,u_1})$.  For each $u\geq_i v\geq_j w$, we need a $2$-morphism $F_{u,v,w}\from F(\phi_{v,w})\circ F(\phi_{u,v}) \to F(\phi_{u,w})$.  Suppose that the sequence defining $F(\phi_{u,v})$ is $u\geqslant_1 u_1\cdots \geqslant_1 u_i=v$, that defining $F(\phi_{v,w})$ is $v\geqslant_1 v_1,\cdots \geqslant_1 v_j=w$ and that defining $F(\phi_{u,w})$ is $u\geqslant_1 u_1'\cdots \geqslant_1 u_{i+j}'=w$.  We then need a bijection of decorated sets
\[
(F(\phi_{v_{j-1},w})\circ \cdots \circ F(\phi_{v,v_1})) \circ (F(\phi_{u_{i-1},v})\circ \cdots \circ F(\phi_{u,u_1}))\to F(\phi_{u_{i+j-1}',w})\cdots \circ F(\phi_{u,u_1'})
\]
Such a bijection is obtained by taking a composition of bijections of the form $\Id \circ F_{x,y,y',z} \circ \Id$ as in the statement of Lemma \ref{lem:212}.  Condition (\ref{cond:c2}) of Lemma \ref{lem:212} guarantees that the bijection of decorated sets thus constructed is independent of the choices of the $F_{x,y,y',z}$.

To simplify the notation for the proof, for any objects $u\geq v$ of $(\two^n)^p$, write $A_{u,v}=F(\phi_{u,v})$.  We need to define $2$-isomorphisms in $\burn_K$, $\psi_{g,u,v} \from \psi_{g,v}\circ A_{u,v} \to A_{gu,gv} \circ \psi_{g,u}$ for all $u\geq v$ so that \ref{itm:e1} and \ref{itm:e2} hold.  Recall that in the construction of $F$, for each $u\geq_i v$, we selected a sequence $u\geqslant_1 u_1\cdots \geqslant_1 u_i=v$, and set $A_{u,v}=A_{u_{i-1},v}\circ \cdots \circ A_{u,u_1}$.  We have a diagram, where the solid arrows represent 2-morphisms, and the dashed arrow has not yet been defined:

\[
\begin{tikzpicture}[baseline={(current  bounding  box.center)},xscale=7,yscale=1] 
\node (a0) at (0,1) {\small $\psi_{g,v}\circ A_{u_{i-1},v}\circ \cdots \circ A_{u,u_1}$};
\node (a1) at (0.75,1) {\small  $A_{gu_{i-1},gv}\circ \psi_{g,u_{i-1}}\circ \cdots \circ A_{u,u_1}$};
\node (a2) at (1.25,1) {\small $ \dots $};
\node (a3) at (1.725,1) {\small $ A_{gu_{i-1},gv}\circ \cdots \circ A_{gu,gu_1}\circ \psi_{g,u}$};
\node (b0) at (0,0) {\small $\psi_{g,v} \circ A_{u,v}$};
\node (b1) at (1.725,0) {\small $A_{gu,gv}\circ \psi_{g,u}$};

\draw[->] (a0) -- (a1) node[pos=0.5,anchor=north] {};
\draw[->] (a1) -- (a2) node[pos=0.5,anchor=north] {};

\draw[->] (a2) -- (a3) node[pos=0.5,anchor=north] {};

 \draw[->] (a0) -- (b0) node[pos=0.2,anchor=east] {};
\draw[->] (a3) -- (b1) node[pos=0.5,anchor=south east]{};
\draw[->,dashed] (b0) -- (b1) {};

\end{tikzpicture}
\]
\noindent The vertical $2$-morphisms are given by the construction of $F$: the left one is part of the definition, and the right one arises from a sequence of bijections of the form $\Id \circ F_{x,y,y',z} \circ \Id$, as in the proof of Lemma \ref{lem:212}.  Although the decomposition of the right vertical 2-morphism into the $F_{x,y,y',z}$ is not well-defined, the resulting composite is well-defined.  The horizonal 2-morphisms in the top row are instances of the maps $\psi_{g,x,y}$ as in (\ref{itm:hyp2}) of the lemma, for objects $x\geqslant_1 y$ of $(\two^n)^p$.
We define the 2-morphism $\psi_{g,u,v}$ (taking the place of the dashed arrow) to make the diagram commutative.  

For checking that \ref{itm:e1} holds, we draw the following schematic figures, interpreted as in Definition \ref{def:external-action}, which the determined reader can translate into equations.  Let us set up some notation.  Fix $g,h\in \ZZ_p$.  Say that in the definition of $F$, we have selected the sequences $u\geqslant_1 u_1\geqslant_1\dots \geqslant_1 u_i = v$ and $hu\geqslant_1 u_1'\geqslant_1\dots \geqslant_1 u_i'=hv$ and $ghu\geqslant_1 u_1''\geqslant_1\dots \geqslant_1 u_i''=ghv$ to define $A_{u,v},A_{hu,hv},A_{ghu,ghv}$, respectively.  Consider the 2-morphism 
\[\mathbb{E}_1'\from \psi_{g,hv}\circ \psi_{h,v}\circ A_{u_{i-1},v}\circ \dots \circ A_{u,u_1}\to 
A_{u_{i-1}'',ghv}\circ\dots \circ A_{u,u_{1}}\circ \psi_{g,hu}\circ\psi_{h,u}
\]
defined as the composite:
\begin{equation}\label{eq:composite-1}
\mathbb{E}_1'=\begin{tikzpicture}[baseline={(current  bounding  box.center)},xscale=3.4,yscale=1.4]
\node (u0) at (0,0) {$ F(u)$};
\node (u1) at (0,-1) {$ F(u_1)$};
\node (u2) at (0,-1.7) {$\vdots$};
\node (u3) at (0,-2.4) {$F(u_{i-1})$};
\node (uf) at (0,-3.4) {$F(u_i)=F(v)$};

\node (b0) at (1,0) {$ F(hu)$};
\node (b1) at (1,-1) {$ F(hu_1)$};
\node (b2) at (1,-1.7) {$\vdots$};
\node (b3) at (1,-2.4) {$F(hu_{i-1})$};
\node (bf) at (1,-3.4) {$F(hu_i)=F(hv)$};

\node (c0) at (2,0) {$ F(hu)$};
\node (c1) at (2,-1) {$ F(u_1')$};
\node (c2) at (2,-1.7) {$\vdots$};
\node (c3) at (2,-2.4) {$F(u_{i-1}')$};
\node (cf) at (2,-3.4) {$F(u_i')=F(hv)$};

\node (d0) at (3,0) {$ F(ghu)$};
\node (d1) at (3,-1) {$ F(gu_1')$};
\node (d2) at (3,-1.7) {$\vdots$};
\node (d3) at (3,-2.4) {$F(gu_{i-1}')$};
\node (df) at (3,-3.4) {$F(ghu_i)=F(ghv)$};

\node (f0) at (4,0) {$ F(ghu)$};
\node (f1) at (4,-1) {$ F(u_1'')$};
\node (f2) at (4,-1.7) {$\vdots$};
\node (f3) at (4,-2.4) {$F(u_{i-1}'')$};
\node (ff) at (4,-3.4) {$F(u_i'')=F(ghv)$};

\draw[->] (u0) -- (u1) node[pos=0.5,anchor=north] {};
\draw[->] (u1) -- (u2) node[pos=0.5,anchor=north] {};
\draw[->] (u2) -- (u3) node[pos=0.5,anchor=north] {};
\draw[->] (u3) -- (uf) node[pos=0.5,anchor=north] {};

\draw[->] (u0) -- (b0) node[pos=0.5,anchor=north] {};
\draw[->] (u1) -- (b1) node[pos=0.5,anchor=north] {};
\draw[->] (u3) -- (b3) node[pos=0.5,anchor=north] {};
\draw[->] (uf) -- (bf) node[pos=0.5,anchor=north] {};

\draw[->] (c0) -- (d0) node[pos=0.5,anchor=north] {};
\draw[->] (c1) -- (d1) node[pos=0.5,anchor=north] {};
\draw[->] (c3) -- (d3) node[pos=0.5,anchor=north] {};
\draw[->] (cf) -- (df) node[pos=0.5,anchor=north] {};

\draw[double,->] (uf) -- (b3) node[pos=0.5,anchor=north] {};
\draw[double,->] (u1) -- (b0) node[pos=0.5,anchor=north] {};

\draw[->] (b0) -- (b1) node[pos=0.5,anchor=north] {};
\draw[->] (b1) -- (b2) node[pos=0.5,anchor=north]{};
\draw[->] (b2) -- (b3) node[pos=0.5,anchor=north] {};
\draw[->] (b3) -- (bf) node[pos=0.5,anchor=north] {};

\draw[double,->] (cf) -- (d3) node[pos=0.5,anchor=north] {};
\draw[double,->] (c1) -- (d0) node[pos=0.5,anchor=north] {};

\draw[->] (c0) -- (c1) node[pos=0.5,anchor=north] {};
\draw[->] (c1) -- (c2) node[pos=0.5,anchor=north]  {};
\draw[->] (c2) -- (c3) node[pos=0.5,anchor=north] {};
\draw[->] (c3) -- (cf) node[pos=0.5,anchor=north] {};

\draw[->] (d0) -- (d1) node[pos=0.5,anchor=north] {};
\draw[->] (d1) -- (d2) node[pos=0.5,anchor=north]  {};
\draw[->] (d2) -- (d3) node[pos=0.5,anchor=north] {};
\draw[->] (d3) -- (df) node[pos=0.5,anchor=north] {};

\draw[->] (f0) -- (f1) node[pos=0.5,anchor=north] {};
\draw[->] (f1) -- (f2) node[pos=0.5,anchor=north]  {};
\draw[->] (f2) -- (f3) node[pos=0.5,anchor=north] {};
\draw[->] (f3) -- (ff) node[pos=0.5,anchor=north] {};

\draw[double,->] (b2) -- (c2) node[pos=.5,anchor=south] {};
\draw[double,->] (d2) -- (f2);

\end{tikzpicture}\end{equation}

Each of the diagonal 2-morphisms is of the form $\psi_{k,x,y}$ for some $k\in G$ and $x\geqslant_1 y$ objects of $(\two^n)^p$.  The horizontal 2-morphisms are composites of 2-morphisms of the form $\Id \circ_1 F_{x,y,y',z}\circ_1 \Id$ for $x\geqslant_1 y,y'\geqslant_1 z$.  

Define $\mathbb{E}_1=\alpha_{g,h,u}^{-1}\circ_1\mathbb{E}_1'\circ_1 \alpha_{g,h,v}$.  To verify \ref{itm:e1}, we need to show that $\mathbb{E}_1$ is the same as the $2$-morphism $\mathbb{E}_2$ defined by:
\begin{equation}\label{eq:composite-2}
\mathbb{E}_2=\begin{tikzpicture}[baseline={(current  bounding  box.center)},xscale=3.5,yscale=1.4]

\node (c0) at (2,0) {$ F(u)$};
\node (c1) at (2,-1) {$ F(u_1)$};
\node (c2) at (2,-1.7) {$\vdots$};
\node (c3) at (2,-2.4) {$F(u_{i-1})$};
\node (cf) at (2,-3.4) {$F(u_i)=F(v)$};

\node (d0) at (3,0) {$ F(ghu)$};
\node (d1) at (3,-1) {$ F(ghu_1)$};
\node (d2) at (3,-1.7) {$\vdots$};
\node (d3) at (3,-2.4) {$F(ghu_{i-1})$};
\node (df) at (3,-3.4) {$F(ghu_i)=F(ghv)$};

\node (f0) at (4,0) {$ F(ghu)$};
\node (f1) at (4,-1) {$ F(u_1'')$};
\node (f2) at (4,-1.7) {$\vdots$};
\node (f3) at (4,-2.4) {$F(u_{i-1}'')$};
\node (ff) at (4,-3.4) {$F(u_i'')=F(ghv)$};

\draw[double,->] (cf) -- (d3) node[pos=0.5,anchor=north] {};
\draw[double,->] (c1) -- (d0) node[pos=0.5,anchor=north] {};

\draw[->] (c0) -- (c1) node[pos=0.5,anchor=north] {};
\draw[->] (c1) -- (c2) node[pos=0.5,anchor=north]  {};
\draw[->] (c2) -- (c3) node[pos=0.5,anchor=north] {};
\draw[->] (c3) -- (cf) node[pos=0.5,anchor=north] {};

\draw[->] (d0) -- (d1) node[pos=0.5,anchor=north] {};
\draw[->] (d1) -- (d2) node[pos=0.5,anchor=north]  {};
\draw[->] (d2) -- (d3) node[pos=0.5,anchor=north] {};
\draw[->] (d3) -- (df) node[pos=0.5,anchor=north] {};

\draw[->] (f0) -- (f1) node[pos=0.5,anchor=north] {};
\draw[->] (f1) -- (f2) node[pos=0.5,anchor=north]  {};
\draw[->] (f2) -- (f3) node[pos=0.5,anchor=north] {};
\draw[->] (f3) -- (ff) node[pos=0.5,anchor=north] {};

\draw[->] (c0) -- (d0) node[pos=0.5,anchor=north] {};
\draw[->] (c1) -- (d1) node[pos=0.5,anchor=north] {};
\draw[->] (c3) -- (d3) node[pos=0.5,anchor=north] {};
\draw[->] (cf) -- (df) node[pos=0.5,anchor=north] {};

\draw[double,->] (d2) -- (f2);

\end{tikzpicture}
\end{equation}

Again, the diagonal 2-morphisms are of the form $\psi_{k,x,y}$ for $k\in G$ and $x\geqslant_1 y$ objects of $(\two^n)^p$.  The horizontal $2$-morphism comes from composing several of the $F_{x,y,y',z}$ maps.

  We first apply the hypothesis \ref{itm:e1'} to express $\mathbb{E}_2$ as: 
\begin{equation}\label{eq:composite-3}
\begin{tikzpicture}[baseline={(current  bounding  box.center)},xscale=3.5,yscale=1.5]
\node (c0) at (2,0) {$ F(u)$};
\node (c1) at (2,-1) {$ F(u_1)$};
\node (c2) at (2,-1.7) {$\vdots$};
\node (c3) at (2,-2.4) {$F(u_{i-1})$};
\node (cf) at (2,-3.4) {$F(u_i)=F(v)$};

\node (d0) at (3,0) {$ F(hu)$};
\node (d1) at (3,-1) {$ F(hu_1)$};
\node (d2) at (3,-1.7) {$\vdots$};
\node (d3) at (3,-2.4) {$F(hu_{i-1})$};
\node (df) at (3,-3.4) {$F(hu_i)=F(hv)$};

\node (dt) at (3,.8) {};
\node (db) at (3,-4.2) {};

\node (e0) at (4,0) {$ F(ghu)$};
\node (e1) at (4,-1) {$ F(ghu_1)$};
\node (e2) at (4,-1.7) {$\vdots$};
\node (e3) at (4,-2.4) {$F(ghu_{i-1})$};
\node (ef) at (4,-3.4) {$F(ghu_i)=F(ghv)$};

\node (g0) at (5,0) {$ F(ghu)$};
\node (g1) at (5,-1) {$ F(u_1'')$};
\node (g2) at (5,-1.7) {$\vdots$};
\node (g3) at (5,-2.4) {$F(u_{i-1}'')$};
\node (gf) at (5,-3.4) {$F(ghv)$};

\draw[->] (c0) -- (c1) node[pos=0.5,anchor=north] {};
\draw[->] (c1) -- (c2) node[pos=0.5,anchor=north]  {};
\draw[->] (c2) -- (c3) node[pos=0.5,anchor=north] {};
\draw[->] (c3) -- (cf) node[pos=0.5,anchor=north] {};

\draw[->] (d0) -- (d1) node[pos=0.5,anchor=north] {};
\draw[->] (d1) -- (d2) node[pos=0.5,anchor=north]  {};
\draw[->] (d2) -- (d3) node[pos=0.5,anchor=north] {};
\draw[->] (d3) -- (df) node[pos=0.5,anchor=north] {};

\draw[->] (e0) -- (e1) node[pos=0.5,anchor=north] {};
\draw[->] (e1) -- (e2) node[pos=0.5,anchor=north]  {};
\draw[->] (e2) -- (e3) node[pos=0.5,anchor=north] {};
\draw[->] (e3) -- (ef) node[pos=0.5,anchor=north] {};

\draw[->] (g0) -- (g1) node[pos=0.5,anchor=north] {};
\draw[->] (g1) -- (g2) node[pos=0.5,anchor=north]  {};
\draw[->] (g2) -- (g3) node[pos=0.5,anchor=north] {};
\draw[->] (g3) -- (gf) node[pos=0.5,anchor=north] {};

\draw[->] (c0) -- (d0) node[pos=0.5,anchor=north] {};
\draw[->] (c1) -- (d1) node[pos=0.5,anchor=north] {};
\draw[->] (c3) -- (d3) node[pos=0.5,anchor=north] {};
\draw[->] (cf) -- (df) node[pos=0.5,anchor=north] {};

\draw[->] (d0) -- (e0) node[pos=0.5,anchor=north] {};
\draw[->] (d1) -- (e1) node[pos=0.5,anchor=north] {};
\draw[->] (d3) -- (e3) node[pos=0.5,anchor=north] {};
\draw[->] (df) -- (ef) node[pos=0.5,anchor=north] {};

\draw[double,->] (e2) -- (g2) node[pos=0.5,anchor=north] {};

\draw[double,->] (cf) -- (d3) node[pos=0.5,anchor=north] {};
\draw[double,->] (c1) -- (d0) node[pos=0.5,anchor=north] {};

\draw[double,->] (df) -- (e3) node[pos=0.5,anchor=north] {};
\draw[double,->] (d1) -- (e0) node[pos=0.5,anchor=north] {};

\path[->] (c0) edge [bend left=80] node [left] {} (e0) ;

\path[->] 

(cf) edge [bend right=80] node [left] {} (ef) ;

\draw[double,->] (d0) -- (dt);

\draw[double,->] (db) -- (df);

\end{tikzpicture}
\end{equation}

Let $\mathbb{E}_2'=\alpha_{g,h,u}\circ_1\mathbb{E}_2\circ_1 \alpha_{g,h,v}^{-1}$.  To show $\mathbb{E}_1=\mathbb{E}_2$ it then suffices to check $\mathbb{E}_1'=\mathbb{E}_2'$.  Observe that the first two columns (of objects) along with the 1-morphisms among these objects, and the two morphisms between these columns are the same in (\ref{eq:composite-1}) and (\ref{eq:composite-3}).   Thus, it suffices to check that the 2-morphisms $\mathbb{E}_1''$ and $\mathbb{E}_2''$ depicted below agree.  Set

\begin{equation}\label{eq:composite-4}
\mathbb{E}_1''=\begin{tikzpicture}[baseline={(current  bounding  box.center)},xscale=3.5,yscale=1.4]

\node (b0) at (1,0) {$ F(hu)$};
\node (b1) at (1,-1) {$ F(hu_1)$};
\node (b2) at (1,-1.7) {$\vdots$};
\node (b3) at (1,-2.4) {$F(hu_{i-1})$};
\node (bf) at (1,-3.4) {$F(hu_i)=F(hv)$};

\node (c0) at (2,0) {$ F(hu)$};
\node (c1) at (2,-1) {$ F(u_1')$};
\node (c2) at (2,-1.7) {$\vdots$};
\node (c3) at (2,-2.4) {$F(u_{i-1}')$};
\node (cf) at (2,-3.4) {$F(u_i')=F(hv)$};

\node (d0) at (3,0) {$ F(ghu)$};
\node (d1) at (3,-1) {$ F(gu_1')$};
\node (d2) at (3,-1.7) {$\vdots$};
\node (d3) at (3,-2.4) {$F(gu_{i-1}')$};
\node (df) at (3,-3.4) {$F(ghu_i)=F(ghv)$};

\node (f0) at (4,0) {$ F(ghu)$};
\node (f1) at (4,-1) {$ F(u_1'')$};
\node (f2) at (4,-1.7) {$\vdots$};
\node (f3) at (4,-2.4) {$F(u_{i-1}'')$};
\node (ff) at (4,-3.4) {$F(u_i'')=F(ghv)$};

\draw[->] (c0) -- (d0) node[pos=0.5,anchor=north] {};
\draw[->] (c1) -- (d1) node[pos=0.5,anchor=north] {};
\draw[->] (c3) -- (d3) node[pos=0.5,anchor=north] {};
\draw[->] (cf) -- (df) node[pos=0.5,anchor=north] {};

\draw[->] (c0) -- (d0) node[pos=0.5,anchor=north] {};
\draw[->] (c1) -- (d1) node[pos=0.5,anchor=north] {};
\draw[->] (c3) -- (d3) node[pos=0.5,anchor=north] {};
\draw[->] (cf) -- (df) node[pos=0.5,anchor=north] {};

\draw[->] (b0) -- (b1) node[pos=0.5,anchor=north] {};
\draw[->] (b1) -- (b2) node[pos=0.5,anchor=north]{};
\draw[->] (b2) -- (b3) node[pos=0.5,anchor=north] {};
\draw[->] (b3) -- (bf) node[pos=0.5,anchor=north] {};

\draw[double,->] (cf) -- (d3) node[pos=0.5,anchor=north] {};
\draw[double,->] (c1) -- (d0) node[pos=0.5,anchor=north] {};

\draw[->] (c0) -- (c1) node[pos=0.5,anchor=north] {};
\draw[->] (c1) -- (c2) node[pos=0.5,anchor=north]  {};
\draw[->] (c2) -- (c3) node[pos=0.5,anchor=north] {};
\draw[->] (c3) -- (cf) node[pos=0.5,anchor=north] {};

\draw[->] (d0) -- (d1) node[pos=0.5,anchor=north] {};
\draw[->] (d1) -- (d2) node[pos=0.5,anchor=north]  {};
\draw[->] (d2) -- (d3) node[pos=0.5,anchor=north] {};
\draw[->] (d3) -- (df) node[pos=0.5,anchor=north] {};

\draw[->] (f0) -- (f1) node[pos=0.5,anchor=north] {};
\draw[->] (f1) -- (f2) node[pos=0.5,anchor=north]  {};
\draw[->] (f2) -- (f3) node[pos=0.5,anchor=north] {};
\draw[->] (f3) -- (ff) node[pos=0.5,anchor=north] {};

\draw[double,->] (b2) -- (c2) node[pos=.5,anchor=south] {};
\draw[double,->] (d2) -- (f2);

\end{tikzpicture}\end{equation}

and set
\begin{equation}\label{eq:composite-5}
\mathbb{E}_2''=\begin{tikzpicture}[baseline={(current  bounding  box.center)},xscale=3.5,yscale=1.4]

\node (d0) at (3,0) {$ F(hu)$};
\node (d1) at (3,-1) {$ F(hu_1)$};
\node (d2) at (3,-1.7) {$\vdots$};
\node (d3) at (3,-2.4) {$F(hu_{i-1})$};
\node (df) at (3,-3.4) {$F(hu_i)=F(hv)$};

\node (dt) at (3,.8) {};
\node (db) at (3,-4.2) {};

\node (e0) at (4,0) {$ F(ghu)$};
\node (e1) at (4,-1) {$ F(ghu_1)$};
\node (e2) at (4,-1.7) {$\vdots$};
\node (e3) at (4,-2.4) {$F(ghu_{i-1})$};
\node (ef) at (4,-3.4) {$F(ghu_i)=F(ghv)$};

\node (g0) at (5,0) {$ F(ghu)$};
\node (g1) at (5,-1) {$ F(u_1'')$};
\node (g2) at (5,-1.7) {$\vdots$};
\node (g3) at (5,-2.4) {$F(u_{i-1}'')$};
\node (gf) at (5,-3.4) {$F(ghv)$};

\draw[->] (d0) -- (d1) node[pos=0.5,anchor=north] {};
\draw[->] (d1) -- (d2) node[pos=0.5,anchor=north]  {};
\draw[->] (d2) -- (d3) node[pos=0.5,anchor=north] {};
\draw[->] (d3) -- (df) node[pos=0.5,anchor=north] {};

\draw[->] (e0) -- (e1) node[pos=0.5,anchor=north] {};
\draw[->] (e1) -- (e2) node[pos=0.5,anchor=north]  {};
\draw[->] (e2) -- (e3) node[pos=0.5,anchor=north] {};
\draw[->] (e3) -- (ef) node[pos=0.5,anchor=north] {};

\draw[->] (g0) -- (g1) node[pos=0.5,anchor=north] {};
\draw[->] (g1) -- (g2) node[pos=0.5,anchor=north]  {};
\draw[->] (g2) -- (g3) node[pos=0.5,anchor=north] {};
\draw[->] (g3) -- (gf) node[pos=0.5,anchor=north] {};

\draw[->] (d0) -- (e0) node[pos=0.5,anchor=north] {};
\draw[->] (d1) -- (e1) node[pos=0.5,anchor=north] {};
\draw[->] (d3) -- (e3) node[pos=0.5,anchor=north] {};
\draw[->] (df) -- (ef) node[pos=0.5,anchor=north] {};

\draw[double,->] (e2) -- (g2) node[pos=0.5,anchor=north] {};

\draw[double,->] (df) -- (e3) node[pos=0.5,anchor=north] {};
\draw[double,->] (d1) -- (e0) node[pos=0.5,anchor=north] {};

\end{tikzpicture}
\end{equation}

In fact, we may assume without loss of generality that $(u_j)_{j=1,\ldots,i}=(u''_j)_{j=1,\ldots,i}$, and we do so for the rest of the proof of \ref{itm:e1}.  Write $\mathbb{E}_2''((u_j))$ and $\mathbb{E}_1''((u_j),(u_j'))$ to illustrate the dependence on the sequences $(u_j)$ and $(u_j')$.  Note that
\begin{equation}\label{eq:d2}
\mathbb{E}_1''((u_j),(u_j'))=\Phi_{(gu_j'),(ghu_j)}\circ_2\mathbb{E}_2''((h^{-1}u_j'))\circ_2\Phi_{(hu_j),(u_j')},
\end{equation}
where, for sequences $(v_j)_{j=0,\ldots,i}$ and $(v_j')_{j=0,\ldots i}$ with $v_j\geqslant_1 v_{j+1}$ and $v_j'\geqslant_1 v_{j+1}'$ for all $j$ and so that $v_0=v'_0$ and $v_i=v'_i$, the term $\Phi_{(v_j),(v_j')}$ denotes the $2$-morphism 
\[
A_{v_{i-1},v_i}\circ A_{v_{i-2},v_{i-1}}\circ \dots \circ A_{v_0,v_1}\to A_{v_{i-1}',v_i'}\circ A_{v_{i-2}',v_{i-1}'}\circ \dots \circ A_{v_0',v_1'}
\]
obtained as a composite of the maps $F_{x,y,y',z}$ for $x\geqslant_1 y,y'\geqslant_1 z$.  

  We claim that if $(u_j)$ and $(h^{-1}u_j')$ differ at a single entry, then $\mathbb{E}_1''((u_j),(u_j'))$ and $\mathbb{E}_2''((u_j))$ are the same.  Assuming this, let us check that $\mathbb{E}_1''((u_j),(u_j'))$ is independent of $(u_j')$.    
We show that if $(u_j')$ and $(v_j')$ differ in only one entry, then $\mathbb{E}_1''((u_j),(v_j'))=\mathbb{E}_1''((u_j),(u_j'))$.  It then follows by induction that $\mathbb{E}_1''((u_j),(u_j'))$ is independent of $u_j'$, as needed.

To see that $\mathbb{E}_1''((u_j),(v_j'))=\mathbb{E}_1''((u_j),(u_j'))$ for $u_j,u_j'$ and $v_j'$ as above, we use (\ref{eq:d2}), from which we need
\[
\Phi_{(gu_j'),(ghu_j)}\circ_2\mathbb{E}_2''((h^{-1}u_j'))\circ_2\Phi_{(hu_j),(u_j')}=\Phi_{(gv_j'),(ghu_j)}\circ_2\mathbb{E}_2''((h^{-1}v_j'))\circ_2\Phi_{(hu_j),(v_j')}.
\]  
Rearranging, that is equivalent to
\[
\Phi_{(gu_j'),(gv_j')}\circ_2\mathbb{E}_2''((h^{-1}u_j'))\circ_2 \Phi_{(v_j'),(u_j')}=\mathbb{E}_2''((h^{-1}v_j')),
\]
but that is just the claim $\mathbb{E}_1''((h^{-1}v_j'),(u_j'))=\mathbb{E}_2''((h^{-1}v_j'))$, which is a case of our assumption, since $(v_j')$ and $(u_j')$ differ in only a single entry.
That is, it suffices to prove that if $(u_j)$ and $(h^{-1}u_j')$ differ at a single entry then $\mathbb{E}_2''((u_j))$ and $\mathbb{E}_1''((u_j),(u_j'))$ are the same.

It is enough to consider the case $i=2$.  In this case, we must check that the 2-morphisms represented by the diagrams below agree.

\[
\begin{tikzpicture}[baseline={(current bounding box.center)},scale=1.8]=
\node (x0) at (0,0) {$F(hu)$};
\node (x1) at (2,0) {$F(ghu)$};
\node (y0) at (0,-1) {$F(hu_1)$};
\node (z0) at (0,-2) {$F(hv)$};
\node (y1) at (1,-1) {$F(u_1')$};
\node (y2) at (3,-1) {$F(gu_1')$};
\node (y3) at (5,-1) {$F(ghu_1)$};
\node (z1) at (2,-2) {$F(ghv)$};
\draw[->] (x0) to (x1);
\draw[->] (x1) to (y3);
\draw[->] (x0) to (y0);

\draw[->] (y1) to (y2);
\draw[->] (x0) to (y1);
\draw[->] (x1) to (y2);
\draw[->] (y0) to (z0);
\draw[->] (y1) to (z0);
\draw[->] (z0) to (z1);
\draw[->] (y2) to (z1);
\draw[->] (z1) to (y3);
\draw[double,->] (y0) to (y1);
\draw[double,<-] (y2) to (z0);
\draw[double,->] (y1) to (x1);
\draw[double,->] (y2) to (y3);
\end{tikzpicture} \qquad = \qquad \begin{tikzpicture}[baseline={(current bounding box.center)}, scale =1.8]
\node (a1) at (0,0) {$F(ghv)$};
\node (a2) at (0,1) {$F(ghu_1)$};
\node (a3) at (0,2) {$F(ghu)$};
\node (b1) at (-1,0) {$F(hv)$};
\node (b2) at (-1,1) {$F(hu_1)$};
\node (b3) at (-1,2) {$F(hu)$};

\draw[->] (a2) -- (a1);
\draw[->] (a3) -- (a2);
\draw[->] (b2) -- (b1);
\draw[->] (b3) -- (b2);
\draw[->] (b1) -- (a1);
\draw[->] (b2) -- (a2);
\draw[->] (b3) -- (a3);
\draw[double,->] (b1) -- (a2);
\draw[double,->] (b2) -- (a3);

\end{tikzpicture} 
\]

In formulas, that is: 
\begin{align*}
(F_{ghu,gu_1',ghu_1,ghv}\circ_1 \Id)\circ_2(\Id \circ_1 \psi_{g,hu,u_1'})\circ_2 & (\Id\circ_1\psi_{g,u_1',hv}\circ_1\Id)\circ_2 (\Id \circ_1 F_{hu,hu_1,u_1',hv})\\&=(\Id\circ_1 \psi_{g,hu,hu_1})\circ_2 (\psi_{g,hu_1,hv}\circ_1\Id).
\end{align*}
This is exactly an instance of \ref{itm:e2'}.  By the above argument, we have verified \ref{itm:e1}.

The claim \ref{itm:e2} is established by substantially similar techniques (but does not require \ref{itm:e1'}), and is omitted.
The proof of uniqueness up to natural isomorphism is analogous to the proof that $F$ itself is (nonequivariantly) well-defined up to natural isomorphism.
\end{proof}

Let $H$ be a subgroup of $G$.  For a small category $\Cat$ with a $G$-action, let $\Cat^H$, called the $H$-\emph{fixed-point category}, be the subcategory of $\Cat$ whose objects are the objects of $\Cat$ invariant under $H$, and whose arrows are those of $\Cat$ that are invariant under $H$.

\begin{lem}\label{lem:fixed-functor}
Let $\Cat$ be a small category with an action by a finite group $G$, and fix a functor $\ff\col \Cat\to \burn$ with an external action by $G$.  Let $H$ be a subgroup of $G$.  Then there is a well-defined $H$-\emph{fixed-point functor} of $\ff$, written $\ff^H\col \Cat^H\to \burn$, given as follows.  On objects $v$ of $\Cat^H$, $\ff^H$ is defined by $\ff^H(v)=\ff(v)^H$, where $h\in H$ acts on $\ff(v)$ by $\psi_{h,v}$ as in (\ref{itm:34-1}) of Definition \ref{def:external-action}. On morphisms $\phi\from x \to y$ of $\Cat^H$, we define $F^{\mathrm{inv}}(\phi)=s^{-1}(F(x)^H)\cap t^{-1}(F(y)^H)$, viewed as a correspondence $F^H(x)\to F^H(y)$.  Using (\ref{itm:34-2}) of Definition \ref{def:external-action}, $h\in H$ acts on $F^{\mathrm{inv}}(\phi)$ by $\psi_{h,\phi}$.  For morphisms $\phi$ in $\Cat^H$, define $F^H(\phi)$ by $F^H(\phi)=(F^{\mathrm{inv}}(\phi))^H$.  Then, given objects $x,y$ and $z$ of $\Cat^H$, along with morphisms $\beta\from x \to y$ and $\gamma \from y \to z$ in $\Cat^H$, the associators $F_{\beta,\gamma}\from F(\gamma)\circ F(\beta)\to F(\gamma\circ\beta)$ for $\beta\from x \to y$ and $\gamma \from y\to z$ of $\ff^H$ restrict to give a bijection 
\begin{equation}\label{eq:good-restriction}
F^H_{\beta,\gamma}\from F^H(\gamma)\circ F^H(\beta)\to F^H(\gamma\circ \beta),
\end{equation}
which we use as the associators of $\ff^H$ in (\ref{itm:32-3}) of Definition \ref{def:lax}.  This data defines a functor $F\from \Cat^H\to \burn$ in the sense of Definition \ref{def:lax}.
	\end{lem}
\begin{proof}
The hypotheses of the lemma give us the data of (\ref{itm:32-1}) and (\ref{itm:32-2}) of Definition \ref{def:lax}.  We must check that $F_{\beta,\gamma}$ does indeed restrict as in (\ref{eq:good-restriction}) to verify that $F^H_{\beta,\gamma}$ is data as in (\ref{itm:32-3}), and moreover check that the pentagon as in Definition \ref{def:lax} commutes.   

Before starting, we note that, for $x,y$ objects of $\Cat^H$ and $\phi\from x \to y$ a morphism in $\Cat^H$, the action on $F^{\mathrm{inv}}(\phi)$ by $\psi_{h,\phi}$ is specified by using that $\psi_{h,y}\circ_1 F^{\mathrm{inv}}(\phi)$ and $F^{\mathrm{inv}}(\phi)\circ_1\psi_{h,x}$ are canonically identified with $F^{\mathrm{inv}}(\phi)$, for all $h\in H$, so that the 2-morphism $\psi_{h,\phi}$ in Definition \ref{def:external-action}(\ref{itm:34-2}) determines a bijection $F^{\mathrm{inv}}(\phi)\to F^{\mathrm{inv}}(\phi)$.  This bijection, written $\psi^{\mathrm{inv}}_{h,\phi}$, is the action of $h\in H$ on $F^\mathrm{inv}(\phi)$.  

Fix objects $x,y,z$ of $\Cat^H$ and morphisms $\beta\from x\to y$, $\gamma \from y \to z$ of $\Cat^H$.  Let us first verify that \[F_{\beta,\gamma}|_{F^H(\gamma)\circ F^H(\beta)}\from F^H(\gamma)\circ F^H(\beta)\to F(\gamma\circ \beta)\]
has image in $F^H(\gamma\circ\beta)$.  
Consider the commuting pentagon from \ref{itm:e2} for $\ff$, for a fixed $h\in H$:
\[
\begin{tikzpicture}[scale=0.8,baseline={(current  bounding  box.center)}]
\node (h0) at (0,0)   {\scriptsize $\psi_{h,z}\circ F(\gamma)\circ F(\beta)$};
\node (h1) at (2.5,2.5)    {\scriptsize $F(\gamma) \circ \psi_{h,y} \circ F(\beta)$};
\node (h2) at (2*2.5,0)  {\scriptsize $F(\gamma)\circ F(\beta) \circ \psi_{h,x}$};
\node (g2) at (2*2.5,-2.5) {\scriptsize $F(\gamma\circ \beta) \circ \psi_{h,x}$};
\node (g1) at (0,-2.5)  {\scriptsize $\psi_{h,z} \circ F(\gamma\circ \beta)$};

\draw[->]
(h0) edge node[auto] {\scriptsize $\psi_{h,\gamma}\circ_1 \Id$} (h1)
(h1) edge node[right] {\scriptsize $\Id \circ_1 \psi_{h,\beta}$} (h2)
(h2) edge node[auto] {\scriptsize $F_{\beta,\gamma} \circ_1 \Id$} (g2)
(h0) edge node[left] {\scriptsize$\Id \circ_1 F_{\beta,\gamma}$} (g1)
(g1) edge node[auto] {\scriptsize $\psi_{h,\gamma\circ \beta}$} (g2);
\end{tikzpicture}
\]

Recall that the objects of this diagram are correspondences from $F(x)$ to $F(z)$, and the arrows of this diagram are bijections of correspondences.  We consider the diagram formed by considering only the subsets of the above correspondences that have source in $F^H(x)$ and target in $F^H(z)$, as follows.  To obtain the following diagram, we have used that the 2-morphisms in \ref{itm:e2} respect source and target maps.  The arrows in the resulting diagram are again bijections.
\[
\begin{tikzpicture}[scale=1.2,baseline={(current  bounding  box.center)}]
\node (h0) at (-2,0)   {\scriptsize $(\psi_{h,z}\circ F(\gamma)\circ F(\beta))\cap s^{-1}(F^H(x))\cap t^{-1}(F^H(z))$};
\node (h1) at (2.5,2.5)    {\scriptsize $(F(\gamma) \circ \psi_{h,y} \circ F(\beta))\cap s^{-1}(F^H(x))\cap t^{-1}(F^H(z))$};
\node (h2) at (7,0)  {\scriptsize $(F(\gamma)\circ F(\beta) \circ \psi_{h,x})\cap s^{-1}(F^H(x))\cap t^{-1}(F^H(z))$};
\node (g2) at (2*3,-2.5) {\scriptsize $(F(\gamma\circ \beta) \circ \psi_{h,x})\cap s^{-1}(F^H(x))\cap t^{-1}(F^H(z))$};
\node (g1) at (-1,-2.5)  {\scriptsize $(\psi_{h,z} \circ F(\gamma\circ \beta))\cap s^{-1}(F^H(x))\cap t^{-1}(F^H(z))$};

\draw[->]
(h0) edge node[auto] {\scriptsize $\psi_{h,\gamma}\circ_1 \Id$} (h1)
(h1) edge node[right] {\scriptsize $\Id \circ_1 \psi_{h,\beta}$} (h2)
(h2) edge node[auto] {\scriptsize $F_{\beta,\gamma} \circ_1 \Id$} (g2)
(h0) edge node[left] {\scriptsize$\Id \circ_1 F_{\beta,\gamma}$} (g1)
(g1) edge node[auto] {\scriptsize $\psi_{h,\gamma\circ \beta}$} (g2);
\end{tikzpicture}
\]
However, for any object $w$ of $\Cat^H$ and $h_1\in H$, the restriction of $\psi_{h_1,w}$ to $F^H(w)$ is canonically identified with the identity 1-morphism $F^H(w)\to F^H(w)$.  In particular, the above diagram is canonically identified with 
\begin{equation}\label{eq:pre-very-restricted}
\begin{tikzpicture}[scale=1.2,baseline={(current  bounding  box.center)}]
\node (h0) at (-2,0)   {\scriptsize $(F(\gamma)\circ F(\beta))\cap s^{-1}(F^H(x))\cap t^{-1}(F^H(z))$};
\node (h1) at (2.5,2.5)    {\scriptsize $(F(\gamma) \circ \psi_{h,y} \circ F(\beta))\cap s^{-1}(F^H(x))\cap t^{-1}(F^H(z))$};
\node (h2) at (7,0)  {\scriptsize $(F(\gamma)\circ F(\beta) )\cap s^{-1}(F^H(x))\cap t^{-1}(F^H(z))$};
\node (g2) at (2*3,-2.5) {\scriptsize $(F(\gamma\circ \beta) )\cap s^{-1}(F^H(x))\cap t^{-1}(F^H(z))$};
\node (g1) at (-1,-2.5)  {\scriptsize $( F(\gamma\circ \beta))\cap s^{-1}(F^H(x))\cap t^{-1}(F^H(z))$};

\draw[->]
(h0) edge node[auto] {\scriptsize $\psi_{h,\gamma}\circ_1 \Id$} (h1)
(h1) edge node[right] {\scriptsize $\Id \circ_1 \psi_{h,\beta}$} (h2)
(h2) edge node[auto] {\scriptsize $F_{\beta,\gamma} $} (g2)
(h0) edge node[left] {\scriptsize$ F_{\beta,\gamma}$} (g1)
(g1) edge node[auto] {\scriptsize $\psi_{h,\gamma\circ \beta}$} (g2);
\end{tikzpicture}
\end{equation}
Note that the bottom row is naturally identified with the bijection 
\[\psi_{h,\gamma\circ\beta}^{\mathrm{inv}}\from F^{\mathrm{inv}}(\gamma\circ\beta)\to F^{\mathrm{inv}}(\gamma\circ\beta).\]

We have a further commutative diagram by restricting the maps $\psi_{h,\gamma}\circ_1\Id$ and $\Id\circ_1 F_{\beta,\gamma}$ from the previous diagram.  Note that the arrows are no longer necessarily bijections - they need only be maps of sets:

\begin{equation}\label{eq:very-restricted}
\begin{tikzpicture}[scale=0.8,baseline={(current  bounding  box.center)}]
\node (h0) at (0,0)   {\scriptsize $ F^{\mathrm{inv}}(\gamma)\circ F^{\mathrm{inv}}(\beta)$};
\node (h1) at (2.5,2.5)    {\scriptsize $F^{\mathrm{inv}}(\gamma) \circ F^{\mathrm{inv}}(\beta)$};
\node (h2) at (2*2.5,0)  {\scriptsize $F^{\mathrm{inv}}(\gamma)\circ F^{\mathrm{inv}}(\beta) $};
\node (g2) at (2*2.5,-2.5) {\scriptsize $F^{\mathrm{inv}}(\gamma\circ \beta)$};
\node (g1) at (0,-2.5)  {\scriptsize $ F^{\mathrm{inv}}(\gamma\circ \beta)$};

\draw[->]
(h0) edge node[auto] {\scriptsize $\psi^{\mathrm{inv}}_{h,\gamma}\circ_1 \Id$} (h1)
(h1) edge node[right] {\scriptsize $\Id \circ_1 \psi^{\mathrm{inv}}_{h,\beta}$} (h2)
(h2) edge node[auto] {\scriptsize $F_{\beta,\gamma}$} (g2)
(h0) edge node[left] {\scriptsize$F_{\beta,\gamma}$} (g1)
(g1) edge node[auto] {\scriptsize $\psi^{\mathrm{inv}}_{h,\gamma\circ \beta}$} (g2);
\end{tikzpicture}
\end{equation}
Because (\ref{eq:very-restricted}) commutes, we have equality of the following composites:
\begin{equation}\label{eq:very-restricted2}
F_{\beta,\gamma}\circ_2(\Id\circ_1\psi^{\mathrm{inv}}_{h,\beta})\circ_2(\psi_{h,\gamma}^{\mathrm{inv}}\circ_1\Id)= \psi^{\mathrm{inv}}_{h,\gamma\circ\beta}\circ_2 F_{\beta,\gamma}.
\end{equation}
On the subset $F^H(\gamma)\circ F^H(\beta)\subset F^{\mathrm{inv}}(\gamma)\circ F^{\mathrm{inv}}(\beta)$, the 2-morphism $ (\Id\circ_1\psi^{\mathrm{inv}}_{h,\beta})\circ_2(\psi^{\mathrm{inv}}_{h,\gamma}\circ_1\Id)$ is the identity.  Thus, we have:
\[
F_{\beta,\gamma}|_{F^H(\gamma)\circ F^H(\beta)}=(\psi_{h,\gamma\circ\beta}^{\mathrm{inv}}\circ_2 F_{\beta,\gamma})|_{F^H(\gamma)\circ F^H(\beta)}.
\]
Thus, the image of $F_{\beta,\gamma}|_{F^H(\gamma)\circ F^H(\beta)}$ in $F^\mathrm{inv}(\gamma\circ\beta)$ is preserved by $\psi_{h,\gamma\circ\beta}$.  That is, the image of $F_{\beta,\gamma}|_{F^H(\gamma)\circ F^H(\beta)}$ is in $F^H(\gamma\circ\beta)$.  

The map of sets $F_{\beta,\gamma}|_{F^H(\gamma)\circ F^H(\beta)}$ is injective by construction.  We must check surjectivity.  Suppose $\nu\in F^H(\gamma\circ\beta)$.  First, we show that the element $F_{\beta,\gamma}^{-1}(\nu)\in F(\gamma)\circ F(\beta)$ is in $F^\mathrm{inv}(\gamma)\circ F^\mathrm{inv}(\beta)$.

If $F^{-1}_{\beta,\gamma}(\nu)\not\in F^\mathrm{inv}(\gamma)\circ F^\mathrm{inv}(\beta)$, then 
\begin{equation}\label{eq:proof-of-surjectivity-of-fixed}
(\Id\circ_1 \psi_{h,\beta})\circ_2 (\psi_{h,\gamma}\circ_1\Id)F^{-1}_{\beta,\gamma}(\nu)\not=F^{-1}_{\beta,\gamma}(\nu),
\end{equation}
for some $h\in H$, by the following observations.  Let $F^{-1}_{\beta,\gamma}(\nu)=(\nu_2,\nu_1)$, with $\nu_2\in F(\gamma)$ and $\nu_1\in F(\beta)$ where $s(\nu_2)=t(\nu_1)$.  Then $(\Id\circ_1 \psi_{h,\beta})\circ_2 (\psi_{h,\gamma}\circ_1\Id) (\nu_2,\nu_1)$ can be written $(\nu_2',\nu_1')$ for some $\nu_2'\in F(\gamma)$ and $\nu_1'\in F(\beta)$.  By $(\nu_2,\nu_1)\not\in F^\mathrm{inv}(\gamma)\circ F^\mathrm{inv}(\beta)$, we have $t(\nu_1)\neq t(\nu_1')$ for an appropriate choice of $h$.  Thus $(\nu_2,\nu_1)$ cannot be $(\nu_2',\nu_1')$, giving (\ref{eq:proof-of-surjectivity-of-fixed}).  This contradicts commutativity of (\ref{eq:pre-very-restricted}), so that $F^{-1}_{\beta,\gamma}(\nu)\in F^{\mathrm{inv}}(\gamma)\circ F^{\mathrm{inv}}(\beta)$.  

To see that $F^{-1}_{\beta,\gamma}(\nu)\in F^H(\gamma)\circ F^H(\beta)$, observe from (\ref{eq:very-restricted2}) that \[F_{\beta,\gamma}((\psi^{\mathrm{inv}}_{h,\gamma}\nu_2,\psi^{\mathrm{inv}}_{h,\beta}\nu_1))=F_{\beta,\gamma}(\nu_2,\nu_1),\]
so that $\nu_2$ and $\nu_1$ are both $H$-fixed, by injectivity of $F_{\beta,\gamma}$.  

We have then established that, for all objects $x,y,z$ of $\Cat^H$ and morphisms $\beta\from x\to y$ and $\gamma \from y \to z$ in $\Cat^H$, that $F^H_{\beta,\gamma}$ defines a bijection $F^H(\gamma)\circ F^H(\beta)\to F^H(\gamma\circ \beta)$.  It remains to check that the pentagon of Definition \ref{def:lax} commutes.  This follows immediately from the fact that $F^H_{\beta,\gamma}$ is a restriction of $F_{\beta,\gamma}$.  

\end{proof}
\begin{defn}\label{def:singularity}
Fix a group $K$.  Let $\Cat$ be a small category with an action by a finite group $G$, and fix a functor $\ff\col \Cat\to \burn_K$ with an external action by $G$.  Let $H$ be a subgroup of $G$.  Call $F$ a $H$-\emph{singular} functor if there exists $u$ an object of $\Cat^H$ and $x\in F(u)$ so that the decorated bijection $\psi_h$ between $F(u)$ and $F(u)$ has $s^{-1}(x)=t^{-1}(x)$, but $\sigma(s^{-1}(x))\neq 1\in K$.  Otherwise, say $F$ is $H$-\emph{nonsingular}.  If $F$ is nonsingular for all subgroups $H\subset G$, we just say that $F$ is \emph{nonsingular}.  
\end{defn}

Let $\Cat$ be a small category with an action by a finite group $G$: as a matter of convention, we regard any functor $\ff\col \Cat\to \burn$ with an external action by $G$ as nonsingular. 

\begin{lem}\label{lem:fixed-point-singularity}
	Fix a group $K$.  Let $\Cat$ be a small category with an action by a finite group $G$, and fix a functor $\ff\col \Cat\to \burn_K$ with an external action by $G$.  Let $H$ be a subgroup of $G$, and say $\ff$ is $H$-nonsingular. Then there is a well-defined $H$-\emph{fixed-point functor} of $\ff$, written $\ff^H\col \Cat^H\to \burn_K$, given as follows.   On objects $v$ of $\Cat$, $\ff^H$ is defined by $\ff^H(v)=(\forgot\ff(v))^H$, where $h\in H$ acts on $\forgot\ff(v)$ as in Lemma \ref{lem:fixed-functor}.  On morphisms $\phi\from x \to y$ of $\Cat^H$, we define $F^{\mathrm{inv}}(\phi)=s^{-1}(F^H(x))\cap t^{-1}(F^H(y))$, viewed as a decorated correspondence $F^H(x)\to F^H(y)$.  Using (\ref{itm:34-2}) of Definition \ref{def:external-action}, $h\in H$ acts on $F^{\mathrm{inv}}(\phi)$ by $\psi_{h,\phi}$.  For morphisms $\phi$ in $\Cat^H$, define $F^H(\phi)$ by $F^H(\phi)=(F^{\mathrm{inv}}(\phi))^H$.  Then, given objects $x,y$ and $z$ of $\Cat^H$, along with morphisms $\beta\from x \to y$ and $\gamma \from y \to z$ in $\Cat^H$, the associators $F_{\beta,\gamma}\from F(\gamma)\circ F(\beta)\to F(\gamma\circ\beta)$ for $\beta\from x \to y$ and $\gamma \from y\to z$ of $\ff^H$ restrict to give a bijection 
	\begin{equation}\label{eq:good-restriction2}
	F^H_{\beta,\gamma}\from F^H(\gamma)\circ F^H(\beta)\to F^H(\gamma\circ \beta),
	\end{equation}
	which we use as the associators of $\ff^H$ in (\ref{itm:32-3}) of Definition \ref{def:lax}.  This data defines a functor $F\from \Cat^H\to \burn_K$ in the sense of Definition \ref{def:lax}.  
\end{lem}
\begin{proof}
	The proof of this lemma is completely analogous to that of Lemma \ref{lem:fixed-functor}, and amounts to checking that the constructions there are compatible with decorations, for a $H$-nonsingular functor; the details are omitted. 
\end{proof}

Finally, a similar argument shows:

\begin{lem}\label{lem:action-on-fix-functors}
	Fix notation as in Lemma \ref{lem:fixed-point-singularity}.  The $H$-fixed-point functor $F^H$ admits a $N(H)/H$-external action, where $N(H)$ is the normalizer of $H$ in $G$, by restriction of the external action on $F$, compatible with the $N(H)/H$-action on $\Cat^H$.  
\end{lem}

\subsection{Natural transformations}\label{sec:nat-transform-burn}
To relate different functors to the Burnside category, we will need the following notion:

\begin{defn}\label{def:nattrans}
  Let $\Cat$ be a small category.  A \emph{natural transformation} $\eta \from F_1 \to F_0$ between
  $2$-functors $F_1,F_0\from \Cat\to \burn_K$ is a
  strictly unitary $2$-functor
  $\eta\from \two\times \Cat \to \burn_K$ so that
  $\eta|_{\{1\}\times \Cat}=F_1$ and
  $\eta|_{\{0\}\times \Cat}=F_0$.  A natural transformation of functors $F_1,F_0\from \two^{np}\to \burn_K$ with external action by $\ZZ_p$, where $\ZZ_p$ acts on $\two^{np}$ by permuting the coordinates, is such an $\eta$, itself admitting an external action (where $\two\times \two^{np}$ has the product $\ZZ_p$-action, where $\ZZ_p$ acts trivially on $\two$.).
\end{defn}

We usually refer to `natural transformations with external action' as `natural transformations' where it will not cause confusion.

For $\Cat=\two^n$ or $\two^{np}$, a natural transformation (functorially) induces a chain map between
the totalizations of Burnside functors, which we write as
$\Tot_\degh(\eta)\from\Tot_\degh(F_1)\to \Tot_\degh(F_0)$, for any homomorphism $\degh\from K \to \ZZ_2$.  (In fact, for a natural transformation with external action by $\ZZ_p$, $\Tot_\degh(\eta)$ is $\ZZ[\ZZ_p]$-equivariant; cf.\ Lemma \ref{lem:total-equiv}.)

Many of the natural transformations we will encounter will be
sub-functor inclusions or quotient functor surjections. Given a
functor $F\from\two^{np}\to\burn_K$ with external action, a \emph{sub-functor with external action}
(\resp \emph{quotient functor}) $H\from\two^{np}\to\burn_K$
is a functor that satisfies:
\begin{enumerate}[leftmargin=*]
\item $H(v)\subset F(v)$ for all $v\in\two^{np}$, and so that the external action of $\ZZ_p$ restricts to an action on the set $H(v)$.
\item $H(\phi_{u,v})\subset F(\phi_{u,v})$ for all $u\geq v$, with the
  source and target maps being restrictions of the
  corresponding ones on $F(\phi_{u,v})$, and so that the action of $\ZZ_p$ preserves $H$ (in the natural sense).
\item $s^{-1}(x)\subset H(\phi_{u,v})$ (\resp
  $t^{-1}(y)\subset H(\phi_{u,v})$) for all $u\geq v$ and for all $x\in
  H(u)$ (\resp $y\in H(v)$).
  Equivalently, $H(\phi_{u,v}) = \bigcup_{x\in H(u)} s\inv(x)$ (\resp $\bigcup_{y\in H(v)} t\inv(y)$).
\end{enumerate}
If $H$ is a sub- (\resp quotient) functor of $F$, then there is
a natural transformation $H\to F$ (\resp $F\to H$), and the
induced chain map $\Tot(H)\to \Tot(F)$ (\resp
$\Tot(F)\to\Tot(H)$) is an inclusion (\resp a quotient
map) of chain complexes (in fact, a $\ZZ_p$-equivariant map of chain complexes).  See also \cite[Section 3.7]{oddkh}.

\begin{defn}\label{def:burn-cofib-sequence}
  If $J$ is a sub-functor with external action of $F\from\two^{np}\to\burn_K$, then the
  functor $L$ defined as $L(v)=F(v)\setminus J(v)$ and
  $L(\phi_{u,v})=\cup_{y\in L(v)}t^{-1}(y)\subset F(\phi_{u,v})\setminus J(\phi_{u,v})$ is a quotient
  functor of $F$ (and vice versa). Such a sequence
  \[
  J\to F\to L
  \]
  is called a \emph{cofibration sequence} of Burnside functors; it
  induces the short exact sequence
  \[
  0\to\Tot(J)\to\Tot(F)\to\Tot(L)\to0
  \]
  of chain complexes.
\end{defn}

\subsection{Stable equivalence of functors}

 In the sequel, we will be
interested not just in functors
$F\from \two^n\to \burn_K$, but in \emph{stable} functors,
which are pairs $(F,R)$, for $F$ a functor $F\from \two^n\to \burn_K$ with external action by $G$, and for $R$ an element of the real representation ring of $G$, with $R$ a linear combination of copies of the trivial representation and the regular representation.  In case $G=\{1\}$, we view stable functors as pairs $(F,r)$ for $r$ an integer, referring to $r$ copies of the trivial representation, and $F$ a functor $\two^n\to\burn_K$.  Note that a functor $F\from \two^N \to \burn_K$ with external action by $G=\{1\}$ is the same information as a functor $F\from \two^N\to \burn_K$ without external action. 
We denote the regular representation of $G$ by $\mathbb{R}(G)$.  For an orthogonal $G$-representation $V$, write $V^+$ for its one-point compactification, considered as a pointed space by taking the point at infinity as the basepoint.  We will also write $\Sigma^RF$ for $(F,R)$.

Let $\detg_G=\tilde{H}^*(\mathbb{R}(G)^+)$ as a graded $\ZZ[G]$-module.  We define the totalization of the stable functor $(F,r\mathbb{R}+s\mathbb{R}(G))$, by $\Tot((F,r\mathbb{R}+s\mathbb{R}(G))=\Tot(F)[r]\otimes_\ZZ \detg_G^{\otimes s}$, where $\Tot(F)[r]$ denotes the (ordinary) totalization shifted up by $r$.  If $s<0$, we make sense of the above formula using the (graded) dual of $\detg_G$.  In this section we will describe when two such stable
functors are equivalent, following Definition 3.6 of \cite{oddkh}. 

A \emph{face inclusion} $\iota$ is a functor $\two^n \to \two^N$ that
is injective on objects and preserves the relative gradings.  Note
that self-equivalences $\iota\from \two^n \to \two^n$ are face
inclusions.  Consider a face inclusion $\iota\from \two^n \to
\two^N$ and a functor $F \from \two^n \to \burn_K$. We define an
induced functor $F_\iota\from \two^N \to \burn_K$, which is
uniquely determined by requiring that $F=F_\iota\circ \iota$, and such
that for $v\in \Ob(\two^N)\setminus \Ob(\iota(\two^n))$, we have
$F_\iota(v)=\varnothing$. For a face inclusion $\iota$, we define
$|\iota|=|\iota(v)|-|v|$ for any $v\in\two^n$, which is
independent of $v$ since $\iota$ is assumed to preserve relative
gradings.  For any functor $F$, and face inclusion
$\iota$ as above, 
\begin{equation}\label{eq:totalization-iso}
	\Tot_K(F_\iota) \; \cong \; \Sigma^{|\iota|}\Tot_K(F)
\end{equation}
To construct such an isomorphism let $\mathfrak{c}_{F,v}$ denote the canonical isomorphism of $\ZZ[K]$-modules $\Abelianize_K(F_\iota(\iota(v)))\to \Abelianize_K(F(v))$, for $v$ an object of $\two^n$.  For a function $\sigma\from \Ob(\two^n)\to \{+1,-1\}=\ZZ_2$, whose value on $v\in\Ob(\two^n)$ will be denoted $\sigma_v$, we define an isomorphism of graded $\ZZ[K]$-modules $\rho_F\from \Tot_K(F_{\iota})\to \Sigma^{|\iota|}\Tot_K(F)$ given by sending the summand $\Abelianize_K(F_\iota(\iota(v)))$ to $\Abelianize_K(F(v))$ by $\rho_{F,\iota,v}=\sigma_v\mathfrak{c}_{F,v}$.  

We next determine under what conditions on $\{\sigma_v\}_{v\in\Ob(\two^n)}$ the map $\rho_F$ is an isomorphism of chain complexes.  For any $v\geqslant_1 w$ objects of $\two^n$, we need that 
\begin{equation}\label{eq:definition-of-a-cochain}
(-1)^{s_{\iota(v),\iota(w)}}\sigma_w=\sigma_v(-1)^{s_{v,w}+|\iota|}.
\end{equation}
To see this, consider the cellular cochain complex of the $n$-dimensional cube, with $\ZZ_2$-coefficients: $\cellC^*([0,1]^{n};\ZZ_2)$.  The assignment $(-1)^{s_{u,v}}$ defines a cochain in $\cellC^1([0,1]^n;\ZZ_2)$ whose coboundary is the constant $2$-cochain that evaluates to $-1$ on all $2$-dimensional faces of $[0,1]^n$.  Similarly, the assignment sending a pair of objects $u\geqslant_1 v$ of $\two^n$ to $(-1)^{|\iota|+s_{\iota(u),\iota(v)}}$ is a 1-cochain, again with coboundary the constant cochain evaluating to $-1$. 
Since $H^1([0,1]^n;\ZZ_2)=0$, $(-1)^{|\iota|+s_{\iota(u),\iota(v)}}$ and $(-1)^{s_{u,v}}$ are cohomologous.  
The condition that $\sigma\in \cellC^{0}([0,1]^{n};\ZZ_2)$ has $(\delta \sigma)(u,v)=(-1)^{|\iota|+s_{\iota(u),\iota(v)}+s_{u,v}}$ is precisely (\ref{eq:definition-of-a-cochain}).  Thus, $\sigma$ satisfying (\ref{eq:definition-of-a-cochain}) exist and so (\ref{eq:totalization-iso}) holds. Moreover, any two cochains $\sigma$ satisfying (\ref{eq:definition-of-a-cochain}) differ by a cocycle in $\cellC^0([0,1]^n;\ZZ_2)$.  Since $H^0([0,1]^n;\ZZ_2)=\ZZ_2$, any functions $\sigma\from \Ob(\two^n)\to\ZZ_2$ satisfying (\ref{eq:definition-of-a-cochain}) are $(\pm 1)$-multiples of each other.  Thus the isomorphism in (\ref{eq:totalization-iso}) is canonical up to sign.

As a consequence of (\ref{eq:totalization-iso}), the following holds for any homomorphism $\degh\from K\to\ZZ_2$, for $F,\iota$ as above:
\[
\Tot_\degh(F_\iota) \; \cong \; \Sigma^{|\iota|}\Tot_\degh(F).
\]

For $\ZZ_p$ acting on $\two^{np}=(\two^n)^p$ and $\two^{Np}=(\two^N)^p$ by cyclic permutation, an \emph{equivariant face inclusion} $\iota \from \two^{np}\to \two^{Np}$ will be a face inclusion so that $g\iota=\iota g$ for all $g\in \ZZ_p$.  

Let $\iota$ be an equivariant face inclusion, and $F\from \two^{np}\to \burn_K$ a functor with external action by $G$.  Then $F_\iota$, defined as above, admits a $\ZZ_p$-external action using the data of (\ref{itm:34-1}) and (\ref{itm:34-2}) of Definition \ref{def:external-action} from the functor $F$ with $\ZZ_p$-external action.

\begin{lem}\label{lem:det-naturality} Fix $F\from \two^{np}\to \burn_K$ and an equivariant face inclusion $\iota\from \two^{np}\to \two^{Np}$.   Say that $\ZZ_p$ acts on $\two^{np}$ and $\two^{Np}$ by cyclic permutation, and $F$ admits a compatible $\ZZ_p$-external action.  Fix a homomorphism $\degh\from K \to \ZZ_2$.  The pair $(F,\iota)$ induces a $\ZZ_p$-equivariant isomorphism between $\Tot_\degh(F_\iota)$ and $\detg_{\ZZ_p}^{|\iota|/p}\otimes\Tot_\degh(F)$, natural up to sign.
\end{lem}
\begin{proof}

	We show that 
	\[
	\Tot_K(F_\iota)\cong \detg_{\ZZ_p}^{|\iota|/p}\otimes\Tot_K(F),
	\]
	which implies the claim of the lemma.
	
	Note that as $\ZZ$-complexes, $\detg_{\ZZ_p}^{|\iota|/p}\otimes\Tot_K(F)$ is naturally identified with $\Sigma^{|\iota|}\Tot_K(F)$.  From the discussion preceding the lemma, we have an isomorphism (well-defined up to sign) of $\ZZ$-complexes
	\[
	\Tot_K(F_\iota)\cong \Sigma^{|\iota|}\Tot_K(F),
	\]
	so it remains to determine the $\ZZ_p$-action on the right-hand side that is compatible with the $\ZZ_p$-action on $\Tot_K(F_\iota)$.  
	
	Let $\ZZ_p$ act on $\cellC^0([0,1]^n;\ZZ_2)$ as follows.  It is enough to define the action on a generator of $\ZZ_p$.  Let $g=1_p\in \ZZ_p$ and $\sigma \in \cellC^0([0,1]^n;\ZZ_2)$, and define
	\[
	(g\sigma)(v)=\sigma(gv)\tau(v)\tau(\iota v),
	\]
	where $\tau$ is as in the discussion preceding Lemma \ref{lem:total-equiv}.

	If $\rho\from \Tot_K(F_\iota) \to \Sigma^{|\iota|}\Tot_K(F)$ is one of the two isomorphisms (well-defined up to sign) as in (\ref{eq:totalization-iso}), associated to a choice of function $\sigma \from \Ob(\two^n)\to \ZZ_2$, then we have:
	\begin{align*}
	\delta(g\sigma)(v,w)&=(-1)^{s_{gv,gw}+g_{\iota gv,\iota gw}+s_{v,w}+s_{gv,gw}+s_{\iota v,\iota w}+s_{\iota gv,\iota gw}+|\iota|}\\&=(-1)^{s_{v,w}+s_{\iota v,\iota w}+|\iota|}=\delta(\sigma)(v,w).
	\end{align*}
	That is, $\sigma$ and $g\sigma$ agree up to an overall sign, since $\delta(g\sigma)=\delta(\sigma)$. In particular, for fixed $g$, the term $(\sigma(v))\cdot((g\sigma)(v))$ is constant in $v\in \Ob(\two^{np})$.  
	
	Assume now that $g=1_p$ is the standard generator of $\ZZ_p$, as an additive group.  We determine $(\sigma(v))\cdot ((g\sigma)(v))$ by considering its evaluation at $v=0\in \Ob(\two^{np})$.  Then
	\[
	(\sigma(v))\cdot ((g\sigma)(v))=\tau(0)\tau(\iota(0)).
	\]
	By definition, $\tau(0)=1$, while it is readily checked that $\tau(\iota(0))=-1$ if and only if $p$ is even and $|\iota|/p$ is odd.
	
	Thus, the isomorphism (\ref{eq:totalization-iso}) is $\ZZ_p$-equivariant if $\ZZ_p$ acts on $\Tot_K(F_{\iota})$ as in Lemma \ref{lem:total-equiv}, and $K$ acts on $\Sigma^{|\iota|}\Tot_K(F)\cong\ZZ[|\iota|]\otimes_\ZZ \Tot_K(F)$ by the sign representation on $\ZZ[|\iota|]$ if $p$ is even, and if $p$ is odd, then $\ZZ_p$ acts on $\ZZ[|\iota|]$ by the trivial representation.  It is direct to check that $\ZZ[|\iota|]$ with this $\ZZ_p$-action is $\detg_{\ZZ_p}$, as needed.
\end{proof}

With this
background, we state the relevant notion of equivalence for stable
functors.

\begin{defn}\label{def:stableq}
  Two stable functors
  $(E_1 \from \two^{m_1} \to \burn_K, q_1)$ and
  $(E_2\from \two^{m_2} \to \burn_K, q_2)$ are \emph{stably
    equivalent} for $\degh\from K \to \ZZ_2$ if there is a sequence of stable functors
  $\{(F_i \from \two^{n_i} \to \burn_K, r_i)\}$
  ($0\leq i \leq \ell$) with $\Sigma^{q_1}E_1=\Sigma^{r_0}F_0$ and
  $\Sigma^{q_2}E_2=\Sigma^{r_\ell}F_\ell$ such that for each pair
  $\{ \Sigma^{r_i}F_i, \Sigma^{r_{i+1}}F_{i+1}\}$, one of the
  following holds:
\begin{enumerate}[leftmargin=*]
\item\label{itm:stable-1} $(n_i,r_i)=(n_{i+1},r_{i+1})$ and there is a natural
  transformation $\eta\from F_i\to F_{i+1}$ or
  $\eta\from F_{i+1} \to F_i$ such that the induced map $\Tot_\degh(\eta)$
  is a chain homotopy equivalence.
\item\label{itm:stable-2} There is a face inclusion
  $\iota\from \two^{n_i} \hookrightarrow \two^{n_{i+1}}$ such that
  $r_{i+1}=r_i-|\iota|$ and $F_{i+1}=(F_i)_\iota$; or a face inclusion
  $\iota\from \two^{n_{i+1}} \hookrightarrow \two^{n_{i}}$ such that
  $r_{i}=r_{i+1}-|\iota|$ and $F_{i}=(F_{i+1})_\iota$.
  
\end{enumerate}

Two nonsingular stable functors $(E_1,Q_1),(E_2,Q_2)$ with $E_i\from \two^{n_ip}\to \burn_K$ with external action by $\ZZ_p$ (compatible with the action on $\two^{n_ip}$ by cyclic permutation) are \emph{externally stably equivalent} for a given homomorphism $\degh\from K \to \ZZ_2$ if there is a sequence of nonsingular stable functors $\{(F_i \from \two^{n_i} \to \burn_K, R_i)\}$
($0\leq i \leq \ell$) with $(E_1,Q_1)=(F_0,R_0)$ and
$(E_2,Q_2)=(F_\ell,R_\ell)$ such that for each pair
$\{ (F_i,R_i), (F_{i+1},R_{i+1})\}$, one of the
following holds: 
\begin{enumerate}[leftmargin=*]
	\item\label{itm:stable-1-eq} $(n_i,R_i)=(n_{i+1},R_{i+1})$ and there is a natural
	transformation of functors with external actions $\eta\from F_i\to F_{i+1}$ or
	$\eta\from F_{i+1} \to F_i$, such that the induced map, for each subgroup $H\subset G$, $\Tot_\degh(\eta^H)$
	is a chain homotopy equivalence, where $\eta^H$ is the fixed-point functor.
	\item\label{itm:stable-2-eq} There is an equivariant face inclusion
	$\iota\from \two^{n_ip} \hookrightarrow \two^{n_{i+1}p}$ such that
	$R_{i+1}=R_i-(|\iota|/p) \mathbb{R}(G)$ and $F_{i+1}=(F_i)_\iota$; or a face inclusion
	$\iota\from \two^{n_{i+1}p} \hookrightarrow \two^{n_{i}p}$ such that
	$R_{i}=R_{i+1}-(|\iota|/p) \mathbb{R}(G)$ and $F_{i}=(F_{i+1})_\iota$.
	
\end{enumerate}  

We call such a sequence, along with the arrows $\eta,\iota$ between
the $(F_i,R_i)$, a $\degh$-{\emph{external stable equivalence}} between the stable
functors $(E_1,Q_1)$ and $(E_2,Q_2)$.  If the sequence is such that the maps $\eta$
satisfy $\Tot_\degh(\eta^H)$ are chain homotopy equivalences for all $\degh\from K \to \ZZ_2$, we call it a $K$-{\emph{equivariant
    (stable) equivalence}}, and say that $(E_i,Q_i)$ are
$K$-\emph{equivariantly equivalent}.  All external stable equivalences that appear in this paper will be $K$-equivariant.
\end{defn}

An external stable equivalence from $(E_1,Q_1)$ to
$(E_2,Q_2)$ induces a $\ZZ_p$-equivariant chain homotopy equivalence
$\Tot((E_1,Q_1))\to \Tot((E_2,Q_2))$, well-defined up to
choices of inverses of the chain homotopy equivalences involved in its
construction, and an overall sign (since $\sigma$ is well-defined up to an overall sign).

We will also need the notion of a \emph{product} of Burnside functors.  
\begin{defn}\label{def:product-functor}
Given functors $F\from \two^{mp}\to \burn_K$ and $J \from \two^{np} \to \burn$, both with external action by $\ZZ_p$ compatible with the permutation action on $(\two^n)^p$, we define the product $F\times J \from \two^{(m+n)p}\to \burn_K$ as follows
\begin{enumerate}
\item For $(v_1,v_2)\in \two^{mp} \times \two^{np}$, $(F\times J)((v_1,v_2))=F(v_1)\times J(v_2)$.
\item For all $(u_1,u_2)>(v_1,v_2)$, $(F\times J)(\phi_{(u_1,u_2),(v_1,v_2)})=F(\phi_{u_1,v_1})\times J(\phi_{u_2,v_2})$.  The decoration on each element of the correspondence is the decoration of $F(\phi_{u_1,v_1})$.
\item For all $(u_1,u_2)>(v_2,v_2)>(w_1,w_2)$, the map $(F \times J)_{(u_1,u_2),(v_1,v_2),(w_1,w_2)}$ is defined by
\[
(F \times J)_{(u_1,u_2),(v_1,v_2),(w_1,w_2)}(x_1,x_2)=((F)_{u_1,v_1,w_1}(x_1),(J)_{u_2,v_2,w_2}(x_2)),
\]
where, if $u_i=v_i$ or $v_i=w_i$, we set $(F)_{u_i,v_i,w_i}=\Id$ or $(J)_{u_i,v_i,w_i}=\Id$, respectively.  
\end{enumerate}
It is direct to check that this defines a strictly unitary lax $2$-functor $\two^{(n+m)p} \to \oddb$.  A computation verifies that $\Tot_K(F \times J)=\Tot_K(F) \otimes \Tot_K(J)$.  

The $\ZZ_p$-external action on $F\times J$ is given as follows.  On objects, $\psi_{g,(v,w)}$ is given by the product action $\psi_{g,(v,w)} \from F(v)\times J(w) \to F(gv)\times J(gw)$, and similarly for the action on correspondences.   It is direct to confirm that these data satisfy conditions \ref{itm:e1} and \ref{itm:e2} of Definition \ref{def:external-action}.
\end{defn}
\section{Realizations of Burnside functors}\label{sec:disk}

In this section, given a functor $F\from \two^n\to\burn_K$, along with some other choices, we construct an
essentially well-defined spectrum $\Realize{F}$, which is an equivariant spectrum if $K\neq\{1\}$, in a sense that will be made precise in the course of Section \ref{subsec:eqvar}. As a first step, we construct finite CW complexes
$\CRealize{F}_{V}$ for sufficiently large representations $V$, so that increasing the
parameter $V$ corresponds to suspending the CW complex
$\CRealize{F}_{V}$. The finite CW spectrum $\Realize{F}$ is then
defined from this sequence of spaces. The construction of
$\CRealize{F}_{V}$ depends on some auxiliary choices, but its stable
homotopy type does not. Moreover, the spectra constructed from two
stably equivalent Burnside functors will be homotopy equivalent.  Much of this section is either a generalization of, or contained in, \cite[Sect.\ 4]{oddkh}, which itself is mostly a collection of results from \cite{lls1} along with some background on equivariant topology.  The only essentially new material in the present section is Lemma \ref{lem:connctd}.

\subsection{Maps from correspondences}\label{sec:realization-sub-main}
 We start with the construction of
(ordinary) disk maps, following \cite[\S2.10]{lls1}\footnote{In previous papers, starting with \cite{lls1}, but continuing in \cite{lls2},\cite{oddkh}, one works with `box maps.'  The previous papers could have been executed in very close analogy using disk maps as formulated here, obtaining homotopy-equivalent objects; we prefer disk maps in the present paper as they are more suitable for visualizing the group action.}, which the reader may consult for more details.  Let $B^\ell=\{ x\in \RR^\ell\mid ||x||\leq 1\}$, and fix an
identification $S^\ell=B^\ell/\partial$, with $\partial:=\partial B^\ell$, and view $S^\ell$ as a pointed space, with base point the image of $\partial$.  For any subset $B\subset B^\ell$ of the form $B=\{y\in B^\ell\mid ||y-y_0||\leq c, \}$, for some $y_0\in B^\ell$ and $c\in \mathbb{R}_{>0}$ so that $||y_0||+c< 1$, we note that there is a standard identification of $B$ with a copy of $B^\ell$ by the map $\phi$ defined by $x\to (x-y_0)/c$ and so we have a standard identification $S^\ell=B/\partial B$.  In the sequel, by a subdisk $B\subset B^\ell$ we will mean a subset $B$ as above, which we will often identify with $B^\ell$ itself, using  $\phi$.

Given a collection (indexed by $\{1,\dots,t\}$) of sub-disks $B_1,\dots, B_t$ of some disk $B$, so that the $\{B_i\}_{i=1,\ldots, t}$ have
disjoint interiors, there is an induced map
\begin{equation}\label{eq:signedphi}
  S^\ell = B  /\partial   B \to B   /(B \backslash (\mathring{B_1} \cup \dots \cup \mathring{B_t}))= \bigvee^t_{a=1} B_a /\partial B_a = \bigvee^t_{a=1}S^\ell \to S^\ell.
\end{equation}
The last map is the identity on each summand, so that the composition (\ref{eq:signedphi})
has degree $t$.  As observed in \cite{lls1}, this construction is
continuous in the position of the sub-disks.  We let $E(B,t)$
denote the space of (indexed) subdisks with disjoint interiors in $B$, and have a
continuous map $E(B,t) \to \Map(S^\ell,S^\ell)$.

We can generalize the above procedure to associate a map of spheres to
a map of finite sets $A \to Y$, as follows.  Say we have chosen a collection of subdisks $\{B_a\}_{a\in A}$ where the subdisks
$B_a \subset B$, for some fixed disk $B$, have disjoint interiors.  Then we have a
map:
\begin{equation}\label{eq:mapfromset}
  S^\ell = B /\partial B \to B /(B\backslash (\bigcup_{a\in A}\mathring{B_a} )) =\bigvee_{a\in A} B_a /\partial B_a = \bigvee_{a\in A}S^\ell \to \bigvee_{y\in Y} S^\ell
\end{equation}
where the last map is built using the map of sets $A \to Y$.

More generally, we can also create maps from a correspondence of
sets, as follows.  

\begin{construction}\label{const:disk}  Fix finite sets $X,Y$ and fix a finite correspondence $A$ from $X$ to $Y$ with
source map $s$ and target map $t$.  Say that we also have a collection
of disks $B_x$ for $x\in X$.  Finally, take as input also a collection of
sub-disks $B_a \subset  B_{s(a)}$ with disjoint interiors, for $a\in
A$.  We then have an induced map
\begin{equation}\label{eq:mapfromset2}
\bigvee_{x\in X} S^\ell \to \bigvee_{y\in Y} S^\ell,
\end{equation}
by applying, on $B_x$, the map associated to the set map $s^{-1}(x)
\to Y$.  A map constructed this way is said to \emph{refine}
the correspondence $A$.  

\end{construction}
For a pair of finite sets $A,X$, along with a map of sets $s\from A \to X$, and a collection of disks $\{B_x\}_{x\in X}$, let $E(\{B_x\},s,A,X)$ be the space of
collections of labeled sub-disks $\{B_a \subset B_{s(a)} \mid a \in
A\}$ with disjoint interiors.  Then, choosing a correspondence
$\mathbb{A}=(A,s,t)$, and a collection of disks $\{B_x\}_{x\in X}$,
Construction \ref{const:disk} gives a map
$E(\{B_x\},s,A,X) \to \Map(\vee_{x\in X} S^\ell ,\vee_{y\in Y} S^\ell)$.  We
write 
\begin{equation}
\Phi(e,(A,s,t)) \in\Map(\bigvee_{x\in X} S^\ell , \bigvee_{y\in Y}S^\ell)\label{eq:phi}
\end{equation}
for the map associated to $e\in E(\{B_x\},s,A,X)$ and the
correspondence $(A,s,t)$. One of the main points is that, for any disk map
$\Phi(e,(A,s,t))$ refining $(A,s,t)$, the induced map on the $\ell\th$ homology
agrees with the abelianization map 
\[
\Abelianize(A)\from\Abelianize(X)=\widetilde{H}_\ell(\vee_{x\in
  X}S^\ell)\to\Abelianize(Y)=\widetilde{H}_\ell(\vee_{y\in Y}S^\ell).
\]

We now indicate a further generalization of disk maps to cover decorated
correspondences.
\begin{construction}\label{const:k-disk}
Fix a finite group $K$ and finite sets $X$ and $Y$, and a decorated correspondence $(A,s,t,\sigma)$ from
$X$ to $Y$, with $A$ finite and $\sigma\from A \to K$.  Fix also some collection of disks $\{B_x\}_{x \in X}$.  Fix a homomorphism $\refl\col K \to \mathrm{Homeo}(B^\ell)$.  Fix
a collection of subdisks 
$B_a\subset B_{s(a)}$ for $a\in A$.
There is an induced map just as in Construction \ref{const:disk},
but whose construction depends on the decoration $\sigma$, as
follows. For $x\in X$, we have a set map $s^{-1}(x) \to Y$, along with
decorations for each element of $s^{-1}(x)$.  We modify the map refining
$s^{-1}(x) \to Y$ (without decoration) by precomposing with $\refl(\sigma(a))$:
\[
S^\ell = B/\partial B \to B /(B\backslash (\bigcup_{a\in A}\mathring{B_a} )) = \bigvee_{a\in A} B_a /\partial   B_a\xrightarrow{\bigvee  \refl(\sigma(a))} \bigvee_{a\in A} B_a  / \partial B_a = \bigvee_{a\in A}S^\ell  \to \bigvee_{y\in Y}S^\ell.
\]
We say that a map constructed this way $\refl$-\emph{refines} (or, when $\refl$ is clear from context, simply \emph{refines}) the decorated
correspondence $\mathbb{A}=(A,s,t,\sigma)$.  

\end{construction}
As before, we can regard Construction \ref{const:k-disk} as a map
\[
\Phi(e,\mathbb{A})\in \Map( \bigvee_{x\in X} S^\ell ,\bigvee_{y\in Y}S^\ell),
\]
where $e \in E(\{B_x\},s,A,X)$, and $\mathbb{A}=(A,s,t,\sigma)$ is a
decorated correspondence.  Once again, the induced map on the $\ell\th$
homology agrees with the $\degh$-abelianization map, where the homomorphism $\degh$ is defined by setting, for $k\in K$,  $\degh(k)$ to be the topological degree of $\refl(k)$.

For $V$ an orthogonal representation of a finite group $K$, write $B(V)$ for the unit ball of $V$.  

Let $E_{K,V}(\{B_x\},s,A,X)$ denote the set of elements in $E(\{B_x\},s,A,X)$ whose centers lie in $B(V)^K$.  In other words, each element $e$ of $E(\{B_x\},s,A,X)$ is a collection of disks; the element $e$ will be in $E_{K,V}(\{B_x\},s,A,X)$ if and only if the center of each disk in $e$ lies in $B(V)^K$.  

\begin{lem}[{cf. \cite[Lemma 4.5]{oddkh}}]
\label{lem:connctd-pre}
  Let $A$ and $X$ be finite sets, and let $s\from A\to X$ be a map of sets. If $\dim(V^K)\geq k$ then
  $E_{K,V}(\{ B_x\},s,A,X)$ is $(k-2)$-connected.
\end{lem}

\begin{proof}
  The proof is analogous to \cite[Lemma 2.29]{lls1} or \cite[Lemma 4.5]{oddkh}.
\end{proof}

For finite sets $X,Y$ and a finite correspondence $\mathbb{A}=(A,s,t)$ from $X$ to $Y$, it is convenient to abuse notation somewhat and write $E(\{B_x\},s,\mathbb{A},X)$ for $E(\{B_x\},s,A,X)$, as we will do in the following lemma.  
\begin{lem}[{cf. \cite[Lemma 4.6]{oddkh}}]\label{lem:box-composition}
  Fix an $\mathbb{R}$-vector space $V$ and an orthogonal $K$-representation $\refl\from K \to O(V)$.  Let $A$, $B$, $X$, $Y$, and $Z$ be finite sets, and let $s_A\from A\to X$ and $s_B\from B\to Y$ be maps of sets.  Let $\mathbb{A}=(A,s_A,t_A,\sigma_A)$ and $\mathbb{B}=(B,s_B,t_B,\sigma_B)$ be decorated correspondences, from $X$ to $Y$ and from $Y$ to $Z$, respectively. If $e\in E(\{B_x\},s_A,A,X)$ 
   and $f\in E(\{B_y\},s_B,B,Y)$, then there is a unique
  $f\circ_\refl e\in E(\{B_x\},s_{B\circ A},\mathbb{B}\circ \mathbb{A},X)$, 
   so
  that $\Phi(f\circ_\refl e,\mathbb{B}\circ \mathbb{A})=\Phi(f,\mathbb{B})\circ\Phi(e,\mathbb{A})$. Moreover,
  the assignment $E(\{B_y\},s_B,B,Y)\times E(\{B_x\},s_A,A,X)\to
  E(\{B_x\},s_{B\circ A},\mathbb{B}\circ \mathbb{A},X)$, by sending a pair $(f,e)$ to $f\circ_{\refl} e$, is continuous and sends
  $E_{K,V}(\{B_y\},s_B,B,Y)\times E_{K,V}(\{B_x\},s_A,A,X)$ to
  $E_{K,V}(\{B_x\},s_{\mathbb{B}\circ \mathbb{A}},\mathbb{B}\circ\mathbb{A},X)$. 
\end{lem}

\begin{proof}

For $(b,a) \in B \times_Y A$ where $b$ has decoration $g$ and $a$ has decoration $h$, let $e_b\colon B_b\to B_{s_B(b)}$ and $e_a\colon B_a\to B_{s_A(a)}$ denote the corresponding disks in $E(\{B_x\},s_B,B,Y)$ and $E(\{B_x\},s_A,A,X)$, respectively.  Define $B_{(b,a)}\subset B(V)$ to be the subdisk given by the image of
\[
B_b 
\stackrel{e_b}{\longrightarrow} B_{s_B(b)= t_A(a)}\xrightarrow{\refl(h^{-1})} B_{s_B(b)} 
= B_a
\stackrel{e_a}{\longrightarrow}B_{s_A(a)}.
\]
 
This defines $f \circ_\refl e$ as the image of $(f,e)$ under the assignment 
\[
E(\{B_y\},s_B,B,Y)\times E(\{B_x\},s_A,A,X)\to
E(\{B_x\},s_{\mathbb{B}\circ \mathbb{A}},\mathbb{B}\circ\mathbb{A},X)
.\] 
It follows from the definitions that $\Phi(f\circ_\refl e,\mathbb{B}\circ\mathbb{A})=\Phi(f,\mathbb{B})\circ\Phi(e,\mathbb{A})$.

Finally, consider the restriction of the assignment to $E_{K,V}(\{B_y\}, s_B,B,Y) \times E_{K,V}(\{B_x\}, s_A,A,X)$.  It is clear that the above construction takes disks centered on $V^K$ to disks centered on $V^K$, completing the proof.
\end{proof}

If $K$ is abelian and $\mathbb{A}=(A,s,t,\sigma)$ is a decorated finite correspondence between finite sets $X$ and $Y$, then for $e\in E_{K,V}(\{B_x\},s,A,X)$, the induced map $\Phi(e,\mathbb{A})$ is $K$-equivariant.
Note that the condition that $K$ is abelian is necessary, as $\Phi(e,\mathbb{A})$ may be a collapse to a slightly smaller disk (mod boundary), along with multiplication by $g\in K$; in order for this to be $K$-equivariant, we would need $kg=gk$ for all $k \in K$.

\begin{construction}\label{const:g-box-action}
	Fix a finite group $G$, an abelian group $K$, and a finite-dimensional orthogonal $K\times G$-representation $\mathfrak{s}\from K\times G\to O(V)$.  Let $\underline{A}\from \two\to \burn_K$ be a functor with external action $\psi$ by $G$, where $G$ acts on $\two$ trivially.  Write $X=\underline{A}(1)$ and $Y=\underline{A}(0)$, with $\mathbb{A}:=\underline{A}(\phi_{1,0})=(A,s,t,\sigma)$, where as usual we write $s,t,\sigma$ for the source, target, and sign maps of $\mathbb{A}$, respectively.   For each $x\in X$, define the disk $B_x(V)$ to be a copy of $B(V)$; more precisely one may define $B_x(V)=\{x\}\times B(V)$.  We usually write $B_x$ for $B_x(V)$, when the representation $V$ is clear from context.  Write elements of the set $B(X,V)=\amalg_{x\in X} B_x(V)$ as pairs $(x,v)$, where $v\in B(V)$ and $x\in X$.  Then define an action of $K\times G$ on $B(X,V)$ by setting, for $k\in K$ and $g\in G$, 
	\[
	(k\times g)\cdot (x,v)=(g\cdot x,\mathfrak{s}(k\sigma(\psi_{g,v},x)\times g)\cdot v).
	\]
	Here $\sigma(\psi_{g,v},x)$ denotes the decoration $\sigma(s^{-1}(x))$ for the decorated correspondence $\psi_{g,v}\from X\to X$.  This assignment gives a continuous homomorphism $\rho\from K\times G \to \mathrm{Homeo}(B(X,V))$.
	
	Moreover, $G$ acts on $E(\{B_x\},s,A,X)$ as follows.  View each element $e\in E(\{B_x\},s,A,X)$ as a function assigning a disk $e(a)=B_a\subset B_{s(a)}\cong B(V)$ (note that the isomorphism is fixed - each ball $B_a$ is a copy of $B(V)$) to each element $a\in A$.  Then, for $e\in E(\{B_x\},s,A,X)$, define $(g\cdot e)(a)$ to be the image of $\rho(1\times g)(e(g^{-1}a))$ in $B_{s(a)}$.  
\end{construction}

\begin{examp}
	We consider an example of Construction \ref{const:g-box-action} to see how it looks in practice.  In the notation of Construction \ref{const:g-box-action}, let $K=G=\ZZ_2=\{\pm 1\}$, and let $X$ be the $2$-element set $\{x,y\}$.  Let $V=\mathbb{R}_K\oplus \mathbb{R}_G$ where $K\times G$ preserves the decomposition, and the nonidentity element of $K$ acts by $-1$ on $\mathbb{R}_K$ and $1$ on $\mathbb{R}_G$.  The nonidentity element of $G$ acts by $1$ on $\mathbb{R}_K$ and $-1$ on $\mathbb{R}_G$.  Say that the nontrivial element of $G$ acts (by decorated bijection) on $X$ by\[
	x\overset{-1}{\to} y\qquad \mbox{ and } \qquad y\overset{-1}{\to} x.
	\]
	Here the label over the arrows refers to the value, in $K$, of the decoration of the bijection.  The set $B(X,V)$ is then two copies, $B_x(V)$ and $B_y(V)$, of $B(V)$.  The action of $K\subset K\times G$ on $B(X,V)$ is given by, for $(x,v)\in B_x(V)$,
	\[ 
	\rho(k\times 1)(x,v)=(x,kv),
	\]
	and similarly in the $B_y(V)$ factor.  The action of the nontrivial element $g\in G$ on $B_x(V)$ is given by
	\[
	\rho(1\times g)(x,v)=(y,\left(\begin{array}{cc} -1 & 0 \\ 0 & -1 
	\end{array}\right)v),
	\]
	and there is a similar formula for the $B_y(V)$ factor.
\end{examp}

\begin{lem}\label{lem:connctd}
	Use the notation from Construction \ref{const:g-box-action} and let $H$ be a subgroup of $G$.  
 
For any $N>0$, there exists a fixed finite-dimensional representation $V_N$ so that the following holds.  For all finite-dimensional representations $V$ as in Construction \ref{const:g-box-action} for which there is an embedding of $V_N$ in $V$, the fixed-point set of $E_{K,V}(\{B_x(V)\},s,A,X)$ under the action of $H$, denoted $E_{K,V}(\{B_x(V)\},s,A,X)^H$, is $N$-connected, and nonempty.
\end{lem}
\begin{proof}

Fix some $(K\times G)$-representation $\mathfrak{s}$ on $V$, and assume that $\dim V^K\geq 1$.  

Let $Z_{K,V}(\{B_x\}_{x\in X},s,A,X)$ denote the space of injective maps (of sets) $\zeta \from A \to B(X,V^K)$ that lift the map of sets $z\from A \to X=\pi_0(B(X,V^K))$.  The group $G$ acts on $Z_{K,V}(\{B_x\}_{x\in X},s,A,X)$ by $(g\cdot \zeta)(a)=(\mathfrak{s}(g))(\zeta(g^{-1}a))$.  There is a continuous map $\pi\from E_{K,V}(\{B_x\}_{x\in X},s,A,X)\to Z_{K,V}(\{B_x\}_{x\in X},s,A,X)$ by sending balls to their centers.  This map is $K\times G$-equivariant and a homotopy equivalence; here we have used that the centers of disks in $E_{K,V}(\{B_x\}_{x\in X},s,A,X)$ lie in $V^K$.  Moreover, the fixed-point set $E_{K,V}(\{B_x\}_{x\in X},s,A,X)^H$ is sent by $\pi$ to $Z_{K,V}(\{B_x\}_{x\in X},s,A,X)^H$.  Let $\pi^H$ denote the restriction of $\pi$ to $E_{K,V}(\{B_x\}_{x\in X},s,A,X)^H$.  It is straightforward to check that $\pi^H$ is also a homotopy equivalence.  

Thus, it suffices to find conditions under which $Z_{K,V}(\{B_x\}_{x\in X},s,A,X)^H$ is $N$-connected.  We describe the set $Z_{K,V}(\{B_x\}_{x\in X},s,A,X)^H$.  Note that $G$ acts on the set $A$ itself, as follows.  We have, by the definition of external actions, a 2-isomorphism:
\[
\psi_{g,A} \from \psi_{g,Y}\circ A\to A\circ \psi_{g,X}.
\]
However,  as a set, $A$ may be canonically identified with $\psi_{g,Y}\circ A$ by sending an element $a$ of $A$ to the unique pair $(\nu,a)$ of $\psi_{g,Y}\circ A=\psi_{g,Y}\times_{Y} A$, and similarly for identifying $A$ with $A\circ \psi_{g,X}$.  Thus $\psi_{g,A}$ defines a bijection $A\to A$, and the collection of these as $g\in G$ varies defines an action of $G$ on $A$.

Choose an element $a_i\in A$, for $i=1,\ldots,n$, in each of the orbits $H\backslash A$.  By the definition of $Z_{K,V}(\{B_x\}_{x\in X},s,A,X)$ and the action, an element $\zeta$ of $Z_{K,V}(\{B_x\}_{x\in X},s,A,X)^H$ is determined by the restriction $\zeta|_{\{a_i\}_{i=1,\ldots,n}}$, since, for any $a\in A$, with $a=ha_{i_1}$ for some $i_1\in\{1,\ldots,n\}$ and $h\in H$, the assumption that $\zeta\in Z_{K,V}(\{B_x\}_{x\in X},s,A,X)^H$ ensures that $\rho(1\times h)(\zeta(a_{i_1}))=\zeta(a)$.

For each $a_i$, let $S_i\subset H$ be the stabilizer.  For $g\in G$, define $\langle g \rangle \subset G$ to be the subgroup of $G$ generated by $g$.  Let $B=B(V)$ for some $V$ sufficiently large.

Let $Z'_{K,V}(\{B_x\}_{x\in X},s,A,X)$ denote the space of maps of sets $\zeta \from A'=\{a_i\}_{i=1,\ldots,n} \to B(X,V^K)$, that lift the map of sets $z\from A' \to X$, and so that:

\begin{enumerate}[leftmargin=*,label=(D-\arabic*)]
	\item\label{itm:fixedness} The element $\zeta(a_i)$ lies in $V^{K\times S_i}$.
	\item\label{itm:disjointness}  The elements $\{\rho(1\times h)\zeta(a_i)\}_{i=1,\ldots,n;h\in H/S_i}$,  are disjoint in $B(X,V^K)$.  
\end{enumerate}

Alternatively, \ref{itm:disjointness} is equivalent to 
\begin{enumerate}[leftmargin=*,label=(D-\arabic*$'$)]\addtocounter{enumi}{1}
	\item\label{itm:disjointness-2}  For each $x\in X$, the elements $\{\rho(1\times h)\zeta(a_i)\}_{i=1,\ldots,n;h\in H/S_i}$ so that $s(ha_i)=x$ are disjoint in $B(\{x\},V^K)$.  
\end{enumerate}

Moreover, if \ref{itm:disjointness-2} is true for a single $x\in X$, then it is also true for $hx$ for any $h\in H$.  Thus, we need only check \ref{itm:disjointness-2} under the assumption that $X$ is a one-element set $\{x\}$; we will assume $X=\{x\}$ henceforth.  In particular, we have reduced to the case that $H$ acts trivially on $X=\{x\}$, by replacing $H$ by the stabilizer of $x$ in $H$.  

From the above discussion, we have that 
\[
Z'_{K,V}(\{B_x\},s,A,X)=Z_{K,V}(\{B_x\},s,A,X)^H,
\]
so we need only show that $Z'_{K,V}(\{B_x\},s,A,X)$ is highly connected.  

  Note that $Z'_{K,V}(\{B_x\},s,A,X)$ is exactly the set of tuples $(x_1,\dots,x_n)$ where $x_i\in B(V)$ so that 
  \begin{enumerate} \item $x_i\in V^{S_i}$.
  	\item For each $i\in \{1,\ldots,n\}$, the orbit of $x_i$ is isomorphic, as an $H$-set, to $H/S_i$.  Equivalently, for all $g\in H$ which are not in $S_i$, $x_i\not \in V^{\langle g \rangle}$.
  	\item $(x_1,\dots,x_n)\not\in\Delta$ where $\Delta$ is the set of tuples $(x_1,\ldots,x_n)$ for which there is some pair $i\neq j$ with $g_1x_i=g_2x_j$ for some $g_1,g_2\in H$.
  	\end{enumerate}  

That is, $Z'_{K,V}(\{B_x\}_{x\in X},s,A,X)=(\prod_{i=1}^nV^{S_i})-D$ where 
\[
D=\Delta \cup \bigcup_{i=1}^n 
\left (
 (V^{S_1} \times V^{S_2} \times \ldots \times V^{S_{i-1}})
\times 
\bigcup_{g\not \in S_i} V^{\langle g\rangle}
\times 
(V^{S_{i+1}} \times \ldots \times V^{S_n})
\right ).
\]

For given $N>0$, to show that $Z'_{K,V}(\{B_x\}_{x\in X},s,A,X)$ (and therefore also $Z_{K,V}(\{B_x\}_{x\in X},s,A,X)^H$) is $N$-connected, it suffices to show that $D$ has arbitrarily high codimension in $\prod_{i=1}^nV^{S_i}$. We will show next that this can be achieved by constructing a suitably large $V$.  

Let $\RR[G]$ denote the regular representation of $G$.  Recall that $\RR[G]$ satisfies the following two properties:
\begin{enumerate}
\item It contains a copy of the trivial representation; every $g\in G$ acts trivially on this 1-dimensional summand. 
\item For any $1 \neq g \in G$, $g$ acts nontrivially on some irreducible summand of $\RR[G]$. 
\end{enumerate}
Given $1 \neq g \in G$, these two facts show that both the dimension and the codimension of $\RR[G]^{\langle g \rangle}$ are at least 1.

As $D$ is the union of finitely many pieces, it suffices to show that each piece has high codimension, say at least $N+2$.  Choose $V \cong\RR[G]^{\oplus N+2}$, so that $V^{\langle g \rangle}$ has dimension and codimension at least $N+2$.  It is now clear that the non-$\Delta$ pieces of $D$ have codimension at least $N+2$. 
To see that $\Delta$ also has high codimension, observe that $\Delta$ is the (finite) union of subsets homoeomorphic to diagonals $\Delta_{i,j}=\{(x,x)\in V^{S_i}\times V^{S_j}\}$ for $i\neq j$, thickened by the remaining components $\prod_{k \neq i, j} V^{S_k}$. 
The dimension of $\Delta_{i,j}$ is at most $\min (\dim V^{S_i}, \dim V^{S_j})$, so it codimension is at least $\max (\dim V^{S_i}, \dim V^{S_j})$, which is at least $N+2$. 
The codimension of $\Delta_{i,j} \subset V^{S_i} \times V^{S_j}$ is the same as the codimension of $\Delta_{i,j} \times (\prod_{k \neq i, j} V^{S_k}) \subset \prod_{k=1}^n V^{S_k}$. 

The same argument applies to any representation containing $V$, completing the proof of the lemma. 
\end{proof}

\begin{lem}\label{lem:connctd-2}
Maintain the notation from Lemma \ref{lem:connctd}.  For $e\in E_{K,V}(\{B_x(V)\},s,A,X)^H$, the induced map $\Phi(e,\mathbb{A})$ is $K\times H$-equivariant.
\end{lem}
\begin{proof}
This follows from the definition of disk maps, as well as the definition of the $H$-action on $E(\{B_x(V)\},s,A,X)$ in Lemma \ref{lem:connctd}.
\end{proof}

\subsection{Equivariant topology}\label{subsec:eqvar}
Let $\topp$ be the category of well-based topological spaces. A \emph{weak equivalence} $X\to Y$ is a map that
induces isomorphisms on all homotopy groups; typically the spaces we consider are simply connected, in which case the definition reduces to being
isomorphisms on all homology groups.  Homotopy equivalence is a special case of weak equivalence, and for CW complexes the two notions are equivalent.

Let
$G\text{-}\topp$ be the category of well-based topological spaces with a continuous action by a finite group $G$. We also require that the inclusions of fixed points
$X^{H}\to X^{H'}$, for all subgroups $H'<H$ of $G$, are cofibrations.  For pointed $G$-spaces $X$ and $Y$, a map $X\to Y$ is called a
\emph{weak equivalence} if the induced map $X^H\to Y^H$ is a weak
equivalence for all subgroups $H$ of $G$. (We will also sometimes call such a map a $G$-\emph{weak equivalence}, to distinguish it from a map in $G\text{-}\topp$ that is a weak equivalence in the nonequivariant sense.)  A homotopy between $G$-maps $f\from X \to Y$ and $h\from X \to Y$ is an extension to a $G$-equivariant map $k\from X \times I \to Y$, where $X \times I$ is given a $G$-structure by $g(x,i)=(gx,i)$, so that $k|_{X\times 0}=f$ and $k|_{X\times 1} =h$.   A homotopy equivalence in
$G\text{-}\topp$ induces a weak equivalence. For $G$-CW complexes, the two notions are equivalent by the $G$-Whitehead
theorem, see \cite[Theorem 2.4]{Greenlees-May}. For $G$-CW complexes,
a weak equivalence $X\to Y$ induces a weak equivalence between
quotients of fixed points, $X^{H'}/X^H\to Y^{H'}/Y^H$, for all
subgroups $H'<H$ of $G$, and between orbit spaces, $X/H\to Y/H$, for
all subgroups $H$ of $G$. 

\subsection{Homotopy coherence}\label{subsec:homtpy}
In this section, we briefly review homotopy colimits and homotopy
coherent diagrams, following \cite[\S2.9]{lls1}. 

We recall the notion of a homotopy coherent diagram, which is the
data from which a homotopy colimit is constructed.  A homotopy
coherent diagram is intuitively a diagram $F\from \Cat\to
K\text{-}\topp$ which is not commutative, but commutative up to
homotopy, and the homotopies themselves commute up to higher homotopy,
and so on, and for which all the homotopies and higher homotopies are
viewed as part of the data of the diagram. More precisely, we have the
following definition:

\begin{defn}[{\cite[Definition 2.3]{Vogt}}]\label{def:homco}
  A homotopy coherent diagram $F\from\Cat\to K\text{-}\topp$ is an assignment, 
  to each $x \in \Cat$, of a space $F(x) \in K\text{-}\topp$, and for each $n \geq
  1$ and each sequence
    \[
    x_0 \xrightarrow{f_1} x_1 \xrightarrow{f_2} \cdots \xrightarrow{f_n} x_n
    \]
    of composable morphisms in $\Cat$, a continuous map
    \[
    F(f_n,\dots,f_1)\from [0,1]^{n-1} \times F(x_0) \to F(x_n)\phantom{\xrightarrow{f}}
    \]
    with $F(f_n,\dots,f_1)([0,1]^{n-1}\times \{ *\})=*$.  These maps
    are required to satisfy the following compatibility conditions:
\begin{align}
\nonumber F(f_n,\dots,f_1)(t_1,& \dots,t_{n-1})=  \\ 
& \begin{cases}
F(f_n,\dots,f_2)(t_2,\dots,t_{n-1}), &f_1 = \Id\\
F(f_n ,\dots,\hat{f}_i,\dots,f_1)(t_1,\dots,t_{i-1}\cdot t_i,\dots,t_{n-1}),&f_i=\Id, 1 < i < n\\
F(f_{n-1},\dots,f_1)(t_1,\dots,t_{n-2}),&f_n = \Id\\
F(f_n,\dots,f_{i+1})(t_{i+1},\dots, t_{n-1}) \circ F(f_i,\dots,f_1)(t_1,\dots,t_{i-1}), & t_i=0 \\
F(f_n,\dots,f_{i+1}\circ f_i,\dots,f_1)(t_1,\dots,\hat{t}_{i},\dots,t_{n-1}), & t_i=1.
\end{cases}\label{eq:compat}
\end{align}
When $\Cat$ does not contain any non-identity isomorphisms, homotopy
coherent diagrams may be defined only in terms of non-identity
morphisms and the last two compatibility conditions.
\end{defn}

Given a homotopy coherent diagram, we can define its \emph{homotopy
  colimit} in $K\text{-}\topp$, quite concretely, as follows:
\begin{defn}[{\cite[\S5.10]{Vogt}}]\label{def:hoco}
  Given a homotopy coherent diagram $F\from \Cat \to K\text{-}\topp$ the
  \emph{homotopy colimit} of $F$ is defined by
\begin{equation}\label{eq:hoco}
\hoco \; F = \{ * \} \amalg \coprod_{n\geq 0} \coprod_{x_0\xrightarrow{f_1} \cdots\xrightarrow{f_n}x_n} [0,1]^n \times F(x_0) /\sim,
\end{equation}
 where the equivalence relation $\sim$ is given as follows:
\[
(f_n,\dots,f_1;t_1,\dots,t_n;p)\sim\begin{cases}
(f_n,\dots,f_2 ; t_2,\dots,t_n;p),&f_1=\Id\\
(f_n,\dots,\hat{f}_i,\dots,f_1;t_1,\dots,t_{i-1}\cdot t_i,\dots,t_n;p),&f_i=\Id,i>1\\
(f_{n},\dots,f_{i+1};t_{i+1},\dots,t_n;F(f_i,\dots,f_1)(t_1,\dots,t_{i-1},p)), &t_i=0\\
(f_n,\dots,f_{i+1}\circ f_i,\dots,f_1;t_1,\dots,\hat{t}_{i},\dots,t_n;p), &t_i=1,i<n\\
(f_{n-1},\dots,f_1; t_1,\dots,t_{n-1};p), & t_n=1\\
*, & p=*.
\end{cases}
\]
When $\Cat$ does not contain any non-identity isomorphisms, homotopy
colimits may be defined only in terms of non-identity morphisms and the last four equivalence relations.  That is, 
\begin{equation*}\label{eq:hoco-red}
\hoco \; F = \{ * \} \amalg \coprod_{n\geq 0} \coprod_{\substack{x_0\xrightarrow{f_1} \cdots\xrightarrow{f_n}x_n\\\forall i\in \{1,\dots,n\}, f_i\neq \Id}} [0,1]^n \times F(x_0) /\sim',
\end{equation*}
where $\sim'$, in the case $\Cat$ has no non-identity isomorphisms, is the last four cases of the definition of $\sim$.  

In this paper, the categories $\Cat$ will have no non-identity isomorphisms, so we will work with the latter formulation.   
\end{defn}

We will occasionally need the following:

\begin{defn}\label{def:homomorph}[Definition 2.6, \cite{Vogt}]
A \emph{homomorphism} of homotopy coherent diagrams $F_1,F_0\from \Cat \to K\text{-}\topp$ is a collection of maps $\phi_x\from F_1(x) \to F_0(x)$ for each $x\in \Ob(\Cat)$, so that
\[
F_0(f_n,\dots,f_1)(t_1,\dots,t_{n-1})\circ \phi_x= \phi_y\circ F_1(f_n,\dots,f_1)(t_1,\dots,t_{n-1}),
\]
where $f_n\circ\dots\circ f_1\from x\to y\in \Cat$, for all $t_i$.  
\end{defn}

A homotopy-coherent diagram may itself be viewed as a commutative diagram from an auxiliary category as in \cite[Definition 2.3]{Vogt}, and a homomorphism of homotopy coherent diagrams is a homomorphism (of diagrams, in the usual sense) of the associated commutative diagrams from the auxiliary category.

We will need the following properties:
\begin{enumerate}[leftmargin=*,label=(ho-\arabic*)]
\item \label{itm:ho1} Suppose that $F_0,F_1\from \Cat \to
  K\text{-}\topp$ are homotopy coherent diagrams and $\eta\from F_1
  \to F_0$ is a \emph{natural transformation}, defined as a homotopy coherent
  diagram
  \[\eta\from\two\times\Cat\to K\text{-}\topp
  \]
  with $\eta|_{\{i\}\times\Cat}=F_i$, $i=0,1$.  Then $\eta$ induces a
  map $\hoco\,\eta \from \hoco\,F_1 \to \hoco\,F_0$, well-defined up to homotopy, according to \cite[Theorem 5.12]{Vogt}.  If $\eta(x)$
  is a $K$-weak equivalence for each $x\in \Cat$---we will call such an
  $\eta$ a $K$-weak equivalence from $F_1$ to $F_0$---then $\hoco\,\eta$ is a 
 $K$-weak equivalence as well.

  When the spaces involved are $K$-CW complexes, a
  weak equivalence $\eta\from F_1\to F_0$ is also a \emph{homotopy
  equivalence} \cite[Proposition 4.6]{Vogt}, that is, there exist
  $\zeta,\zeta'\from F_0\to F_1$ and
  \[
  \mathfrak{h},\mathfrak{h}'\from\{2\to1\to0\}\times\Cat\to
  K\text{-}\topp,
  \]
  with $\mathfrak{h}|_{\{2\to1\}\times\Cat}=\eta$,
  $\mathfrak{h}|_{\{1\to0\}\times\Cat}=\zeta$,
  $\mathfrak{h}|_{\{2\to0\}\times\Cat}=\Id_{F_0}$, and
  $\mathfrak{h}'|_{\{2\to1\}\times\Cat}=\zeta'$,
  $\mathfrak{h}'|_{\{1\to0\}\times\Cat}=\eta$,
  $\mathfrak{h}'|_{\{2\to0\}\times\Cat}=\Id_{F_1}$.  Here we write $\Id_{F_0}$ (similarly $\Id_{F_1}$) for the `identity' natural transformation $\two \times \Cat \to K\text{-}\topp$, which restricts to $F_0$ on $\{0\}\times \Cat$ and $\{1\}\times \Cat$, and where the morphism $\Id_{F_0}(\phi_{1,0}\times c)$, for $c$ an object of $\Cat$, is the identity $F_0(c)\to F_0(c)$; there are also well-defined higher homotopies. 

\item\label{itm:ho-new}
 A homomorphism $F_1 \to F_0\from \Cat \to K\text{-}\topp $ of homotopy coherent diagrams induces a $K$-equivariant map $\hoco\, F_1 \to \hoco \, F_0$.  This map on homotopy colimits satisfies a certain compatibility with \ref{itm:ho1}, as in \cite[Proposition 7.1]{Vogt}.
\item\label{itm:ho3}
For any subgroup $H$ of $K$, define the fixed-point diagram $F^H \from \Cat\to \topp$ by setting $F^H(x)$ to be the fixed points $F(x)^H$.  Then there is a natural homeomorphism
\[
(\hoco\, F)^H\simeq \hoco (F^H),
\]
see \cite[(ho-3)]{oddkh}
\item \label{itm:ho2} Suppose that $F\from \Cat \to \topp$ and $G\from \Dat \to \topp$.  Then there is an induced functor $F \wedge G \from \Cat\times\Dat \to \topp$ with $(F\wedge G)(v\times w)=F(v)\wedge G(w)$.  Then there is a natural (in homomorphisms of homotopy coherent diagrams) weak equivalence $(\hoco\, F)\wedge(\hoco \, G)\to \hoco (F\wedge G)$.
\item \label{itm:ho4} Let $L \from \Cat \to \Dat$ be a functor between small categories.  Given $d\in \Ob(\Dat)$, the \emph{undercategory} of $d$ is defined as follows.  It has objects $\{(c,f)\mid c\in \Cat, f\from d\to L(c)\}$, and arrows $\Arr((c,f),(c',f'))=\{ g \from c\to c'\mid f'=L(g)\circ f\}$.  We write $d \downarrow L$ for the undercategory of $d$.  The functor $L$ is called \emph{homotopy cofinal} if for each $d\in \Ob(\Dat)$, the undercategory $d\downarrow L$ has contractible nerve. 

For a homotopy coherent diagram $F\from \Dat \to K\text{-}\topp$, there is an induced homotopy coherent diagram $F\circ L \from \Cat \to K\text{-}\topp$.  If $F(j)$ is cofibrant for all $j\in \Ob(\Dat)$ and $L$ is homotopy cofinal, then the natural map
\[
\hoco\, F\circ L \to \hoco\, F
\]
is a homotopy equivalence. This follows from the version for homotopy limits in \cite{bousfield-kan}; cf.\ \cite[Sect.\ 2.9, (ho-4)]{lls1}.
\end{enumerate}

\subsection{Little disks refinement}
\label{sec:little-boxes}
With this background, we are ready to review the little box
realization construction of \cite[\S5]{lls1} and generalize to functors to $\burn_K$.  Assume from now on that $K$ is abelian.
\begin{defn}\label{def:spacref}
  Fix a small category $\Cat$ and a strictly unitary $2$-functor $F
  \from \Cat \to \burn_K$, as well as a finite abelian group $K$.  Fix a pair of finite-dimensional orthogonal $K$-representations $V_1$ and $V_2$, where the action of $K$ on $V_1$ is denoted $\refl$.  A \emph{spatial
    refinement} of $F$ modeled on $(V_1,V_2)$ is a homotopy coherent diagram $\widetilde{F}
  \from \Cat \to K\text{-}\topp$ such that
\begin{enumerate}[leftmargin=*]
\item For any $u\in \Cat$, $\widetilde{F}(u)=\vee_{x\in F(u)} B(V)/\partial B(V)$, where $V:=V_1\oplus V_2$.    
\item For any sequence of morphisms $u_0 \xrightarrow{f_1} \cdots
  \xrightarrow{f_n} u_n$ in $\Cat$ and any $(t_1,\dots,t_{n-1})\in
  [0,1]^{n-1}$ the map
  \[
  \widetilde{F}_k(f_n,\dots,f_1) (t_1,\dots,t_{n-1})\from \bigvee_{x\in F(u_0)} B(V)/\partial B(V) \to \bigvee_{x\in F(u_n)}B(V)/\partial B(V)  \]
  is a $K$-equivariant disk map $\refl\oplus\Id_{V_2}$-refining the correspondence $F(f_n
  \circ \dots \circ f_1)$ (note that $F(f_n
  \circ \dots \circ f_1)$ is naturally isomorphic to $F(f_n)
  \times_{F(u_{n-1})} \dots \times_{F(u_1)} F(f_1)$).
\end{enumerate}
\end{defn}
This definition extends \cite[Definition 5.1]{lls1} and \cite[Definition 4.11]{oddkh}.

The main technical result that makes it possible to construct spatial refinements from Burnside functors is as follows.

\begin{prop}[{cf. \cite[Proposition 4.12]{oddkh},  \cite[Proposition 5.22]{lls1}}]
\label{prop:cube-non-special}
  Let $\Cat$ be a small category in which every sequence of composable
  non-identity morphisms has length at most $n$, and let $F\from
  \Cat\to\burn_K$ be a strictly unitary 2-functor, with $K$ a finite abelian group.
\begin{enumerate}[leftmargin=*]
\item Fix a finite-dimensional orthogonal $K$-representation $\refl\from K\to O(V_1)$.  For $V_2$ a sufficiently large $K$-representation, there is a spatial refinement of
  $F$ modeled on $(V_1,V_2)$.
   \label{itm:real1-pre}
\item Fix a finite-dimensional orthogonal $K$-representation $\refl\from K\to O(V_1)$.  For $V_2$ a sufficiently large $K$-representation, any two spatial refinements of $F$ modeled on $(V_1,V_2)$ are weakly equivalent.  \label{itm:real2-pre}
\item\label{itm:cube-non-special-3} Fix a finite-dimensional orthogonal $K$-representations $\refl\from K\to O(V_1)$, and $V_2,W$.  If $\widetilde{F}$ is a spatial refinement of $F$ modeled on $(V_1,V_2)$ then the result of suspending each $\widetilde{F}(u)$ and
  $\widetilde{F}(f_n, \dots, f_1)$ by $W$ gives a spatial refinement of $F$ modeled on $(V_1,V_2\oplus W)$.  \label{itm:real3-pre}
\end{enumerate}
\end{prop}

\begin{proof}
  This is entirely analogous to the proof of Proposition 4.12 of \cite{oddkh}; see Proposition 4.20 of \cite{oddkh} for equivariant aspects. 
\end{proof}

\subsection{Realization of cube-shaped diagrams}
Finally, in this section we will discuss how to construct a CW complex
$\CRealize{F}$, and then a spectrum
$\Realize{F}$, from a diagram $F\from
\two^n\to\burn_K$.  We assume in this section that $K$ is abelian.  Let $\two_+$ be the category with objects
$\{0,1,*\}$ and unique non-identity morphisms $1\to0$ and $1\to *$, and
let $\two^n_+=(\two^n)\amalg \ast$ where, for $v\in\two^n-\{0^n\}$, there is a unique arrow $v\to \ast$, and $\Arr(0^n,\ast)=\emptyset$.

Let $\widetilde{F}\from \two^n\to K\text{-}\topp$ be a spatial refinement of $F$ modeled on $(V_1,V_2)$ for some finite-dimensional orthogonal $K$-representations $V_1$ and $V_2$, and let $\widetilde{F}^+\from \two^n_+ \to
K\text{-}\topp$ be the diagram obtained from $\widetilde{F}$ by setting
$\widetilde{F}^+(\ast)=\pt$; we will sometimes abuse notation by also calling $\widetilde{F}^+$ a spatial refinement of $F$.  Let $\CRealize{F}_{(V_1,V_2)}$ be the
homotopy colimit of $\widetilde{F}^+$ (we will sometimes suppress $(V_1,V_2)$ from the notation).  Sometimes we write $||\widetilde{F}^+||$ to indicate dependence on the choice of spatial refinement.  We call
$\CRealize{F}_{(V_1,V_2)}$ a \emph{(spatial) realization} of $F\from \two^n
\to\burn_K$ for the pair $(V_1,V_2)$.

\begin{cor}[{cf. \cite[Corollary 5.6]{lls1}, \cite[Corollary 4.14]{oddkh}}]
\label{cor:burnlimwelldef}
  Fix a finite-dimensional orthogonal $K$-representation $V_1$.  For $V_2$ a sufficiently large finite-dimensional orthogonal $K$-representation, the realization $\CRealize{F}_{(V_1,V_2)}$ is well-defined up
  to weak equivalence in $K\text{-}\topp$.
\end{cor}
\begin{proof}
  This follows from Proposition
  \ref{prop:cube-non-special} and properties of homotopy colimits
  \ref{itm:ho1}.
\end{proof}

The homotopy colimit $\CRealize{F}$ may be given various CW
structures;  First, from Definition \ref{def:hoco}, there is the
\emph{standard} CW structure, with cells $[0,1]^m \times B_x$,
parameterized by tuples $(f_m,\dots,f_1)$ subject to some relations.  Usually, this will not be a $K$-CW decomposition (as some cells may be, for example, fixed by the action of $K$, but not fixed pointwise, as in the definition of a $K$-CW structure).

We have a second cell structure on $\CRealize{F}$, the \emph{fine} structure, which is obtained from the standard structure by subdividing each cell $[0,1]^m\times B_x$ into $K$-cells (to see that this is possible, see for example \cite{adams} or \cite{illman}).  This is an equivariant cell structure, and so in particular $\CRealize{F}$ has the homotopy type of a $K$-CW complex, although otherwise we will not use this structure.

Further, $\CRealize{F}$ has the \emph{coarse} cell structure of \cite[Section
6]{lls1}.  There they construct a CW structure on $\CRealize{F}$
for $F$ a (not $K$-decorated) Burnside functor, with cells formed by taking unions of the balls $[0,1]^m\times B_x(V)$, so that there is exactly one (non-basepoint) cell
$\cell(x)$ for each $x\in\amalg_u F(u)$.  
The coarse cell structure generalizes in a straightforward way to $K$-equivariant realizations to give a CW complex structure on $||F||$ so that the action of $K$ permutes the cells, but it is not a $K$-CW-structure.  In the sequel, we will treat $||F||$ as a CW complex with the coarse cell structure.

\begin{prop}\label{prop:totalization}
  Fix a finite abelian group $K$, and a finite-dimensional orthogonal $K$-representation $\refl\from K \to O(V_1)$.  If $F\from\two^n\to\burn_K$ is a functor, then the shifted reduced (coarse) cellular complex
  $\redcellC(\CRealize{F}_{(V_1,V_2)})[-\dim V_1-\dim V_2]$ is isomorphic to the totalization
  $\Tot_\degh(F)$ with the cells mapping to the corresponding generators, where $\degh\from K\to \ZZ_2$ is the topological degree of $\refl$. If
  $\eta\from F_1\to F_0$ is a natural transformation of Burnside
  functors, then there is an induced $K$-equivariant cellular map $\eta_*\from \CRealize{F_1}\to\CRealize{F_0}$, and the induced cellular chain map agrees with
  $\Tot_\degh(\eta)$.
\end{prop}

\begin{proof}
  This follows from the proof of Proposition 4.16 of \cite{oddkh}.
\end{proof}

We will use the notion of an equivariant spectrum as in \cite[Chapter XII]{alaska}, which we recall here.  Fix a finite group $G$ and a complete $G$-universe $\mathcal{U}$, that is, an infinite-dimensional real inner product space equipped with an orthogonal $G$-action, so that, for each finite-dimensional orthogonal $G$-representation $V$, $\mathcal{U}$ contains the direct sum of countably many copies of $V$.  We refer to finite-dimensional $G$-subspaces of $\mathcal{U}$ as \emph{indexing spaces}.  A $G$-\emph{prespectrum} $E$ indexed on $\mathcal{U}$ is a family of based $G$-spaces $E(V)$, running over all indexing spaces $V$, together with maps:
\[
\sigma_{V,W}\from \Sigma^{W-V}E(V)\to E(W),
\] 
whenever $V\subset W$, where $\Sigma^{W-V}$ denotes suspension by the orthogonal complement of $V$ in $W$, and where it is required that $\sigma_{V,V}=\Id$.  The following diagram is also required to commute:

\[
\begin{tikzpicture}[scale=1.4,baseline={(current bounding box.center)}]

\node (a1) at (0,0) {$\Sigma^{Z-W}\Sigma^{W-V}E(V)$};
\node (a2) at (2.5,0) {$\Sigma^{Z-W}E(W)$};
\node (b1) at (0,-1.5) {$\Sigma^{Z-V}E(V)$};
\node (b2) at (2.5,-1.5) {$E(Z)$};
\draw[->] (a1) -- (a2) ;
\draw[->] (b1) -- (b2);
\draw[->] (a1) -- (b1) node[pos=.5,anchor=east] {\scriptsize $\cong$};
\draw[->] (a2) -- (b2);

\end{tikzpicture}
\]
A \emph{spectrum} $E$ is a prespectrum so that the adjoints of the maps $\sigma_{V,W}$, given by
\[
\tilde{\sigma}_{V,W}\from E(V)\to \Omega^{W-V}E(W),
\]
are homeomorphisms.

The forgetful functor $\ell$ from spectra to prespectra has a left adjoint $L$, called \emph{spectrification}, so that for $E$ a spectrum, $L\ell E=E$. 

A map of prespectra (or spectra) $f\from E_1\to E_2$ consists of a $G$-equivariant map $E_1(V)\to E_2(V)$ for all indexing spaces $V$, respecting the structure maps.

A \emph{homotopy} of maps of spectra $f_1,f_2\from E_1\to E_2$ is a $G$-equivariant map of prespectra $h\from E_1\wedge I_+\to E_2$, where $I_+$ is the unit interval with a disjoint basepoint added, so that $h|_{E_1\wedge i}=f_i$ for $i=1,2$.  
 
 In order to construct a spectrum from a Burnside functor, we will also consider \emph{virtual representations}, as considered in \cite{costenoble-waner} or \cite{costenoble-may-waner}.
 \begin{defn}\label{def:virtual-rep}
 	For a finite group $K$ (with complete universe $\mathcal{U}$ fixed), the category of \emph{virtual representations} of $K$ has objects the pairs $(V,W)$ for indexing spaces $V,W$.  The pair $(V,W)$ we will usually write as $V\ominus W$.  A \emph{(virtual) map }  $V_1\ominus W_1\to V_2\ominus W_2$ is the equivalence class of a pair of $K$-equivariant isometries:
 	\begin{align*}
 	f\from V_1\oplus Z_1 &\to V_2\oplus Z_2,\\
 	g\from W_1\oplus Z_1 &\to W_2\oplus Z_2,
 	\end{align*}
 	where $Z_1$ and $Z_2$ are indexing spaces.  The equivalence relation between pairs $(f,g)$ as above is generated by setting $(f,g)\sim (f\oplus k,g\oplus k)$ where $k$ is any $K$-isometry $T_1\to T_2$ of indexing spaces, and where $f\oplus k $ and $g\oplus k$ are defined by:
 	\begin{align*}
 	f\oplus k\from V_1\oplus (Z_1+T_1) &\to V_2\oplus (Z_2+T_2),\\
 	g\oplus k\from W_1\oplus (Z_1+T_1) &\to W_2\oplus (Z_2+T_2),
 	\end{align*}
 \end{defn}
 We will write $(V_1\ominus W_1)+ (V_2\ominus W_2)$ for $(V_1\oplus V_2)\ominus (W_1\oplus W_2)$.  We call virtual representations that are of the form $V\ominus 0$ \emph{ordinary representations}.  Associated to any virtual representation there is a natural element of the representation ring of $K$.

\begin{construction}\label{const:output-spectra}
	Fix a finite abelian group $K$, and a finite-dimensional orthogonal $K$-representation $\refl\from K \to O(V_1)$.  Let $\Cat$ be a small category in which every sequence of composable non-identity morphisms has length at most $n$, for some $n$, and let $(F\from \Cat\to \burn_K,W)$ be a stable functor, where we take a virtual representation representative of $W$ as $W=W_1\ominus W_2$.  We define a spectrum $|(F,W)|_{V_1}$ as follows.  We start with the definition of a prespectrum $|(F,W)|_{\mathcal{P},V_1}$.  First, fix $V_2$ a sufficiently large indexing space, and define $|(F,W)|_{\mathcal{P},V_1}(V_2)$ as follows.  We require that the virtual representation $W\oplus V_2$ is isomorphic to a (ordinary) representation, and set $|(F,W)|_{\mathcal{P},V_1}(V_2)=||F||_{(V_1,V_2+W)}$.  The resulting space $|(F,W)|_{\mathcal{P},V_1}(V_2)$ depends on some choices, but its weak-equivalence class is well-defined by Corollary \ref{cor:burnlimwelldef}.  For indexing spaces $V_3$ containing $V_2$, set 
	\[
	|(F,W)|_{\mathcal{P},V_1}(V_3)=\Sigma^{V_3-V_2}||F||_{(V_1,V_3+W)},
	\]
	with the structure maps acting by the identity in the suspension factor.  For indexing spaces $V_3$ so that $V_2\not\subset V_3$, we set $|(F,W)|_{\mathcal{P},V_1}(V_3)=\ast$, a single basepoint, with trivial structure maps.
	
	It is readily confirmed that $|(F,W)|_{\mathcal{P},V_1}$ is a prespectrum, depending on the choices of a virtual representation $W_1\ominus W_2$ underyling $W$, an indexing space $V_2$, as well as a spatial realization of $(V_1,V_2\oplus W)$.  We set $|(F,W)|_{V_1}=L|(F,W)|_{\mathcal{P},V_1}$.  Said differently, $|(F,W)|_{V_1}$ is homotopy equivalent to the suspension spectrum $\Sigma^{W-V_2}\Sigma^{\infty}(||(F,W)||_{(V_1,V_2)})$.
	
	Proposition \ref{prop:spectrum-realization-well-def} confirms that this spectrum is well-defined up to equivariant homotopy. 
\end{construction}

We record a result of \cite{oddkh} (there it is proved for $K=\ZZ_2$; the more general proof is no different):
\begin{prop}[{\cite[Lemma 4.17 ]{oddkh}}]\label{prop:spectrum-realization-well-def}
Let $(F\from \two^n \to \burn_K,W)$ be a stable Burnside functor and $\refl\from K \to O(V_1)$ a finite-dimensional orthogonal representation of $K$.  The spectrum realization $\Realize{\Sigma^W F}_{V_1}$ is well-defined up to $K$-equivariant stable homotopy equivalence.  For stable Burnside functors $(F_i,W_i)$ for $i=1,2$, a $K$-equivariant stable equivalence $\eta\from \Sigma^{W_1}F_1\to \Sigma^{W_2}F_2$ induces a $K$-equivariant stable homotopy equivalence 
\[
|\eta|\from \Realize{\Sigma^{W_1}F_1}_{V_1}\to \Realize{\Sigma^{W_2}F_2}_{V_1}.
\]
\end{prop}
\begin{proof}
	This follows using the fact that $||(F,W)||_{(V_1,V_2)}$ is well-defined up to weak equivalence for $V_2$ sufficiently large, as well as Proposition \ref{prop:cube-non-special}(\ref{itm:cube-non-special-3}).  Here we use that $|\Sigma^W F|_{V_1}$ is the suspension spectrum $\Sigma^{W-V_2}\Sigma^{\infty}(||(F,W)||_{(V_1,V_2)})$.  The construction of $|\eta|$ is as in the proof of Lemma 4.17 of \cite{oddkh}. 
\end{proof}

\section{External actions and realization}\label{sec:external}

Our goal in this section will be to show that, for a Burnside functor $\ff$ with an external action $\psi$, a suitable realization of $\ff$ admits a $G$-action, and the fixed-point set can be explicitly described as a realization of yet another Burnside functor.  In Section \ref{sec:ho-co-external} we deal with some generalities on homotopy coherent diagrams, then specialize to homotopy coherent diagrams from Burnside functors in Section \ref{sec:realizing-external}.  Throughout this section we assume that $K$ is a finite abelian group.  

\subsection{External actions on homotopy coherent diagrams}\label{sec:ho-co-external}

\begin{defn}\label{def:space-ext-action}
Let $F\from \Cat \to \topp$ a homotopy coherent diagram, where $\Cat$ is a small category so that there is some $n$ for which each sequence of composable non-identity morphisms has length at most $n$.  Say that a finite group $G$ acts on $\Cat$ by $\psi$.  An \emph{external action} $\bar{\psi}$ of $G$ on $F$, compatible with $\psi$, is defined as follows.  An external action consists of a homomorphism $\bar{\psi} \from G \to \mathrm{Homeo} (\bigvee_{c\in \Ob(\Cat)} F(c))$ lifting the group action $\psi$ of $G$ on $\Ob(\Cat)$ (and preserving the basepoint).  The action $\bar{\psi}$ is required to `commute with composition' in the following sense:
\begin{equation}\label{eq:pre-equivariance}
\bar{\psi}_g(F(f_i,\dots,f_1)(t_1,\dots, t_{i-1})(y))=F(\psi_g(f_i),\dots,\psi_g(f_1))(t_1,\dots,t_{i-1})(\bar{\psi}_g y),
\end{equation}
for all $g\in G$ and $y\in F(c)$.
For a functor $F \from \Cat \to K\text{-}\topp$, an external action on $F$ is as above but further requiring that the $K$ and $G$ actions commute.  
\end{defn}

\begin{rmk}\label{rmk:ho-action}
A homotopy coherent diagram with external action by $G$ may be thought of as an analogue of a $G$-space in the category of homotopy coherent diagrams.  First, note that a pointed $G$-space $X$ may be viewed as a functor $\underline{X}\from BG \to \topp$, where $BG$ is the category with one object, and morphisms $G$.  A more flexible notion (though equivalent for many purposes, see \cite{dwyer-kan-smith},\cite{cooke}) is a homotopy coherent diagram $\underline{X} \from BG \to \topp$.  

Consider the case of a small category $\Cat$ \emph{without} a $G$-action.  Then one might define a ``$G$-equivariant'' diagram as a homotopy-coherent diagram $BG\times \Cat \to \topp$.  
For the case of present interest, that is, for a small category $\Cat$ with $G$-action, we need a `twisted' version of the above definition, as in Definition \ref{def:space-ext-action}.  
\end{rmk}

\begin{prop}\label{prop:preconditions}
Fix $F \from \Cat \to \topp$, where $\Cat$ has an action $\psi$ and $F$ admits an external action, as in Definition \ref{def:space-ext-action}.  Then the homotopy colimit $\hoco\, F$ admits a $G$-action by 
\[
g(f_m,\dots, f_1;t_1,\dots,t_m;y) = (\psi_g f_m,\dots, \psi_g f_1;t_1,\dots, t_m ; \bar{\psi}_gy).
\]
Similarly, if $F\from \Cat \to K\text{-}\topp$ admits an external action by $G$, the homotopy colimit in $K\text{-}\topp$ inherits a $K\times G$-action by the same formula.
\end{prop}

\begin{proof}
This consists of unraveling the Definition \ref{def:hoco} of homotopy colimits and applying the condition (\ref{eq:pre-equivariance}).  We work with the version of the homotopy colimit in which no nonidentity isomorphisms appear in the index category (as is possible from our hypotheses on $\Cat$).  One first sees by directly considering Definition \ref{def:hoco} that $G$ acts on the homotopy colimit (as a set), and the continuity of the $G$-action in Definition \ref{def:space-ext-action} implies that the $G$-action on the homotopy colimit is continuous.  The $K$-equivariant version is analogous.  
\end{proof}

\begin{defn}\label{def:external-weak-equivalence}
Let $F_1,F_2 \from \Cat \to K\text{-}\topp$ be homotopy coherent diagrams, where $\Cat$ has an action $\psi$ and $F_1$ and $F_2$ admit external actions, all as in Definition \ref{def:space-ext-action}.  We say that $F_1$ and $F_2$ are \emph{externally weakly equivalent} (usually shortened to \emph{weakly equivalent} if the context is clear) if there is a diagram $F_3 \from \two \times \Cat\to K\text{-}\topp$, where $\two \times \Cat$ is given the product $G$-action with $G$ acting trivially on $\two$, so that $F_3|_{i\times \Cat}=F_{i+1}$ and so that $F_3$ itself has an external action.  Furthermore, we require that the maps $F_1(x)\to F_2(x)$ are weak equivalences of $K$-spaces for each $x\in \Cat$.
\end{defn}

\begin{lem}\label{lem:fixed-points}
Let $H$ be a subgroup of $G$, and let $F \from \Cat \to \burn_K$ be as in Definition \ref{def:space-ext-action}.  The $H$-fixed-point set $(\hoco\, F)^H$ of $\hoco\, F$ is the homotopy colimit of the homotopy coherent diagram $F^H\col \Cat^H \to K\text{-}\topp$ with entries $F^H(u)=F(u)^H$, and with morphisms defined as follows.  For a sequence of composable morphisms $(f_i,\ldots,f_1)$ of $\Cat$, with $u\in \Ob(\Cat)$ the domain of $f_1$, the map $F^H(f_i,\dots,f_1)(t_1,\dots,t_{i-1})$ is given by the restriction of the map $F(f_i,\dots,f_1)(t_1,\dots,t_{i-1})$ to $F(u)^H$.  
\end{lem}
\begin{proof}
We describe the fixed-point set of $\hoco \, F$ explicitly.  First, by the construction of homotopy colimits, by applying the relations iteratively, each non-basepoint point in $\hoco\, F$ may be represented (uniquely) by a tuple $(f_m,\dots, f_1;t_1,\dots,t_m;y)$ for $m\geq 0$, with none of $t_i=0,1$.  Such a point is in the fixed-point set if and only if 
\[
(f_m,\dots,f_1)=(h f_m,\dots,h f_1)
\]
as tuples in the set $\Arr(\Cat)$ of homomorphisms of $\Cat$, and $y=hy$. 
That is, the $f_i$ must come from the $H$-fixed arrows, i.e.\ elements of $\Arr(\Cat^H)$, and the lemma follows.
\end{proof}

\subsection{Realizations}\label{sec:realizing-external}

We start with a minor modification of the Construction \ref{const:g-box-action}:
\begin{construction}\label{const:g-box-action-2}
	Fix a finite group $G$, and a finite abelian group $K$, and a pair of finite-dimensional orthogonal $K\times G$-representations $\mathfrak{s}_i\from K\times G\to O(V_i)$ for $i=1,2$.  Let $\underline{X}\from \{*\}\to \burn_K$ be a functor with external action $\psi$ by $G$, where $G$ acts on the one-point category $\{*\}$ (that is the category with a unique object $*$ and no non-identity morphisms) trivially.  Write $X$ for $\underline{X}(\ast)$.  Let $V=V_1\oplus V_2$, and define $B(X,V)=\amalg_{x\in X} B_x(V)$.  Then define an action of $K\times G$ on $B(X,V)$ by setting, for $k\in K$, $g\in G$ and $(v_1,v_2)\in V_1\oplus V_2=V$,
	\[
	(k\times g)\cdot (x,v_1,v_2)=(g\cdot x,\mathfrak{s}_1(k\sigma(\psi_{g,v},x)\times g)\cdot v_1,\mathfrak{s}_2(k\times g)\cdot v_2).
	\]
	Here $\sigma(\psi_{g,v},x)$ denotes the decoration $\sigma(s^{-1}(x))$ for the decorated correspondence $\psi_{g,v}\from X\to X$.  This assignment gives a continuous homomorphism $\rho\from K\times G \to \mathrm{Homeo}(B(X,V))$.
	 
\end{construction}

\begin{lem}\label{lem:conditions}
Let $K$ be an abelian group. Fix a small category $\Cat$ such that (1) there is some $n$ for which each sequence of composable non-identity morphisms has length at most $n$ and (2) $\Cat$ has an action $\psi$ by a finite group $G$.  Fix a Burnside functor $\ff\col \Cat \to \burn_K$ where $\ff$ admits an external action $\bar{\psi}$ by $G$, compatible with $\psi$.   

Let $\tilde{F}$ be a spatial refinement for $\ff$, modeled on $K$-representations $(V_1,V_2)$.  Suppose that $V_1$ and $V_2$ admit $G$-actions commuting with the actions of $K$.  We define an action of $K\times G$ on $B(\bigcup_{u\in \Ob(\Cat)}F(u) ,V_1\oplus V_2)$, as in Construction \ref{const:g-box-action-2}.  Suppose, for each $g\in G$, $u\in \Ob(\Cat)$, $x\in F(u)$, and $p\in B_x/\partial B_x$, that
\begin{equation}\label{eq:equivariance}
g(\tilde{F}(f_i,\dots,f_1)(t_1,\dots, t_{i-1})(p))=\tilde{F}(g(f_i),\dots,g(f_1))(t_1,\dots,t_{i-1})(g p).
\end{equation}   
In this notation, the term $gp$ does not refer to the action of $g\in G$ applied to $p\in V_1\oplus V_2$.  Rather, we view $p$ as an element of $B(\bigcup_{u\in \Ob(\Cat)} F(u),V_1\oplus V_2)$, and $gp$ is the result of applying the action, as in Construction \ref{const:g-box-action-2}, of $G$ on $B(\bigcup_{u\in \Ob(\Cat)} F(u),V_1\oplus V_2)$ to $p$, and similarly for the left-hand side of (\ref{eq:equivariance}).  

Then $\hoco \,\widetilde{F}$ admits a $G$-action, commuting with its natural $K$-action, given by 
\[
(f_m,\dots, f_1;t_1,\dots,t_{m-1};y) \to (gf_m,\dots, g f_1;t_1,\dots, t_{m-1} ; gy).
\]

\end{lem}

The following special case will be our main use of the Lemma.  If $F\from \two^{np} \to \burn_K$ is a Burnside functor admitting an external action by $\ZZ_p$ compatible with the action of $\ZZ_p$ on $\two^{np}$ by permutation of coordinates, with $\widetilde{F}$ satisfying the conditions of Lemma \ref{lem:conditions} for $\Cat=\two^{np}$, we have that $||\widetilde{F}^+||$ admits a $\ZZ_p$-action, commuting with its natural $K$-action, as above.

\smallskip

\noindent \emph{Proof of Lemma \ref{lem:conditions}.}
This follows directly from Proposition \ref{prop:preconditions}.  \qed 
\smallskip

\begin{defn}\label{def:g-coherence}
We call a spatial refinement $\hodit$ of a Burnside functor $\ff\col \Cat\to \burn_K$ with a $G$-external action satisfying (\ref{eq:equivariance}) a $G$-\emph{coherent refinement} of $\ff$.
\end{defn}
We next try to build a homotopy coherent diagram satisfying the conditions of Lemma \ref{lem:conditions}.  The key is to provide a suitable generalization of Proposition 5.2 of \cite{lls1}.

\begin{prop}[{cf. \cite[Proposition 5.2]{lls1},  \cite[Proposition 4.12]{oddkh}}]
\label{prop:5.2-analog}
Let $\Cat$ be a small category admitting a $G$-action $\psi$, in which every sequence of composable
  non-identity morphisms has length at most $n$, and let $F\from
  \Cat\to\burn_K$ be a strictly unitary 2-functor admitting an external $G$-action.  Fix a finite-dimensional orthogonal $K\times G$-representation $\refl \from K \times G\to O(V_1)$.  
\begin{enumerate}[leftmargin=*]
\item There exists some finite-dimensional $K\times G$-representation $W$ so that the following holds.  For all finite-dimensional representations $V_2$ of $K\times G$ which admit an embedding of $W$, there exists a $G$-coherent refinement of $F$ modeled on $(V_1,V_2)$.  \label{itm:real1}
\item\label{itm:real2} There exists some finite-dimensional $K\times G$-representation $W$ so that the following holds.  For all finite-dimensional representations $V_2$ of $K\times G$ which admit an embedding of  $W$, any two $G$-coherent refinements of $F$ modeled on $(V_1,V_2)$ are $K$-weakly equivalent.  Let $\eta$ denote the functor $\eta\from \two\times \Cat\to \burn_K$ that is two copies of $F$ (on $\{0\}\times \Cat$ and $\{1\}\times \Cat$) along with identity arrows along the $\two$-factor.  Then, if $\widetilde{F}_0$ and $\widetilde{F}_1$ are $G$-coherent refinements of $F$ modeled on $(V_1,V_2)$, there exists a $G$-coherent refinement $\widetilde{\eta}$ of $\eta$, so that $\widetilde{\eta}|_{\{0\}\times\Cat }=\widetilde{F}_0$ and  $\widetilde{\eta}|_{ \{1\}\times\Cat}=\widetilde{F}_1$.
\item If $\widetilde{F}_{(V_1,V_2)}$ is a $G$-coherent refinement of $F$ modeled on $(V_1,V_2)$,
  then the result of suspending each $\widetilde{F}_{(V_1,V_2)}(u)$ and
  $\widetilde{F}_{(V_1,V_2)}(f_n, \dots, f_1)$ by a $K\times G$ representation $V'$ gives a $G$-coherent spatial
  refinement of $F$ modeled on $(V_1,V_2\oplus V')$.  \label{itm:real3}
\end{enumerate}
\end{prop}
\begin{proof}

For Item \ref{itm:real1}, we inductively construct a spatial refinement $\widetilde{F}$.  

First, choose representatives $a_\omega$ of the orbits of $\Arr(\Cat)$ under the action of $G$.  For each representative $a_\omega$, let $S_\omega \subset G$ be its stabilizer subgroup.  For each $a_\omega$, choose a $K\times S_\omega$-equivariant disk map refining $F(a_\omega)$; such exist by Lemmas \ref{lem:connctd} and \ref{lem:connctd-2}.  Then, define the maps associated to each $a\in \Arr(\Cat)$ by, if $ga_\omega=a$, setting $g\widetilde{F}(a_\omega)g^{-1}=:\widetilde{F}(ga_\omega)$.  Here, recall that $g\in G$ acts via Construction \ref{const:g-box-action-2}.

It follows from the construction of $\widetilde{F}(a_\omega)$ that $\widetilde{F}(a)$ is independent of the choice of $g$ so that $ga_\omega=a$ holds.  Let us see that the maps constructed thus satisfy Lemma \ref{lem:conditions} for $i=1$.  We need to check that
\[
g\widetilde{F}(f)(p)=\widetilde{F}(gf)(gp)
\] 
for all $g\in G$, arrows $f$ of $\Cat$, and points $p\in B(V)$.  By hypothesis, $f=ha_\omega$ for some $a_\omega$ and $h\in G$.  Then $\widetilde{F}(ha_\omega)$ is defined by $h\widetilde{F}(a_\omega)h^{-1}$, and the $i=1$ case of (\ref{eq:equivariance}) follows readily, using that $s\widetilde{F}(a_\omega)s^{-1}=\widetilde{F}(a_\omega)$ for $s\in S_\omega$.
 
Fix $\ell\geq 1$ and suppose that for any sequence $v_0\to^{f_1}\dots \to^{f_\ell} v_\ell$ of non-identity morphisms we have chosen a map $e_{f_1,\dots,f_\ell}\from [0,1]^{\ell-1}\to E_{K,V}(\{B_x\}_{x \in F(v_0)},s_{F(f_\ell\circ \dots \circ f_1)},F(f_\ell\circ\dots\circ f_1),F(v_0))$, compatible in that:
\begin{align*}
e_{f_\ell,\dots,f_1}(t_1,\dots,t_{i-1},0,t_{i+1},\dots,t_{\ell-1})&=e_{f_\ell,\dots,f_i}(t_{i+1},\dots,t_{\ell-1})\circ e_{f_{i-1},\dots,f_1}(t_1,\dots,t_{i-1})\\
e_{f_\ell,\dots,f_1}(t_1,\dots,t_{i-1},1,t_{i+1},\dots,t_{\ell-1})&=e_{f_\ell,\dots,f_i\circ f_{i-1},\dots,f_1}(t_1,\dots,t_{i-1},t_{i+1},\dots,t_{\ell-1})
\end{align*}
and satisfying the $i=\ell$ condition of Lemma \ref{lem:conditions}.  

Choose representatives $a_\omega$ for the orbits of the $G$-action on the set of all composable tuples $(f_1,\dots,f_\ell)$, where the $f_i$ are morphisms $v_0\to^{f_1}\dots \to^{f_{\ell+1}} v_{\ell+1}$ for $v_i$ objects of $\Cat$, with stabilizers $S_\omega$ as before.  Here, $G$ acts on the set of composable tuples by acting on each of the morphisms in a composable tuple.
Then for the induction step, given $a_\omega=(f_1,\dots,f_{\ell+1})$ where $v_0
  \xrightarrow{f_1} \cdots \xrightarrow{f_{\ell+1}} v_{\ell+1}$ is a composable sequence of arrows, we
  have a continuous map
  \[
  S^{\ell-1}=\partial ([0,1]^\ell) \to E_{K,V}(\{B_x \mid x \in
  F(v_0)\},s_{F(f_{\ell+1}\circ\dots\circ f_1)},F(f_{\ell+1}\circ\dots\circ f_1),F(v_0))
  \]
  defined by
  \begin{equation}\label{eq:homconds3}
    \begin{split}
      (t_1,\dots,t_{i-1},0,t_{i+1},\dots,t_\ell) &\mapsto  e_{f_{\ell+1},\dots,f_{i+1}}(t_{i+1},\dots,t_{\ell})\circ e_{f_i,\dots,f_1}(t_1,\dots,t_{i-1}) \\
      (t_1,\dots,t_{i-1},1,t_{i+1},\dots,t_\ell) &\mapsto
      e_{f_{\ell+1},\dots,f_{i+1}\circ f_i,\dots,
        f_1}(t_1,\dots,t_{i-1},t_{i+1},\dots, t_{\ell}).
    \end{split}
  \end{equation}
  By the argument from Lemma \ref{lem:connctd}, this map extends to a map, call it
  $e_{f_{\ell+1},\dots,f_1}$, from $[0,1]^\ell$, which is $K\times S_\omega$-equivariant.  Define $e_{f_{\ell+1}',\dots,f_1'}$, for $(f_{\ell+1}',\dots,f_1')=ga_\omega$ for some $g\in G$, by $ge_{a_\omega}g^{-1}$.  This is well-defined as in the $i=1$ case (independent of the choice of $g$ for which $(f_{\ell+1}',\dots,f_1')=ga_\omega$) and gives that the collection of $e_{(f_{\ell+1}',\dots,f_1')}$ thus defined satisfy the $i=\ell+1$ case of Lemma \ref{lem:conditions}.  

We have used that external actions respect composition, as in Definition \ref{def:external-action}, in order to see that each $ge_{a_\omega}g^{-1}$ is a family of disk maps refining the composite correspondence $F(f_{\ell+1}'\circ\dots\circ f_1')$.  

By definition, the
  maps
  \[
  \Phi(e_{f_m,\dots,f_1},F(f_m\circ\dots \circ f_1))\from [0,1]^{m-1} \times \bigvee_{x\in F(v_0)} B_x(V)/\partial B_x(V) \to \bigvee_{x\in F(v_{m})}B_x(V) /\partial B_x(V),
  \]
running over the set of all tuples of composable nonidentity arrows $v_0\to^{f_1}\dots \to^{f_m} v_m$, assemble to form a homotopy coherent diagram. 

Next we address point \ref{itm:real2}.  Fix $G$-coherent refinements $\widetilde{F}_i$ of $F$, for $i=0,1$.  It suffices to construct a $G$-coherent refinement $\widetilde{\eta}\from \two\times \Cat \to K\text{-}\topp$ with $\widetilde{\eta}|_{\{i\}\times \Dat}=\widetilde{F_i}$.  Using the construction in the proof of Item \ref{itm:real1} we can construct such $\widetilde{\eta}$, and it follows from the definitions that $\widetilde{\eta}(\phi_{1,0} \times \Id_u)$, for each $u\in \Ob(\Dat)$, will be a homotopy equivalence (where $\phi_{1,0}$ is the unique nonidentity morphism in $\two$).  By \ref{itm:ho1} we have that $\widetilde{F}_0,\widetilde{F}_1$ are then $K$-weakly equivalent.  

Item \ref{itm:real3} is clear. 
\end{proof}

Let us consider the fixed-point sets of the homotopy colimit constructed in Lemma \ref{lem:conditions}.  We state the following result only for Burnside functors from the cube category; the result for general $\Cat$ as in Lemma \ref{lem:conditions} differs only notationally and will not be needed.  Henceforth, we will always view $\two^{np}$ as a category with $\ZZ_p$-action by permuting the coordinates.   The $\ZZ_p$-fixed-point set is readily identified with a copy of $\iota\from \two^n \rightarrow (\two^n)^p$, which we call the \emph{canonical embedding} of cube categories.

\begin{lem}\label{lem:fix-point-coincidence} 
Let $\widetilde{F}$ be a $\ZZ_p$-coherent refinement of $F\from \two^{np}\to\burn_K$, a nonsingular Burnside functor with external action by $\ZZ_p$, compatible with the permutation action on $\two^{np}$.  Say that $\widetilde{F}$ is modeled on $(V_1,V_2)$.  For $H$ a subgroup of $\ZZ_p$, the $H$-fixed-point set, $||\widetilde{F}||^H$, is a $K\times (\ZZ_p/H)$-equivariant realization of the fixed-point Burnside functor $\ff^H$, modeled on $(V_1^H,V_2^H)$.  That is, \[||\widetilde{F}||^H_{(V_1,V_2)}=||\widetilde{F}^H||_{(V_1^H,V_2^H)}.\]  
Additionally, $\widetilde{F}^H$ is a $K$-equivariant refinement of $F^H$.

For a $\ZZ_p$-external natural transformation $\eta\from \two^{np}\times \two \to \burn_K$, where $\ZZ_p$ acts by permutation on $\two^{np}$ and trivially on $\two$, we have similarly $||\tilde{\eta}||^H_{(V_1,V_2)}=||\tilde{\eta}^H||_{(V_1^H,V_2^H)}$.  Finally, $\tilde{\eta}^H$ is a $K$-equivariant refinement of $\eta^H$.  
\end{lem}
\begin{proof}
By Lemma \ref{lem:fixed-points}, $(\hoco\, \widetilde{\ff}^+)^H$ is described explicitly, by restricting to the sub-homotopy-coherent diagram $(\widetilde{F}^H)^+\from (\two^{np})^H_+ \to K\text{-}\topp$.  Recall that $\tilde{F}^H$ has an explicit description as in Lemma \ref{lem:fixed-points}.  We have
\[
(\hoco\, \widetilde{\ff}^+)^H=\hoco_{(\two^{np})^H_+}\, (\widetilde{F}^H)^+.
\]
The homotopy coherent diagram $\widetilde{F}^H$ is a $K$-equivariant refinement of $F^H$, and is in fact $\ZZ_p/H$-coherent (with $\ZZ_p/H$ external action as in Lemma \ref{lem:action-on-fix-functors}), by unwrapping the definitions.  The previous equation then shows that $||\widetilde{F}||^H=\hoco\, (\widetilde{F^H})^+$ for the $\ZZ_p/H$-coherent refinement $\widetilde{F^H}:=\widetilde{F}^H$.  The claim that  $||\widetilde{F}||^H_{(V_1,V_2)}=||\widetilde{F}^H||_{(V_1^H,V_2^H)}$ follows.

The remaining claims in the lemma are proved entirely analogously.  \end{proof}

\begin{lem}\label{lem:sub-burn-func}
Let $\two^n \subset \two^{np}$ be the canonical embedding.  Fix a nonsingular Burnside functor $\ff\col \two^{np} \to \burn_K$ with external action, where $\ff$ admits an external action lifting the $\ZZ_p$-action on $\two^{np}$.  We will denote both actions by $\psi$.  

Let $\ff$ be a $\ZZ_p$-coherent refinement.  Then the fixed-point set $(\hoco\, \ff^+)^{\ZZ_p}$ is $\hoco\, \widetilde{J}^+$, for $\widetilde{J}$ some $K$-equivariant refinement of $\ff^{\ZZ_p}$.
\end{lem}
\begin{proof}
This follows immediately from Lemma \ref{lem:fix-point-coincidence}.
\end{proof}

We can now discuss how the realizations of different Burnside functors are related.

\begin{lem}[{cf. \cite[Lemma 4.15]{oddkh}}]
\label{lem:415}
A cofibration sequence $J \to  F \to H$ of functors with external action $\two^{np}\to \burn_K$ compatible with the permutation action on $\two^{np}$, upon realization, induces a cofibration sequence in $(K\times \ZZ_p)\text{-}\topp$.  In general, any external natural transformation $\eta \from F_1 \to F_0$ of Burnside functors $\two^{np} \to \burn_K$ induces a $(K\times \ZZ_p)$-equivariant map on sufficiently large realizations $||\widetilde{F}_1||\to ||\widetilde{F}_0||$, well-defined up to $K\times\ZZ_p$-equivariant homotopy.
\end{lem}
\begin{proof}
The proof is parallel to that of Lemma 4.15 of \cite{oddkh}, which produces a $K$-equivariant map of realizations as a Puppe map.   We will need some of the details in the proof of Lemma \ref{lem:416}, so we go over the argument.

If $\eta\from\two^{np+1}\to\burn_K$ is the
  natural transformation, then $(F_0)_{\iota_0}$ is a subfunctor and
  $(F_1)_{\iota_1}$ is the corresponding quotient functor, where
  $\iota_i\from\two^{np}\to\two^{np+1}$ is the face inclusion to
  $\{i\}\times\two^{np}$.  For a fixed $K\times \ZZ_p$-representation $V_1$ and for any sufficiently large $K\times \ZZ_p$-representation $V_2$, we have $\ZZ_p$-coherent realizations of $F_0$ and $F_1$, and obtain a cofibration sequence, as in the proof of Lemma 4.15 of \cite{oddkh}:
  \begin{equation}\label{eq:sss-415}
    \CRealize{(F_0)_{\iota_0}}_{(V_1,V_2)}\to\CRealize{\eta}_{(V_1,V_2)}\to\CRealize{(F_1)_{\iota_1}}_{(V_1,V_2)}.
  \end{equation}
  However, $\CRealize{(F_0)_{\iota_0}}_{(V_1,V_2)}=\CRealize{F_0}_{(V_1,V_2)}$, while
  $\CRealize{(F_1)_{\iota_1}}_{(V_1,V_2)}=\Sigma\CRealize{F_1}_{(V_1,V_2)}$ since
  $\CRealize{F_1}_{(V_1,V_2)}$ is constructed as a homotopy colimit over $\two_+^n$, while
  $\CRealize{(F_1)_{\iota_1}}_{(V_1,V_2)}$ is constructed as a homotopy colimit over
  $\two_+^{n+1}$. Therefore, the Puppe map
  \[
    \CRealize{(F_1)_{\iota_1}}_{(V_1,V_2)}=\Sigma\CRealize{F_1}_{(V_1,V_2)}=\CRealize{F_1}_{(V_1,V_2\oplus \mathbb{R})}\to\Sigma\CRealize{(F_0)_{\iota_0}}_{(V_1,V_2)}=\Sigma\CRealize{F_0}_{(V_1,V_2)}=\CRealize{F_0}_{(V_1,V_2\oplus \mathbb{R})}
  \]
  is the required map.
To see that the map is also $\ZZ_p$-equivariant, we use that, under the hypothesis of Lemma \ref{lem:415}, the cofibration sequence itself is $\ZZ_p$-equivariant, from which the Puppe map can be chosen to be $\ZZ_p$-equivariant.  Well-definedness of the map up to $K\times \ZZ_p$-equivariant homotopy follows from the well-definedness of Puppe maps.

\end{proof}

Write $\eta_*$ for the map $||F_1||_{(V_1,V_2)}\to ||F_0||_{(V_1,V_2)}$ as in Lemma \ref{lem:415}.  Recall that neither the coarse nor standard CW structures need be equivariant CW structures.

\begin{lem}[{cf. \cite[Proposition 4.16]{oddkh}}]
\label{lem:416}
Let $F \from \two^{np} \to \burn_K$ be a Burnside functor with $\ZZ_p$-external action compatible with the permutation action, and fix a coherent realization $\CRealize{F}$ modeled on $(V_1,V_2)$ with $V_2$ sufficiently large.  Then:
\begin{enumerate} \item The shifted reduced coarse cellular complex $\redcellC(\CRealize{F})[-\dim V_1-\dim V_2]$ is isomorphic (as a chain complex) to the totalization $\Tot_\degh(F)$, with the cells mapping to the corresponding generators.  Here, $\degh$ is the topological degree of the action of $K$ on $V_1$.  
	\item\label{itm:2-416}  If $\eta \from F_1 \to F_0$ is an external natural transformation, then the map $\eta_*\from ||F_1|| \to ||F_0||$ is cellular with respect to the coarse CW structure, and the induced cellular map on the coarse structure agrees with $\Tot_\degh(\eta)$.  
	\item\label{itm:3-416}  If $F$ is nonsingular, the restriction to fixed points $(\eta_*)^H\from ||F_1||^H \to ||F_0||^H$, for $H$ a subgroup of $\ZZ_p$, is $K\times (\ZZ_p/H)$-equivariantly homotopic to  $(\eta^H)_*\from ||F_1^H||_{(V_1^H,V_2^H)}\to ||F_0^H||_{(V_1^H,V_2^H)}$, the map of realizations induced by the $H$-fixed-point functor $\eta^H\from F^H_1\to F^H_0$.  Here we have used $||F_i||^H_{(V_1,V_2)}=||F_i^H||_{(V_1^H,V_2^H)}$, for suitable realizations, by Lemma \ref{lem:fix-point-coincidence}.  
	\item\label{itm:4-416} Finally, $(\eta_*)^H$ is a cellular map on the coarse CW structures on $||F_i||^H_{(V_1,V_2)}$.  The induced cellular chain map on the $H$-fixed points, in the coarse CW structure of the $H$-fixed-point set, is $\Tot_\degh(\eta^H)$. 
	\end{enumerate}
\end{lem}

\begin{proof}
The first two claims are just Proposition \ref{prop:totalization}, and do not involve the external action.    
 
   Let us fix $\tilde{\eta}$ a $\ZZ_p$-equivariant spatial refinement of $\eta$, with restriction $\tilde{F}_i$ on $\two^{np}\times\{i\}$ for $i=0,1$.  By taking fixed points in the cofibration sequence (\ref{eq:sss-415}), we have a cofibration sequence
   \begin{equation}\label{eq:cofib-2}
   ||\tilde{F}_0||^H\to ||\tilde{\eta}||^H\to ||\tilde{F}_1||^H.
   \end{equation}  The $K\times (\ZZ_p/H)$-equivariant homotopy type of the Puppe map for the sequence (\ref{eq:cofib-2}) is exactly $(\eta_*)^H$, using that $\eta_*$ is defined as a Puppe map.  We have that $\tilde{\eta}^H$ is a refinement of $\eta^H$, by Lemma \ref{lem:fix-point-coincidence}, and so the cofibration sequence above defines the homotopy class of the map $(\eta^H)_*$ in Proposition \ref{prop:totalization}.   
   
   The claim (\ref{itm:4-416}) is a consequence of items (\ref{itm:2-416}) and (\ref{itm:3-416}).
\end{proof}

\begin{lem}\label{lem:414}
Let $F \from \two^{np} \to \burn_K$ be a nonsingular Burnside functor with external action by $\ZZ_p$.  For a fixed finite-dimensional orthogonal $K\times \ZZ_p$-representation $V_1$ and for $V_2$ sufficiently large, $||F||_{(V_1,V_2)}$ is well-defined up to weak equivalence in $(K\times \ZZ_p)\text{-}\topp$.
\end{lem}
\begin{proof}
Fix coherent refinements $\widetilde{F}_0$ and $\widetilde{F}_1$ modeled on $(V_1,V_2)$ for $V_2$ large.  By Proposition \ref{prop:5.2-analog}(\ref{itm:real2}), there is a homotopy coherent diagram $\widetilde{\eta} \from \two\times \two^{np} \to K\text{-}\topp$ so that $\widetilde{\eta}|_{i\times \two^{np}}=\widetilde{F}_i$.  We need to show that, for any subgroup $H\subset K\times \ZZ_p$, the induced map
\[
(\eta_*)^H\colon ||\tilde{F}_1||^H\to ||\tilde{F}_0||^H
\]
is a (nonequivariant) homotopy equivalence. 

 We treat first the case that $H$ is contained in the $\ZZ_p$ factor.  We observe that for any homomorphism $\degh\from K\to \ZZ_2$, $\Tot_\degh(\eta^H)$ is the identity.  By Lemma \ref{lem:416}(\ref{itm:4-416}), $\eta^H$ induces a map between cellular chain complexes 
\[
\redcellC(||\tilde{F}_1||^H)\to \redcellC(||\tilde{F}_0||^H),
\]
which may be identified with $\Tot_\degh(\eta^H)$ up to a shift.  Then, for each subgroup $H\subset \ZZ_p$, $\eta^H$ is a $K\times (\ZZ_p/H)$-equivariant map which is \emph{a priori} a homotopy equivalence in the nonequivariant sense, since $\Tot_\degh(\eta^H)$ is identified with the identity map.  

For a general subgroup $H\subset K\times \ZZ_p$ (that is, a subgroup which need not be contained in the $\ZZ_p$ factor), let $H'$ denote the image of $H$ in $\ZZ_p$.  It is a consequence of the formula for the action of $K\times G$ on the homotopy colimit in Lemma \ref{lem:conditions} that $||\tilde{F}_i||^{H}$ is a spatial realization of the fixed-point functor $F^{H'}$, modeled on $(V_1^H,V_2^H)$, for $i=1,2$.  By the same argument, $||\tilde{\eta}||^H$ is a spatial realization of the fixed-point functor $\eta^{H'}$, modeled on $(V_1^H,V_2^H)$.  As in the proof of Lemma \ref{lem:415}, the map
\begin{equation}\label{eq:eta-general-h}
(\eta_*)^H\colon ||\tilde{F}_0||^H\to ||\tilde{F}_1||^H
\end{equation}
is then identified up to homotopy with the Puppe map 
\[
(\eta^{H'})_*\colon ||F_0^{H'}||_{(V_1^H,V_2^H)}\to ||F_1^{H'}||_{(V_1^H,V_2^H)}.
\]
We have already established that the Puppe maps $(\eta_*)^{H''}$ are homotopy equivalences when $H''$ is a subgroup of $\ZZ_p\subset K\times \ZZ_p$, so $\eta^{H}$ in (\ref{eq:eta-general-h}) is an equivalence for all $H$.  This establishes that $||F||_{(V_1,V_2)}$ is well-defined up to weak equivalence, as needed.
\end{proof}

In order to describe the relationship between realizations of externally stably equivalent Burnside functors, we need a further object.  Let $J_p\from \two^p \to \burn_K$  be the nonsingular Burnside functor (with external action by $\ZZ_p$) with $J_p(1^p)$ a $1$-element set, and $J_p(v)=\emptyset$ for $v\neq 1^p$.  

\begin{lem}\label{lem:equivariant-spheres}
For any pair of finite-dimensional orthogonal $K\times \ZZ_p$ representations $(V_1,V_2)$, the realization of $J_p$ satisfies $||J_p||_{(V_1,V_2)}=\Sigma^{V_1\oplus V_2} (\mathbb{R}(\ZZ_p))^+$.
\end{lem}
\begin{proof}
By \ref{itm:ho4}, it suffices to prove the lemma in the case that $\widetilde{J}_p(1^p)=S^0$, with the trivial $\ZZ_p$-action.  From \cite[Proposition 6.1]{lls1}, there is a nonequivariant identification of $||\widetilde{J}_p||$ with $M_p\times [0,2]/\partial (M_p\times [0,2])$, where $M_p$ is the permutahedron on $p$ symbols.  

Recall that the permutahedron $M_p$ is the $(p-1)$-dimensional convex hull of the orbit of $(1,2,\ldots, p)\in \mathbb{R}^p$ under the action of the symmetric group $S_p$ on $p$ letters (where $S_p$ acts on $\mathbb{R}^p$ by permuting the coordinates).  

Moreover, $M_p$ has the structure of a cubical complex, so that the $(m-1)$-dimensional cubes are in correspondence with sequences $(u^0>u^1>\dots > u^{m-1})$ of objects of $\two^{p}$ with $u^i\neq 0^p\in \Ob(\two^{p})$ for all $i$.  

The $\ZZ_p$-action on $M_p\times [0,2]/\partial (M_p\times [0,2])$ determined by its isomorphism with $||\widetilde{J}_p||$ is given by permuting the cubes (taking each cube $[0,1]^{m-1}$ labeled by a sequence $(u^0>u^1>\dots> u^{m-1})$ to the cube $[0,1]^{m-1}$ labeled by some other sequence $(v^0>v^1>\dots> v^{m-1})$ by the identity map $[0,1]^{m-1}\to [0,1]^{m-1}$) and acting by the identity on the $[0,2]$ factor, by direct inspection of the proof of \cite[Proposition 6.1]{lls1}.  

The identification between $M_p$, viewed as a cubical complex and viewed as a convex hull in $\mathbb{R}^p$ is such that each cube is mapped linearly to $\mathbb{R}^p$.  In particular, the \emph{vertices} of the polytope $M_p$ are exactly the $(p-1)$-tuples $(0,\dots,0)$ in the cubes associated to maximal sequences $(u^0>u^1>\dots u^{m-1})$ of objects of $\two^{p}$ with $u^i\neq 0^p\in \Ob(\two^{np})$ for all $i$.  These maximal sequences are in bijection with permutations of the set $\{1,\dots,p\}$.   

Finally, the $\ZZ_p$-action on $M_p$ induced by the isomorphism with $||\widetilde{J}_p||$ agrees with the $\ZZ_p$-action on the convex hull by the permutation $(1\dots p-1,p)\to (2\dots p, 1)$.  To see this, by the preceding paragraph it suffices to identify the two actions on the vertices.  The identification on vertices follows from the descriptions of the two actions above.

Then $M_p\times [0,2]$ is identified (not metrically, but topologically) with the unit ball in $\RR(\ZZ_p)$, and the lemma follows.
\end{proof}

We also need a simple fact about indexing categories:

\begin{lem}\label{lem:index-cat}
The natural $\ZZ_p$-equivariant functor $\two^p_+\times \two^{np}_+\to \two^{(n+1)p}_+$ is homotopy cofinal.
\end{lem}
\begin{proof}
This is straightforward. 
\end{proof}

\begin{prop}\label{prop:stable-equivalences-realization}
An external $K$-equivariant stable equivalence $(E_1,W_1) \to (E_2,W_2)$ of stable nonsingular functors $(E_1 \from \two^{n_1p} \to \burn_K, W_1)$ and $(E_2\from \two^{n_2p} \to \burn_K, W_2)$ induces a $K\times \ZZ_p$-equivariant homotopy equivalence $|\Sigma^{W_1}E_1|_{V_1}\to |\Sigma^{W_2}E_2|_{V_1}$ for any finite-dimensional orthogonal $K\times \ZZ_p$-representation $V_1$.
\end{prop}
\begin{proof}
We need only check that the operations (\ref{itm:stable-1}) and (\ref{itm:stable-2}) of Definition \ref{def:stableq} induce equivariant homotopy equivalences.  

For operation (\ref{itm:stable-1}), say we have a natural transformation $F_1\to F_{2}$ of Burnside functors with external action.  Fix a finite-dimensional orthogonal $K\times \ZZ_p$-representation $V_1$.  Associated to a natural transformation with external action, there is a map $||F_1||_{(V_1,V_2)} \to ||F_2||_{(V_1,V_2)}$ for any realizations, for $V_2$ sufficiently large, by Lemma \ref{lem:415}.  By Lemma \ref{lem:416}, the resulting (equivariant) map is a $K$-homotopy equivalence $||F_1^H||_{(V_1,V_2)} \to ||F_2^H||_{(V_1,V_2)}$ for all $H$ (having applied the Whitehead theorem on each fixed-point set).  By the $G$-Whitehead theorem, we have that $F_1$ and $F_2$ are $(K\times \ZZ_p)$-equivariantly homotopy equivalent.  

We next deal with operation (\ref{itm:stable-2}).  It will suffice to show that for the face inclusion $\iota \from \two^{np}\to \two^{(n+1)p}$ and any Burnside functor $F\from \two^{np}\to \burn_K$, that $\Sigma^{\mathbb{R}(\ZZ_p)}||F||_{(V_1,V_2)}$ is (equivariantly) homotopy equivalent to $||F_\iota||_{(V_1,V_2)}$ for any $V_2$ sufficiently large.  We will check this using the relationship of homotopy colimits to smash products.

First, observe by \ref{itm:ho2} that we have a natural $K$-weak equivalence:
\begin{equation}\label{eq:natural-map}
(\hoco_{\two^p_+} \widetilde{J}^+) \wedge (\hoco_{\two^{np}_+} \widetilde{F}^+) \to \hoco_{\two^p_+\times \two^{np}_+} (\widetilde{J}^+\wedge \widetilde{F}^+).
\end{equation}
We must check that this map is $\ZZ_p$-equivariant.  To do so, we would like to use naturality of the map in \ref{itm:ho2}.  In order to use that naturality, we need to use the external action to generate homomorphisms of homotopy coherent diagrams.

Choose a generator $g\in \ZZ_p$.  Let $\mathbf{F}_{g^{-1}}\from \two^{np}_+\to \two^{np}_+$ be the functor of the action of $g^{-1}$ on $\two^{np}_+$.  We consider the pullback homotopy coherent diagram $\mathbf{F}_{g^{-1}}(\widetilde{F}^+)=\widetilde{F}^+\circ \mathbf{F}_{g^{-1}}$.  The external action of $\ZZ_p$ on $\widetilde{F}^+$ is precisely the data of a homomorphism of homotopy coherent diagrams $\Psi_{g}\from \mathbf{F}_{g^{-1}}(\widetilde{F}^+)\to \widetilde{F}^+$, by Definition \ref{def:space-ext-action}.  We then obtain a well-defined $K$-equivariant map, by \ref{itm:ho-new},
\[
\hoco\, \mathbf{F}_{g^{-1}}(\widetilde{F}^+) \to \hoco\, \widetilde{F}^+.
\]
Note that $\hoco\, \mathbf{F}_{g^{-1}}(\widetilde{F}^+)$ is not identical to $\hoco \, \widetilde{F}^+$.  However, there is a natural (in homomorphisms of homotopy coherent diagrams with external $G$-action) homeomorphism between $\hoco\, \mathbf{F}_{g^{-1}}(\widetilde{F}^+)$ and $\hoco \, \widetilde{F}^+$
defined essentially by relabeling, as follows.

Recall that $\hoco \, \widetilde{F}^+$ is defined as a quotient of the disjoint union of certain spaces $[0,1]^n\times \widetilde{F}(x_0)$, parameterized by tuples of arrows $(f_1,\dots,f_n)$ of $\two^{np}$ with $f_i\from x_{i-1}\to x_i$ for objects $x_i$ of $\two^{np}$.  Now $\hoco\, \mathbf{F}_{g^{-1}}(\widetilde{F}^+)$ is a quotient of a disjoint union of spaces $[0,1]^n\times \mathbf{F}_{g^{-1}}(\widetilde{F}^+)(x_0)$ parameterized by tuples of arrows $(f_1,\dots,f_n)$ of $\two^{np}$ with $f_i\from x_{i-1}\to x_i$ for objects $x_i$ of $\two^{np}$.  We have a homeomorphism 
\begin{equation}\label{eq:relabel-homeo}
\hoco \, \widetilde{F}^+\to \hoco\, \mathbf{F}_{g^{-1}}(\widetilde{F}^+)
\end{equation}
by identifying the space $[0,1]^n \times \widetilde{F}(x_0)$, associated to the tuple $(f_1,\dots,f_n)$ to the space $[0,1]^n\times \mathbf{F}_{g^{-1}}(\widetilde{F}^+)(gx_0)$ associated to the tuple $(gf_1,\dots,gf_n)$.  Here, we have used that $\mathbf{F}_{g^{-1}}(\widetilde{F}^+)(gx_0)$ and $\widetilde{F}(x_0)$ are equal.  We choose the identification between $[0,1]^n\times \tilde{F}(x_0)$ and $[0,1]^n\times \mathbf{F}_{g^{-1}}(\widetilde{F}^+)(gx_0)$ to be the identity on the $[0,1]^n$ factor.  It is direct to check that this identification is compatible with the gluing rules in Definition \ref{def:hoco}.

  All of this discussion applies equally well, replacing $\widetilde{F}^+$ with $\widetilde{J}^+$ or $\widetilde{F}^+\wedge \widetilde{J}^+$.  In fact, we have a commutative diagram, where the vertical arrows are homomorphisms:
 \[
\begin{tikzpicture}[baseline={([yshift=-.8ex]current  bounding  box.center)},xscale=7.5,yscale=1.5]
\node (a0) at (0,1) {$(\hoco_{\two^p_+} (\mathbf{F}_{g^{-1}}\widetilde{J}^+)) \wedge (\hoco_{\two^{np}_+} (\mathbf{F}_{g^{-1}}\widetilde{F}^+))$};
\node (a1) at (1,1) {$\hoco_{\two^p_+\times \two^{np}_+}(\mathbf{F}_{g^{-1}} (\widetilde{J}^+\wedge \widetilde{F}^+))$};
\node (b0) at (0,0) {$(\hoco_{\two^p_+} \widetilde{J}^+) \wedge (\hoco_{\two^{np}_+} \widetilde{F}^+)$};
\node (b1) at (1,0) {$\hoco_{\two^p_+\times \two^{np}_+} (\widetilde{J}^+\wedge \widetilde{F}^+)$};

\draw[->] (a0) -- (a1) node[pos=0.5,anchor=north] {};
 \draw[->] (a0) -- (b0) node[pos=0.5,anchor=east] {};
\draw[->] (a1) -- (b1) node[pos=0.5,anchor=west]{};
\draw[->] (b0) -- (b1) node[pos=0.5,anchor=north] {};

\end{tikzpicture}
\]
Moreover, the $\ZZ_p$-action on $\hoco \,\widetilde{F}^+$ fits into the commutative diagram
 \[
\begin{tikzpicture}[baseline={([yshift=-.8ex]current  bounding  box.center)},xscale=2,yscale=2.5]
\node (a0) at (0,1) {$\hoco_{\two^{np}_+} (\mathbf{F}_{g^{-1}}\widetilde{F}^+)$};
\node (a1) at (1,0.5) {$\hoco_{\two^{np}_+} \widetilde{F}^+$};
\node (b0) at (0,0) {$\hoco_{\two^{np}_+} \widetilde{F}^+$};

 \draw[->] (a0) -- (b0) node[pos=0.5,anchor=east] {$\Psi_g$};
\draw[->] (a1) -- (b0) node[pos=0.3,anchor=north]{$g$};
\draw[->] (a0) -- (a1) {};

\end{tikzpicture}
\]
where the vertical arrow is the map induced by the homomorphism, and the diagonal arrow labeled by $g$ is as in the definition of the $\ZZ_p$-action on $\hoco\, \widetilde{F}^+$.  The remaining diagonal map is the inverse of the identification from (\ref{eq:relabel-homeo}). The analogous diagrams for $(\hoco_{\two^p_+} \widetilde{J}^+) \wedge (\hoco_{\two^{np}_+} \widetilde{F}^+)$ and $\hoco_{\two^p_+\times \two^{np}_+} (\widetilde{J}^+\wedge \widetilde{F}^+)$ also commute.  Using the above commutative square, and the naturality of \ref{itm:ho2} with respect to homomorphisms, we see that (\ref{eq:natural-map}) is $\ZZ_p$-equivariant.  Moreover, we have that the map
\[
(\hoco_{\two^p_+} \widetilde{J}^+) \wedge (\hoco_{\two^{np}_+} \widetilde{F^H}^+) \to \hoco_{\two^p_+\times \two^{np}_+} (\widetilde{J}^+\wedge \widetilde{F^H}^+).
\]
is a $K$-weak equivalence for each subgroup $H$ of $\ZZ_p$, by the same argument.  That is, (\ref{eq:natural-map}) is a $K\times \ZZ_p$-weak equivalence.

We note that $\widetilde{J}^+\wedge \widetilde{F}^+$ is the pullback of some $G$-coherent spatial refinement $\widetilde{J \times F}^+$ under $L\from \two^p_+\times \two^{np}_+\to \two^{(n+1)p}_+$, as a consequence of the definitions. 
Moreover, it is immediate from the definitions that $J\times F=F_\iota$.

Using Lemma \ref{lem:index-cat} and \ref{itm:ho4}, we have a $K$-homotopy-equivalence
\[
\hoco_{\two^p_+\times \two^{np}_+} (\widetilde{J}^+\wedge \widetilde{F}^+)\simeq \hoco_{\two^{(n+1)p}_+} (\widetilde{J\times F}^+).
\]
Moreover, this homotopy-equivalence is once again equivariant with respect to the $\ZZ_p$-action, because the homotopy equivalence is natural in the involved diagrams.  For each subgroup $H\subset \ZZ_p$, we obtain a similar map of $H$-fixed-point sets.  However, the same hypotheses we have used to this point apply to the $H$-fixed-point sets, since they come from refinements of the Burnside functor $F^H$, according to Lemma \ref{lem:fix-point-coincidence}.  That is, the map on $H$-fixed-point sets is also a homotopy equivalence, and by the $G$-Whitehead Theorem, we have obtained:  
\[
(\hoco_{\two^p_+} \widetilde{J}^+) \wedge (\hoco_{\two^{np}_+} \widetilde{F}^+)\simeq \hoco_{\two^{(n+1)p}_+} (\widetilde{J\times F}^+),
\]
now $K\times \ZZ_p$-equivariantly.  Applying Lemma \ref{lem:equivariant-spheres}, the result follows.
\end{proof}

\section{Applications to Khovanov spectra and Khovanov homology}\label{sec:app}

In this section, we recall the definition and main properties of Khovanov spectra from \cite{lls1}, as well as the generalization of the Lawson-Lipshitz-Sarkar construction to the odd Khovanov case \cite{oddkh}.  

Fix a link $L$ with diagram $D$, from which we obtain the Khovanov functor $\khfunc(D) \from \two^n\to \ZZ\Mod$ and the odd Khovanov functor $\khofunc(D) \from \two^n \to \ZZ\Mod$; we will often omit the diagram $D$ from the notation where it is clear from context.  In \cite{lls1}, Lawson-Lipshitz-Sarkar extended $\khfunc \from \two^n\to \ZZ\Mod$ to a stable Burnside functor $\khburn \from \two^n \to \burn$:

\[
\begin{tikzpicture}[baseline={([yshift=-.8ex]current  bounding  box.center)},xscale=2.5,yscale=1.5]
\node (a0) at (0,0) {$\two^n$};
\node (a1) at (1,0) {$\ZZ\Mod$};
\node (b1) at (1,1) {$\burn$};

\draw[->] (a0) -- (a1) node[pos=0.5,anchor=north] {\scriptsize
  $\khfunc$}; \draw[->] (b1) -- (a1) node[pos=0.2,anchor=east] {};
\draw[->,dashed] (a0) -- (b1) node[pos=0.5,anchor=south east]{\scriptsize $\khburn$};

\end{tikzpicture}
\]
In \cite{oddkh}, $\AbFunc_o(L)$ was extended to a stable functor $\khoburn \from \two^n \to \burn_{\ZZ_2}$, so that for $\dig=0 \from \ZZ_2\to \ZZ_2$, $\Tot_{\dig}(\khoburn)=\khfunc$ and for $\dig=\Id$, $\Tot_{\dig}(\khoburn)=\khofunc$.  In \cite{oddkh}, it was shown that the equivariant stable-equivalence class of $\khoburn(D)$ is an invariant of the link $L$.  From $\khoburn$, one can construct an infinite family $\X_n(L)$ of Khovanov spaces (or spectra), well-defined up to stable homotopy, whose definition we will see in Section \ref{sec:kh-burn-func}.

Once we have recalled these definitions, we will see in Section \ref{subsec:63} that the machinery of Sections \ref{sec:burn}--\ref{sec:external} applies to the Khovanov-Burnside functors $\khoburn$ and $\khburn$.  That is, we will show that $\khburn$ and $\khoburn$, as well as their annular analogs, admit external actions in various settings.  This is largely, but not entirely, formal.  In Section \ref{subsec:fixed-point}, we will show that the fixed-point functors of these Khovanov-Burnside functors agree with certain \emph{annular Khovanov-Burnside functors}.  This section is not formal, and relies on understanding the relationship between resolution configurations in the periodic link and the quotient link; this becomes particularly complicated in the odd case.  In Section \ref{subsec:reidemeister}, we show that the external actions are well-defined; this section is largely formal once an understanding of the fixed-point functors is obtained.  One can also obtain the results of this section without knowing the fixed-point functor explicitly, but it is somewhat easier with the results of Section \ref{subsec:fixed-point} in hand.  Here we also wrap up the construction of space-level invariants, proving Theorem \ref{thm:submain} using the tools from Section \ref{sec:external}.  We end with some spectral sequences in Section \ref{subsec:smith}, and some questions in Section \ref{subsec:questions}.

\subsection{The Khovanov-Burnside functor}\label{sec:kh-burn-func}

The purpose of this section is to explicitly describe the various Burnside functors we will use (cf.\ \cite[Sect.\ 6]{lls2}). 

We start by recalling the construction of the functor $\khoburn(D)$, for $D$ a diagram of an oriented link $L$, with $n$ ordered crossings, and a choice of orientation of crossings, as well as a choice of edge assignment as in Section \ref{subsec:okh-cpx}, and finally an ordering of the circles of each resolution.  Following Lemma \ref{lem:212}, it suffices to define $\khoburn(D)$ on objects, edges $\phi_{u,v}$ with $u\geqslant_1 v$, and across two-dimensional faces of the cube
$\two^n$.  For brevity, we write $\khoburn$ for $\khoburn(D)$.  On objects, set
\[
\khoburn(u)=\KhGen(u).
\]
For each edge $u\geqslant_1 v$ in $\two^n$, and each element
$y \in \khoburn(v)$, write
\[
\AbFunc_{o}(\phi^{\op}_{v,u})(y)=\sum_{x\in \khoburn(u) } \epsilon_{x,y} x.
\]
Note each $\epsilon_{x,y}\in \{ -1,0,1\}$.
Define
\[
\khoburn(\phi_{u,v})=\{ (y,x) \in \khoburn(v) \times \khoburn(u) \mid
\epsilon_{x,y}=\pm 1 \},
\]
where the sign on elements of $\khoburn(\phi_{u,v})$ is given by
$\epsilon_{x,y}$ of the pair, and the source and target maps are the
natural ones.

We need only define the $2$-morphisms across $2$-dimensional faces.
In fact, there is a unique
choice of $2$-morphisms compatible with the preceding data.  To be
more specific, for any $2$-dimensional face
$u\geqslant_1v,v'\geqslant_1w$, and any pair
$(x,y) \in \khoburn(u) \times \khoburn(w)$, there is a unique bijection
between
\[
A_{x,y}:=s^{-1}(x)\cap t^{-1}(y)\subset \khoburn(\phi_{v,w})\times_{\khoburn(v)} \khoburn(\phi_{u,v})
\] 
and
\[
A_{x,y}':=s^{-1}(x)\cap t^{-1}(y)\subset \khoburn(\phi_{v',w})\times_{\khoburn(v')} \khoburn(\phi_{u,v'})
\]
that preserves the signs.  (That is, the signed sets
$A_{x,y},A_{x,y}'$ both have at most one element of any given sign.)  Indeed, the only resolution configurations for which $A_{x,y}$ has more than one element are the ladybug configurations.  The unique sign-preserving matching turns out to be the right ladybug matching of \cite{lshomotopytype} for a type X edge assignment, and is the left ladybug matching for a type Y edge assignment.  This completes the description of a strictly unitary lax $2$-functor $\khoburn_0(D)$ associated to the data as above.  We call the identification (for any $x,y$) of sets $A_{x,y}$ and $A_{x,y}'$ above the {ladybug matching}. Recall also that we work with \emph{stable} Burnside functors; that is, pairs of a Burnside functor and an integer.  We define the \emph{(odd) Khovanov-Burnside functor} by $\khoburn=(\khoburn_0,-n_-)$, the Burnside functor shifted down by $n_-$. 
It follows from the construction that we may write $\khoburn$ as a sum over quantum gradings: $\khoburn=\amalg_j \khoburn^j$, where $\khoburn^j$ is the subfunctor from Khovanov generators in quantum grading $j$.

 Recall that the equivariant stable equivalence class of $\khoburn$ is a link invariant, as in the following theorem.  In what follows, let $\tilde{\mathbb{R}}$ be the nontrivial one-dimensional real representation of $\ZZ_2$. 

\begin{thm}[{\cite[Theorem 1.7 ]{oddkh}}]\label{thm:oddkh-main}
The equivariant stable equivalence class of the stable functor $\khoburn$ is independent of the choices in its construction, and is a link invariant. Let $\widetilde{\khoburn}^j_{n}$ be a spatial refinement of $\khoburn^j$ in sufficiently high dimension, modeled on $\tilde{\mathbb{R}}^n$.  Then the $\ZZ_2$-equivariant stable homotopy type of the spatial realization $\X_n^j=||(\widetilde{\khoburn}^j_{n})^+||$ is a link invariant.  Moreover, there is a CW structure on $\X_n^j$ for which the reduced cellular chain complex $\redcellC^*(\X_n^j)=\oddKhCx^j(L;\ZZ)=\Tot_{\Id}(\khoburn^j)^*$ if $n$ odd, or $\KhCx^j(L;\ZZ)=\Tot_{\dig=0}(\khoburn^j)^*$ if $n$ is even.
\end{thm}

Let $\khburn=\forgot (\khoburn)$ denote the stable Burnside functor obtained by applying the forgetful functor $\burn_{\ZZ_2}\to \burn$; call this the \emph{even Khovanov-Burnside functor}; it agrees with the construction of \cite{lls1}. We illustrate an example Khovanov-Burnside functor in Figure \ref{fig:burnside-dots}.

\begin{figure}

\begin{tikzcd}[column sep={.8cm,between origins}]
	& \underset{1}{\bullet} \arrow{rrr}{a_1} & 
					& &\underset{1}{\bullet} & \\
	\underset{x_1}{\bullet}\arrow[bend left]{rrrr}{a_2} &  & \underset{x_2}{\bullet} \arrow[bend right]{rr}{a_3}& 
					& \underset{y_1}{\bullet} \\
	& \underset{x_1x_2}{\bullet} & & & & 
\end{tikzcd}
\caption{An example Burnside functor $F\from \two^1 \to \burn_{\ZZ_2}$.  We visualize elements of $F(1),F(0)$ as dots, and regard the morphism $F(\phi_{1,0})$ as a collection of arrows.  Here, we let $F(1)=\{1,x_1,x_2,x_1x_2\}$, the set of Khovanov generators associated to a resolution configuration of two circles, and $F(0)=\{1,y_1\}$, the set of Khovanov generators associated to a single circle. Set $F(\phi_{1,0}) = \{a_1, a_2, a_3\}$; $s(a_i)$ is given by the tail of the arrow $a_i$, and $t(a_i)$ is given by the head of the arrow $a_i$.  This is the Khovanov-Burnside functor associated to two circles merging to a single circle. Generators at the same height have the same quantum grading.} 
\label{fig:burnside-dots}
\end{figure}
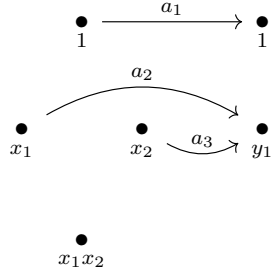

Next, we address the construction of the stable \emph{annular Khovanov-Burnside functor} $\akhoburn(D)=\amalg_{j,k} \akhoburn^{j,k}(D)\from \two^n \to \burn_{\ZZ_2}$ associated to a diagram $D$ of an annular link $L$, along with an ordering of the $n$ crossings, an orientation of the crossings, a choice of edge assignment, and an ordering of circles at each resolution.  We define, for $u\in \two^n$,
\[
\akhoburn^{j,k}(u)=\KhGen^{j,k}(u).
\]
For each edge $u\geqslant_1 v$ in $\two^n$, and each element
$y \in \akhoburn(v)$, write
\[
\oddafunc(\phi^{\op}_{v,u})(y)=\sum_{x\in \akhoburn(u) } \epsilon_{x,y} x.
\]
Define
\[
\akhoburn(\phi_{u,v})=\{ (y,x) \in \akhoburn(v) \times \akhoburn(u) \mid
\epsilon_{x,y}=\pm 1 \},
\]
where the sign on elements of $\akhoburn(\phi_{u,v})$ is given by
$\epsilon_{x,y}$ of the pair, and the source and target maps are the
natural ones.  The matching along $2$-dimensional faces is obtained from that of $\khoburn$, and the formal desuspension of $\akhoburn$ is also inherited from $\khoburn$.  We have the following theorem:

\begin{thm}\label{thm:annular-burn-exists}
The equivariant stable equivalence class of the functor $\akhoburn(D)$ is independent of the choices involved in its construction, and is an invariant of the annular link $L$.  Let $\widetilde{\akhoburn}^{j,k}_{n}$ be a spatial refinement of $\akhoburn^{j,k}$ in sufficiently high dimension, modeled on $\tilde{\mathbb{R}}^n$.  Then the $\ZZ_2$-equivariant stable homotopy type of the spatial realization $\Akhspace^{j,k}_n=||(\widetilde{\akhoburn}^{j,k}_{n})^+||$ is a link invariant. Moreover, there is a CW structure on $||(\widetilde{\akhoburn}^{j,k}_{n})^+||$ for which the reduced cellular chain complex $\redcellC^*(\Akhspace^{j,k}_n)=\oddAkc^j(L;\ZZ)=\Tot_{\Id}(\akhoburn^{j,k})^*$ if $n$ odd, or $\Akc^j(L;\ZZ)=\Tot_{\dig=0}(\akhoburn^{j,k})^*$ if $n$ is even.
\end{thm}
\begin{proof}
This follows from keeping track of the annular gradings in the invariance proof of the equivariant stable equivalence class of $\khoburn$.
\end{proof}

We write $\akhburn\from \two^n \to \burn$ for the \emph{even annular Khovanov-Burnside functor}, obtained from $\akhoburn$ by applying the functor $\burn_{\ZZ_2}\to \burn$.

\subsection{Equivariant Khovanov-Burnside functors}\label{subsec:63}

In this section, we apply the machinery from Sections \ref{sec:burn}--\ref{sec:external} to construct Burnside functors with external action.  
We first outline the notation used in this section. 
Let $p > 1$ be an integer, and consider a $p$-periodic link $\tilde L$ with (annular) periodic diagram $\tilde D$. Let $\psi$ denote the rotation of the annulus by $2\pi/p$.

We abuse notation by writing $\psi$ also for the actions on $\tilde{L}$ and $\tilde{D}$.   
The quotient link $L = \tilde L / \psi$ has (annular) diagram $D = \tilde D / \psi$, with $n$ crossings.
We refer to information relating to $\tilde L$ as `upstairs' and information relating to the quotient $L$ as `downstairs.'  

\begin{thm}\label{thm:functor-with-action}
Let $\tilde{L}$ be a $p$-periodic link.  Then there is a natural $\ZZ_p$-external action on $\akhburn(\tilde{L})$ and $\khburn(\tilde{L})$, whose external equivariant stable equivalence class is an invariant of the equivariant isotopy type of the link $\tilde{L}$.  If $p$ is odd, then there is a natural $\ZZ_p$-external action on $\akhoburn(\tilde{L})$ and $\khoburn(\tilde{L})$, whose external equivariant stable equivalence class is an invariant of the equivariant isotopy type of the link $\tilde{L}$. 
\end{thm}
We will prove this theorem over the course of the next few sections.  We start with the construction.

\begin{prop}\label{prop:equivariant-constr}
Let $\tilde{D}$ be a $p$-periodic link diagram.  In terms of the diagram $\tilde{D}$, there is a well-defined $\ZZ_p$-external action on $\akhburn(\tilde{D})$ and $\khburn(\tilde{D})$.  
\end{prop}
\begin{proof}
Recall from Section \ref{subsec:periodic-links} that $\ZZ_p$ acts on $\amalg_{u\in \two^{np}} \KhGen^{j,k}(u)$ for any $j,k\in \ZZ$.  We write $\psi u$ for the generator obtained from the Khovanov generator $u$ by rotation by $2\pi/p$.  

For each $v\in \two^{np}$ and $i\in \ZZ_p$, we have a (bijective) correspondence
\[E_{\psi^i,v}\colon \akhburn(v)\to \akhburn(\psi^iv),
\]
by sending the generator $x\in \akhburn(v)=\KhGen(v)$ to $\psi^ix\in \akhburn(\psi^iv)$.  We view $E_{\psi^i,v}$ as a $1$-morphism in $\burn$.

For $u\geqslant_1 v$, it is easy to check that there are natural bijections $\psi\from \akhburn(\phi_{u,v})\to \akhburn(\phi_{\psi u,\psi v})$ and $\psi \from \khburn(\phi_{u,v})\to \khburn(\phi_{\psi u, \psi v})$. (Here we have abused notation: the bijections are induced by the rotation $\psi$, hence the name.)  Moreover, if $a\in \akhburn(\phi_{u,v})$, one can check directly that $s(\psi a)=\psi s(a)$ and $t(\psi a)=\psi t(a)$.  For $u\geq_1 v$, let
\[
E_{\psi^i,\phi_{u,v}}\colon E_{\psi^i,v}\circ \akhburn(\phi_{u,v})\to \akhburn(\phi_{\psi^iu,\psi^iv})\circ E_{\psi^i,u}
\]
be the $2$-morphism in $\burn$ induced by $\psi^i\from \akhburn(\phi_{u,v})\to \akhburn(\phi_{\psi^iu,\psi^iv})$.  Note that $E_{\psi^i,\phi_{u,v}}$ is indeed a $2$-morphism, since $\psi^i\from \akhburn(\phi_{u,v})\to \akhburn(\phi_{\psi^iu,\psi^iv})$ is compatible with source and target maps.

There are similar constructions of $E_{\psi^i,v}$ and $E_{\psi^i,\phi_{u,v}}$ for $\khburn$ in place of $\akhburn$.  

We are almost in the situation of Lemma \ref{lem:212-equivariant}: in order to apply that lemma and obtain an external action $E$ on $\akhburn(\tilde{D})$ and $\khburn(\tilde{D})$, we need to check the conditions \ref{itm:e1'} and \ref{itm:e2'}.  The hypothesis \ref{itm:e1'} holds because, for each pair $u\geq_1 v$, and integers $i,j$, there is at most one $2$-morphism between $E_{\psi^{i+j},v}\circ \akhburn (\phi_{u,v})$ and $\akhburn(E_{\psi^{i+j}u,\psi^{i+j}v})\circ E_{\psi^{i+j},u}$.  For \ref{itm:e2'}, we need only show that the $\ZZ_p$-action respects the ladybug matching.  However, this is also essentially automatic; let us see how formal properties of the ladybug matching guarantee this.  First of all, we need only consider squares $u\geqslant_1 v,v'\geqslant_1w$ in $\two^{np}$ so that $\khburn(\phi_{u,v})$ (or $\akhburn(\phi_{u,v})$) has two elements, otherwise the diagram in the hypotheses \ref{itm:e2'} of Lemma \ref{lem:212-equivariant} is automatically commutative.  That is, we may assume the resolution configuration associated to $u\geq_2 w$ is a ladybug configuration. 

Then the arrows from the action in \ref{itm:e2'} are obtained from the maps
\[
\KhGen(u)\to \KhGen(\psi u),
\]
(similarly for $v,v',w$), obtained by rotating the resolution $D_u$ to $D_{\psi u}$ (using the fact that $\khburn(\phi_{u',u''})$ is a subset of the product $\khburn(u')\times \khburn(u'')$ for any $u', u''$ in $\two^{np}$).  Finally, the ladybug matching is an invariant of planar isotopy \cite[Lemma 5.8]{lshomotopytype},  so the diagram commutes.  Lemma \ref{lem:212-equivariant} then implies that there is a $\ZZ_p$-external action on $\akhburn(\tilde{D})$ and $\khburn(\tilde{D})$, as needed.  
\end{proof}

We next generalize this to the odd case.  We will need an auxiliary lemma.  

\begin{lem}\label{lem:eqvar-cube-cat}
Suppose $p$ is odd.  Let $\eqcellC([0,1]^{np};\f_2)$ be the subcomplex of $\cellC^*([0,1]^{np};\f_2)$ consisting of elements fixed by the $\ZZ_p$-action, with respect to the product cell structure on $[0,1]^{np}$.  Then $H^2(\cellC^{\ZZ_p}([0,1]^{np};\f_2))=H^1(\cellC^{\ZZ_p}([0,1]^{np};\f_2))=0$.
\end{lem}
\begin{proof}
Say $c\in \cellC^{\ZZ_p,1}([0,1]^{np};\f_2)$ has $\delta c=0$.  Now, $c=\delta e$ for some $e\in \cellC^0([0,1]^{np};\f_2)$.  Then $(1+\psi)\delta e=0$, from which we see that $(1+\psi)e$ is a cocycle.  However, the only cocycles in $\cellC^0([0,1]^{np};\f_2)$ are the constant cochains evaluating to $0$ or $1$ on all vertices of $[0,1]^{np}$.  The cochain evaluating to $1$ is not in the image of $1+\psi$, since the image of $1+\psi$ is characterized as those cochains so that on each $\ZZ_p$-orbit, the sum of evaluation over the orbit is $0$.  Thus $(1+\psi)e=0$, and so $c$ is the boundary of an invariant cochain, as needed.

Next, say $c\in \cellC^{\ZZ_p,2}([0,1]^{np};\f_2)$ with $\delta c=0$.  Say $c=\delta e$ for $e\in \cellC^1([0,1]^{np};\f_2)$.  As before $(1+\psi) e$ is a cocycle, so there exists $f$ with $\delta f =(1+\psi)e$.  Then $\delta (1+\psi+\dots+\psi^{p-1})f=0$. That is, $(1+\psi+\dots+\psi^{p-1})f$ is either the constant $0$-cocycle or the constant $1$-cocycle.  Since $p$ is odd, we obtain that, in the former case, $f$ must vanish on invariant vertices of $[0,1]^{np}$, and in the latter case, $f$ evaluates to $1$ on the invariant vertices.  However, adding the nontrivial cocycle to $f$ still produces a cochain $f'$ so that $\delta f'=(1+\psi)e$, and so we may assume that $f$ vanishes on all the invariant vertices of $[0,1]^{np}$, and that $(1+\psi+\dots+\psi^{p-1})f=0$.  That is, $f$ lives in a free $\ZZ_p$-submodule of $\cellC^0([0,1]^{np};\f_2)$, and using $(1+\psi+\dots+\psi^{p-1})f=0$, it follows that $f=(1+\psi)g$ for some $g\in \cellC^0([0,1]^{np};\f_2)$.  Then $\delta g=e+e'$ for some $e'$ in the image of multiplication by $(1+\psi+\dots+\psi^{p-1})$, since $\cellC^1([0,1]^{np};\f_2)$ is a free $\f_2[\ZZ_p]$-module.  Then, since $\delta^2=0$, we have $\delta e'=c$.  Finally, the image of $(1+\psi+\dots+\psi^{p-1})$ on $\cellC^1([0,1]^{np};\f_2)$ is equal to the set of invariant cochains in degree $1$, so $c$ is the boundary of an invariant cochain, as needed.
\end{proof}

\begin{prop}\label{prop:odd-equivariant-constr}
Suppose $p$ is odd.  Let $\tilde{D}$ be a $p$-periodic link diagram.  Then there is a well-defined $\ZZ_p$-external action on $\akhoburn(\tilde{D})$ and $\khoburn(\tilde{D})$.  Moreover, this $\ZZ_p$-external action is nonsingular (see Definition \ref{def:singularity}).
\end{prop}
\begin{proof}
We begin by choosing an equivariant orientation of crossings for $\tilde{D}$, by which we mean that for each orbit of the $np$ crossings of $\tilde{D}$ under the action of $\ZZ_p$, we choose a representative crossing, orient it, and then use the $\ZZ_p$-action to define an orientation of crossings for all crossings in the same orbit.  

Next, we need to show that there exists an equivariant edge assignment.  By this, we mean that the function $\epsilon$ as in Section \ref{subsec:okh-cpx} can be chosen so that $\epsilon_{v,u}=\epsilon_{\psi v,\psi u}$.  For $p=2$, this is not generally possible, as the may be confirmed by drawing the usual picture of the Hopf link.  However, recall that an edge assignment amounts to the choice of an element $\epsilon \in \cellC^1([0,1]^{np};\f_2)$ with coboundary $\delta \epsilon = \Omega(\tilde{D})$ (tacitly identifying $\ZZ_2=\{\pm 1\}$ with $\f_2$).  We first observe that $\Omega(\tilde{D})$ is $\ZZ_p$-equivariant, since the odd resolution configuration $C_{u,w}$ for $u\geq_2 w$ is planar isotopic to the odd resolution configuration $C_{\psi u,\psi w}$, and since $\Omega(\tilde{D})_{u,w}$ is determined by the isotopy type of $C_{u,w}$ for each $u\geq_2 w\in \two^{np}$.  The condition $\epsilon_{u,v}=\epsilon_{\psi u,\psi v}$ means that we require $\epsilon \in \cellC^{\ZZ_p,1}([0,1]^{np};\f_2)$.  By Lemma \ref{lem:eqvar-cube-cat}, such $\epsilon$ exists.  

Finally, we must also choose orderings of the circles at each resolution.  In fact, any ordering of circles will do. 
We must now describe the action of $\psi$ on $\KhGen$.  Forgetting the sign, we have $\psi$ takes $\KhGen(u)\to \KhGen(\psi u)$ as in the proof of Proposition \ref{prop:equivariant-constr}.  Say $Z(\tilde{D}_u)=\{a_1,\dots,a_{\ell_1}\}$ so that $a_1<\dots<a_{\ell_1}$ and $Z(\tilde{D}_{\psi u})=\{b_1,\dots,b_{\ell_1}\}$ so that $b_1<\dots<b_{\ell_1}$.  For $x=a_1\otimes \dots \otimes a_k\in \KhGen(u)$ taken to $b_{\sigma(1)}\otimes \dots\otimes  b_{\sigma(k)}\in \KhGen(\psi u)$, the sign is just $\mathrm{sgn}(\sigma)$.  

We have now constructed $\psi$ on objects of $\akhoburn(\tilde{D})$ and $\khoburn(\tilde{D})$.  Since the edge assignment is equivariant, for $u\geq_1 v$, we have actions $\psi\from \akhoburn(\phi_{u,v})\to \akhoburn(\phi_{\psi u,\psi v})$ and $\psi \from \khoburn(\phi_{u,v})\to \khoburn(\phi_{\psi u, \psi v})$. 
The proof of the proposition now proceeds as in the proof of Proposition \ref{prop:equivariant-constr}.  To see nonsingularity of the resulting external action, consider any Khovanov generator $x\in \KhGen(u)$ fixed by $\psi$ (viewed as a bijection, not a signed bijection).  In particular, we have $\psi u=u$. 
Every invariant generator $x$ is a product of terms coming from nontrivial circles of $\tilde{D}_u$ and products $x_{i_1}\dots x_{i_p}$ of trivial circles related by rotation. 
So, to prove nonsingularity, it suffices to check that the sign is 1 for invariant generators $x$ which come from a nontrivial circle of $\tilde{D}_u$, or for a single product $x =x_{i_1}\dots x_{i_p}$ of trivial circles related by rotation. (Note that if $x=1$, $\psi$ acts as the identity.)
In the former case, certainly $\psi$ takes $x\to x$ with sign $1$.  In the latter, $\psi$ acts by some permutation of $x_{i_1}\dots x_{i_p}$. 
To verify that the sign of $\psi$ is $1$, it suffices to check a particular ordering of the circles of $u$.  To see this, note that reordering the circles changes the action of $\psi$ (viewed as a permutation of $\{1,\dots,\ell\}$ using the ordering of the circles) by conjugation.  Ordering the trivial circles in a $\ZZ_p$-orbit by order of appearance, going counterclockwise starting from an arc $\tilde{\gamma}$, we see that, indeed, $\psi$ acts with sign $+1$ on all invariant generators; here we have used that $p$ is odd.
\end{proof}

\subsection{Fixed-point functors}\label{subsec:fixed-point}

In this section, we find the fixed-point Burnside functors of the equivariant Khovanov-Burnside functors constructed above.  The main result is the following.

Write $\iota \from \two^n \into \two^{np}$ for the canonical embedding. 

\begin{thm}\label{thm:burnside-fixed-pts}
Let $\tilde{D}$ be a $p$-periodic link diagram (with $p>1$), with quotient diagram $D$.  The Khovanov fixed-point functors are
	\begin{enumerate}
		\item $\akhburn(D) = \khburn(\tilde{D})^{\ZZ_p}$
		\item  $\akhburn^{j,k}(D) = \akhburn^{pj-(p-1)k,k}(\tilde{D})^{\ZZ_p}$,
	\end{enumerate}
for any pair of quantum and ($k$)-gradings $(j, k)$.  If $p$ is odd, we further have, for suitable choices of crossing orientations, edge assignments, and circle orderings at each resolution:
	\begin{enumerate}\setcounter{enumi}{2}
		\item $\akhoburn(D)=\khoburn(\tilde{D})^{\ZZ_p},$
		\item  $\akhoburn^{j,k}(D)=\akhoburn^{pj-(p-1)k,k}(\tilde{D})^{\ZZ_p}.$
	\end{enumerate}
\end{thm}
\begin{proof}

Let us first address the case of $F=\khburn(\tilde{D})$; that is, let us see that $F^{\ZZ_p}=\akhburn(D)$.  By Lemma \ref{lem:212-equivariant}, and the fact that the fixed-point category of $\two^{np}$ is the image of the canonical embedding $\iota:\two^n\to \two^{np}$, it suffices to identify $F^{\ZZ_p}(\iota u)$ for each $u\in \two^n$, as well as the correspondences $F^{\ZZ_p}(\phi_{\iota u,\iota v})$ for $u\geqslant_1 v$, and finally to identify the $2$-morphisms associated to $2$-dimensional faces of $\two^n$.  Proposition \ref{prop:identify-equiv-gens} shows that $F^{\ZZ_p}(\iota u)$ is canonically identified with $\akhburn(u)$, and that the quantum gradings are as in the statement.  We package the proof that the $1$-morphisms are correct as Proposition \ref{prop:long-diffs} below, and the claim about $2$-morphisms as Lemma \ref{lem:ladybug-compatible}.  Assuming those lemmas, the present theorem follows directly.
\end{proof}

\begin{prop}\label{prop:long-diffs}
Let $\tilde{D}$ be a $p$-periodic link diagram, with $p>1$.  Fix $u\geqslant_1 v\in \Ob(\two^n)$ and consider a sequence of objects of $\two^{np}$ given by $\iota u\geqslant_1 u_1\dots\geqslant_1 u_p=\iota v$.  Then 
\begin{align}\label{eq:khburn-fix}
\khburn(\tilde{D})^{\ZZ_p}(\phi_{u_{p-1},\iota v}\circ\dots\circ \phi_{\iota u,u_1})&\cong\akhburn(D)(\phi_{u,v}),\\
\akhburn(\tilde{D})^{\ZZ_p}(\phi_{u_{p-1},\iota v}\circ\dots\circ \phi_{\iota u,u_1})&\cong\akhburn(D)(\phi_{u,v}),\nonumber
\end{align}
where $\cong$ denotes natural isomorphism.  Further, if $p$ is odd, then:
\begin{align*}
\khoburn(\tilde{D})^{\ZZ_p}(\phi_{u_{p-1},\iota v}\circ\dots\circ \phi_{\iota u,u_1})&\cong\akhoburn(D)(\phi_{u,v}),\\
\akhoburn(\tilde{D})^{\ZZ_p}(\phi_{u_{p-1},\iota v}\circ\dots\circ \phi_{\iota u,u_1})&\cong\akhoburn(D)(\phi_{u,v}),
\end{align*}
for appropriate choices for $\tilde{D},D$ of crossing orientations and edge assignments, and of circle orderings at each resolution.
\end{prop}
\begin{proof}

First consider the case for $F=\khburn(\tilde D)$. By commutativity of the $2$-dimensional faces of the cube, it suffices to show the identification of 1-morphisms for any particular path $\{u_i\}_i$.   

Note that in the sequence of resolutions $D_{u},D_{u_1},\ldots,D_{u_{p}}$, there may be circles, present in each of the $D_{u_i}$, which are inactive for each of the elementary cobordisms from $D_{u_{i}}$ to $D_{u_{i-1}}$; it is easy to see that the isomorphisms in the statement of Proposition \ref{prop:long-diffs} hold if and only if they hold for the sequence of resolutions with the inactive circles deleted.  The proof of the Proposition then amounts to a case-by-case check of the three different types of merges; see Figure \ref{fig:equiv-ann-merges}.  

\begin{figure}

\begin{tikzpicture}[scale=1,baseline={(current bounding box.center)}]
\node (x) at (0,0) {$\mathbb{X}$};
\draw circle (.5cm) circle (1 cm);
\node (vv) at (0,-1.5) {$\mathbb{V}\otimes \mathbb{V}$}; 
\end{tikzpicture}
$\map{\Delta^{p-1}m }$
\begin{tikzpicture}[scale=1,baseline={(current bounding box.center)}]
\node (x) at (0,0) {$\mathbb{X}$};
\draw
(10:.5cm) -- (10:1cm)
arc (10:62:1cm) 
-- (62:.5cm)
arc (62:10:.5cm);
\draw
(82:.5cm) -- (82:1cm)
arc (82:134:1cm) 
-- (134:.5cm)
arc (134:82:.5cm);
\draw
(154:.5cm) -- (154:1cm)
arc (154:206:1cm) 
-- (206:.5cm)
arc (206:154:.5cm);
\draw
(226:.5cm) -- (226:1cm)
arc (226:278:1cm) 
-- (278:.5cm)
arc (278:226:.5cm);
\draw
(298:.5cm) -- (298:1cm)
arc (298:350:1cm) 
-- (350:.5cm)
arc (350:298:.5cm);
\node (wp) at (0,-1.5) {$\mathbb{W}^{\otimes p}$}; 
\end{tikzpicture}
%
%
%
\qquad
\begin{tikzpicture}[scale=1,baseline={(current bounding box.center)}]
\node (x) at (0,0) {$\mathbb{X}$};
\draw circle (.5cm);
\draw
(10:.75cm) -- (10:1cm)
arc (10:62:1cm) 
-- (62:.75cm)
arc (62:10:.75cm);
\draw
(82:.75cm) -- (82:1cm)
arc (82:134:1cm) 
-- (134:.75cm)
arc (134:82:.75cm);
\draw
(154:.75cm) -- (154:1cm)
arc (154:206:1cm) 
-- (206:.75cm)
arc (206:154:.75cm);
\draw
(226:.75cm) -- (226:1cm)
arc (226:278:1cm) 
-- (278:.75cm)
arc (278:226:.75cm);
\draw
(298:.75cm) -- (298:1cm)
arc (298:350:1cm) 
-- (350:.75cm)
arc (350:298:.75cm);
\node (vwp) at (0,-1.5) {$\mathbb{V}\otimes \mathbb{W}^{\otimes p}$}; 
\end{tikzpicture}
$\map{m^p}$
\begin{tikzpicture}[scale=1,baseline={(current bounding box.center)}]
\node (x) at (0,0) {$\mathbb{X}$};
\draw
(10:.5cm) -- (10:1cm)
arc (10:62:1cm) 
-- (62:.5cm)
arc (62:82:.5cm);
\draw
(82:.5cm) -- (82:1cm)
arc (82:134:1cm) 
-- (134:.5cm)
arc (134:154:.5cm);
\draw
(154:.5cm) -- (154:1cm)
arc (154:206:1cm) 
-- (206:.5cm)
arc (206:226:.5cm);
\draw
(226:.5cm) -- (226:1cm)
arc (226:278:1cm) 
-- (278:.5cm)
arc (278:298:.5cm);
\draw
(298:.5cm) -- (298:1cm)
arc (298:350:1cm) 
-- (350:.5cm)
arc (350:370:.5cm);
\node (v) at (0,-1.5) {$\mathbb{V}$}; 
\end{tikzpicture} 
%
%

\qquad
\begin{tikzpicture}[scale=1,baseline={(current bounding box.center)}]
\node (x) at (0,0) {$\mathbb{X}$};
\draw
(10:.25cm) -- (10:.5cm)
arc (10:62:.5cm) 
-- (62:.25cm)
arc (62:10:.25cm);
\draw
(82:.25cm) -- (82:.5cm)
arc (82:134:.5cm) 
-- (134:.25cm)
arc (134:82:.25cm);
\draw
(154:.25cm) -- (154:.5cm)
arc (154:206:.5cm) 
-- (206:.25cm)
arc (206:154:.25cm);
\draw
(226:.25cm) -- (226:.5cm)
arc (226:278:.5cm) 
-- (278:.25cm)
arc (278:226:.25cm);
\draw
(298:.25cm) -- (298:.5cm)
arc (298:350:.5cm) 
-- (350:.25cm)
arc (350:298:.25cm);
\draw
(10:.75cm) -- (10:1cm)
arc (10:62:1cm) 
-- (62:.75cm)
arc (62:10:.75cm);
\draw
(82:.75cm) -- (82:1cm)
arc (82:134:1cm) 
-- (134:.75cm)
arc (134:82:.75cm);
\draw
(154:.75cm) -- (154:1cm)
arc (154:206:1cm) 
-- (206:.75cm)
arc (206:154:.75cm);
\draw
(226:.75cm) -- (226:1cm)
arc (226:278:1cm) 
-- (278:.75cm)
arc (278:226:.75cm);
\draw
(298:.75cm) -- (298:1cm)
arc (298:350:1cm) 
-- (350:.75cm)
arc (350:298:.75cm);

\node (wp) at (0,-1.5) {$\mathbb{W}^{\otimes 2p}$}; 
\end{tikzpicture}
$\map{m^p}$
\begin{tikzpicture}[scale=1,baseline={(current bounding box.center)}]
\node (x) at (0,0) {$\mathbb{X}$};
\draw
(10:.25cm) -- (10:1cm)
arc (10:62:1cm) 
-- (62:.25cm)
arc (62:10:.25cm);
\draw
(82:.25cm) -- (82:1cm)
arc (82:134:1cm) 
-- (134:.25cm)
arc (134:82:.25cm);
\draw
(154:.25cm) -- (154:1cm)
arc (154:206:1cm) 
-- (206:.25cm)
arc (206:154:.25cm);
\draw
(226:.25cm) -- (226:1cm)
arc (226:278:1cm) 
-- (278:.25cm)
arc (278:226:.25cm);
\draw
(298:.25cm) -- (298:1cm)
arc (298:350:1cm) 
-- (350:.25cm)
arc (350:298:.25cm);
\node (wp) at (0,-1.5) {$\mathbb{W}^{\otimes p}$}; 
\end{tikzpicture}
\caption{The three equivariant annular merges, with $p = 5$ illustrated. Here, $\Delta$ stands for `split' and $m$ stands for `merge.'  }
\label{fig:equiv-ann-merges}
\end{figure}

First, say $\phi^\op_{v,u}$ represents a $\mathbb{V}\otimes\mathbb{V} \to \mathbb{W}$ merge.  The $p$-cover is illustrated by the top left picture in Figure \ref{fig:equiv-ann-merges}.  In that figure, $\iota v$ is the diagram consisting of two concentric circles, while $\iota u$ is the diagram consisting of a single $\ZZ_p$-orbit of circles.  Then $F(\iota v)$ has four invariant generators  $\{1,x_1,x_2,x_1x_2\}$, where $x_1,x_2\in Z(\tilde{D}_{\iota v})$, and $F(\iota u)$ has two invariant generators $\{1,y_1\dots y_p\}$, where $y_1,\dots, y_p\in Z(\tilde{D}_{\iota u})$, all lying in the same $\ZZ_p$-orbit.

The first map $\phi^\op_{\iota v,u_{p-1}}$ is a merge, and then all the following maps $ \{ \phi^\op_{u_i, u_{i-1}}\}_{0 < i < p}$ are split maps.  Recall that we use the convention that $u_0=\iota u$ and $u_p=\iota v$.  It is straightforward to check that $F(\phi_{\iota u,\iota v})^{\ZZ_p}\cong\{a_1,a_2\}$ with source and target maps $s(a_i)=y_1\dots y_p$ and $t(a_i)=x_i$.  Thus, $F(\phi_{\iota u,\iota v})^{\ZZ_p}$ is naturally isomorphic to $\akhburn(D)(\phi_{u,v})$ for this case.

If $\phi^\op_{v,u}$ represents a $\mathbb{V} \otimes \mathbb{W} \to \mathbb{V}$ merge, then all $p$ maps $\{\phi^\op_{u_i, u_{i-1}}\}_{1 \leq i \leq p}$ are merge maps.  The invariant generators at $\iota v$ are $\{1,x,y_1\ldots y_p,xy_1\ldots y_p\}$ with $x\in Z(D_{\iota v})$ a nontrivial circle and where $y_i\in Z(D_{\iota v})$ are trivial circles, forming a single $\ZZ_p$-orbit.  The invariant generators at $\iota u$ are $\{1,z\}$ for $z\in Z(D_{\iota u})$.  The correspondence $F(\phi_{\iota u,\iota v})^{\ZZ_p}=\{a_1,a_2\}$ with $s(a_1)=1,s(a_2)=z$ and target $t(a_1)=1, t(a_2)=x$.  We then observe that $F(\phi_{\iota u,\iota v})^{\ZZ_p}$ is naturally isomorphic to $\akhburn(D)(\phi_{u,v})$ in this case as well.

A similar situation occurs for the case $\mathbb{W}\otimes\mathbb{W} \to \mathbb{W}$. The invariant generators at $\tilde v$ are $\{1,x_1\dots x_p,y_1\dots y_p,x_1\dots x_py_1\dots y_p\}$, where $x_i\in Z(D_{\iota v})$ are all in the same $\ZZ_p$-orbit, and similarly for $y_i\in Z(D_{\iota v})$.  The invariant generators at $\iota u$ are $\{1,z_1\dots z_p\}$ where $z_i\in Z(D_{\iota u})$ lie in the same $\ZZ_p$-orbit.  One may quickly check that $F(\phi_{\iota u,\iota v})^{\ZZ_p}=\{a_1,a_2,a_3\}$ with $s(a_1)=1,t(a_1)=1$, and $s(a_2)=x_1\dots x_p$, $t(a_2)=z_1\dots z_p$, and finally $s(a_3)=y_1\dots y_p$, $t(a_3)=z_1\dots z_p$.  It follows readily that $F(\phi_{\iota u,\iota v})^{\ZZ_p}$ is naturally isomorphic to $\akhburn(D)(\phi_{u,v})$ in this case.  

The above cases, along with duality (for the corresponding equivariant split maps) \cite[Section 10]{lls1} show that equation (\ref{eq:khburn-fix}) holds.

The case of $F=\akhburn(\tilde D)$ is very similar, so we omit the details here.

Next we treat the case $F=\khoburn(\tilde D)$.  We have already seen that, if we forget the signs, $(\forgot F)^{\ZZ_p}=\akhburn(D)$.  Now, $\akhoburn(D)$ can be viewed as a way of sprinkling signs on the correspondences of $\akhburn(D)$ (and similarly for $\khoburn(\tilde{D})$ relative to $\khburn(\tilde{D})$), and we need to say that these sprinklings respect the equality of Burnside functors in (\ref{eq:khburn-fix}).  

Recall that in order to define $F$, we needed to choose the data of an (equivariant) orientation of crossings, as well as an equivariant edge assignment.  Say we have fixed these data.  Now, in order to define $\akhoburn(D)$, we need an orientation of crossings of $D$, as well as an edge assignment of $D$.  We choose the orientation of crossings coming from taking the quotient of the orientation of crossings of $\tilde{D}$.  In order to compare $\akhoburn(D)$ with $F$, we must find a way to define an edge assignment on $D$, given the edge assignment upstairs.  We start with the following lemma.  Recall that $\KhGen(\tilde{D})^{\ZZ_p}$ upstairs is identified with $\KhGen(D)$ downstairs, using the choice of an arc $\tilde{\gamma}$, as in the discussion after Proposition \ref{prop:identify-equiv-gens}.

\begin{lem}\label{lem:fake-functor}
Let $C$ be an index-$1$ annular resolution configuration, with associated odd annular Khovanov projective functor $\oddafunc' \from \two^\op \to \ZZ\text{-}\mathrm{Mod}$.  Let $p$ be odd, and let $\tilde{C}$ denote the $p$-cover of $C$, with some choice of lift of $\gamma$ to $\tilde{\gamma}$.  Set $v_i=0^{p-i}1^{i}\in \Ob(\two^{p})$.  Let $\AbFunc_o'\from (\two^p)^\op\to \ZZ\text{-}\mathrm{Mod}$ denote the odd Khovanov projective functor associated to $\tilde{C}$.  Then
\begin{equation}\label{eq:fake-agree}
(\AbFunc_o'(\phi^\op_{v_{p-1},v_p})\circ \dots \circ \AbFunc_o'(\phi^\op_{v_0,v_{1}}))_{\ZZ_p}=\oddafunc'(\phi^\op_{0,1}).
\end{equation}
Here we have written $(\AbFunc_o'(\phi^\op_{v_{p-1},v_p})\circ \dots \circ \AbFunc_o'(\phi^\op_{v_0,v_{1}}))_{\ZZ_p}$ to denote the restriction of the composite $(\AbFunc_o'(\phi^\op_{v_{p-1},v_p})\circ \dots \circ \AbFunc_o'(\phi^\op_{v_0,v_{1}}))$ to the submodule of $\AbFunc_o'(0^p)$ spanned by $\ZZ_p$-invariant generators, and then its projection to the submodule of $\AbFunc_o'(1^p)$ spanned by $\ZZ_p$-invariant generators.  Recall that the ordering of the arcs and circles of $\tilde{C}$ are defined with respect to the lift $\tilde{\gamma}$.  
\end{lem}
\begin{proof}

The proof is a case-by-case check of index-$1$ annular resolution configurations.  First, consider the resolution configuration associated to a merge $\mathbb{V}\otimes \mathbb{V} \to \mathbb{W}$.  That is, say we have the following picture in the base:

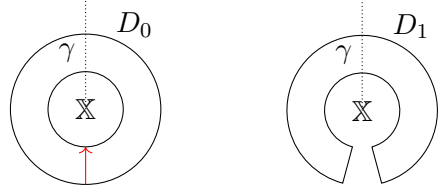
\[
\begin{tikzpicture}[scale=1,baseline={(current bounding box.center)}]

\node (x) at (0,0) {$\mathbb{X}$};
\draw circle  (1cm)
          circle (.5cm);
\path (60:1.3cm) node (v0)  {$D_0$};
\draw[->,red] (0,-1) -- (0,-0.5);
\draw[densely dotted] (0,0) -- node[auto]{$\gamma$} (0,1.5);

\end{tikzpicture}\qquad\qquad
\begin{tikzpicture}[scale=1,baseline={(current bounding box.center)}]
\node (x) at (0,0) {$\mathbb{X}$};
\path (60:1.3cm) node (v0)  {$D_1$};
\draw[densely dotted] (0,0) -- node[auto]{$\gamma$} (0,1.5);
\draw(-75:0.5cm) -- (-75:1cm)
  arc (-75:255:1cm) -- (255:0.5cm) arc (255:-75:0.5cm) ;
\end{tikzpicture}
\]

For this case, consider Figure \ref{fig:vvw-odd}, which illustrates the proof for $p=5$; the proof for general $p$ is entirely analogous, and is omitted.
\begin{figure}

\centering

\begin{tikzpicture}[scale=1,baseline={(current bounding box.center)}]
\draw circle (1cm)
circle (.5cm);
\draw[densely dotted] (0,0) -- node[near end,anchor=south east]{$\tilde{\gamma}$} (0,1.5);
\path (60:1.5cm) node (v0)  {$\tilde{D}_{0^5}$};
\draw[->,red] (36:1cm) -- (36:0.5cm);
\draw[->,red] (36*3:1cm) -- (36*3:0.5cm);
\draw[->,red] (36*5:1cm) -- (36*5:0.5cm);
\draw[->,red] (36*7:1cm) -- (36*7:0.5cm);
\draw[->,red] (36*9:1cm) -- (36*9:0.5cm);
\end{tikzpicture}\qquad \qquad
\begin{tikzpicture}[scale=1,baseline={(current bounding box.center)}]
\draw[densely dotted] (0,0) -- node[near end,anchor=south east]{$\tilde{\gamma}$} (0,1.5);
\path (60:1.5cm) node (v0)  {$\tilde{D}_{1^5}$};
\path (90:0.7cm) node[anchor=west] (b5) {$y_5$};
\path (90-72:0.7cm) node (b4) {$y_4$};
\path (90-72*2:0.7cm) node (b3) {$y_3$};
\path (90-72*3:0.7cm) node (b2) {$y_2$};
\path (90-72*4:0.7cm) node (b1) {$y_1$};
\draw (54:0.5cm) --(54:1cm)
arc (54:54+66:1cm) --(54+66:0.5cm)
arc (54+66:54:0.5cm) --cycle;
\draw (54+72:0.5cm) --(54+72:1cm)
arc (54+72:54+66+72:1cm) --(54+66+72:0.5cm)
arc (54+66+72:54+72:0.5cm) --cycle;
\draw (54+72*2:0.5cm) --(54+72*2:1cm)
arc (54+72*2:54+66+72*2:1cm) --(54+66+72*2:0.5cm)
arc (54+66+72*2:54+72*2:0.5cm) --cycle;
\draw (54+72*3:0.5cm) --(54+72*3:1cm)
arc (54+72*3:54+66+72*3:1cm) --(54+66+72*3:0.5cm)
arc (54+66+72*3:54+72*3:0.5cm) --cycle;
\draw (54+72*4:0.5cm) --(54+72*4:1cm)
arc (54+72*4:54+66+72*4:1cm) --(54+66+72*4:0.5cm)
arc (54+66+72*4:54+72*4:0.5cm) --cycle;
\end{tikzpicture}
\caption{\textbf{The $\mathbb{V}\otimes \mathbb{V}\to \mathbb{W}$ case for $p=5$.} The map downstairs is $x_1,x_2\to x$, where $x_1,x_2\in Z(D_0)$, $x\in Z(D_1)$.  Let $Z(\tilde{D}_{0^5})=\{\tilde{x}_1,\tilde{x}_2\}$, the elements over $x_1,x_2$, and let $Z(\tilde{D}_{1^5})=\{y_1,\dots,y_5\}$, related by the action of $\ZZ_5$.  Then upstairs the map on fixed points is $\tilde{x}_i\to (y_5-y_1)(y_1-y_2)(y_2-y_3)(y_3-y_4)y_5=y_1y_2y_3y_4y_5$.  Moreover, the element $1\in \AbFunc_o'(0^5)$ is sent to a term in $\AbFunc_o'(1^5)$ which is killed by projection to the summand of invariant generators, and the element $\tilde{x}_1\tilde{x}_2$ is annihilated by $\AbFunc_o'(\phi^{\mathrm{op}}_{0^5,1^5})$.  This verifies Lemma \ref{lem:fake-functor} in this example.
}
\label{fig:vvw-odd}
\end{figure}
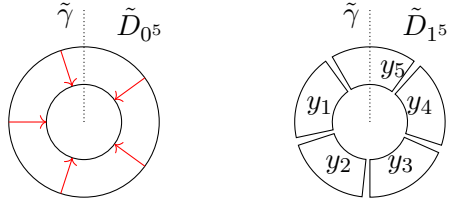

Next, consider the resolution configuration associated to a merge $\mathbb{V}\otimes \mathbb{W} \to \mathbb{V}$.  In this case, both upstairs and downstairs there are only merge maps, from which the Lemma follows readily.

Next, consider the case $\mathbb{W}\otimes \mathbb{W} \to \mathbb{W}$.  In this case, again, upstairs there are only merges, from which the result is immediate.  Note here that the ordering of the circles in the lift is chosen according to the discussion in Section \ref{subsec:periodic-links-kh}.

Next, consider the case $\mathbb{W}\to \mathbb{W}\otimes \mathbb{W}$.  For this, downstairs we have $Z(D_0)=\{x\}$, $Z(D_1)=\{y_1,y_2\}$ and upstairs $\KhGen(0^p)^{\ZZ_p}=\{1,z_1\dots z_p\}$ and $\KhGen(1^p)^{\ZZ_p}=\{1,w_1^1\dots w^1_p,w^2_1\dots w^2_p,w^1_1\dots w^1_p w_1^2\dots w^2_p\}$, where the ordering is so that $y_1<y_2$ and $w_1^1 \leq w_i^1< w^2_1\leq w^2_i$ for all $i$, and $\{z_i\},\{w_i^1\},\{w_i^2\}$ are the orbits of circles $z_1,w_1^1,w^2_1$, respectively, under $\ZZ_p$.  

Downstairs, having fixed an orientation of crossing going from $y_1$ to $y_2$, we have:
\[
\oddafunc'(\phi^\op_{0,1})(1)=y_1-y_2 \qquad \qquad \oddafunc'(\phi^\op_{0,1})(x)=y_1y_2.
\]
Upstairs, we observe, using the definition of the odd Khovanov projective functor: 
\begin{align*}
\AbFunc_o'(\phi^\op_{v_{p-1},v_p})\circ \dots \circ \AbFunc_o'(\phi^\op_{v_{0},v_1})(1)&= (w^1_1-w^2_1)\dots (w^1_p-w^2_p),\\
\AbFunc_o'(\phi^\op_{v_{p-1},v_p})\circ \dots \circ \AbFunc_o'(\phi^\op_{v_{0},v_1})(z_1\dots z_p)&= (w^1_1-w^2_1)\dots (w^1_p-w^2_p)(z_1\dots z_p)\\
&=w^1_1\dots w^1_p w^2_1\dots w^2_p.
\end{align*}
From this calculation, we have obtained the Lemma in the $\mathbb{W}\to \mathbb{W}\otimes \mathbb{W}$ case.

The cases $\mathbb{V}\to \mathbb{W} \otimes \mathbb{V}$ and $\mathbb{W} \to \mathbb{V}\otimes \mathbb{V}$ are very similar to the cases we have done so far, and we omit them; this completes the proof of Lemma \ref{lem:fake-functor}.
\end{proof}

Now, we must see how to translate from an (equivariant, type X) edge assignment $\tilde{\epsilon}$ on $\tilde{D}$ to an edge assignment on $D$.  Fix $u\geqslant_1 v \in \two^n$.  We define $v_i\in \two^{np}$ by $v_i=(v)^{p-i}(u)^{i}$, as elements of $(\two^n)^p$ for $0\leq i \leq p$.  We then define an element $\epsilon \in \cellC^1([0,1]^{n};\ZZ_2)$ by:
\[
\epsilon_{u,v}=\tilde{\epsilon}_{v_{p},v_{p-1}}\dots \tilde{\epsilon}_{v_{1},v_0}.
\]

Recall the definition of the \emph{obstruction cocycle} $\Omega(D)$ from Section \ref{subsec:okh-cpx}.  Any cochain $c\in \cellC^1([0,1]^{n};\ZZ_2)$ for which $\delta c=\Omega(D)$ gives a functor $\akhoburn(D)_{c}\from \two^n\to\burn_{\ZZ_2}$, the odd annular Khovanov-Burnside functor with edge assignment $\epsilon$, whose stable equivalence class is well-defined, i.e.\ independent of $c$.  To proceed, we need to confirm that $\delta \epsilon=\Omega(D)$.  We will work with the type X obstruction cocycle; the following lemma also holds for the type Y obstruction cocycle, if the edge assignment upstairs is chosen to be type Y (the proof below immediately generalizes to the type Y case).

\begin{lem}\label{lem:edge-magic}
For $\epsilon\in \cellC^1([0,1]^{n};\ZZ_2)$ as defined above, we have $\delta \epsilon=\Omega(D)$.
\end{lem}
\begin{proof}
For $x\in \cellC^2([0,1]^n;\ZZ_2)$ and $u\geq_2 w\in \two^n$, we write $x_{u,w}$ for the evaluation of $x$ on the copy of $[0,1]^2$ corresponding to the pair $(u,w)$.  We need to check that for each $2$-dimensional face $u\geqslant_1 v,v'\geqslant_1 w$, that $(\delta \epsilon)_{u,w}=\Omega(D)_{u,w}$.  There are two cases to consider.

The first case is that $\oddafunc'(\phi^\op_{v,u})\oddafunc'(\phi^\op_{w,v})\neq 0$.  Then $\Omega(D)_{u,w}$ is determined as follows:
\[
\oddafunc'(\phi^\op_{v,u})\oddafunc'(\phi^\op_{w,v})=\oddafunc'(\phi^\op_{v',u})\oddafunc'(\phi^\op_{w,v'})
\]
if and only if $\Omega(D)_{u,w}=1$.  However, if $\oddafunc'(\phi^\op_{v,u})\oddafunc'(\phi^\op_{w,v})=0$, more data is needed to determine $\Omega(D)_{u,w}$.  For comparison, if we worked with $\AbFunc_o'$ in place of $\oddafunc'$, more data is needed to define $\Omega(D)_{u,w}$ only for ladybug resolution configurations $C_{u,w}$ (in that case $\Omega(D)_{u,w}=-1$ for type X edge assignments, etc.).  

Let us consider the case where $\oddafunc'(\phi^\op_{v,u})\oddafunc'(\phi^\op_{w,v})\neq 0$.  Write $w_i=w^{p-i}v^{i}$, $w'_i=w^{p-i}v'^i$ and $v_i=v^{p-i}u^i$, $v_i'=v'^{p-i}u^i$, as objects in $\two^{np}$.  Then, using Lemma \ref{lem:fake-functor},
\begin{align*}
&\tilde{\epsilon}_{v_{p},v_{p-1}}\dots \tilde{\epsilon}_{v_1,v_{0}}\tilde{\epsilon}_{w_p,w_{p-1}}\dots \tilde{\epsilon}_{w_1,w_{0}}=\tilde{\epsilon}_{v'_p,v'_{p-1}}\dots \tilde{\epsilon}_{v'_1,v'_{0}}\tilde{\epsilon}_{w'_p,w'_{p-1}}\dots \tilde{\epsilon}_{w'_1,w'_{0}}\\
\text{if and only if}
\qquad &\oddafunc'(\phi^\op_{v,u})\oddafunc'(\phi^\op_{w,v})= \oddafunc'(\phi^\op_{v',u})\oddafunc'(\phi^\op_{w,v'}),
\end{align*}
since 
\begin{align*}
\tilde{\epsilon}_{v_{p},v_{p-1}}\dots \tilde{\epsilon}_{v_1,v_{0}}\tilde{\epsilon}_{w_{p},w_{p-1}}\dots \tilde{\epsilon}_{w_1,w_{0}}\AbFunc_o'(\phi^\op_{v_{p-1},v_p})\circ \dots \circ \AbFunc_o'(\phi^\op_{v_0,v_{1}})\AbFunc_o'(\phi^\op_{w_{p-1},w_p})\circ \dots \circ \AbFunc_o'(\phi^\op_{w_0,w_{1}})&\\ 
=\tilde{\epsilon}_{v'_{p},v'_{p-1}}\dots \tilde{\epsilon}_{v'_1,v'_{0}}\tilde{\epsilon}_{w'_{p},w'_{p-1}}\dots \tilde{\epsilon}_{w'_1,w'_{0}}\AbFunc_o'(\phi^\op_{v'_{p-1},v'_p})\circ \dots \circ \AbFunc_o'(\phi^\op_{v'_0,v'_{1}})\AbFunc_o'(\phi^\op_{w'_{p-1},w'_p})\circ \dots \circ \AbFunc_o'(\phi^\op_{w'_0,w'_{1}}).&
\end{align*}
The latter equality holds because $\tilde{\epsilon}$ is an edge assignment.  
We have verified $(\delta \epsilon)_{u,w}=\Omega(D)_{u,w}$ on all faces of the first case.

We next treat the second case, or faces where $\oddafunc'(\phi^\op_{v,u})\oddafunc'(\phi^\op_{w,v})=0$.  We start by cataloging such faces:
\begin{lem}\label{lem:bad-ann-func}
Say $u\geqslant_1 v\geqslant_1 w$ and let $C_{u,w}$ be an index-$2$ odd annular resolution configuration so that
\begin{equation}\label{eq:bad-ann-func}
\oddafunc'(\phi^\op_{v,u})\oddafunc'(\phi^\op_{w,v})=0.
\end{equation}

Then either the underlying resolution configuration of $C_{u,w}$ is type $X$ or $Y$, or $C_{u,w}$ consists of three concentric nontrivial circles $C_1,C_2,C_3$ with $C_1,C_2$ joined by an arc, as well as $C_2,C_3$ joined by an arc, or the dual configuration of the latter.
\end{lem}
\begin{proof}
The proof of this Lemma is a case-by-case check.  
\end{proof}

Next, we check that $(\delta \epsilon)_{u,w}=\Omega(D)_{u,w}$ for configurations $C_{u,w}$ of type X or Y, such faces are always of the second case.  We may as well assume now that $u=11,v=10,v'=01,w=00$, to simplify notation.  First consider $C_{u,w}$ of type  X.  There are four annular resolution configurations to consider, pictured in Figure \ref{fig:ann-danger}.

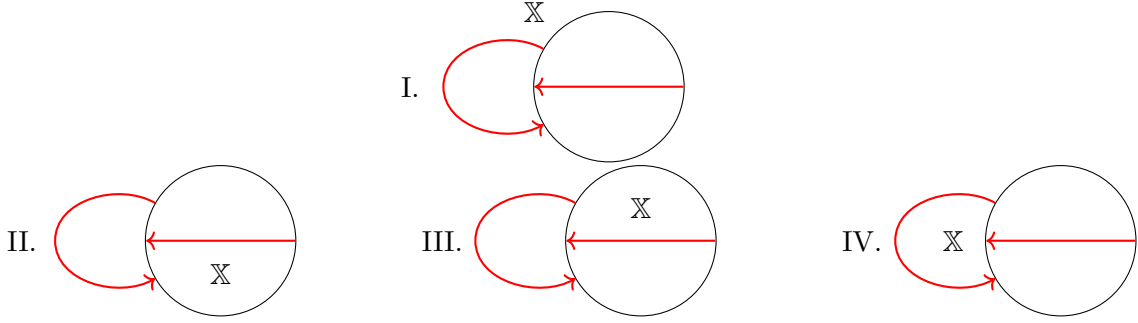
\begin{figure}

\centering

 \begin{tikzpicture}[scale=0.22,baseline={(current  bounding  box.center)}]
\node (i) at (-12,0) {\textrm{I}.};
\node (x) at (-4.5,4.5) {$\mathbb{X}$};
         \node[inner sep=0pt, outer sep=0pt,draw,shape=circle,minimum width=2cm] (a) at (0,0) {};
\draw[thick, red,->] (4.5,0) --(-4.5,0);
         \draw[thick,red,->] (a) to[out=150,in=90] (-10,0) to[out=-90,in=210] (a);
       \end{tikzpicture}\qquad \qquad

 \begin{tikzpicture}[scale=0.22,baseline={(current  bounding  box.center)}]
\node (i) at (-12,0) {\textrm{II}.};
\node (x) at (0,-2) {$\mathbb{X}$};
         \node[inner sep=0pt, outer sep=0pt,draw,shape=circle,minimum width=2cm] (a) at (0,0) {};
\draw[thick, red,->] (4.5,0) --(-4.5,0);
         \draw[thick,red,->] (a) to[out=150,in=90] (-10,0) to[out=-90,in=210] (a);
       \end{tikzpicture}\qquad \qquad
 \begin{tikzpicture}[scale=0.22,baseline={(current  bounding  box.center)}]
\node (i) at (-12,0) {\textrm{III}.};
\node (x) at (0,2) {$\mathbb{X}$};
         \node[inner sep=0pt, outer sep=0pt,draw,shape=circle,minimum width=2cm] (a) at (0,0) {};
\draw[thick, red,->] (4.5,0) --(-4.5,0);
         \draw[thick,red,->] (a) to[out=150,in=90] (-10,0) to[out=-90,in=210] (a);
       \end{tikzpicture}\qquad \qquad
 \begin{tikzpicture}[scale=0.22,baseline={(current  bounding  box.center)}]
\node (i) at (-12,0) {\textrm{IV}.};
\node (x) at (-6.5,0) {$\mathbb{X}$};
         \node[inner sep=0pt, outer sep=0pt,draw,shape=circle,minimum width=2cm] (a) at (0,0) {};
\draw[thick, red,->] (4.5,0) --(-4.5,0);
         \draw[thick,red,->] (a) to[out=150,in=90] (-10,0) to[out=-90,in=210] (a);
       \end{tikzpicture}
\caption{The annular resolution configurations of type X.
}
\label{fig:ann-danger}
\end{figure}

 Recall that we need to show 
\begin{equation}\label{eq:negative-epsilon}
\tilde{\epsilon}_{v_{p},v_{p-1}}\dots \tilde{\epsilon}_{v_1,v_{0}}\tilde{\epsilon}_{w_{p},w_{p-1}}\dots \tilde{\epsilon}_{w_1,w_{0}}=-\tilde{\epsilon}_{v'_{p},v'_{p-1}}\dots \tilde{\epsilon}_{v'_1,v'_{0}}\tilde{\epsilon}_{w'_{p},w'_{p-1}}\dots \tilde{\epsilon}_{w'_1,w'_{0}}.
\end{equation}
However, we have, by definition of an edge assignment,
\[
\prod_{i=1}^p \tilde{\epsilon}_{v_{i},v_{i-1}}\tilde{\epsilon}_{v'_{i},v'_{i-1}}\tilde{\epsilon}_{w_{i},w_{i-1}}\tilde{\epsilon}_{w'_{i},w'_{i-1}}=\prod_{(a,c)\in \indset}\Omega(\tilde{D})_{a,c},
\]
where $\indset$ is defined as follows.  Note that an element of $\two^{2p}$, say $c=c_1\ldots c_{2p}$ determines elements $c_L=c_1c_3\ldots c_{2p-1}\in \two^p$ and $c_R=c_2c_4\ldots c_{2p}\in \two^p$.  The objects $c_L,c_R\in \two^p$ will be called the first and second $\two^p$ factors of $c$.  The set $\indset$ is the set of pairs $(a,c)$ with $c=0^2x\in (\two^2)^p$ for some $x\in (\two^2)^{p-1}$, and $a$ is the result of replacing the rightmost $0$ in the first $\two^p$-factor of $c$ with a $1$, and the rightmost $0$ in the second $\two^p$-factor with a $1$.
For example, $(0^41^2,0^6)$ and $(01^5, 0^31^3)$ are both in $\indset$ for $p=3$; in the latter pair, $c_L = 0^21$ and $c_R = 01^2$.
We visualize the product $\prod \Omega(\tilde{D})$ as a product with a term for each face of a grid, whose vertices are objects of $(\two^{2})^p$.  We draw this as follows in the $p=3$ case, with only a few vertices labeled: 
\[
\begin{tikzpicture}[scale=.8]
\draw (0,0) grid (3,3);
\node[anchor=south east] (start) at (0,3) {$(0^6)$};
\node[anchor=south] (move1) at (.8,3) {$(0^51)$};
\node[anchor=south] (move2) at (2.2,3) {$(0^3101)$};
\node[anchor=south west] (right) at (3,3) {$(010101)$};
\node[anchor=north east] (left) at (0,0) {$(101010)$};
\node[anchor = north west] (end) at (3,0) {$(1^6)$};
\end{tikzpicture}
\]

Each of the faces of this grid $G$, corresponding to $a\geqslant_1 b,b'\geqslant_1 c\in (\two^2)^p$, is assigned a label in $\{\mathrm{A,C,X,Y}\}$ according to the type of the corresponding odd resolution configuration $\tilde{D}_{a,c}$.  Sometimes, we will assign the faces of the grid a $\pm 1$, using that $\Omega(\tilde{D})_{a,c}=1$ for faces of type C,Y and is $-1$ for faces of type A,X.  We will work to understand this grid in cases \textrm{I}--\textrm{IV}.  For instance, we will see below that for case \textrm{I} and $p=3$, the grid is:  
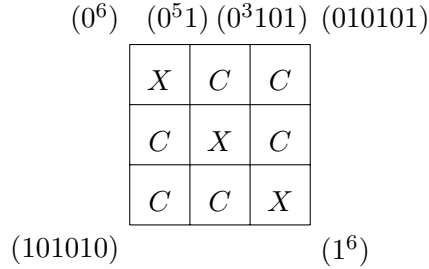
\[
\begin{tikzpicture}[scale=.8]
\draw (0,0) grid (3,3);
\node[anchor=south east] (start) at (0,3) {$(0^6)$};
\node[anchor=south] (move1) at (.8,3) {$(0^51)$};
\node[anchor=south] (move2) at (2.2,3) {$(0^3101)$};
\node[anchor=south west] (right) at (3,3) {$(010101)$};
\node[anchor=north east] (left) at (0,0) {$(101010)$};
\node[anchor = north west] (end) at (3,0) {$(1^6)$};
\node (tt) at (.5,2.4) {X};
\node (mm) at (1.5,1.4) {X};
\node (bb) at (2.5,.4) {X};
\node (mt) at (.5,1.4) {C};
\node (bt) at (.5,.4) {C};
\node (mb) at (2.5,1.4) {C};
\node (tb) at (2.5,2.4) {C};
\node (tm) at (1.5,2.4) {C};
\node (bm) at (1.5,0.4) {C};
\end{tikzpicture}
\]

Given a vertex $c\in \verti(G)$, with vertex $b\in \verti(G)$ directly below, and $b'\in \verti(G)$ directly to the right, we call $D_{b}$ the \emph{left resolution} of $D_c$, and $D_{b'}$ the \emph{right resolution} of $D_c$.   Note that each edge of the grid corresponds to resolving a crossing that is entirely contained within a single sector (recalling the notation of sectors from Section \ref{subsec:periodic-links}), and so we may label each edge of the grid by the sector in which the corresponding surgery occurs.

First we treat the configuration \textrm{I}.  Here, upstairs we have a picture as follows, illustrated for $p=3$:
\[
\begin{tikzpicture}[scale=.5]
\node (x) at (0,0) {$\mathbb{X}$};
          \begin{scope}[shift={(30:3cm)}]
            \draw (0,0) circle (1cm);
            \draw[thick,red,<-] (-1,0) --(1,0);
            \draw[thick,red,->] (150:1cm) to[out=150,in=90] (-2,0) to[out=-90,in=210] (210:1cm);
          \end{scope}
\begin{scope}[shift={(-90:3cm)}, rotate=-120 ]
            \draw (0,0) circle (1cm);
            \draw[thick,red,<-] (-1,0) --(1,0);
            \draw[thick,red,->] (150:1cm) to[out=150,in=90] (-2,0) to[out=-90,in=210] (210:1cm);
          \end{scope}
\begin{scope}[shift={(150:3cm)},rotate=-240]
           \draw (0,0) circle (1cm);
            \draw[thick,red,<-] (-1,0) --(1,0);
            \draw[thick,red,->] (150:1cm) to[out=150,in=90] (-2,0) to[out=-90,in=210] (210:1cm);
         \end{scope}
\draw[dotted] (0,0) -- (90:4cm);
\draw[dotted] (0,0) -- (-30:4cm);
\draw[dotted] (0,0) -- (-150:4cm);
\node (gamm) at (0,4) {$\tilde{\gamma}$};
\end{tikzpicture}
\]
Let $G$ denote the grid associated to such a configuration.  It is immediate from the definitions that all the faces on the main diagonal of $G$ are type X.  Now, for each off-diagonal face $D$, we see that one of the resolutions performed must be a merge.  Moreover, each off-diagonal resolution configuration is disconnected.  Inspecting the list of odd $2$-dimensional resolution configurations, any such configuration is of type C.  Thus, the total number of faces of type A or X is odd, which is equivalent to (\ref{eq:negative-epsilon}), since each face of type A or X contributes a factor of $-1$, while faces of type C and Y do not.  So, we have verified Lemma \ref{lem:edge-magic} in this case.

Next, we treat case \textrm{II}. The picture upstairs is as follows, again illustrated for $p=3$:
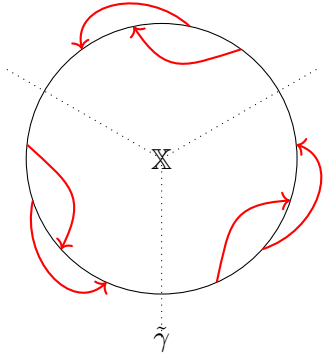
\[
\begin{tikzpicture}[scale=.6]
\node (x) at (0,0) {$\mathbb{X}$};
\draw[dotted] (0,0) -- (-90:4cm);
\draw[dotted] (0,0) -- (30:4cm);
\draw[dotted] (0,0) -- (150:4cm);

\draw (0,0) circle (3cm);

\draw[thick,red,->] (-90+24:3cm) .. controls (1.5,-1.5) and (1.5,-1)  .. (-90+3*24:3cm);
\draw[thick,red,->] (-90+24*2:3cm) .. controls (3.5,-1.6) and (4,0) .. (-90+4*24:3cm);

\draw[cm={cos(120),-sin(120),sin(120),cos(120),(0cm,0cm)},thick,red,->] (-90+24:3cm) .. controls (1.5,-1.5) and (1.5,-1)  .. (-90+3*24:3cm);
\draw[cm={cos(120),-sin(120),sin(120),cos(120),(0cm,0cm)},thick,red,->] (-90+24*2:3cm) .. controls (3.5,-1.6) and (4,0) .. (-90+4*24:3cm);

\draw[cm={cos(240),-sin(240),sin(240),cos(240),(0cm,0cm)},thick,red,->] (-90+24:3cm) .. controls (1.5,-1.5) and (1.5,-1)  .. (-90+3*24:3cm);
\draw[cm={cos(240),-sin(240),sin(240),cos(240),(0cm,0cm)},thick,red,->] (-90+24*2:3cm) .. controls (3.5,-1.6) and (4,0) .. (-90+4*24:3cm);

\node (gamm) at (0,-4) {$\tilde{\gamma}$};
\end{tikzpicture}
\]

It is readily checked once again that all of the diagonal faces are type X.  Fix an off-diagonal face with upper-left hand vertex at $a\in \mathrm{Vert}(G)$, whose left-resolution is in the $q\th$ sector and whose right-resolution is in the $r\neq q\th$ sector.  Write $\two^2_t$ for the $t\th$ factor of $\two^2$ in $(\two^2)^p$.  Then the resulting resolution configuration depends only on the initial condition of $c$ in $\two^2_r$ and $\two^2_q$.  To see this, consider the restriction of $D_{a,c}$ to a sector $S_t$ outside of $S_q$ and $S_r$.  It will be an arc connecting the boundary components $\partial^+ S_t$ and $\partial^- S_t$ (where the positive (\resp negative) boundary $\partial^+ S_t$ (\resp $\partial^-S_t$) of a sector $S_t$ will denote the end obtained by traversing counterclockwise (\resp clockwise)), as well as some disjoint circles, no matter the restriction of $c$ to $\two^2_t$.  In particular, the resulting two-dimensional resolution configuration $D_{a,c}$ is formed by drawing the parts of the resolution configuration in the $q$ and $r$ sectors, and attaching these on their boundaries; see for example Figure \ref{fig:type-2s}.

Next, fix $c\in \verti(G)$, the upper-left hand corner of a square $a,b,b',c$ in $G$, where $D_b$ is the left resolution and $D_{b'}$ is the right resolution.  Say the pair $a\geq_2 c$ differs only in entries $e_1,e_2$, where $e_1$ is in the $q\th$-sector and $e_2$ is in the $r\th$-sector.  Let $a_q,a_r,c_q,c_r$ denote the restrictions of $a$ and $c$ to $\two^2_q$, $\two^2_r$, respectively, and recall that the type of the resolution configuration $D_{a,c}$ depends only on $a_q,c_q,a_r,c_r$.  Note furthermore that the only $c$ in the grid for which $c_q=c_r=0^2$ is $c=0^{2p}$, which does not participate in an off-diagonal face.  So, we need only consider pairs $(a,c)$ with $(c_q,c_r)\neq (0^2,0^2)$.  We list all such resolution configurations and their types in Figure \ref{fig:type-2s}.
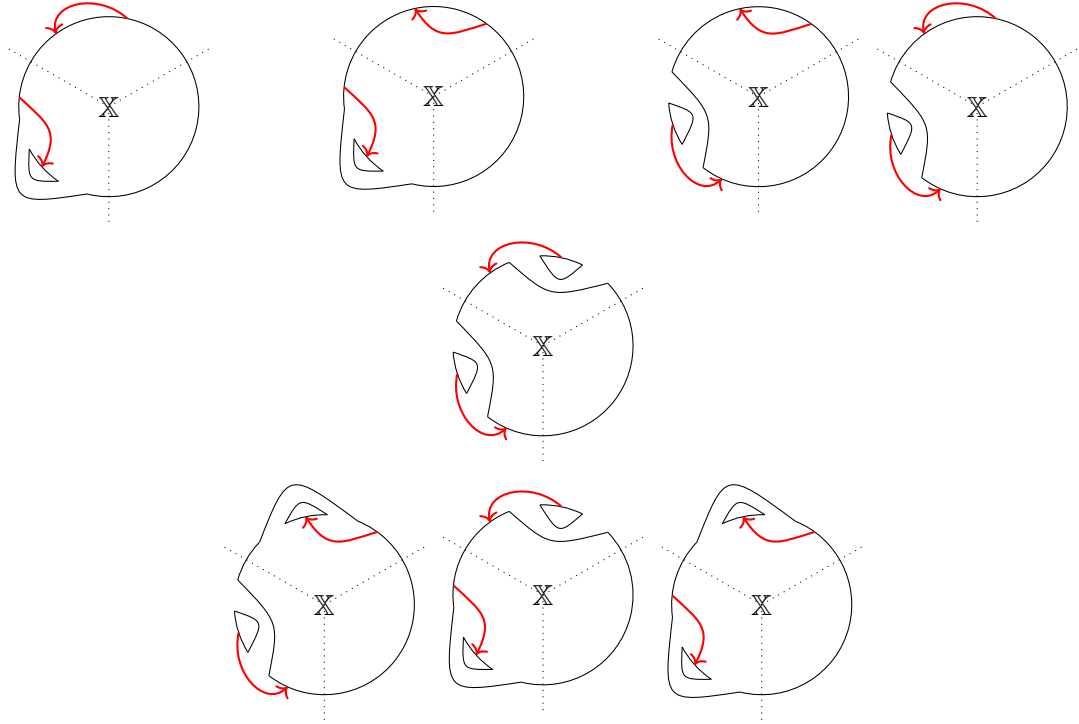
\begin{figure}

\centering

 \begin{tikzpicture}[scale=0.3991,baseline={(current  bounding  box.center)}]
\node (x) at (0,0) {$\mathbb{X}$};
\draw[dotted] (0,0) -- (-90:4cm);
\draw[dotted] (0,0) -- (30:4cm);
\draw[dotted] (0,0) -- (150:4cm);
\draw(-90-24+10:3cm) arc (-90-24+10:150+2*24-10:3cm) ;
\draw (-90-3*24+10:3cm) arc (-90-3*24+10:-90-24-10:3cm);
\draw (-90-3*24-10:3cm) .. controls (-3.3,-3.3) .. (-90-24+10:3cm);
\draw (-90-3*24+10:3cm) .. controls (-2.7,-2.5) .. (-90-24-10:3cm);
\draw[cm={cos(240),-sin(240),sin(240),cos(240),(0cm,0cm)},thick,red,->] (-90+24*2:3cm) .. controls (3.5,-1.6) and (4,0) .. (-90+4*24:3cm);
\draw[cm={cos(120),-sin(120),sin(120),cos(120),(0cm,0cm)},thick,red,->] (-90+24:3cm) .. controls (1.5,-1.5) and (1.5,-1)  .. (-90+3*24:3cm);
       \end{tikzpicture}\qquad \qquad
 \begin{tikzpicture}[scale=0.3991,baseline={(current  bounding  box.center)}]

\node (x) at (0,0) {$\mathbb{X}$};
\draw[dotted] (0,0) -- (-90:4cm);
\draw[dotted] (0,0) -- (30:4cm);
\draw[dotted] (0,0) -- (150:4cm);
\draw(-90-24+10:3cm) arc (-90-24+10:150+2*24-10:3cm) ;
\draw (-90-3*24+10:3cm) arc (-90-3*24+10:-90-24-10:3cm);
\draw (-90-3*24-10:3cm) .. controls (-3.3,-3.3) .. (-90-24+10:3cm);
\draw (-90-3*24+10:3cm) .. controls (-2.7,-2.5) .. (-90-24-10:3cm);
\draw[cm={cos(240),-sin(240),sin(240),cos(240),(0cm,0cm)},thick,red,->] (-90+24:3cm) .. controls (1.5,-1.5) and (1.5,-1)  .. (-90+3*24:3cm);

\draw[cm={cos(120),-sin(120),sin(120),cos(120),(0cm,0cm)},thick,red,->] (-90+24:3cm) .. controls (1.5,-1.5) and (1.5,-1)  .. (-90+3*24:3cm);
       \end{tikzpicture}\qquad \qquad
 \begin{tikzpicture}[scale=0.3991,baseline={(current  bounding  box.center)}]

\node (x) at (0,0) {$\mathbb{X}$};
\draw[dotted] (0,0) -- (-90:4cm);
\draw[dotted] (0,0) -- (30:4cm);
\draw[dotted] (0,0) -- (150:4cm);
\draw(-90-2*24+10:3cm) arc (-90-2*24+10:150+24-10:3cm) ;
\draw (-90-4*24+10:3cm) arc (-90-4*24+10:-90-2*24-10:3cm);
\draw (-90-4*24+10:3cm) .. controls (-2,-.55) .. (-90-2*24-10:3cm);
\draw (-90-2*24+10:3cm) .. controls (-1.5,-.6) .. (150+24-10:3cm);
\draw[cm={cos(120),-sin(120),sin(120),cos(120),(0cm,0cm)},thick,red,->] (-90+24*2:3cm) .. controls (3.5,-1.6) and (4,0) .. (-90+4*24:3cm);
\draw[cm={cos(240),-sin(240),sin(240),cos(240),(0cm,0cm)},thick,red,->] (-90+24:3cm) .. controls (1.5,-1.5) and (1.5,-1)  .. (-90+3*24:3cm);
       \end{tikzpicture}
 \begin{tikzpicture}[scale=0.3991,baseline={(current  bounding  box.center)}]

\node (x) at (0,0) {$\mathbb{X}$};
\draw[dotted] (0,0) -- (-90:4cm);
\draw[dotted] (0,0) -- (30:4cm);
\draw[dotted] (0,0) -- (150:4cm);
\draw(-90-2*24+10:3cm) arc (-90-2*24+10:150+24-10:3cm) ;
\draw (-90-4*24+10:3cm) arc (-90-4*24+10:-90-2*24-10:3cm);
\draw (-90-4*24+10:3cm) .. controls (-2,-.55) .. (-90-2*24-10:3cm);
\draw (-90-2*24+10:3cm) .. controls (-1.5,-.6) .. (150+24-10:3cm);
\draw[cm={cos(120),-sin(120),sin(120),cos(120),(0cm,0cm)},thick,red,->] (-90+24*2:3cm) .. controls (3.5,-1.6) and (4,0) .. (-90+4*24:3cm);
\draw[cm={cos(240),-sin(240),sin(240),cos(240),(0cm,0cm)},thick,red,->] (-90+24*2:3cm) .. controls (3.5,-1.6) and (4,0) .. (-90+4*24:3cm);
       \end{tikzpicture}\qquad \qquad
 \begin{tikzpicture}[scale=0.3991,baseline={(current  bounding  box.center)}]

\node (x) at (0,0) {$\mathbb{X}$};
\draw[dotted] (0,0) -- (-90:4cm);
\draw[dotted] (0,0) -- (30:4cm);
\draw[dotted] (0,0) -- (150:4cm);
\draw(-90-2*24+10:3cm) arc (-90-2*24+10:30+24-10:3cm) ;
\draw (30+24*3+10:3cm) arc (30+24*3+10:150+24-10:3cm);
\begin{scope}[rotate=240]
\draw (-90-4*24+10:3cm) arc (-90-4*24+10:-90-2*24-10:3cm);
\draw (-90-4*24+10:3cm) .. controls (-2,-.55) .. (-90-2*24-10:3cm);
\draw (-90-2*24+10:3cm) .. controls (-1.5,-.6) .. (150+24-10:3cm);
\end{scope}

\draw (-90-4*24+10:3cm) arc (-90-4*24+10:-90-2*24-10:3cm);
\draw (-90-4*24+10:3cm) .. controls (-2,-.55) .. (-90-2*24-10:3cm);
\draw (-90-2*24+10:3cm) .. controls (-1.5,-.6) .. (150+24-10:3cm);
\draw[cm={cos(240),-sin(240),sin(240),cos(240),(0cm,0cm)},thick,red,->] (-90+24*2:3cm) .. controls (3.5,-1.6) and (4,0) .. (-90+4*24:3cm);
\draw[cm={cos(120),-sin(120),sin(120),cos(120),(0cm,0cm)},thick,red,->] (-90+24*2:3cm) .. controls (3.5,-1.6) and (4,0) .. (-90+4*24:3cm);
       \end{tikzpicture}

 \begin{tikzpicture}[scale=0.3991,baseline={(current  bounding  box.center)}]

\node (x) at (0,0) {$\mathbb{X}$};
\draw[dotted] (0,0) -- (-90:4cm);
\draw[dotted] (0,0) -- (30:4cm);
\draw[dotted] (0,0) -- (150:4cm);
\draw(-90-2*24+10:3cm) arc (-90-2*24+10:30+24*2-10:3cm) ;
\draw (30+24*4+10:3cm) arc (30+24*4+10:150+24-10:3cm);
\begin{scope}[rotate=240]
\draw (-90-3*24+10:3cm) arc (-90-3*24+10:-90-24-10:3cm);
\draw (-90-3*24-10:3cm) .. controls (-3.3,-3.3) .. (-90-24+10:3cm);
\draw (-90-3*24+10:3cm) .. controls (-2.7,-2.5) .. (-90-24-10:3cm);
\end{scope}

\draw (-90-4*24+10:3cm) arc (-90-4*24+10:-90-2*24-10:3cm);
\draw (-90-4*24+10:3cm) .. controls (-2,-.55) .. (-90-2*24-10:3cm);
\draw (-90-2*24+10:3cm) .. controls (-1.5,-.6) .. (150+24-10:3cm);
\draw[cm={cos(240),-sin(240),sin(240),cos(240),(0cm,0cm)},thick,red,->] (-90+24:3cm) .. controls (1.5,-1.5) and (1.5,-1)  .. (-90+3*24:3cm);
\draw[cm={cos(120),-sin(120),sin(120),cos(120),(0cm,0cm)},thick,red,->] (-90+24*2:3cm) .. controls (3.5,-1.6) and (4,0) .. (-90+4*24:3cm);

       \end{tikzpicture}
 \begin{tikzpicture}[scale=0.3991,baseline={(current  bounding  box.center)}]
\node (x) at (0,0) {$\mathbb{X}$};
\draw[dotted] (0,0) -- (-90:4cm);
\draw[dotted] (0,0) -- (30:4cm);
\draw[dotted] (0,0) -- (150:4cm);
\draw(-90-24+10:3cm) arc (-90-24+10:30+24-10:3cm) ;
\draw (30+24*3+10:3cm) arc (30+24*3+10:150+2*24-10:3cm);
\begin{scope}[rotate=240]
\draw (-90-4*24+10:3cm) arc (-90-4*24+10:-90-2*24-10:3cm);
\draw (-90-4*24+10:3cm) .. controls (-2,-.55) .. (-90-2*24-10:3cm);
\draw (-90-2*24+10:3cm) .. controls (-1.5,-.6) .. (150+24-10:3cm);
\end{scope}

\draw (-90-3*24+10:3cm) arc (-90-3*24+10:-90-24-10:3cm);
\draw (-90-3*24-10:3cm) .. controls (-3.3,-3.3) .. (-90-24+10:3cm);
\draw (-90-3*24+10:3cm) .. controls (-2.7,-2.5) .. (-90-24-10:3cm);
\draw[cm={cos(240),-sin(240),sin(240),cos(240),(0cm,0cm)},thick,red,->] (-90+24*2:3cm) .. controls (3.5,-1.6) and (4,0) .. (-90+4*24:3cm);
\draw[cm={cos(120),-sin(120),sin(120),cos(120),(0cm,0cm)},thick,red,->] (-90+24:3cm) .. controls (1.5,-1.5) and (1.5,-1)  .. (-90+3*24:3cm);

       \end{tikzpicture}
 \begin{tikzpicture}[scale=0.3991,baseline={(current  bounding  box.center)}]
\node (x) at (0,0) {$\mathbb{X}$};
\draw[dotted] (0,0) -- (-90:4cm);
\draw[dotted] (0,0) -- (30:4cm);
\draw[dotted] (0,0) -- (150:4cm);
\draw(-90-24+10:3cm) arc (-90-24+10:30+24*2-10:3cm) ;
\draw (30+24*4+10:3cm) arc (30+24*4+10:150+2*24-10:3cm);
\begin{scope}[rotate=240]
\draw (-90-3*24+10:3cm) arc (-90-3*24+10:-90-24-10:3cm);
\draw (-90-3*24-10:3cm) .. controls (-3.3,-3.3) .. (-90-24+10:3cm);
\draw (-90-3*24+10:3cm) .. controls (-2.7,-2.5) .. (-90-24-10:3cm);
\end{scope}

\draw (-90-3*24+10:3cm) arc (-90-3*24+10:-90-24-10:3cm);
\draw (-90-3*24-10:3cm) .. controls (-3.3,-3.3) .. (-90-24+10:3cm);
\draw (-90-3*24+10:3cm) .. controls (-2.7,-2.5) .. (-90-24-10:3cm);
\draw[cm={cos(240),-sin(240),sin(240),cos(240),(0cm,0cm)},thick,red,->] (-90+24:3cm) .. controls (1.5,-1.5) and (1.5,-1)  .. (-90+3*24:3cm);
\draw[cm={cos(120),-sin(120),sin(120),cos(120),(0cm,0cm)},thick,red,->] (-90+24:3cm) .. controls (1.5,-1.5) and (1.5,-1)  .. (-90+3*24:3cm);
       \end{tikzpicture}
\caption{The off-diagonal resolution configurations in case \textrm{II}. The first four configurations are realized up to isotopy by expressions of the form $(*1,0*)\in \two^2_q\times \two^2_r$ and their permutations, while the latter four are obtained from permutations of $(*1,*1)\in \two^2_q\times \two^2_r$.}
\label{fig:type-2s}
\end{figure}
Indeed, we see that all the off-diagonal faces of $G$ are type C, which completes case \textrm{II} (since type X faces appear an odd number of times on the diagonal).

Case \textrm{III} is quite similar to case \textrm{II} and is omitted.

Finally, we address case \textrm{IV}.  The picture upstairs is as follows (illustrated for $p=5$): 
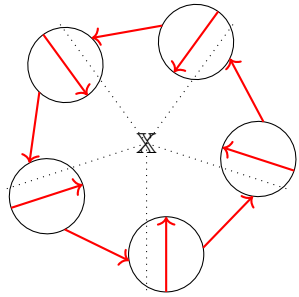
\[
\begin{tikzpicture}[scale=.5]
\node (x) at (0,0) {$\mathbb{X}$};
\draw[dotted] (0,0) -- (-90:4cm);
\draw[dotted] (0,0) -- (-90+72:4cm);
\draw[dotted] (0,0) -- (-90+2*72:4cm);
\draw[dotted] (0,0) -- (-90+3*72:4cm);
\draw[dotted] (0,0) -- (-90+4*72:4cm);
\begin{scope}[shift={(-90+10:3cm)}]
\draw (0,0) circle (1cm);
\draw[thick,->,red] (0,-1) -- (0,1);
\end{scope}
\begin{scope}[shift={(-90+72+10:3cm)}, rotate=72]
\draw (0,0) circle (1cm);
\draw[thick,->,red] (0,-1) -- (0,1);
\end{scope}
\begin{scope}[shift={(-90+72*2+10:3cm)}, rotate=72*2]
\draw (0,0) circle (1cm);
\draw[thick,->,red] (0,-1) -- (0,1);
\end{scope}
\begin{scope}[shift={(-90+72*3+10:3cm)}, rotate=72*3]
\draw (0,0) circle (1cm);
\draw[thick,->,red] (0,-1) -- (0,1);
\end{scope}
\begin{scope}[shift={(-90+72*4+10:3cm)}, rotate=72*4]
\draw (0,0) circle (1cm);
\draw[thick,->,red] (0,-1) -- (0,1);
\end{scope}

\begin{scope}[rotate=10]
\draw[red,->,thick] (-90+18.434:3.16cm) -- (-90+72-18.434:3.16cm);
\end{scope}
\begin{scope}[rotate=10+72]
\draw[red,->,thick] (-90+18.434:3.16cm) -- (-90+72-18.434:3.16cm);
\end{scope}
\begin{scope}[rotate=10+72*2]
\draw[red,->,thick] (-90+18.434:3.16cm) -- (-90+72-18.434:3.16cm);
\end{scope}
\begin{scope}[rotate=10+72*3]
\draw[red,->,thick] (-90+18.434:3.16cm) -- (-90+72-18.434:3.16cm);
\end{scope}
\begin{scope}[rotate=10+72*4]
\draw[red,->,thick] (-90+18.434:3.16cm) -- (-90+72-18.434:3.16cm);
\end{scope}
\end{tikzpicture}
\]
We order the crossings so that the edges forming a pentagon correspond to the first factor $\two^p\to (\two^2)^p$ and the other edges correspond to the second factor $\two^p \to (\two^2)^p$.  

We divide length $1$-arrows in $(\two^2)^p$ into two sets as follows.  Recall that each arrow $\phi_{v,u}^\op$ for $u\geqslant_1 v$ can be recorded as the element $v\in(\two^2)^p$, but with one of the $1,0$-entries of $v$ replaced by a $\ast$ to denote the entry that changes between $v,u$.  If $*$ is at an odd position in $\two^p$ (that is, $*$ occurs in the first $\two$-factor in some copy $\two^2\subset (\two^2)^p$), we call $\phi_{v,u}^\op$ a \emph{left edge}, otherwise a \emph{right edge}.  Similarly, an index-$2$ resolution configuration from $u\geq_2 w$ can be described by an element in $(\two^2)^p$ with two bits replaced by $*$.  

Note that resolving a right edge on some resolution $D_c$ is a split, unless $c=(10)^p$.  Further, resolving a left edge is a merge unless $c=(10)^k(00)(10)^{p-k-1}$ for some $k$.  
Further, any resolution configuration $D_{u,w}$ for which $\phi^\op_{w,v}$ is a split and $\phi^\op_{v',u}$ is a split, while $\phi^\op_{w,v'}$ and $\phi^{\op}_{v,u}$ are merges, has type C.  
We then need only consider faces in $G$ containing the vertex $(10)^p$ or some $(10)^k(00)(10)^{p-k-1}$.  However, $(10)^k(00)(10)^{p-k-1}$ is a vertex of $G$ if and only if $k=0$.  So, we see in fact that only the lower left-hand cornered can be of type other than C.  The picture in the lower left-hand corner is:
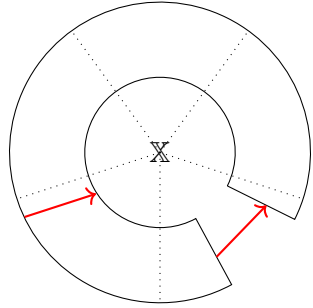
\[
\begin{tikzpicture}[scale=.5]
\node (x) at (0,0) {$\mathbb{X}$};
\draw[dotted] (0,0) -- (-90:4cm);
\draw[dotted] (0,0) -- (-90+72:4cm);
\draw[dotted] (0,0) -- (-90+2*72:4cm);
\draw[dotted] (0,0) -- (-90+3*72:4cm);
\draw[dotted] (0,0) -- (-90+4*72:4cm);

\draw (-90+72-18.434+10:2cm) -- (-90+72-18.434+10:4cm)
  arc (-90+72-18.434+10:270+18.434+10:4cm) -- (270+18.434+10:2cm) arc (270+18.434+10:-90+72-18.434+10:2cm) ;
\begin{scope}[shift={(-90+72*4+10:3cm)}, rotate=72*4]
\draw[thick,->,red] (0,-1) -- (0,1);
\end{scope}
\begin{scope}[rotate=10]
\draw[red,->,thick] (-90+18.434:3.16cm) -- (-90+72-18.434:3.16cm);
\end{scope}
\end{tikzpicture}
\]
This is a type X face, and so the proof is completed for case \textrm{IV}.

Translating the above proof to type Y faces is immediate.  The grid is the same in each case, with type X faces replaced with type Y faces.  

The only case that remains to check is that of three concentric circles (and its dual).  We fix an orientation of edges as below; the case of other orientations is similar.
\[
\begin{tikzpicture}[scale=.5]
\node (x) at (0,0) {$\mathbb{X}$};
\draw[dotted] (0,0) -- (-90:4cm);
\draw[dotted] (0,0) -- (30:4cm);
\draw[dotted] (0,0) -- (150:4cm);
\draw (0,0) circle (1cm);
\draw (0,0) circle (2cm);
\draw (0,0) circle (3cm);
\begin{scope}[rotate=-90]
\draw[->,thick,red] (40:2cm) -- (40:1cm);
\draw[->, thick, red] (80:2cm) -- (80:3cm);
\end{scope}
\begin{scope}[rotate=30]
\draw[->,thick,red] (40:2cm) -- (40:1cm);
\draw[->, thick, red] (80:2cm) -- (80:3cm);
\end{scope}
\begin{scope}[rotate=-210]
\draw[->,thick,red] (40:2cm) -- (40:1cm);
\draw[->, thick, red] (80:2cm) -- (80:3cm);
\end{scope}
\end{tikzpicture}
\]
We order the crossings so that the outer edges correspond to the first factor $\two^p\to (\two^2)^p$ and the inner edges correspond to the second factor $\two^p \to (\two^2)^p$.  
The upper left-hand corner of $G$ is readily seen to be a type C configuration, since it consists of two merges.  We note that the next configuration on the diagonal of $G$ is a type X face:
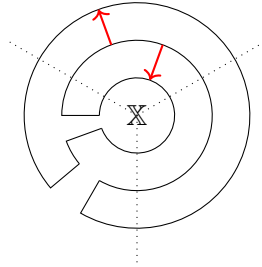
\[
\begin{tikzpicture}[scale=.5]
\node (x) at (0,0) {$\mathbb{X}$};
\draw[dotted] (0,0) -- (-90:4cm);
\draw[dotted] (0,0) -- (30:4cm);
\draw[dotted] (0,0) -- (150:4cm);

\draw (-120:2cm) --(-120:3cm) arc (-120:-120+360-20:3cm)  -- (-120+360-20:2cm);
\draw (-90-50:2cm) arc (-90-50:-90-70:2cm)-- (-90-70:1cm);
\draw (-90-30:2cm) arc (-90-30:-90-90+360:2cm) -- (-180:1cm);
\draw (-90-70:1cm) arc (-90-70:-90-90+360:1cm);
\begin{scope}[rotate=30]
\draw[->,thick,red] (40:2cm) -- (40:1cm);
\draw[->, thick, red] (80:2cm) -- (80:3cm);
\end{scope}
\end{tikzpicture}
\]
 In fact, all other faces on the diagonal are type X, since the arcs outside of the `active' sector, up to isotopy, do not depend on $c$, as is illustrated below: 
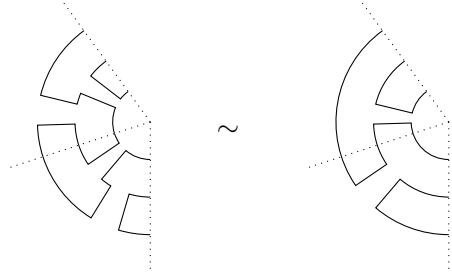
\[
\begin{tikzpicture}[scale=.5,baseline={(current  bounding  box.center)}]
\draw[dotted] (0,0) -- (-90:4cm);
\draw[dotted] (0,0) -- (-90-72:4cm);
\draw[dotted] (0,0) -- (-90-72*2:4cm);
\begin{scope}[rotate=-72-90]
\draw (0:1cm) arc (0:24-7.65:1cm) -- (24-7.65:2cm) arc (24-7.65:0:2cm);
\draw (0:3cm) arc (0:24*2-7.65:3cm) -- (24*2-7.65:2cm) arc (24*2-7.65:24+7.65:2cm) -- (24+7.65:1cm) arc (24+7.65:72:1cm);
\draw (72:3cm) arc (72:2*24+7.65:3cm) -- (2*24+7.65:2cm) arc (2*24+7.65:72:2cm);
\end{scope}
\begin{scope}[rotate=-72*2-90]
\draw (0:1cm) arc (0:24-7.65:1cm) -- (24-7.65:2cm) arc (24-7.65:0:2cm);
\draw (0:3cm) arc (0:24*2-7.65:3cm) -- (24*2-7.65:2cm) arc (24*2-7.65:24+7.65:2cm) -- (24+7.65:1cm) arc (24+7.65:72:1cm);
\draw (72:3cm) arc (72:2*24+7.65:3cm) -- (2*24+7.65:2cm) arc (2*24+7.65:72:2cm);
\end{scope}
\end{tikzpicture}
\qquad \sim\qquad 
\begin{tikzpicture}[scale=.5,baseline={(current  bounding  box.center)}]
\draw[dotted] (0,0) -- (-90:4cm);
\draw[dotted] (0,0) -- (-90-72:4cm);
\draw[dotted] (0,0) -- (-90-72*2:4cm);
\begin{scope}[rotate=-72*2-90]
\draw (0:1cm) arc (0:48-7.65:1cm) -- (48-7.65:2cm) arc (48-7.65:0:2cm);
\draw (0:3cm) arc (0:48*2-7.65:3cm) -- (48*2-7.65:2cm) arc (48*2-7.65:48+7.65:2cm) -- (48+7.65:1cm) arc (48+7.65:144:1cm);
\draw (144:3cm) arc (144:2*48+7.65:3cm) -- (2*48+7.65:2cm) arc (2*48+7.65:144:2cm);
\end{scope}
\end{tikzpicture}
\]
In particular, there are an even number of faces of type X on the diagonal.

For $u\in (\two^2)^p$, let $|u|_1$ (\resp $|u|_2$) denote the number of $1$'s occurring in the first copy of $\two^p \to (\two^2)^p$ (\resp second copy).  Now suppose $D_{a,c}$ is an index-$2$ resolution configuration such that $|c|_1> |c|_2$, for $a,c\in \verti(G)$; suppose $b$ is the left resolution and $b'$ is the right resolution.  Such resolution configurations are, up to isotopy:
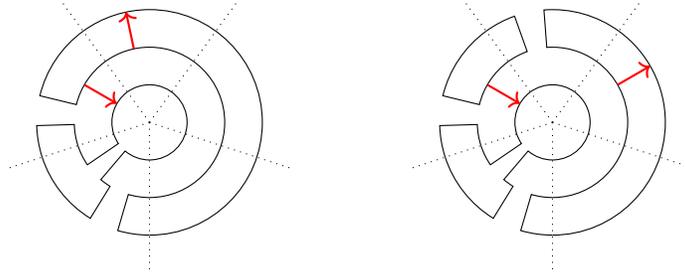
\[
\begin{tikzpicture}[scale=.5, baseline={(current  bounding  box.center)}]
\draw[dotted] (0,0) -- (-90:4cm);
\draw[dotted] (0,0) -- (-90-72:4cm);
\draw[dotted] (0,0) -- (-90-72*2:4cm);
\draw[dotted] (0,0) -- (-90-72*3:4cm);
\draw[dotted] (0,0) -- (-90-72*4:4cm);
\begin{scope}[rotate=-72-90]
\draw (0:1cm) arc (0:24-7.65:1cm) -- (24-7.65:2cm) arc (24-7.65:0:2cm);
\draw (0:3cm) arc (0:24*2-7.65:3cm) -- (24*2-7.65:2cm) arc (24*2-7.65:24+7.65:2cm) -- (24+7.65:1cm) arc (24+7.65:72:1cm);
\draw (72:3cm) arc (72:2*24+7.65:3cm) -- (2*24+7.65:2cm) arc (2*24+7.65:72:2cm);
\end{scope}
\begin{scope}[rotate=-72*2-90]
\draw (0:1cm) arc (0:72:1cm);
\draw (0:3cm) arc (0:72-24-7.65:3cm) -- (72-24-7.65:2cm) arc (72-24-7.65:0:2cm);
\draw (72:3cm) arc (72:2*24+7.65:3cm) -- (2*24+7.65:2cm) arc (2*24+7.65:72:2cm);
\end{scope}
\draw (-90:3cm) arc (-90:-90+3*72:3cm);
\draw (-90:2cm) arc (-90:-90+3*72:2cm);
\draw (-90:1cm) arc (-90:-90+3*72:1cm);

\draw[->,thick,red] (72*2-90+2*24:2cm) -- (72*2-90+2*24:3cm);
\draw[->,thick,red] (3*72-90+24:2cm) -- (3*72-90+24:1cm) ;
\end{tikzpicture}
\qquad \qquad
\begin{tikzpicture}[scale=.5,baseline={(current  bounding  box.center)}]
\draw[dotted] (0,0) -- (-90:4cm);
\draw[dotted] (0,0) -- (-90-72:4cm);
\draw[dotted] (0,0) -- (-90-72*2:4cm);
\draw[dotted] (0,0) -- (-90-72*3:4cm);
\draw[dotted] (0,0) -- (-90-72*4:4cm);
\begin{scope}[rotate=-72-90]
\draw (0:1cm) arc (0:24-7.65:1cm) -- (24-7.65:2cm) arc (24-7.65:0:2cm);
\draw (0:3cm) arc (0:24*2-7.65:3cm) -- (24*2-7.65:2cm) arc (24*2-7.65:24+7.65:2cm) -- (24+7.65:1cm) arc (24+7.65:72:1cm);
\draw (72:3cm) arc (72:2*24+7.65:3cm) -- (2*24+7.65:2cm) arc (2*24+7.65:72:2cm);
\end{scope}
\begin{scope}[rotate=-72*2-90]
\draw (0:1cm) arc (0:72:1cm);
\draw (0:3cm) arc (0:72-24-7.65:3cm) -- (72-24-7.65:2cm) arc (72-24-7.65:0:2cm);
\draw (72:3cm) arc (72:2*24+7.65:3cm) -- (2*24+7.65:2cm) arc (2*24+7.65:72:2cm);
\end{scope}
\begin{scope}[rotate=-72*3-90]
\draw (0:1cm) arc (0:72:1cm);
\draw (0:3cm) arc (0:72-24-7.65:3cm) -- (72-24-7.65:2cm) arc (72-24-7.65:0:2cm);
\draw (72:3cm) arc (72:2*24+7.65:3cm) -- (2*24+7.65:2cm) arc (2*24+7.65:72:2cm);
\end{scope}
\draw (-90:3cm) arc (-90:-90+2*72:3cm);
\draw (-90:2cm) arc (-90:-90+2*72:2cm);
\draw (-90:1cm) arc (-90:-90+2*72:1cm);
\draw[->,thick,red] (3*72-90+24:2cm) -- (3*72-90+24:1cm) ;
\draw[->,thick,red] (72-90+2*24:2cm) -- (72-90+2*24:3cm);
\end{tikzpicture}
\]
From these, we observe that $\phi^\op_{c,b'}$ is a merge and $\phi^\op_{b',a}$ is a split, while $\phi^\op_{c,b}$ is a split and $\phi^{\op}_{b,a}$ is a merge.  Any such resolution configuration has type C.  For any $c$ with $|c|_2>|c|_1$, it turns out similarly that $D_{a,c}$ is type C.  Therefore, the total number of faces of type A and X is even.  Then, as in the discussion of case I in the proof of Lemma \ref{lem:edge-magic}, for the case of three concentric circles downstairs, $(\delta \epsilon)_{u,w}=1=\Omega(D)_{u,w}$.  (The case of three concentric circles, with the orientation of edges changed, results in replacing the type X faces on the diagonal by type Y faces.)  

We omit the case dual to three concentric circles; it follows by application of techniques similar to above.

Since there is at most one signed matching compatible with the Khovanov-Burnside functor, we have that the matching specified above is the ladybug matching.  This completes the proof of Lemma \ref{lem:edge-magic}.
\end{proof}
In turn, Lemma \ref{lem:fake-functor} and Lemma \ref{lem:edge-magic} complete the proof Proposition \ref{prop:long-diffs}.  
\end{proof} 

We next deal with the case of $2$-morphisms for the even Khovanov functor. 

\begin{lem}\label{lem:ladybug-compatible}
Let $u\geqslant_1 v,v'\geqslant_1 w\in \two^n$. The bijection 
\[
	\khburn(\phi_{\iota v,\iota w})^{\ZZ_p}\circ \khburn(\phi_{\iota u,\iota v})^{\ZZ_p}\to \khburn(\phi_{\iota v',\iota w})^{\ZZ_p}\circ \khburn(\phi_{\iota u,\iota v'})^{\ZZ_p}
\] is the ladybug matching.
\end{lem}
	\begin{proof}
   This is quite similar to, but more straightforward than, the proof of Lemma \ref{lem:edge-magic} above.  First of all, there is only something to check if the configuration $D_{u,w}$ downstairs is a ladybug (so there is no analogue of the three-concentric-circles case in the previous proof).  Moreover, $\khburn(\phi_{\iota v,\iota w})^{\ZZ_p}\circ \khburn(\phi_{\iota u,\iota v})^{\ZZ_p}=\emptyset$ for configurations of type \textrm{II} and \textrm{III} (appearing in the proof of Lemma \ref{lem:edge-magic}).  That is, we need only consider index-$2$ annular resolution configurations downstairs of types \textrm{I} and \textrm{IV}.

The case \textrm{I} is a similar calculation to Lemma \ref{lem:edge-magic} and is omitted.  For case \textrm{IV}, we argue inductively.  For odd $p$, we are already done by the comment before the proof, and it is a straightforward calculation to verify that the Lemma holds for the case $p=2$.  

For $p$ odd, the Lemma follows as a consequence of Lemma \ref{lem:edge-magic}, since an edge assignment determines the ladybug matching. 

Say we have verified case \textrm{IV} for fixed $p'$. We show how to verify it for $p=2p'$.  The resolution configuration $\tilde{D}$ upstairs is formed from $p'$ sectors of the form below; the dotted lines indicate the boundary of one of the $p'$ sectors, and the dashed line further bisects this sector into two of the $p = 2p'$ sectors:

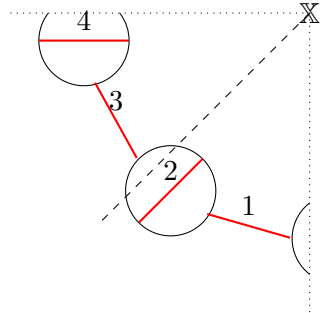
\[
\begin{tikzpicture}
\node (x) at (0,0) {$\mathbb{X}$};
\draw[dotted] (0,0) -- (-90:4cm);
\draw[dotted] (0,0) -- (-180:4cm);
\draw[dashed] (0,0) -- (-135:4cm);
\begin{scope}[shift={(-135+7:3cm)}, rotate=-45]
\draw (0,0) circle (.6cm);
\draw[thick,red] (0,-.6) -- (0,.6) node[black,pos=0.5,anchor=south] {$2$};
\end{scope}

\begin{scope}[shift={(-90:3cm)}]
\draw (0,.475) arc (180-52.34:180+52.34:.6cm);
\end{scope}

\begin{scope}[shift={(-180:3cm)},rotate=-90]
\draw (0,-.475) arc (-180+52.34:180-52.34:.6cm);
\draw[thick,red] (.366,-.6) -- (.366,.6) node[black,pos=0.5,anchor=south] {$4$};
\end{scope}

\begin{scope}[rotate=-180+10]
\draw[thick,red] (8:3cm) -- (30:3cm) node[black,pos=0.5,anchor=south] {$3$};
\end{scope}
\begin{scope}[rotate=-180+55]
\draw[thick,red] (8:3cm) -- (30:3cm) node[black,pos=0.5,anchor=south] {$1$};
\end{scope}

\end{tikzpicture}
\]

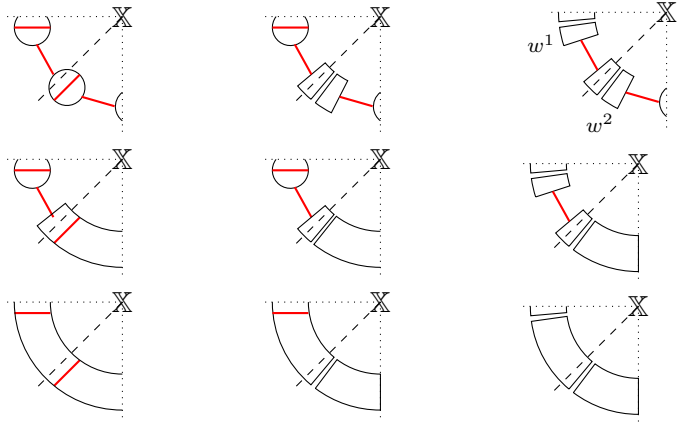
\begin{figure}
\centering
\begin{minipage}{.2\textwidth}
\begin{tikzpicture}[scale=0.3991,baseline={(current  bounding  box.center)}] 
\node (x) at (0,0) {$\mathbb{X}$};
\draw[dotted] (0,0) -- (-90:4cm);
\draw[dotted] (0,0) -- (-180:4cm);
\draw[dashed] (0,0) -- (-135:4cm);
\begin{scope}[shift={(-135+7:3cm)}, rotate=-45]
\draw (0,0) circle (.6cm);
\draw[thick,red] (0,-.6) -- (0,.6);
\end{scope}

\begin{scope}[shift={(-90:3cm)}]
\draw (0,.475) arc (180-52.34:180+52.34:.6cm);
\end{scope}

\begin{scope}[shift={(-180:3cm)},rotate=-90]
\draw (0,-.475) arc (-180+52.34:180-52.34:.6cm);
\draw[thick,red] (.366,-.6) -- (.366,.6);
\end{scope}

\begin{scope}[rotate=-180+10]
\draw[thick,red] (8:3cm) -- (30:3cm);
\end{scope}
\begin{scope}[rotate=-180+55]
\draw[thick,red] (8:3cm) -- (30:3cm);
\end{scope}
\end{tikzpicture}
\begin{tikzpicture}[scale=0.3991,baseline={(current  bounding  box.center)}] 
\node (x) at (0,0) {$\mathbb{X}$};
\draw[dotted] (0,0) -- (-90:4cm);
\draw[dotted] (0,0) -- (-180:4cm);
\draw[dashed] (0,0) -- (-135:4cm);

\begin{scope}[shift={(-180:3cm)},rotate=-90]
\draw (0,-.475) arc (-180+52.34:180-52.34:.6cm);
\draw[thick,red] (.366,-.6) -- (.366,.6);
\end{scope}
\begin{scope}[rotate=-180+10]
\draw[thick,red] (8:3cm) -- (30:3cm);
\end{scope}
\begin{scope}[shift={(-135+7:3cm)}, rotate=-45]
\draw[thick,red] (0,-.6) -- (0,.6);
\end{scope}
\draw (-90:3.6cm) arc (-90:-90-45-7:3.6cm) -- (-90-45-7.3:2.4cm) arc (-90-45-7.32:-90:2.4cm);

\end{tikzpicture}

\begin{tikzpicture}[scale=0.3991,baseline={(current  bounding  box.center)}] 
\node (X) at (0,0) {$\mathbb{X}$};
\draw[dotted] (0,0) -- (-90:4cm);
\draw[dotted] (0,0) -- (-180:4cm);
\draw[dashed] (0,0) -- (-135:4cm);
\node (y) at (-160+4:4.3cm) {\scriptsize $y$};
\node (x) at (-160+4:1.8cm) {\scriptsize $x$};

\begin{scope}[shift={(-180:3cm)},rotate=-90]
\draw[thick,red] (.366,-.6) -- (.366,.6);
\end{scope}

\begin{scope}[shift={(-135+7:3cm)}, rotate=-45]
\draw[thick,red] (0,-.6) -- (0,.6);
\end{scope}
\draw[thick] (-90:3.6cm) arc (-90:-180:3.6cm);
\draw[thick] (-90:2.4cm) arc (-90:-180:2.4cm);

\end{tikzpicture}
\end{minipage}
\begin{minipage}{.2\textwidth}
\begin{tikzpicture}[scale=0.3991,baseline={(current  bounding  box.center)}] 
\node (x) at (0,0) {$\mathbb{X}$};
\draw[dotted] (0,0) -- (-90:4cm);
\draw[dotted] (0,0) -- (-180:4cm);
\draw[dashed] (0,0) -- (-135:4cm);
\begin{scope}[rotate=-90]
\draw (-45+5:3.6cm) arc (-45+5:-45-5:3.6cm) -- (-45-5:2.4cm) arc (-45-5:-45+5:2.4cm) -- (-45+5:3.6cm);
\end{scope}

\begin{scope}[rotate=-90]
\draw (-45+8:3.6cm) arc (-45+8:-45+18:3.6cm) -- (-45+18:2.4cm) arc (-45+18:-45+8:2.4cm) -- (-45+8:3.6cm);
\end{scope}

\begin{scope}[shift={(-90:3cm)}]
\draw (0,.475) arc (180-52.34:180+52.34:.6cm);
\end{scope}

\begin{scope}[shift={(-180:3cm)},rotate=-90]
\draw (0,-.475) arc (-180+52.34:180-52.34:.6cm);
\draw[thick,red] (.366,-.6) -- (.366,.6);
\end{scope}

\begin{scope}[rotate=-180+10]
\draw[thick,red] (8:3cm) -- (30:3cm);
\end{scope}
\begin{scope}[rotate=-180+55]
\draw[thick,red] (8:3cm) -- (30:3cm);
\end{scope}
\end{tikzpicture}

\begin{tikzpicture}[scale=0.3991,baseline={(current  bounding  box.center)}] 
\node (x) at (0,0) {$\mathbb{X}$};
\draw[dotted] (0,0) -- (-90:4cm);
\draw[dotted] (0,0) -- (-180:4cm);
\draw[dashed] (0,0) -- (-135:4cm);
\node (z) at (-170+4:4.2cm) {\scriptsize $z$};

\begin{scope}[shift={(-180:3cm)},rotate=-90]
\draw[thick] (0,-.475) arc (-180+52.34:180-52.34:.6cm);
\draw[thick,red] (.366,-.6) -- (.366,.6);
\end{scope}
\begin{scope}[rotate=-180+10]
\draw[thick,red] (8:3cm) -- (30:3cm);
\end{scope}

\draw[thick] (-90:3.6cm) arc (-90:-90-45+7:3.6cm) -- (-90-45+7.3:2.4cm) arc (-90-45+7.32:-90:2.4cm);
\begin{scope}[rotate=-90]
\draw (-45+5:3.6cm) arc (-45+5:-45-5:3.6cm) -- (-45-5:2.4cm) arc (-45-5:-45+5:2.4cm) -- (-45+5:3.6cm);
\end{scope}
\end{tikzpicture}

\begin{tikzpicture}[scale=0.3991,baseline={(current  bounding  box.center)}]
\node (x) at (0,0) {$\mathbb{X}$};
\draw[dotted] (0,0) -- (-90:4cm);
\draw[dotted] (0,0) -- (-180:4cm);
\draw[dashed] (0,0) -- (-135:4cm);

\draw (-180:2.4cm) arc (-180:-45-90+5:2.4cm) -- (-45-90+5:3.6cm) arc (-45-90+5:-180:3.6cm);
\begin{scope}[shift={(-180:3cm)},rotate=-90]
\draw[thick,red] (.366,-.6) -- (.366,.6);
\end{scope}

\begin{scope}[rotate=-90]
\draw (-45+8:3.6cm) arc (-45+8:0:3.6cm) -- (0:2.4cm) arc (0:-45+8:2.4cm) -- (-45+8:3.6cm);
\end{scope}

\end{tikzpicture}
\end{minipage}
\begin{minipage}{.2\textwidth}

\begin{tikzpicture}[scale=0.3991,baseline={(current  bounding  box.center)}] 
\node (x) at (0,0) {$\mathbb{X}$};
\draw[dotted] (0,0) -- (-90:4cm);
\draw[dotted] (0,0) -- (-180:4cm);
\draw[dashed] (0,0) -- (-135:4cm);
\begin{scope}[rotate=-90]
\draw (-45+5:3.6cm) arc (-45+5:-45-5:3.6cm) -- (-45-5:2.4cm) arc (-45-5:-45+5:2.4cm) -- (-45+5:3.6cm);
\end{scope}

\node (w1) at (-170+4:4.3cm) {\scriptsize $w^1 $};
\node (w2) at (-170+45+4:4.3cm) {\scriptsize $w^2 $};
\begin{scope}[rotate=-90]
\draw[thick] (-45+8:3.6cm) arc (-45+8:-45+18:3.6cm) -- (-45+18:2.4cm) arc (-45+18:-45+8:2.4cm) -- (-45+8:3.6cm);
\end{scope}

\begin{scope}[shift={(-90:3cm)}]
\draw (0,.475) arc (180-52.34:180+52.34:.6cm);
\end{scope}

\begin{scope}[rotate=-122]
\draw[thick] (-45+5:3.6cm) arc (-45+5:-45-5:3.6cm) -- (-45-5:2.4cm) arc (-45-5:-45+5:2.4cm) -- (-45+5:3.6cm);
\end{scope}

\draw (-180:3.6cm) arc (-180:-180+5:3.6cm) -- (-180+5:2.4cm) arc (-180+5:-180:2.4cm);

\begin{scope}[rotate=-180+10]
\draw[thick,red] (8:3cm) -- (30:3cm);
\end{scope}
\begin{scope}[rotate=-180+55]
\draw[thick,red] (8:3cm) -- (30:3cm);
\end{scope}
\end{tikzpicture}

\begin{tikzpicture}[scale=0.3991,baseline={(current  bounding  box.center)}] 
\node (x) at (0,0) {$\mathbb{X}$};
\draw[dotted] (0,0) -- (-90:4cm);
\draw[dotted] (0,0) -- (-180:4cm);
\draw[dashed] (0,0) -- (-135:4cm);
\begin{scope}[rotate=-90]
\draw (-45+5:3.6cm) arc (-45+5:-45-5:3.6cm) -- (-45-5:2.4cm) arc (-45-5:-45+5:2.4cm) -- (-45+5:3.6cm);
\end{scope}

\begin{scope}[rotate=-90]
\draw (-45+8:3.6cm) arc (-45+8:0:3.6cm) -- (0:2.4cm) arc (0:-45+8:2.4cm) -- (-45+8:3.6cm);
\end{scope}

\begin{scope}[rotate=-122]
\draw (-45+5:3.6cm) arc (-45+5:-45-5:3.6cm) -- (-45-5:2.4cm) arc (-45-5:-45+5:2.4cm) -- (-45+5:3.6cm);
\end{scope}

\draw (-180:3.6cm) arc (-180:-180+5:3.6cm) -- (-180+5:2.4cm) arc (-180+5:-180:2.4cm);

\begin{scope}[rotate=-180+10]
\draw[thick,red] (8:3cm) -- (30:3cm);
\end{scope}

\end{tikzpicture}

\begin{tikzpicture}[scale=0.3991,baseline={(current  bounding  box.center)}] 
\node (x) at (0,0) {$\mathbb{X}$};
\draw[dotted] (0,0) -- (-90:4cm);
\draw[dotted] (0,0) -- (-180:4cm);
\draw[dashed] (0,0) -- (-135:4cm);
\begin{scope}[rotate=-90]
\draw (-45+5:3.6cm) arc (-45+5:-82:3.6cm) -- (-82:2.4cm) arc (-82:-45+5:2.4cm) -- (-45+5:3.6cm);
\end{scope}

\begin{scope}[rotate=-90]
\draw (-45+8:3.6cm) arc (-45+8:0:3.6cm) -- (0:2.4cm) arc (0:-45+8:2.4cm) -- (-45+8:3.6cm);
\end{scope}

\draw (-180:3.6cm) arc (-180:-180+5:3.6cm) -- (-180+5:2.4cm) arc (-180+5:-180:2.4cm);

\end{tikzpicture}
\end{minipage}

\caption{\textbf{Resolution configurations appearing in Lemma \ref{lem:ladybug-compatible}}.  Here are pictured (one sector of) the resolution configurations invariant under the action of $\ZZ_{p'}\subset \ZZ_p$.  The configurations on the top row are $\tilde{D}_{0^4}$, $\tilde{D}_{0100}, \tilde{D}_{0101}$, followed by the row $\tilde{D}_{1000},\tilde{D}_{1100},\tilde{D}_{1101}$ and finally $\tilde{D}_{1010},\tilde{D}_{1110},\tilde{D}_{1^4}$.  We have simplified the indexing by writing the indices for the quotient diagram $\tilde{D}/\ZZ_{p'}$. }\label{fig:jumbotron}
\end{figure}
We now draw the grid $G$ as in the odd case, except that we order the crossings using the ordering of $(\two^4)^{p'}$, rather than $(\two^2)^p$.  That just means that in the above picture, we resolve all edges labeled `1' (\resp 2) before any of those labeled `3' (\resp 4).
The $\ZZ_{p'}$-fixed resolutions look as in Figure \ref{fig:jumbotron}, in one of the $p'$ sectors. 
In the configuration $\tilde{D}_{1010}$, label the inner circle by $x$ and the outer circle by $y$.

Using our inductive hypothesis (and looking at the ladybug matching on $\tilde{D}/\ZZ_{p'}$), the circle $x$ is matched with $z_1\dots z_{p'}$, where $z_i$ are the circles in $\tilde{D}_{1100}$ that intersect (the dotted)  sector boundaries.  A further use of our inductive hypothesis matches $z_1\dots z_{p'}$ with the product $w^1_1\dots w^1_{p'}w^2_1\dots w^2_{p'}$, where the $w^1,w^2$ are as labeled in Figure \ref{fig:jumbotron}.  Note that the generator $w^1_{1}\dots w^2_{p'}$ is indeed $\ZZ_p$-invariant, as are $x$ and $y$.  Taking the quotients of $\tilde{D}_{1010}$ and $\tilde{D}_{0101}$ by $\ZZ_p$, we see that $\bar{x}$, the generator downstairs corresponding to $x$, indeed corresponds, under the right ladybug matching, to $\bar{w}$, the generator downstairs corresponding to the product $w^1_1\dots w^2_{p'}$.  This establishes the inductive step, and completes the Lemma.
\end{proof}

\subsection{Well-definedness of the action}\label{subsec:reidemeister}

In this section we show that, for a $p$-periodic link $\tilde{L}$, (1) the $\ZZ_p$-external stable equivalence class of the Burnside functor $\khburn$ is an invariant of $\tilde{L}$; (2) if $p$ is odd, the external equivariant stable equivalence class of $\khoburn$ is an invariant of $\tilde{L}$; and (3) the corresponding statements for the annular functors $\akhburn$ and $\akhoburn$ hold.

\emph{Proof of Theorem \ref{thm:functor-with-action}.}  Throughout the proof we will usually abbreviate `(equivariant) external stable equivalence class' to `equivalence class,' where it will cause no confusion.  We start with the case of $p$ odd and $\khoburn$.  We must first show that the equivalence class of $\khoburn(\tilde{D})$, for a fixed diagram $\tilde{D}$, is an invariant of the choices made in its construction.  Namely, we show independence of the orientation of crossings, the (equivariant) edge assignment, and the ordering of the circles $a_i$ at each resolution.  The proof of these claims almost follows verbatim from the start of the proof of Theorem 1.7 of \cite{oddkh}.

\begin{itemize}[leftmargin=*]
\item \textbf{Edge assignment:} Let $\epsilon,\epsilon'$ be two
  different equivariant edge assignments of the same type.  As noted in \cite[Lemma 2.2]{ors},
  $\epsilon\epsilon'$ is a (multiplicative) ($\ZZ_p$-invariant) cochain in
  $\cellC^1([0,1]^n;\ZZ_2)$.  By Lemma \ref{lem:eqvar-cube-cat}, $\epsilon\epsilon'$ is the coboundary of an invariant $0$-cochain
  $\alpha$ on the cube of resolutions.  That is, there is a map
  $\alpha\from \two^n\to \{ \pm 1\}$, so that for any $v\geqslant_1w$
  $\alpha(v)\alpha(w)=\epsilon(\phi^\op_{w,v})\epsilon'(\phi^\op_{w,v})$.
  If $F_0$ and $F_1$ are the corresponding functors
  $\two^n\to\oddb$, we construct a stable equivalence using the functor $F_2\from \two^{n+1} \to \burn_{\ZZ_2}$, defined by $F_2|_{i\times \two^n}=F_i$, and on the arrows between the two copies of $\two^n$ using the signed (identity) correspondence $F_1(v) \to F_2(v)$ determined by $\alpha$.  That is, we apply the sign reassignment by $\alpha$ in the language of \cite[Definition 3.5]{oddkh}.  Using the invariance of $\alpha$, we see that $F_2$ admits an external action.  It is straightforward that this natural transformation induces quasi-isomorphisms on the totalization of all fixed-point functors, finishing this check.

\item \textbf{(Equivariant) Orientations at crossings:} Recall that \cite[Lemma 2.3]{ors} asserts that for oriented diagrams $(L,o)$ and $(L,o')$
  and an edge assignment $\epsilon$ for $(L,o)$, there exists an edge
  assignment of the same type $\epsilon'$ for $(L,o')$ so that
  $\oddKhCx(L,o,\epsilon)\cong \oddKhCx(L,o',\epsilon')$.  The  isomorphism constructed in that Lemma respects the Khovanov generators, and so induces an isomorphism of Burnside functors.  The natural generalization to the equivariant setting also holds; that is, for a change of equivariant orientation of crossing, the corresponding odd Khovanov chain complexes are identified (and $\epsilon'$ is equivariant), from which independence of $\khoburn$ follows.  (Independence of the (equivariant) orientations of
  crossings can also be proved using (equivariant) Reidemeister II moves twice, as
  in \cite[Figure~4.5]{SSS-geometric-perturb}.)
  
  \item \textbf{Type of edge assignment:} \cite[Lemma 2.4]{ors} shows
  that an edge assignment $\epsilon$ of a link diagram with oriented crossings
  $(L,o)$ of type $X$ can also be viewed as a type $Y$ edge
  assignment for some orientation $o'$.  That is, the type X
  Burnside functor associated to $(L,o,\epsilon)$ is already the
  type Y Burnside functor associated to
  $(L,o',\epsilon)$.  In fact, if $L$ is a periodic link diagram, the orientation $o'$ constructed in \cite{ors} is equivariant.  Moreover, the identification of the Burnside functors is equivariant, handling this case.
  
  \item \textbf{Ordering of circles at each resolution:} We
  must check that reordering the circles of a resolution results in an equivalent Burnside functor.  For this, let
  $\KhGen(u)$ and $\KhGen'(u)$ denote the Khovanov generators for two
  differing (equivariant) orderings of the circles for a fixed equivariant link diagram.  These
  orderings are related by a bijection from $\KhGen(u)$ to
  $\KhGen'(u)$.  One checks directly that these bijections relate
  the two functors $F_1,F_2\from\two^n\to\oddb$ by a sign reassignment, which, moreover, commutes with the action of  $\ZZ_p$.
\end{itemize}

We now assume that the ordering of the circles upstairs is chosen as at the end of Section \ref{subsec:periodic-links}.
We show how to check invariance of $\khoburn$ under Reidemeister moves by upgrading the proof for chain complexes to Burnside functors, as is done in \cite{lls2}, \cite{lshomotopytype}, with the only change that we keep track of the external action in the course of the proof.  We will work out the details in the case of a Reidemeister I move; this case will make clear what modifications are necessary to the usual invariance proof of $\khoburn$ (without external action) for Reidemeister II and III moves.  Indeed, the proof of invariance is largely an iterated version of the usual invariance proof of Khovanov homology.

\begin{figure}
\centering
\begin{tikzpicture}[baseline={([yshift=-.8ex]current bounding
      box.center)},scale=.5]
\node (x) at (0,0) {$\mathbb{X}$};
\draw[dotted] (0,0) -- (-90:4cm);
\draw[dotted] (0,0) -- (30:4cm);
\draw[dotted] (0,0) -- (150:4cm);
\begin{scope}[shift={(90:3cm)}]

 \node[circle,thick,inner sep=0,outer sep=0,draw,dotted] at (0,0)
    {\begin{tikzpicture}[scale=0.04]  \begin{scope}[rotate=90]\draw[solid] (-2,-2)
        to[out=45,in=-45,looseness=2] (-2,12);\end{scope}\end{tikzpicture}};
\end{scope}
\begin{scope}[shift={(90+120:3cm)}]

 \node[circle,thick,inner sep=0,rotate=90+120,outer sep=0,draw,dotted] at (0,0)
    {\begin{tikzpicture}[scale=0.04]  \draw[solid] (-2,-2)
        to[out=45,in=-45,looseness=2] (-2,12);\end{tikzpicture}};
\end{scope}
\begin{scope}[shift={(90+240:3cm)}]

 \node[circle,thick,inner sep=0,rotate=90+240,outer sep=0,draw,dotted] at (0,0)
    {\begin{tikzpicture}[scale=0.04]  \draw[solid] (-2,-2)
        to[out=45,in=-45,looseness=2] (-2,12);\end{tikzpicture}};
\end{scope}
\node (dd) at (3,3) {$\tilde{D}$};
\end{tikzpicture}\qquad \qquad \qquad
\begin{tikzpicture}[baseline={([yshift=-.8ex]current bounding
      box.center)},scale=.5]
\node (x) at (0,0) {$\mathbb{X}$};
\draw[dotted] (0,0) -- (-90:4cm);
\draw[dotted] (0,0) -- (30:4cm);
\draw[dotted] (0,0) -- (150:4cm);
\begin{scope}[shift={(90:3cm)}]

 \node[circle,thick,rotate=90,inner sep=0,outer sep=0,draw,dotted] at (0,0) {\begin{tikzpicture}[scale=0.04]\draw[solid] (-2,12) to (8,2)
  to[out=-45,in=-90] (12,5) to[out=90,in=45] (8,8); \node[draw=none,crossing] at (5,5) {};
  \draw[solid] (-2,-2) to (8,8);\end{tikzpicture}};
\end{scope}
\begin{scope}[shift={(90+120:3cm)}]
 \node[circle,thick,rotate=90+120,inner sep=0,outer sep=0,draw,dotted] at (0,0) {\begin{tikzpicture}[scale=0.04]\draw[solid] (-2,12) to (8,2)
  to[out=-45,in=-90] (12,5) to[out=90,in=45] (8,8); \node[draw=none,crossing] at (5,5) {};
  \draw[solid] (-2,-2) to (8,8);\end{tikzpicture}};
\end{scope}
\begin{scope}[shift={(90+240:3cm)}]
 \node[circle,thick,rotate=90+240,inner sep=0,outer sep=0,draw,dotted] at (0,0) {\begin{tikzpicture}[scale=0.04]\draw[solid] (-2,12) to (8,2)
  to[out=-45,in=-90] (12,5) to[out=90,in=45] (8,8); \node[draw=none,crossing] at (5,5) {};
  \draw[solid] (-2,-2) to (8,8);\end{tikzpicture}};
\end{scope}
\node (dd) at (3,3) {$\tilde{D}'$};
\end{tikzpicture}
\caption{An equivariant Reidemeister I move.  The left-hand image denotes a periodic link diagram $\tilde{D}$ (with $p=3$ pictured), with a $\ZZ_p$-orbit of a certain unknotted arc in picked out.  The right-hand image denotes the periodic link diagram $\tilde{D}'$ obtained by performing a Reidemeister I move along each arc of the orbit.}
\label{fig:reid-1}
\end{figure}
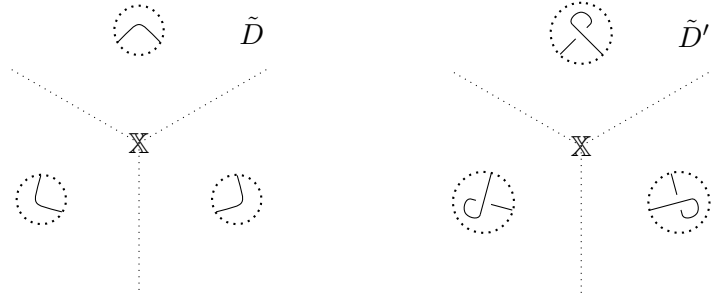

Let $\tilde D$ be a periodic link diagram, and let $\tilde D'$ be a diagram that differs from $\tilde D$ by only an equivariant Reidemeister 1 (R1) move, which consists of $p$ usual Reidemeister moves in the same orbit. See Figure \ref{fig:reid-1}, where we choose one of the R1 moves for concreteness. Let $F_1$ denote the odd Khovanov-Burnside functor of $\tilde{D}$, and $F_2$ that of $\tilde{D}'$. 

From its definition $\KhGen(\tilde{D}')=\amalg_{i\in \two^p} \KhGen(\tilde{D}'_i)$, where $\tilde{D}'_i$ denotes the resolution of $\tilde{D}'$ by resolving the orbit of the R1-crossing according to $i\in \two^p$.  Let $C$ denote the subcomplex spanned by all the generators of $\amalg_{j\neq 0^p} \KhGen(\tilde{D}_j)$ as well as the generators of $\KhGen(\tilde{D}_0)$ that do not contain the product $a_1\dots a_p$, where the $a_i$ are as in Figure \ref{fig:reid-1-part-2}.
Iterating the usual proof \cite[Section 3.5.1]{natancat} of Reidemeister I invariance shows that $C$ is acyclic.  

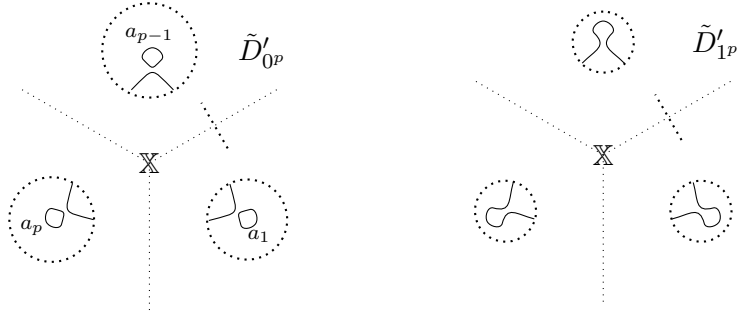
\begin{figure}
\centering
\begin{tikzpicture}[baseline={([yshift=-.8ex]current bounding
      box.center)},scale=.5]
\node (x) at (0,0) {$\mathbb{X}$};
\draw[dotted] (0,0) -- (-90:4cm);
\draw[dotted] (0,0) -- (30:4cm);
\draw[dotted] (0,0) -- (150:4cm);

\begin{scope}[shift={(30:2cm)}]
\draw[thick, dotted] (-60:.7cm) -- (120:.7cm);
\end{scope}
\begin{scope}[shift={(90:3cm)}]
 \node[thick,circle,rotate=90,inner sep=0,outer sep=0,draw,dotted] at (0,0) {\begin{tikzpicture}[scale=0.04]\draw[solid] (-2,-2)
        to[out=45,in=-45,looseness=2] (-2,12);\draw[solid] (8,8) to[out=-135,in=135,looseness=2] (8,2)
  to[out=-45,in=-90] (12,5) to[out=90,in=45] (8,8);
\node[rotate=-90] (a1) at (16,6) {\scriptsize $a_{p-1}$};
\end{tikzpicture}};
\end{scope}
\begin{scope}[shift={(90+120:3cm)}]
 \node[thick,circle,rotate=90+120,inner sep=0,outer sep=0,draw,dotted] at (0,0) {\begin{tikzpicture}[scale=0.04]\draw[solid] (-2,-2)
        to[out=45,in=-45,looseness=2] (-2,12);\draw[solid] (8,8) to[out=-135,in=135,looseness=2] (8,2)
  to[out=-45,in=-90] (12,5) to[out=90,in=45] (8,8);
\node[rotate=-90+240] (a1) at (17,4) {\scriptsize $a_{p}$};\end{tikzpicture}};
\end{scope}
\begin{scope}[shift={(90+240:3cm)}]

  \node[thick,circle,rotate=90+240,inner sep=0,outer sep=0,draw,dotted] at (0,0) {\begin{tikzpicture}[scale=0.04]\draw[solid] (-2,-2)
        to[out=45,in=-45,looseness=2] (-2,12);\draw[solid] (8,8) to[out=-135,in=135,looseness=2] (8,2)
  to[out=-45,in=-90] (12,5) to[out=90,in=45] (8,8);
\node[rotate=-90-240] (a1) at (15,2) {\scriptsize $a_1$};
\end{tikzpicture}};
\end{scope}
\node (dd) at (3,3) {$\tilde{D}'_{0^p}$};
\end{tikzpicture}\qquad \qquad\qquad
\begin{tikzpicture}[baseline={([yshift=-.8ex]current bounding
      box.center)},scale=.5]
\node (x) at (0,0) {$\mathbb{X}$};
\draw[dotted] (0,0) -- (-90:4cm);
\draw[dotted] (0,0) -- (30:4cm);
\draw[dotted] (0,0) -- (150:4cm);

\begin{scope}[shift={(30:2cm)}]
\draw[thick, dotted] (-60:.7cm) -- (120:.7cm);
\end{scope}
\begin{scope}[shift={(90:3cm)}]
 \node[thick,circle,rotate=90,inner sep=0,outer sep=0,draw,dotted] at (0,0) {\begin{tikzpicture}[scale=0.04]\draw[solid] (-2,12) to[out=-45,in=-135,looseness=1.5] (8,8)
  to[out=45,in=90] (12,5) to[out=-90,in=-45] (8,2) to[out=135,in=45,looseness=1.5] (-2,-2);\end{tikzpicture}};
\end{scope}
\begin{scope}[shift={(90+120:3cm)}]
 \node[thick,circle,rotate=90+120,inner sep=0,outer sep=0,draw,dotted] at (0,0) {\begin{tikzpicture}[scale=0.04]\draw[solid] (-2,12) to[out=-45,in=-135,looseness=1.5] (8,8)
  to[out=45,in=90] (12,5) to[out=-90,in=-45] (8,2) to[out=135,in=45,looseness=1.5] (-2,-2);\end{tikzpicture}};
\end{scope}
\begin{scope}[shift={(90+240:3cm)}]
 \node[thick,circle,rotate=90+240,inner sep=0,outer sep=0,draw,dotted] at (0,0) {\begin{tikzpicture}[scale=0.04]\draw[solid] (-2,12) to[out=-45,in=-135,looseness=1.5] (8,8)
  to[out=45,in=90] (12,5) to[out=-90,in=-45] (8,2) to[out=135,in=45,looseness=1.5] (-2,-2);\end{tikzpicture}};
\end{scope}
\node (dd) at (3,3) {$\tilde{D}'_{1^p}$};
\end{tikzpicture}
\caption{Some resolutions of the link diagram $\tilde{D}$.  The ellipses to the upper-right record that we have omitted all but three sectors of the periodic link diagram $\tilde{D}'$.  }
\label{fig:reid-1-part-2}
\end{figure}

Furthermore, $\oddKhCx(\tilde{D})$ is naturally identified with $\oddKhCx(\tilde{D}')/C$.  We have a quotient map
\begin{equation}\label{eq:reide-1-map}
\oddKhCx(\tilde{D}') \to \oddKhCx(\tilde{D}),
\end{equation}
which is a chain homotopy equivalence (because $C$ is acyclic).  This map is induced from a subfunctor inclusion $\khoburn(\tilde{D})\to \khoburn(\tilde{D}')$, in that (\ref{eq:reide-1-map}) is the dual map on totalizations:
\[
\Tot(F_2)^* \to \Tot(F_1)^*.
\]
Here we have used Theorem \ref{thm:oddkh-main} to relate the Khovanov chain complex with the totalizations.  We have a ($\ZZ_2$-equivariant) stable equivalence $F_1 \to F_2$, but we have not yet seen that it is an external equivariant stable equivalence. We must also show that the induced map
\[
\Tot(F_2^{\ZZ_q})^* \to \Tot(F_1^{\ZZ_q})^*
\]
is a homotopy equivalence for each $q>1$ dividing $p$.  For this, let $b_1,\dots, b_{p/q}$ denote the images of the Reidemeister circles $a_i$ in the quotient $\tilde{D}/\ZZ_q$.  Consider the subcomplex $E$ of $\oddAkc(\tilde{D}'/\ZZ_q)$ generated as before by all generators except those of $(\tilde{D}'/\ZZ_q)_{0^{p/q}}$ that contain the product $b_1\dots b_{p/q}$.  As usual, one checks that $E$ is acyclic, and $\oddAkc(\tilde{D}/\ZZ_q)=\oddAkc(\tilde{D}'/\ZZ_q)/E$, so the map  
\begin{equation}\label{eq:annular-1-map}
\oddAkc(\tilde{D}'/\ZZ_q)\to \oddAkc(\tilde{D}/\ZZ_q)
\end{equation}
is a quasi-isomorphism.

Moreover, the subfunctor inclusion $\khoburn(\tilde{D})\to \khoburn(\tilde{D}')$ described above passes to an inclusion on $\ZZ_q$-fixed-point functors $\khoburn(\tilde{D})^{\ZZ_q}\to \khoburn(\tilde{D}')^{\ZZ_q}$.  Using the identification in Theorem \ref{thm:burnside-fixed-pts}, the induced map on totalizations is (\ref{eq:annular-1-map}).  Since we have already seen that (\ref{eq:annular-1-map}) is a quasi-isomorphism, we have proved invariance under Reidemeister I moves.  Keeping track also of the maps induced on even Khovanov homology shows that the inclusion $F_1 \to F_2$ is an equivariant stable equivalence of Burnside functors with external action, as needed.  

Invariance under equivariant Reidemeister II and III is shown in much the same way.  That is, for each acyclic subcomplex or quotient complex `move' in the usual proof of invariance of $\khoburn$, as in \cite[Section 5.3]{oddkh}, one iterates the move $p$ times to produce an acyclic sub- (\resp quotient) complex which is equivariant, and whose quotient (\resp dual subcomplex) is homotopy-equivalent to the original complex.  The sub- (quotient) complexes resulting from fixed-point functors can be understood via Theorem \ref{thm:burnside-fixed-pts}; the induced maps on the totalization of the fixed-point functors give chain homotopy equivalences as well, since they are the usual maps used in the proof of invariance of odd annular Khovanov homology (without external action) from \cite[Section 3.2]{grigsby-wehrli-gl11}.  

The proofs of the even version (for all $p>1$) of the Theorem, as well as the two annular versions, are entirely analogous.  \qed

\emph{Proof of Theorem \ref{thm:submain}.}  Let $\X_n(\tilde{L})$ denote an equivariant realization modeled on $\tilde{\mathbb{R}}^n$, where $\ZZ_p$ acts trivially on $\tilde{\mathbb{R}}^n$, of the stable Burnside functor with external action $\khoburn(\tilde{L})$ and similarly let $\mathcal{AKH}_n(L)$ be the realization of $\akhoburn(L)$ modeled on $\tilde{\mathbb{R}}^n$.  

More generally, say $V$ is a finite-dimensional orthogonal $\ZZ_2\times \ZZ_p$-representation, with $p$ odd, as in the statement of Theorem \ref{thm:submain}.  Write $\mathcal{X}_V(\widetilde{L})$ for an equivariant realization of $\khoburn(\widetilde{L})$ modeled on $V$, and similarly for $\akhoburn$.
The statement that the actions are well-defined is the combination of Proposition \ref{prop:stable-equivalences-realization} with Theorem \ref{thm:functor-with-action}.  The fixed-point assertions follow from Theorem \ref{thm:burnside-fixed-pts} combined with Lemma \ref{lem:sub-burn-func}.  The gradings can be recovered from Proposition \ref{prop:identify-equiv-gens}.  
\qed

\subsection{Smith inequalities}\label{subsec:smith}
\label{subsec:rk-ineqs}

We now use the results on fixed-point functors from Section \ref{subsec:fixed-point} to obtain rank inequalities for Khovanov homology. Let $p$ be prime, and $G = \ZZ_p$.

Recall that the classical Smith inequality (\ref{eq:1}) for a finite $G$-CW complex $M$ is obtained by studying two spectral sequences arising from the \emph{Tate bicomplex}:
\begin{align*}
C^{\tate}(M) 
&= (C^*(M;\f_p) \otimes \f_p[\theta, \theta\inv], d^{\tate}) \\
&:=( \cdots \map{1-\psi} C^*(M; \f_p)
			   \map{N(\psi)} C^*(M;\f_p)
			   \map{1-\psi} C^*(M; \f_p)
			   \map{N(\psi)} \cdots),
\end{align*}
where $\psi$ generates the $G$-action on singular cochains $C^*(M;\f_p)$  and $N(\psi)$ is the norm $1 + \psi + \psi^2 + \ldots + \psi^{p-1}$. 
The filtration by $\theta$-degree gives a spectral sequence $E^\bullet$ with $E^1 \cong H^*(M;\f_p) \otimes \f_p[\theta, \theta\inv]$ while the filtration by cohomological degree gives a spectral sequence converging to $H^*(M^G;\f_p)\otimes \f_p[\theta, \theta\inv]$. The assumptions provide sufficient boundedness to conclude that $E^\bullet$ also converges to $H^*(M^G;\f_p)\otimes \f_p[\theta, \theta\inv]$, and the rank inequality follows. (For a more detailed exposition, see  \cite{lipshitz-treumann}, \cite{z-annular-rank}.)

\begin{thm}\label{thm:tate-applied}
For a $p$-periodic link $\tilde{L}$ for prime $p$ (\resp odd prime $p$), with quotient link $L$, and each pair of quantum and $(k)$-gradings $(j,k)$, there is a spectral sequence with
\[E^1 \cong \Akh^{pj - (p-1)k, k}(\tilde L;\f_p)\otimes \f_p[\theta,\theta^{-1}] \qquad (\text{\resp} \oddAkh^{pj - (p-1)k, k}(\tilde L;\f_p)\otimes \f_p[\theta,\theta^{-1}])\]
converging to 
\[ E^\infty \cong \Akh^{j,k}(L;\f_p)\otimes \f_p[\theta,\theta^{-1}] \qquad (\text{\resp} \oddAkh^{j,k}(L;\f_p)\otimes \f_p[\theta,\theta^{-1}]).\]  There is also a spectral sequence with 
\[E^1 \cong \Kh(\tilde L;\f_p)\otimes \f_p[\theta,\theta^{-1}] \qquad (\text{\resp} \kho(\tilde L;\f_p)\otimes \f_p[\theta,\theta^{-1}]) \]
converging to 
\[E^\infty \cong \Akh(L;\f_p)\otimes \f_p[\theta,\theta^{-1}] \qquad (\text{\resp}  \oddAkh(L;\f_p))\otimes \f_p[\theta,\theta^{-1}]).\]

\begin{proof}
First, consider the case of $p$ odd, and odd annular Khovanov homology.  Construct the Tate bicomplex for $\khon(\tilde{L})$ for odd $n$ (here, $\khon(\tilde{L})$ is viewed as a space, without passing to the suspension spectrum).  The $\theta$-degree filtration gives a spectral sequence with the desired $E^1$-page which converges to the homology of the fixed-point set $\khon(\tilde{L})^G$, which by Theorem \ref{thm:submain} is $\Akhspace_n(L)$.  Now, for the case of even Khovanov homology, repeat the above recipe with $n=0$.  

The proof for the spectral sequences starting in the annular case is entirely analogous.   Finally, for the gradings, note that the spectral sequence splits according to the wedge sum components in the CW-realizations.
\end{proof}
\end{thm}

\begin{cor}
\label{cor:cascade}
Maintain the notation from Theorem \ref{thm:tate-applied}.  For each pair of quantum and $(k)$-gradings $(j,k)$, the following rank inequalities hold (for vector spaces over $\f_p$):
\[
\dim \Akh^{pj - (p-1)k, k}(\tilde L;\f_p) \geq  \dim \Akh^{j,k}(L;\f_p)  \mbox{ and }
\]
\[
\dim \oddAkh^{pj - (p-1)k, k}(\tilde L;\f_p) \geq  \dim \oddAkh^{j,k}(L;\f_p).
\]
We also have the rank inequalities (where each object is the sum over all quantum and $(k)$-gradings):
\begin{align*}
& \dim \Akh(\tilde L;\f_p) 
	\geq \dim \Kh(\tilde L;\f_p)
	\geq \dim \Akh(L;\f_p)
	\geq \dim \Kh(L;\f_p) \mbox{ and }\\
& \dim \oddAkh(\tilde L;\f_p) 
	\geq \dim \kho(\tilde L;\f_p)
	\geq \dim \oddAkh(L;\f_p)
	\geq \dim \kho(L;\f_p).
\end{align*}
\end{cor}
\begin{proof}
The $\Akh$-to-$\Kh$ inequalities follow from the filtration of the Khovanov complex \cite{roberts-dbc}. The middle inequalities follow from Theorem \ref{thm:tate-applied}.
\end{proof}

\subsection{Questions}\label{subsec:questions}
We conclude with some questions about the construction of equivariant Khovanov spaces.  Fix throughout a $p$-periodic link $\tilde{L}$ with quotient $L$.

\begin{enumerate}[leftmargin=*,label=(q-\arabic*)]
\item We have not attempted to relate the totalization $\Tot(\khoburn)$ (that is, the equivariant odd/even Khovanov complex) with any particular CW chain complex of $\X_n(\tilde{L})$, viewed as a $\ZZ_p$-equivariant space.  This would be useful to understand in order to relate the constructions of this paper with the equivariant Khovanov homology (or an odd version of same) constructed by Politarczyk \cite{politarczyk-kh}.  In more generality, it would be desirable to better understand a $\ZZ_p$-equivariant cell decomposition of $\X_n(\tilde{L})$, so that, for example, the space $\X_0(\tilde{L})$ could be related to the space constructed in \cite{bps-2018}.

\item A better understanding of the case of even $p$ for the odd Khovanov-Burnside functor $\khoburn$ would be desirable.  In particular, the techniques of this paper are sufficient to show that for a given periodic diagram $\tilde{D}$ of $\tilde{L}$, the functor $\khoburn(\tilde{D})$ admits a $\ZZ_p$-external action.  However, it is not immediately clear that this action is a link invariant. Moreover, the resulting external action need not be nonsingular.  It is not clear to the authors whether (for $n\geq 1$) Theorem \ref{thm:submain} (including the statement about fixed points) also holds for even $p$; we do not know of a counterexample.

\item Are there applications of the construction of the present paper to showing that some links are not periodic?  Borodzik-Politarczyk-Silvero \cite{bps-2018} have obtained such applications; are there further applications that require using the odd theory?

\item Willis \cite{willis-stabilization} showed that the Khovanov homotopy type of torus links $T(n,m)$ stabilizes as $m\to \infty$.  How does this stabilization interact with the $\ZZ_m$-action?  

\end{enumerate}

\bibliographystyle{myalpha}
\bibliography{localization}
\vspace{0.3in}

\end{document}